\documentclass[12pt]{amsart}
\usepackage[utf8]{inputenc}
\usepackage{format}

\newcommand{\nc}{\newcommand}
\newcommand{\CO}{\mathcal{O}}
\nc{\on}{\operatorname}
\nc{\ra}{\rightarrow}
\nc\iso{\,\vphantom{j^{X^2}}\smash{\overset{\sim}{\vphantom{\rule{0pt}{0.20em}}\smash{\longrightarrow}}}\,}
\nc{\al}{\alpha}
\nc{\la}{\lambda}
\nc{\CA}{\mathcal{A}}
\nc{\BZ}{\mathbb{Z}}
\nc{\CU}{\mathcal{U}}

\title[Hikita conjecture for nilpotent orbits and parabolic Slodowy varieties]{Around Hikita-Nakajima conjecture for nilpotent orbits and parabolic Slodowy varieties}

\begin{document}

\begin{abstract}
Let $G$ be a complex reductive algebraic group. In \cite{LMBM} Ivan Losev, Lucas Mason-Brown and the third-named author suggested a symplectic duality between nilpotent Slodowy slices in $\fg^\vee$ and affinizations of certain $G$-equivariant covers of special nilpotent orbits. In this paper, we study the various versions of Hikita conjecture for this pair. We show that the original statement of the conjecture does not hold for the pairs in question and propose a refined version. We discuss the general approach towards the proof of the refined Hikita conjecture 
and prove this refined version for the parabolic Slodowy varieties, which includes many of the cases considered in \cite{LMBM} and more. Applied to the setting of \cite{LMBM}, the refined Hikita conjecture explains the importance of special unipotent ideals from the symplectic duality point of view.

We also discuss applications of our results. 
In the appendices, we discuss some classical questions in Lie theory that relate the refined version and the original version. We also explain how one can use our results to simplify some proofs of known results in the literature. As a combinatorial application of our results we observe an interesting relation between the geometry of Springer fibers and left Kazhdan-Lusztig cells in the corresponding Weyl group.
\end{abstract}

\maketitle

\section{Introduction}
\subsection{Symplectic duality}
    The main and most general objects of interest in this paper are conical symplectic singularities. We give the exact definition and the most important properties in \cref{conical sympl sing}, but for now we want to focus on an important special class of examples. For a connected reductive group $G$ and a finite-dimensional representation ${\bf{N}}$ of $G$ Braverman, Finkelberg and Nakajima in \cite{nakajima_coulombI}, \cite{BFNII} gave a mathematical definition of the Coulomb branch of cotangent type $\cM_C(G, {\bf{N}})$ of a $3$-dimensional $\cN=4$ supersymmetric gauge theory attached to the pair $(G,{\bf{N}})$. In \cite{bellamy} Bellamy proved that $\cM_C(G, {\bf{N}})$ has symplectic singularities (see also \cite{weekes_sympl_sing}). In \cite{Coulomb_noncotangent} the mathematical definition of a Coulomb branch was extended to the representations of non-cotangent type (under some ``anomaly cancelation'' condition), and the authors conjectured that the resulting varieties have symplectic singularities. In many cases it is known to be true, which leads to a big class of conical symplectic singularities coming as Coulomb branches of $3$-dimensional $\cN=4$ supersymmetric gauge theories. 

    A $3$-dimensional $\cN=4$ supersymmetric gauge theory also comes with its Higgs branch, in many cases coinciding with the hyper-K{\"a}hler quotient $N\ham G$. In \cite{Intriligator1996} Intrilligator and Seiberg suggested the ``mirror duality" between $3$-dimensional $\cN=4$ supersymmetric gauge theories that exchange Higgs and Coulomb branches. Thus, in many cases, the Higgs and the Coulomb branches of a $3$-dimensional $\cN=4$ supersymmetric gauge theory are conical symplectic singularities, and it is natural to expect that their (different) structural properties are governed by the same properties of the gauge theory, and thus are interchanged between each other. 

    This idea leads to the notion of a symplectic duality, formulated in \cite{BPW}, \cite{BPWII} (interestingly, the authors of these papers came up with the idea of symplectic duality independently of the works of physicists). Namely, we expect to have many pairs of conical symplectic singularities $(X^\vee, X)$ that have some of their structural properties interchanged, see \cref{conical sympl sing} for details. Roughly, one of the expectations is that the Lie algebra of a maximal torus $T_X$ acting on  $X$ is interchanged with the space of deformations of $X^\vee$ and vice versa. One also expects that in cases when both $X$ and $X^\vee$ have symplectic resolutions $Y$, $Y^\vee$, there is a (canonical) bijection $Y^{T_X} \simeq (Y^\vee)^{T_{Y^\vee}}$ and both of these sets are {\emph{finite}}.

    It is important to make two remarks. First, the Higgs branch of a $3$-dimensional $\cN=4$ supersymmetric gauge theory may not be normal (and therefore cannot be a conical symplectic singularity), and the Coulomb branch may not be conical. Moreover, even in the case of quiver gauge theories where we can define and compute Higgs and Coulomb branches, assuming they are both conical symplectic singularities we cannot prove that the properties of symplectic duality hold, although it is expected to be true.

    Second, the definition of symplectic duality makes sense without mentioning $3$-dimensional $\cN=4$ supersymmetric gauge theories. Many examples of expected symplectic dual $(X^\vee, X)$ do not have a known gauge theory that produces $X$ and $X^\vee$ as Higgs and Coulomb branches respectively. Crucially, this includes most of the examples considered in this text. Constructing such a theory is an extremely interesting problem that is, however, not discussed in this paper. In work in progress, Finkelberg, Hanany, and Nakajima realize some but not all of the examples considered in our paper as Higgs and Coulomb branches of orthosymplectic quiver gauge theories, and we are grateful to the authors for explaining their results to us. 
    
\subsection{Parabolic Slodowy slices}
    In this paper we focus on some specific cases of $(X^\vee, X)$ that are expected to be symplectic dual. Let $G$ be a complex semisimple Lie group, $G^\vee$ be the Langlands dual of $G$, and $\fg$, $\fg^\vee$ be the corresponding Lie algebras. Let $e^\vee\in \fg^\vee$ be a nilpotent element, and $\chi^\vee=(e^\vee, \bullet)\in (\fg^\vee)^*$. Let $S(e^\vee) \subset \mathfrak{g}$ be the Slodowy slice to $e^\vee$, and $S(\chi^\vee)\subset (\fg^\vee)^*$ be its image under the identification $\fg^\vee\simeq (\fg^\vee)^*$ given by the Killing form. Consider the intersection $X^\vee=S(\chi^\vee)\cap \cN^\vee$ of $S(\chi^\vee)$ with the nilpotent cone. This is a conical symplectic singularity that admits a symplectic resolution. In \cite[Section 9.3]{LMBM} Losev, Mason-Brown and the third-named author suggest a candidate for a symplectic dual $X$, we recall it in \cref{sect_refined_BVLS}. It generalizes a special case that appeared in the literature before, see, e.g., \cite[Section 10.2.2]{BPWII}. Let $\fl^\vee\subset \fg^\vee$ be a Levi subalgebra containing $e^\vee$, and assume that $e^\vee$ is a regular nilpotent element in $\fl^\vee$. Let $\fl\subset \fg$ be the Langlands dual of $\fl^\vee$, $L\subset G$ the corresponding subgroup, and choose a parabolic $P\subset G$ with Levi factor $L$. Then $X=\Spec(\CC[T^*(G/P)])$. We note that it admits a symplectic resolution by $T^* (G/P)$. Then $X^\vee$ and $X$ are symplectic dual to each other. 

    An important generalization of the last example is as follows. Let $\fp, \fq \subset \fg$ be two different parabolic subalgebras, and $\fl, \fm \subset \fg$ be the corresponding Levi subalgebras. Consider the generalized Springer resolution map $\rho\colon T^*(G/Q)\to \fg^*$. Pick a regular nilpotent element $e\in \fl$. We define $\widetilde{S}(\fp, \fq)=\rho^{-1}(S(\chi))$, and let $X=\Spec(\CC[\widetilde{S}(\fp, \fq)])$ to be also denoted by $S(\mathfrak{p},\mathfrak{q})$.
    We have the natural bijection between roots of $\mathfrak{g}$ and roots of $\mathfrak{g}^\vee$. It follows that parabolic subalgebras $\mathfrak{p}, \mathfrak{q} \subset \mathfrak{g}$ determine the parabolic subalgebras $\mathfrak{p}^\vee, \mathfrak{q}^\vee \subset \mathfrak{g}^\vee$ with Levi subalgebras $\mathfrak{l}^\vee$ and $\mathfrak{m}^\vee$ respectively.
    We expect that $X$ is symplectic dual to $X^\vee=S(\fq^\vee, \fp^\vee)$. 
    The pair $(X^\vee, X)$ is the main object of our interest. 

    Let $\mathfrak{X}(\mathfrak{l})$ be the space of characters of $\mathfrak{l}$. Let us mention that the morphism $\widetilde{S}(\mathfrak{p},\mathfrak{q}) \rightarrow S(\mathfrak{p},\mathfrak{q})$ deforms naturally over $\mathfrak{X}(\mathfrak{l})$, we will denote this deformation by $\widetilde{S}_{\mathfrak{X}(\mathfrak{l})}(\mathfrak{p},\mathfrak{q}) \rightarrow S_{\mathfrak{X}(\mathfrak{l})}(\mathfrak{p},\mathfrak{q})$.

\subsection{Hikita-Nakajima conjectures}    We are interested in (different versions of) the Hikita-Nakajima conjecture for $(X^\vee, X)$ as above. We explore the general setting in \cref{section Hikita general}, here we focus on the main object of interest. Let $T\subset G$ be a maximal torus. The Hikita conjecture \cite{Hikita} claims the following.
    \begin{conj}\label{Hikita_intro}
        There is a graded algebra isomorphism $H^*(\widetilde{S}(\fp^\vee, \fq^\vee))\simeq \CC[S(\mathfrak{q},\mathfrak{p})^T]$,
    \end{conj}
    where $S(\mathfrak{q},\mathfrak{p})^T$ are \emph{schematic} fixed points.
    In this paper, we consider a more general version of the conjecture, that we call the equivariant Hikita-Nakajima conjecture. Let $Z_{L^\vee}$ be the center of $L^\vee$, and $\CC^\times$ be the contracting torus acting on $X^\vee$. Then we can consider a $\CC[\fX(\fl), \hbar]$-algebra $H_{Z_{L^\vee}\times \CC^\times}^* (\widetilde{S}(\fp^\vee, \fq^\vee))$. On the dual side $X$ admits a deformation $S_{\mathfrak{X}(\mathfrak{l})}(\mathfrak{p},\mathfrak{q})$ over $\fX(\fl)$, and let $\cA_{\fX(\fl), \hbar}(X)$ be its  quantization over $\mathfrak{X}(\mathfrak{l})$. Pick a generic cocharacter $\nu\colon \CC^\times \to T$, and let $\cC_{\nu}(\cA_{\fX(\fl), \hbar}(X))$ be the Cartan subquotient. Note that it is also a $\CC[\fX(\fl), \hbar]$-algebra. 
    \begin{conj}\label{equiv_HN_intro}
        There is a graded isomorphism of $\CC[\fX(\fl), \hbar]$-algebras $H_{Z_{L^\vee}\times \CC^\times}^* (\widetilde{S}(\fp^\vee, \fq^\vee))\simeq \cC_{\nu}(\cA_{\fX(\fl, \hbar)}(X))$.
    \end{conj}
    Note that specializing to the point $(0,0)\in \mathfrak{X}(\fl)\oplus \CC$ \cref{equiv_HN_intro} recovers \cref{Hikita_intro}. Specializing to $\hbar=0$,  \cref{Hikita_intro} recovers the following conjecture (to be called Hikita-Nakajima conjecture).
    \begin{conj}\label{HN_intro}
    There is a graded isomorphism of $\mathbb{C}[\mathfrak{X}(\mathfrak{l})]$-algebras $H^*_{Z_{L^\vee}}(\widetilde{S}(\mathfrak{p}^\vee,\mathfrak{q}^\vee)) \simeq \mathbb{C}[S(\mathfrak{q},\mathfrak{p})^{\nu(\mathbb{C}^\times)}]$.
    \end{conj}
    
    We finish this section by listing the known results of the Hikita and (equivariant) Hikita-Nakajima conjectures. 
    
    The Hikita conjecture was proven for the case of Hilbert scheme of points on $\mathbb{A}^2$, parabolic type $A$ Slodowy slices and hypertoric varieties in \cite[Theorem 1.1, Theorem A.1, Theorem B.1]{Hikita}. In \cite{Shlykov2020}, Shlykov has proven the case of $\mathbb{A}^2/\Gamma$ where $\Gamma$ is a finite subgroup of $\on{SL}_2(\mathbb{C})$.  
    In \cite[Theorem 1.5]{KamnitzerTingleyWebsterWeeksYacobi}, Kamnitzer, Tingley, Webster, Weekes and Yacobi have proven Hikita conjecture for the ADE slices in the affine Grassmanian. 
    They have also proven the equivariant version for type A quivers. Moreover,  the equivariant Hikita-Nakajima conjecture in the case of ADE slices in the affine Grassmanian that admit a symplectic resolution can be deduced from the results of Kamnitzer, Tingley, Webster, Weekes and Yacobi (see \cite[Section 8.3]{KamnitzerTingleyWebsterWeeksYacobi}, \cite[Section 1.3 and Theorem 1.5]{catO_yang}, see also \cite[Section 6.6]{Kamnitzer_symplectic}). The results of \cite{Shlykov2020} were recently generalized to the equivariant and quantum case in \cite{chen2024}. In \cite{linus_hikita} Setiabrata studies Hikita conjecture for $\mathcal{N}_{\mathfrak{sl}_n}/\!/\!/T$. In \cite{ksh}, Shlykov and the second author proved the Hikita-Nakajima conjecture for ADHM spaces (also known as Gieseker varieties) and initiated the approach towards the proof of Hikita-Nakajima conjecture that we will be using in this paper.

\subsection{Main results}
    Below, we list the main results obtained in this paper and discuss the structure of the text. The first important new result is obtained in \cref{sect_counter}. Namely, we claim the following.
    \begin{prop}\label{counterexample_intro}
        The Conjectures \ref{Hikita_intro} and \ref{equiv_HN_intro} do not hold as stated for the pair 
        \begin{equation*}
        X^\vee=\Spec(\CC[\widetilde{S}(\fp^\vee, \fq^\vee)]), ~X=\Spec(\CC[\widetilde{S}(\fq, \fp)]).
        \end{equation*}
    \end{prop}
    In fact we construct explicit counterexamples when $\fq=\fb$, i.e. $X=\Spec(\CC[T^*(G/P)])$. For $e^\vee\in \fl^\vee$ a regular  nilpotent let $\widetilde{S}(\chi^\vee)=\widetilde{S}(\fp^\vee, \fb^\vee)$. The key issues with Conjectures \ref{Hikita_intro}, \ref{equiv_HN_intro} are as follows, see \cref{sect_counter} for details.
    \begin{itemize}
        \item The algebra $H_{Z_{L^\vee}\times \CC^\times}^* (\widetilde{S}(\chi^\vee))$ is free over $\CC[\fX(\fl), \hbar]$, while the algebra $\cC_{\nu}(\cA_{\fX(\fl), \hbar}(X))$ may have $\CC[\fX(\fl), \hbar]$-torsion.
        \item In many cases the algebra $\CC[X^T]$ is cyclic indecomposable over $\CC[\fh^*]$ while $H^*(\widetilde{S}(\chi^\vee))$ is not indecomposable in general.
    \end{itemize}

    To remedy these issues in \cref{section Hikita general}, we propose \cref{conj_most_general_weak_Hikita} as a replacement of the equivariant Hikita-Nakajima conjecture in a general setting. For the pair 
    \begin{equation*}    X^\vee=\Spec(\CC[\widetilde{S}(\fp^\vee, \fq^\vee)]),~X=\Spec(\CC[\widetilde{S}(\fq, \fp)])
    \end{equation*}
    it leads to \cref{main_prop_weak_hikita} and \cref{main_th_weak_hikita}, that are the main results of the paper. Let us start with the case $\hbar=0$.

\subsubsection{Refined Hikita-Nakajima for $\hbar=0$}    For $\hbar=0$, we are comparing the following algebras:
    \begin{equation*}
    H^*_{Z_{L^\vee}}(\widetilde{S}(\mathfrak{p}^\vee,\mathfrak{q}^\vee))~\text{vs}~\mathbb{C}[S_{\mathfrak{X}(\mathfrak{l})}(\mathfrak{q},\mathfrak{p})^{\nu(\mathbb{C}^\times)}].
    \end{equation*}
    Note that we have the natural embedding \begin{equation}\label{equation_restr_to_fixed_parabolic}
    H^*_{Z_{L^\vee}}(\widetilde{S}(\mathfrak{p}^\vee,\mathfrak{q}^\vee)) \hookrightarrow H^*_{Z_{L^\vee}}(\widetilde{S}(\mathfrak{p}^\vee,\mathfrak{q}^\vee)^{Z_{L^\vee}})
    \end{equation}
    and the latter algebra is just the direct sum of polynomial algebras labeled by the (finite) set of fixed points $\widetilde{S}(\mathfrak{p}^\vee_-,\mathfrak{q}^\vee_-)^{Z_{L^\vee}}$ that we describe explicitly in Appendix \ref{app_torus_fixed}. By localization theorem, the embedding (\ref{equation_restr_to_fixed_parabolic}) is an {\emph{isomorphism generically}} over $H^*_{Z_{L^\vee}}(\on{pt})$. Moreover, we have a natural restriction homomorphism 
    \begin{equation}\label{eq_restr_partial_to_slodowy}
    H^*_{Z_{L^\vee}}(T^*\mathcal{Q}^\vee) \rightarrow  H^*_{Z_{L^\vee}}(\widetilde{S}(\mathfrak{p}^\vee,\mathfrak{q}^\vee))
    \end{equation}
    which is  {\emph{generically surjective}}. We propose to replace the algebra $H^*_{Z_{L^\vee}}(\widetilde{S}(\mathfrak{p}^\vee,\mathfrak{q}^\vee))$ appearing in the Hikita-Nakajima conjecture by the {\emph{image}} 
    \begin{equation}\label{eq_prop_replacement_LHS_hbar_zero}
    \on{Im}(H^*_{Z_{L^\vee}}(T^*\mathcal{Q}^\vee) 
    \rightarrow H^*_{Z_{L^\vee}}(\widetilde{S}(\mathfrak{p}^\vee,\mathfrak{q}^\vee)).
    \end{equation}
    Note that the fiber of the image (\ref{eq_prop_replacement_LHS_hbar_zero}) over a {\emph{generic}} point $\la \in \on{Spec}H^*_{Z_{L^\vee}}(\on{pt})$ is {\emph{equal}} to the fiber of $H^*_{Z_{L^\vee}}(\widetilde{S}(\mathfrak{p}^\vee,\mathfrak{q}^\vee))$ at $\lambda$ (use generic surjectivity of (\ref{eq_restr_partial_to_slodowy})). So, considered as a coherent sheaf on $\on{Spec}H^*_{Z_{L^\vee}}(\on{pt})$, it has rank equal to the number of fixed points $\widetilde{S}(\mathfrak{p}^\vee,\mathfrak{q}^\vee)^{Z_{L^\vee}}$. It then follows from Nakayama lemma that the fibers of the image (\ref{eq_prop_replacement_LHS_hbar_zero}) over arbitrary $\la$ have dimension {\emph{at least}} $|\widetilde{S}(\mathfrak{p}^\vee,\mathfrak{q}^\vee)^{Z_{L^\vee}}|$. So, our construction does not replace $H^*_{Z_{L^\vee}}(\widetilde{S}(\mathfrak{p}^\vee,\mathfrak{q}^\vee))$ by some  ``smaller'' sheaf of algebras, it rather modifies the sheaf $H^*_{Z_{L^\vee}}(\widetilde{S}(\mathfrak{p}^\vee,\mathfrak{q}^\vee))$ at the special points $\lambda \in \on{Spec}H^*_{Z_{L^\vee}}(\on{pt})$.
    Note, in particular, it is {\emph{not}} true in general that the fiber of the image (\ref{eq_prop_replacement_LHS_hbar_zero}) at $\lambda$ is equal to the image $H^*_{Z_{L^\vee}}(T^*\mathcal{Q}^\vee)_\la 
    \rightarrow H^*_{Z_{L^\vee}}(\widetilde{S}(\mathfrak{p}^\vee,\mathfrak{q}^\vee)_{\la}$. Of course, we have a surjective map in one direction, but it might {\emph{not}} be an isomorphism. In other words, there are elements in the fiber of (\ref{eq_prop_replacement_LHS_hbar_zero}) that can  {\emph{not}} be obtained as images of the elements of $H^*_{Z_{L^\vee}}(T^*\mathcal{Q}^\vee)_{\la}$. These elements should rather  be thought as ``limits'' of the images of the elements of $H^*_{Z_{L^\vee}}(T^*\mathcal{Q}^\vee)_{\la'}$ for generic $\la'$  as $\la'$ goes to $\la$.

    So, we replace the LHS of the Hikita-Nakajima conjecture with the image (\ref{eq_prop_replacement_LHS_hbar_zero}). Let us now describe what we do with the RHS. 

    The analog of the map (\ref{eq_restr_partial_to_slodowy}) for the RHS is the pull-back homomorphism
    \begin{equation}\label{eq_restr_slod_to_part_slod}
    \mathbb{C}[S_{\mathfrak{X}(\mathfrak{l})}(\mathfrak{b},\mathfrak{p})^{\nu(\mathbb{C}^\times)}] \rightarrow \mathbb{C}[S_{\mathfrak{X}(\mathfrak{l})}(\mathfrak{q},\mathfrak{p})^{\nu(\mathbb{C}^\times)}]
    \end{equation}
    induced by the natural morphism 
   $
    S_{\mathfrak{X}(\mathfrak{l})}(\mathfrak{q},\mathfrak{p}) \rightarrow  S_{\mathfrak{X}(\mathfrak{l})}(\mathfrak{b},\mathfrak{p})$
    (recall that $S_{\mathfrak{X}(\mathfrak{l})}(\mathfrak{q},\mathfrak{p})$ is the affinization of $\widetilde{S}_{\mathfrak{X}(\mathfrak{l})}(\mathfrak{q},\mathfrak{p})$, which, by the definition, is the preimage of $S_{\mathfrak{X}(\mathfrak{l})}(\mathfrak{b},\mathfrak{p})$ under the natural morphism $T^*\mathcal{P} \rightarrow \mathfrak{g}^* \times_{\mathfrak{h}^*/W} \mathfrak{X}(\mathfrak{l})$). One can show that the sources in (\ref{eq_restr_partial_to_slodowy}), (\ref{eq_restr_slod_to_part_slod}) are isomorphic, i.e., the Hikita-Nakajima conjecture holds for the pair $(X^\vee=T^*\mathcal{Q}^\vee,\,X=\widetilde{S}(\chi))$. So, (\ref{eq_restr_slod_to_part_slod}) is indeed a reasonable analog of (\ref{eq_restr_partial_to_slodowy}). 
    
 One can show that the morphism  (\ref{eq_restr_slod_to_part_slod}) is generically surjective, so the natural guess  for the replacement of the RHS of the Hikita-Nakajima conjecture would be taking the image of (\ref{eq_restr_slod_to_part_slod}). This is {\emph{almost}} the case. Namely, the algebra $\mathbb{C}[S_{\mathfrak{X}(\mathfrak{l})}(\mathfrak{q},\mathfrak{p})^{\nu(\mathbb{C}^\times)}]$ is {\emph{not}} torsion free over $\mathbb{C}[\mathfrak{X}(\mathfrak{l})]$ in general but the algebra $H^*_{Z_{L^\vee}}(\widetilde{S}(\mathfrak{p}^\vee,\mathfrak{q}^\vee))$ is. To remedy this inconsistency, we replace $\mathbb{C}[S_{\mathfrak{X}(\mathfrak{l})}(\mathfrak{q},\mathfrak{p})^{\nu(\mathbb{C}^\times)}]$  by its image under the natural pullback morphism 
    \begin{equation}\label{equation_pullback_to_resol}
     \mathbb{C}[S_{\mathfrak{X}(\mathfrak{l})}(\mathfrak{q},\mathfrak{p})^{\nu(\mathbb{C}^\times)}] \rightarrow  \mathbb{C}[\widetilde{S}_{\mathfrak{X}(\mathfrak{l})}(\mathfrak{q},\mathfrak{p})^{\nu(\mathbb{C}^\times)}].
    \end{equation}
    This morphism should be compared with the (injective) morphism (\ref{equation_restr_to_fixed_parabolic}). Note that the morphism (\ref{equation_restr_to_fixed_parabolic}) is an {\emph{isomorphism generically}}. Summarizing, our replacement for the RHS of the Hikita-Nakajima conjecture is the image:
    \begin{equation*}
    \on{Im}( \mathbb{C}[S_{\mathfrak{X}(\mathfrak{l})}(\mathfrak{b},\mathfrak{p})^{\nu(\mathbb{C}^\times)}]  \rightarrow \mathbb{C}[S_{\mathfrak{X}(\mathfrak{l})}(\mathfrak{q},\mathfrak{p})^{\nu(\mathbb{C}^\times)}] \rightarrow  \mathbb{C}[\widetilde{S}_{\mathfrak{X}(\mathfrak{l})}(\mathfrak{q},\mathfrak{p})^{\nu(\mathbb{C}^\times)}]).
    \end{equation*}
    As before, this image is a natural modification of the original algebra $\mathbb{C}[S_{\mathfrak{X}(\mathfrak{l})}(\mathfrak{q},\mathfrak{p})^{\nu(\mathbb{C}^\times)}]$ coinciding with it at the generic point.

    So, for $\hbar=0$, our main result is the following.
    \begin{theorem}\label{intro_thm_classical}
    There is an isomorphism of graded $\mathbb{C}[\mathfrak{X}(\mathfrak{l})]$-algebras:
    \begin{multline*}
        \on{Im}(H^*_{Z_{L^\vee}}(T^*\mathcal{Q}^\vee) 
    \rightarrow H^*_{Z_{L^\vee}}(\widetilde{S}(\mathfrak{p}^\vee,\mathfrak{q}^\vee))) \simeq \\
    \simeq   \on{Im}( \mathbb{C}[S_{\mathfrak{X}(\mathfrak{l})}(\mathfrak{b},\mathfrak{p})^{\nu(\mathbb{C}^\times)}]  \rightarrow \mathbb{C}[S_{\mathfrak{X}(\mathfrak{l})}(\mathfrak{q},\mathfrak{p})^{\nu(\mathbb{C}^\times)}] \rightarrow  \mathbb{C}[\widetilde{S}_{\mathfrak{X}(\mathfrak{l})}(\mathfrak{q},\mathfrak{p})^{\nu(\mathbb{C}^\times)}]).
    \end{multline*}
    \end{theorem}


  For the statement of the main theorem and details we refer to the main text, here we limit ourselves to expressing the main idea and illustrating it for $\hbar=0$. 
    
\subsubsection{Refined equivariant Hikita-Nakakjima}
The general case is similar to the $\hbar=0$ situation. The LHS of \cref{intro_thm_classical} is replaced by $\on{Im}(H^*_{Z_{L^\vee} \times \mathbb{C}^{\times}}(T^*\mathcal{Q}^\vee) \rightarrow H^*_{Z_{L^\vee} \times \mathbb{C}^\times}(\widetilde{S}(\mathfrak{p}^\vee,\mathfrak{q}^\vee)))$. The replacement of the RHS can be described as follows. First of all, the map $S_{\mathfrak{X}(\mathfrak{l})}(\mathfrak{q},\mathfrak{p}) \rightarrow  S_{\mathfrak{X}(\mathfrak{l})}(\mathfrak{b},\mathfrak{p})$ quantizes to the homomorphism $\mathcal{A}_{\mathfrak{h}^*,\hbar}(S(\mathfrak{b},\mathfrak{p})) \rightarrow \mathcal{A}_{\mathfrak{X}(\mathfrak{l}),\hbar}(S(\mathfrak{q},\mathfrak{p}))$ inducing the homomorphism $\cC_\nu(\mathcal{A}_{\mathfrak{h}^*,\hbar}(S(\mathfrak{b},\mathfrak{p}))) \rightarrow \cC_\nu(\mathcal{A}_{\mathfrak{X}(\mathfrak{l}),\hbar}(S(\mathfrak{q},\mathfrak{p})))$ quantizing (\ref{eq_restr_slod_to_part_slod}). The pull-back map (\ref{equation_pullback_to_resol}) is quantized as follows. We consider the {\emph{sheaf}} quantization $\mathcal{D}_{\mathfrak{X}(\mathfrak{l}),\hbar}(\widetilde{S}(\mathfrak{q},\mathfrak{p}))$ of $\mathcal{O}_{\widetilde{S}(\mathfrak{q},\mathfrak{p})}$ and then (\ref{equation_pullback_to_resol}) is quantized by the natural map 
\begin{equation*}
\cC_\nu(\mathcal{A}_{\mathfrak{X}(\mathfrak{l}),\hbar}(S(\mathfrak{q},\mathfrak{p}))) \rightarrow \Gamma(\widetilde{S}(\mathfrak{q},\mathfrak{p})^{\nu(\mathbb{C}^\times)},\cC_\nu(\mathcal{D}_{\mathfrak{X}(\mathfrak{l}),\hbar}(\widetilde{S}(\mathfrak{q},\mathfrak{p})))).
\end{equation*}
So, the main result of the paper is the following theorem.
\begin{theorem}\label{intro_thm_refined_hn}
There is an isomorphism of graded $\mathbb{C}[\mathfrak{X}(\mathfrak{l}),\hbar]$-algebras:
\begin{multline}\label{eq_our_refined_iso}
\on{Im}(H^*_{Z_{L^\vee} \times \mathbb{C}^{\times}}(T^*\mathcal{Q}^\vee) \rightarrow H^*_{Z_{L^\vee} \times \mathbb{C}^\times}(\widetilde{S}(\mathfrak{p}^\vee,\mathfrak{q}^\vee))) \simeq \\ \simeq \on{Im}(\cC_\nu(\mathcal{A}_{\mathfrak{h}^*,\hbar}(S(\mathfrak{b},\mathfrak{p}))) \rightarrow \cC_\nu(\mathcal{A}_{\mathfrak{X}(\mathfrak{l}),\hbar}(S(\mathfrak{q},\mathfrak{p}))) \rightarrow \Gamma(\widetilde{S}(\mathfrak{q},\mathfrak{p})^{\nu(\mathbb{C}^\times)},\cC_\nu(\mathcal{D}_{\mathfrak{X}(\mathfrak{l}),\hbar}(\widetilde{S}(\mathfrak{q},\mathfrak{p}))))).
\end{multline}
\end{theorem}

One possible way to think about Theorem \ref{intro_thm_refined_hn} is the following. 
Both algebras $H^*_{Z_{L^*\vee} \times \CC^\times}(\widetilde{S}(\mathfrak{p}^\vee,\mathfrak{q}^\vee))$ and $\cC_\nu(\mathcal{A}_{\hbar,\mathfrak{X}(\mathfrak{l})}({S}(\mathfrak{q},\mathfrak{p})))$ have a natural action of the ``big" (polynomial) algebra $\CC[\mathfrak{h}^*,\hbar]^{W_M} \otimes \CC[\mathfrak{h}^*,\hbar]$ coming from the embedding $\widetilde{S}(\mathfrak{p}^\vee,\mathfrak{q}^\vee)\subset T^*\cQ^\vee$ and the quantization of the map $S(\fq, \fp)\to S(\fb,\fp)$. Assuming torsion freeness of $\cC_\nu(\mathcal{A}_{\hbar,\mathfrak{X}(\mathfrak{l})}(S(\mathfrak{q},\mathfrak{b})))$ over $\mathbb{C}[\mathfrak{X}(\mathfrak{l}),\hbar]$, our refined equivariant Hikita-Nakajima conjecture claims that the subalgebras of $H^*_{Z_{L^\vee} \times \CC^\times}(\widetilde{S}(\mathfrak{p}^\vee,\mathfrak{q}^\vee))$ and $\cC_\nu(\mathcal{A}_{\hbar,\mathfrak{X}(\mathfrak{l})}(S(\mathfrak{q},\mathfrak{b})))$ generated by $1$ under the action of the algebra $\CC[\mathfrak{h}^*,\hbar]^{W_M} \otimes \CC[\mathfrak{h}^*,\hbar]$ are isomorphic. 

Let us also briefly mention the idea of the proof of \cref{intro_thm_refined_hn} (cf. \cite{ksh}). Recall that symplectic duality predicts the existence of (canonical) bijection 
\begin{equation}\label{intro_bij_fixed_pts}
Y^{T_X} \simeq (Y^{\vee})^{T_{X^\vee}}.
\end{equation}
The existence of the  bijection (\ref{intro_bij_fixed_pts}) can be checked in our example $Y=\widetilde{S}(\mathfrak{q},\mathfrak{p})$, $Y^\vee=\widetilde{S}(\mathfrak{p}^\vee,\mathfrak{q}^\vee)$ and follows from the results of Appendix \ref{app_descr_fixed_points_parabolic}. Now, the idea is to realize both of the algebras appearing in \cref{intro_thm_refined_hn} as {\emph{subalgebras}} of the direct sum $\mathbb{C}[\mathfrak{X}(\mathfrak{l}),\hbar]^{\oplus |Y^{T_X}|}$. The desired embedding of the LHS of (\ref{eq_our_refined_iso}) is given by composing $H^*_{Z_{L^\vee}}(T^*\mathcal{Q}^\vee) 
    \rightarrow H^*_{Z_{L^\vee}}(\widetilde{S}(\mathfrak{p}^\vee,\mathfrak{q}^\vee))$
with the pull-back (\ref{equation_restr_to_fixed_parabolic}) to the torus fixed points  $ H^*_{Z_{L^\vee}}(\widetilde{S}(\mathfrak{p}^\vee,\mathfrak{q}^\vee)) \hookrightarrow H^*_{Z_{L^\vee}}(\widetilde{S}(\mathfrak{p}^\vee,\mathfrak{q}^\vee)^{Z_{L^\vee}})$. The algebra in the RHS of (\ref{eq_our_refined_iso}) is obviously a subalgebra of $\Gamma(\widetilde{S}(\mathfrak{q},\mathfrak{p})^{\nu(\mathbb{C}^\times)},\cC_\nu(\mathcal{D}_{\mathfrak{X}(\mathfrak{l}),\hbar}(\widetilde{S}(\mathfrak{q},\mathfrak{p}))))$. It them follows from the results of \cite[Section 5]{BPWII} (see also \cite[Section 5.2]{LosevCatO}) that we have an isomorphism 
\begin{equation*}
\Gamma(\widetilde{S}(\mathfrak{q},\mathfrak{p})^{\nu(\mathbb{C}^\times)},\cC_\nu(\mathcal{D}_{\mathfrak{X}(\mathfrak{l}),\hbar}(\widetilde{S}(\mathfrak{q},\mathfrak{p})))) \simeq \mathbb{C}[\mathfrak{X}(\mathfrak{l}),\hbar]^{|Y^{T_X}|}.
\end{equation*}
Now, to prove \cref{intro_thm_refined_hn} we need to identify the images of 
\begin{equation}\label{eq_alg_targets}
H^*_{Z_{L^\vee} \times \mathbb{C}^\times}(T^*\mathcal{Q}^\vee),\,\cC_\nu(\mathcal{A}_{\mathfrak{h}^*,\hbar}(S(\mathfrak{b},\mathfrak{p})))
\end{equation}
inside the direct sum of polynomial algebras. We then note that both of the algebras in (\ref{eq_alg_targets})  are naturally quotients of the same algebra $\mathbb{C}[\mathfrak{h}^*,\hbar]^{W_M} \otimes \mathbb{C}[\mathfrak{h}^*,\hbar]$. So, now we just have two maps $\mathbb{C}[\mathfrak{h}^*,\hbar]^{W_M} \otimes \mathbb{C}[\mathfrak{h}^*,\hbar]  \rightrightarrows \mathbb{C}[\mathfrak{X}(\mathfrak{l}),\hbar]^{\oplus |Y^{T_X}|}$ and our goal is to prove that these maps are equal. The key point is that this can already be done at generic $(\lambda,\hbar_0) \in \mathfrak{X}(\mathfrak{l}) \oplus \mathbb{C}$ and reduces to the standard computations. Namely, on the LHS of the conjecture, this is just the computation of weights of $G$-equivariant bundles on $\mathcal{Q}^\vee$ at the torus fixed points. On the RHS, this is the computation of the highest weights of irreducible modules in the regular integral block of the category $\mathcal{O}$ for parabolic $W$-algebras.

\subsubsection{Comparison of refined Hikita-Nakajima and Hikita-Nakajima conjectures}    The connection between our theorem and the original Hikita conjecture is worth elaborating on. First of all, for $\mathfrak{p}=\mathfrak{b}$, i.e., for the pair $(X^\vee=T^*\mathcal{Q}^\vee, X=\widetilde{S}(\chi))$, our Theorem \ref{main_th_weak_hikita} implies the Hikita-Nakajima conjecture without any modifications. The situation is much more complicated in general.
    Specifically, \cref{main_th_weak_hikita} along with certain additional conditions (see Section \ref{weak_vs_original}) would imply that \cref{equiv_HN_intro} and \cref{Hikita_intro}. These conditions are briefly discussed in Section \ref{sec_cor_conj_further} and thoroughly examined in \cite{hoang2}. In fact, these conditions are necessary conditions for \cref{Hikita_intro} to hold (\cite[Propositions 1.1, 1.5]{hoang2}). Therefore, our main result can essentially be viewed as a precise replacement, rather than a weaker version, of the Hikita-Nakajima conjecture for nilpotent orbits. For $\fg$ of type A, \cref{main_th_weak_hikita} uniformly reproduces all type A theorems on the Hikita conjecture (Appendix \ref{app_HN_type_A}). This includes recent results from \cite{Brundan2008}, \cite{KamnitzerTingleyWebsterWeeksYacobi}, as well as classical results from \cite{deConcini1981}, \cite{Tanisaki}. Our theorem also uniformly generalizes other results for general $\fg$ obtained in \cite{carell},\cite{Kumar2012}, and \cite{Carrell_2017}.

    

    

\subsection{Motivations and applications}
  \subsubsection{Connection to unipotent characters}\label{connection to unipotent characters}
        An important motivation for the Hikita conjecture in the case $\fq=\fb$ comes from \cite[Section 9.3]{LMBM}. In the loc.cit. it is shown how a slightly stronger version of \cref{Hikita_intro} explains the appearance of special unipotent central characters in the study of canonical quantizations of orbit covers. Namely, the following diagram was expected to be commutative. 

        \begin{equation*}
\xymatrix{ H^*_{Z(L^\vee) \times \CC^\times}(T^*\mathcal{B}^\vee) \ar[r]^{res} & H^*_{Z_{L^\vee} \times \CC^\times}(\widetilde{S}(\chi^\vee)) \ar[dd]_{\sim} \\
\CC[\mathfrak{h}^*,\hbar] \otimes_{\CC[\mathfrak{h}^*,\hbar]^W} \CC[\mathfrak{h}^*,\hbar] \ar[u]^{\sim} \ar[d]_{\sim} & \\
\cC_\nu(\mathcal{A}_{\hbar,\mathfrak{X}(\mathfrak{l})}(T^*\cB)) \ar[r] & \cC_\nu(\mathcal{A}_{\hbar,\mathfrak{X}(\mathfrak{l})}(\widetilde{S}(\mathfrak{b},\mathfrak{p}))) 
}
\end{equation*}

As mentioned already, there is no isomorphism $H^*_{Z_{L^\vee} \times \CC^\times}(\widetilde{S}(\chi^\vee))\to \cC_\nu(\mathcal{A}_{\hbar,\mathfrak{X}(\mathfrak{l})}(\widetilde{S}(\mathfrak{b},\mathfrak{p})))$ in general. However, the commutative diagram of \cref{main_prop_weak_hikita} is sufficient for the argument in the paper. Similarly, in the case when $e^\vee\in \fl^\vee$ is not regular, \cref{conj_most_general_weak_Hikita} is sufficient for the argument in \cite[Section 9.3]{LMBM}. Understanding the conjecture for non-regular in Levi $e^\vee$ is a topic of an ongoing work in progress of the authors. We refer to \cref{sec_more_gen_settings} for the discussion.

\subsubsection{Geometry of Springer fibers and Kazhdan-Lusztig cells}  
\cref{main_th_weak_hikita} offers a bridge between representation theory and geometry. In particular, the set-theoretic support of the specialization of the image of the ``big" algebra action on $1\in \cC_\nu(\mathcal{A}_{\hbar,\mathfrak{X}(\mathfrak{l})}(\widetilde{S}(\mathfrak{b},\mathfrak{p}))) $ at a parameter $(\lambda, t)\in \fX(\fl)\oplus \CC \hbar$ records the highest weight of certain irreducible $\cU(\fg)$-modules annihilated by the same primitive ideal. Compared with the cohomological side, we obtain a new, highly nontrivial relation between the geometry of Springer fibers and the theory of Kazhdan-Lusztig cells. This relation is mentioned in \cref{combinatorics} and is part of a work in progress.

 \subsection{Structure of the paper}   We finish the section by explaining the structure of the paper. First, we recall the already existing in the literature results that are important for our paper. Section \ref{sec_prelim} focuses on conical symplectic singularities and their deformation. Section \ref{sec_equiv} recalls the basics of equivariant cohomology and provides explicit computations of fixed points of Slodowy varieties. In Section \ref{sec_sympldual} we fix the settings of the symplectic duality adopted in the paper. Section \ref{sect_refined_BVLS} introduces the main symplectic dual pairs considered in the paper following \cite{LMBM} and \cite{BPWII}. The first new results appear in Section \ref{sect_counter}, where we analyze the Hikita conjecture for Slodowy slices and nilpotent orbits and provide explicit counterexamples. Section \ref{sec_hw_parab} is quite technical and provides the necessary background for understanding the highest weights of parabolic $W$-algebras used in the proof of the main result. In Section \ref{equivariant conjecture section}, we state the main conjecture (\cref{conj_most_general_weak_Hikita}) that serves as our replacement of the equivariant Hikita-Nakajima conjecture in the literature and discuss the differences between the two. Section \ref{Section_weak_HN_for_parab_lodowy_proof} states and proves the main result of the paper (\cref{main_th_weak_hikita}). In Section \ref{sec_cor_conj_further}, we discuss the corollaries of the results obtained and potential directions for further study. Finally, we have 6 appendices that fix notations and provide explicit computations used in the paper.
\subsection{Acknowledgements}
    This paper was made possible thanks to the feedback and constant sharing of the ideas from many mathematicians, their help is extremely appreciated. We are especially grateful to Roman Bezrukavnikov, Alexander Braverman, Pavel Etingof, Andrei Ionov, Joel Kamnitzer, Ivan Losev, Hiraku Nakajima, Peng Shan, Ben Webster, and Alex Weekes.


\tableofcontents
\section{Preliminaries}\label{sec_prelim}

 \subsection{Conical symplectic singularities}\label{conical sympl sing}
        Let $X$ be a normal Poisson scheme. Following Beauville \cite{Beauville2000}, we say that $X$ has \emph{symplectic singularities} if
        \begin{itemize}
            \item[(i)] The regular (smooth) locus $X^{\mathrm{reg}}$ admits a symplectic form $\omega$.
            \item[(ii)] There is a resolution of singularities $\pi\colon Y\to X$, such that $\pi^*\omega$ extends to a regular (not necessarily symplectic) $2$-form on $Y$.
        \end{itemize}

        We say that a graded normal Poisson scheme $X$ is a \emph{conical symplectic singularities} if $X$ has symplectic singularities and admits a $\CC^\times$-action, contracting $X$ onto a point. Note that a conical symplectic singularity is automatically affine. We will call $Y \rightarrow X$ a {\emph{symplectic resolution}} of $X$ if the extension of $\pi^*\omega$ to $Y$ is the {\emph{symplectic}} form. Let us list some examples relevant to this paper, see \cref{sect_nilp_orbits_covers_and_bir_ind} and \cref{sec_parabolic_slodowy} for more details.

        \begin{example*}
            \begin{itemize}
                \item[(a)] 
                The Springer resolution $T^*\mathcal{B} \to \cN$ is a symplectic resolution and thus satisfies (ii). More generally, for arbitrary Parabolic subgroup $P \subset G$, the affinization morphism $\pi_{\mathfrak{p}}\colon Y:=T^*\mathcal{P} \rightarrow X=\operatorname{Spec}(T^*\mathcal{P})$ is a conical symplectic resolution.
                \item[(b)] Let $G$ and $\cN$ be as above, and let $\OO\subset \cN$ be a nilpotent orbit. Then $\Spec(\CC[\OO])$ is a conical symplectic singularity, \cite{Panyushev1991}. Moreover, if $\widetilde{\OO}$ is a nilpotent cover, $\Spec(\CC[\widetilde{\OO}])$ is a conical symplectic singularity, \cite[Lemma 2.5]{LosevHC}. Note that it, in general, may not possess a symplectic resolution. We refer interested readers to \cite{Fu_2003} and \cite{Fu2003}.  
                \item[(c)] Let $X=S(\chi) \cap \cN$ be the intersection of the Slodowy slice to $\chi$ with the nilpotent cone. Then $Y=\widetilde{S}(\chi) \rightarrow S(\chi) \cap \cN=X$ is a conical symplectic resolution, where the map $\pi\colon \widetilde{S}(\chi) \to S(\chi) \cap \cN$ is given by restricting the Springer resolution to $X\subset \cN$. 
                \item[(d)] Generalizing $(a)$ and $(c)$, we can take $Y=\widetilde{S}(\chi,\mathfrak{p})$ to be the parabolc Slodowy variety, and $X:=\on{Spec}\CC[\widetilde{S}(\chi,\mathfrak{p})]$.
            \end{itemize}
        \end{example*}

        Suppose that $X$ has symplectic singularities. In the remainder of the section we introduce some of the key objects associated with $X$. 
        
   \subsubsection{$\mathbb{Q}$-factorial terminalizations}     The following was proved in \cite[Prop 2.1]{LosevSRA} following \cite{BCHM}, see also \cite[A.7]{Namikawa3}. 
        \begin{prop}
            There is a birational projective Poisson morphism $\pi\colon \widetilde{X}\to X$ such that
            \begin{itemize}
                \item[(i)] $\widetilde{X}$ has symplectic singularities;
                \item[(ii)] Every Weil divisor on $\widetilde{X}$ has a (nonzero) integer multiple that is Cartier, i.e. $\widetilde{X}$ is $\QQ$-factorial;
                \item[(iii)] $\codim_{\widetilde{X}}(\widetilde{X}\setminus \widetilde{X}^{\mathrm{reg}})\ge 4$;
                \item[(iv)] If $X$ is conical symplectic singularity, then $\widetilde{X}$ admits a $\CC^\times$-action, such that $\pi$ is $\CC^\times$-equivariant. 
            \end{itemize}

             We say that $\widetilde{X}$ is a \emph{$\QQ$-factorial terminalization} of $X$.
        \end{prop}

        It is important to note that $X$ usually admits multiple $\QQ$-factorial terminalizations. In what follows, we fix a choice of $\widetilde{X}$. 
        
\subsubsection{Namikawa space}        
        Define the \emph{Namikawa space} $\mathfrak{h}_X=H^2(\widetilde{X}^{\mathrm{reg}}, \CC)$. The key property of $\fh_X$ is that it classifies Poisson deformations of $\widetilde{X}$. Namely, we have the following result.
        \begin{prop}\label{universal Poisson deformation}\cite[Proposition 2.6]{Losev4}
            There is a graded Poisson deformation $\widetilde{X}_{\fh_X}$ of $\widetilde{X}$ over $\fh_X$, such that for any graded Poisson deformation $\widetilde{X}_B$ over $\Spec(B)$ there is a unique $\CC^\times$-equivariant morphism  $\Spec(B)\to \fh_X$, such that $\widetilde{X}_B\simeq \widetilde{X}_{\fh_X}\times_{\fh_X}\Spec(B)$.
        \end{prop}
        We call the deformation $\widetilde{X}_{\fh_X}\to \fh_X$ the \emph{universal Poisson deformation} of $\widetilde{X}$. The variety $X_{\fh_X}=\Spec(\CC[\widetilde{X}_{\fh_X}])$ is a graded Poisson deformation of $X$ over $\fh_X$. We note that $X_{\fh_X}\to \fh_X$ is not universal in the sense of \cref{universal Poisson deformation}, but it admits a surjective finite map to the universal Poisson deformation of $X$. For the purposes of this paper, we prefer working with $X_{\fh_X}$ over the universal Poisson deformation of $X$.

 \subsubsection{Automorphisms of $X$}        We finish this section with a short discussion of the automorphisms of $X$. Consider the group of Poisson automorphisms of $X$ commuting with the contracting $\mathbb{C}^\times$-action, and let $Z_X$ be its maximal reductive quotient.

    \begin{warning}\label{warning disconnected group}
    Note that $Z_X$ may be disconnected.     
    \end{warning}        

    Let $T_X\subset Z_X$ be the maximal torus. The action of $Z_X$ on $X$ extends to the action of $Z_X$ on $X_{\fh_X}$, and similarly the action of $T_X$ on $X$ lifts to an action on $\widetilde{X}$ and extends to an action on $\widetilde{X}_{\fh_X}$. 
We denote the Lie algebra of $T_X$ by $\ft_X$. In \cref{sec_parabolic_slodowy}, we describe the Namikawa space $\fh_X$ and the Lie algebra $\ft_X$ for the nilpotent Slodowy slice $X=S_\chi$ and the affinization of the cotangent bundle of partial flag variety $X=T^*\cP$. These spaces play an important role in what follows.

\subsubsection{Chambers}\label{subsub_chambers} Let us now discuss stratifications of the spaces $\mathfrak{h}_X$, $\mathfrak{t}_X$ into the union of certain {\emph{chambers}} separated by certain {\emph{walls}}. These stratifications will play an important role in the paper. 

Recall the morphism $\widetilde{X}_{\mathfrak{h}_X} \rightarrow X_{\mathfrak{h}_X}$.  Consider the set $\fh_X^{\mathrm{sing}}$ of all $\chi\in \fh_X$ such that $\widetilde{X}_\chi$ is not affine, i.e. the map $\widetilde{X}_\chi \to X_\chi$ is not an isomorphism. Irreducible components of $\fh_X^{\mathrm{sing}}$ define walls, and thus, we obtain a division of $\fh_X$ into chambers.

The following proposition holds by the results of Namikawa (see  \cite[Main Theorem]{Namikawa_bir} and footnote $2$ in loc.cit.). Let $W$ be the Namikawa-Weyl group (it was introduced in \cite{Namikawa2} 
see also \cite[Section 2.3]{Losev4}). For a $\mathbb{Q}$-factorial terminalization $\widetilde{X} \rightarrow X$ we will denote by $\on{Amp}(\widetilde{X})$ the ample cone of $X$.  

\begin{prop}\label{namikawa's result_descr_of_h_sing}
The set $\mathfrak{h}_X \setminus \mathfrak{h}^{\mathrm{sing}}_X$ is equal to the union $\bigsqcup_{\widetilde{X} \rightarrow X,\, w\in W_{X}} w(\on{Amp}(\widetilde{X}))$. Every irreducible component of $\mathfrak{h}^{\mathrm{sing}}$ is a codimension one linear subspace of $\mathfrak{h}_X$.
\end{prop}

Let us now describe a natural stratification of $\mathfrak{t}_X$ into chambers separated by certain walls. 
We assume that $X$ admits a symplectic resolution $Y$ such that the set $Y^{T_X}$ is finite.

Let $\mathfrak{t}_X^{\mathrm{sing}} \subset \mathfrak{t}_X$ be the subset of  $x \in \mathfrak{t}_X$ such that $Y^{x}=Y^{T_X}$, where by $Y^x$ we mean zeroes of the vector field on $Y$ defined by $x \in \mathfrak{t}_X = \on{Lie}T_X$. Clearly, $\mathfrak{t}_X^{\mathrm{sing}}$ can be explicitly described as follows. Let $\Delta_X \subset \mathfrak{t}_{X}^*$ be the set of all $\mathfrak{t}_X$-weights that appear in $T_{p}Y$ for $y \in Y^{T_X}$. For $\alpha \in \Delta_X$ let $H_{\alpha} \subset \mathfrak{t}_X$ be the hyperplane given by the equation $\langle \alpha,-\rangle = 0$. It is easy to see that 
\begin{equation*}
\mathfrak{t}_{X}^{\mathrm{sing}} = \bigcup_{\alpha \in \Delta_X}H_{\alpha}.
\end{equation*}

\begin{rmk}
There should be an internal definition of $\mathfrak{t}_{X}^{\mathrm{sing}}$ that does not involve a choice of $Y$ and works for arbitrary conical symplectic singularity $X$ (not necessarily having a symplectic resolution). One candidate is the set $\Delta'_X$ defined as follows (see \cite[Section 3.1]{kmp}). Let $\mathbb{C}[X]_+ \subset \mathbb{C}[X]$ be the ideal generated by the elements of positive degree w.r.t. the contracting action. Then $\Delta'_X$ is the set of $T_X$-weights of $\mathbb{C}[X]/(\mathbb{C}[X]_+)^2$. In other words, $\Delta'_X$ is the set of $T_X$-weights of the cotangent space of $X$ at the unique $T_X \times \mathbb{C}^\times$-fixed point.  It is expected that $\Delta_X=\Delta'_X$ when $X$ has a symplectic resolution. 
\end{rmk}

\subsubsection{Equivariant formality}
Fix a symplectic singularity $X$ and assume that it admits a symplectic resolution $Y$. It in principle follows from the results of Kaledin that $H^*_{T_X \times \mathbb{C}^\times}(Y)$ is free over $H^*_{T_X \times \mathbb{C}^\times}(\on{pt})$. On the other hand, if $Y^{T_X}$ is finite, then this claim becomes extremely easy. Since all of the varieties $Y$ considered in this paper will have isolated $T_X$-fixed points we have decided to include the following lemma.

\begin{lemma}\label{lemma_equiv_cohom_resol_free}
Assume that $Y \rightarrow X$ is a conical symplectic resolution s.t. $Y^{T_X}$ is finite. Then $H^*_{T_X \times \mathbb{C}^\times}(Y)$ is free over $H^*_{T_X \times \mathbb{C}^\times}(\on{pt})$.
\end{lemma}
\begin{proof}
It is enough to construct a finite $T_X \times \mathbb{C}^\times$-equivariant paving of $Y$. Let $\nu\colon \mathbb{C}^\times \rightarrow T_X$ be a generic cocharacter of $T_X$.
Consider a cocharacter $\mathbb{C}^\times \rightarrow T_X \times \mathbb{C}^\times$ given by $t \mapsto (\nu(t),t^{N})$ for $N \gg 0$. Then, the induced action of $\mathbb{C}^\times$ contracts $Y$ to the finite set $Y^{T_X}$. Taking attractors to fixed points and using Bialynicki-Birula Theorem (see \cite[Section 4]{Bialynicki-Birula}), we obtain the desired affine paving.
\end{proof}





\subsection{Nilpotent orbits}\label{subsec_saturation}
    Let $G$ be a complex simple group with a fixed maximal torus $T \subset G$. Let $\fg=\on{Lie}G$, $\mathfrak{h}=\on{Lie}T$ be the corresponding Lie algebras. The variety $\fg^*$ admits a natural Poisson structure induced by the commutator on $\fg$, and the symplectic leaves of $\fg^*$ are coadjoint orbits. The coadjoint orbit $\OO\subset \fg^*$ is \emph{nilpotent} if it is stable under the scaling action of $\CC^\times$. We write $\Orb(G)$ for the set of nilpotent coadjoint $G$-orbits. Let $\cN\subset \fg^*$ denote the union of all nilpotent coadjoint orbits. In classical types, we have a combinatorial parameterization of nilpotent orbits. Namely, for a positive integer $n$ let $\cP(n)$ be the set of partitions of $n$, i.e. non-increasing sequences of positive integers $\mathbf{p}=(\mathbf{p}_1, \mathbf{p}_2, \ldots, \mathbf{p}_k)$, such that $n=\sum_i \mathbf{p}_i$. For an odd number $n$, we say that a partition $\mathbf{p}\in \cP(n)$ is of type $B$ if every even member of $\mathbf{p}$ comes with an even multiplicity, and denote the set of type $B$ partitions of $n$ by $\cP_B(n)$. Similarly, for even $n$, we say that a partition $\mathbf{p}$ is of type $C$ (resp. of type $D$) if every odd (resp. even) member of $\mathbf{p}$ comes with an even multiplicity and denote the corresponding set of partitions by $\cP_C(n)$ (resp. $\cP_D(n)$). A partition $\mathbf{p}\in \cP_D(n)$ is called very even if all members of $\mathbf{p}$ are even.

    \begin{prop}[Section 5.1, \cite{CM}]\label{prop:orbitstopartitions}
Suppose $G$ is a simple group of classical type. Then the following are true:
\begin{enumerate}[label=(\alph*)]
    \item If $\fg = \mathfrak{sl}_{n}$, then there is a bijection:
    $$\Orb(G) \xrightarrow{\sim} \mathcal{P}(n) $$
    \item If $\fg = \mathfrak{so}_{2n+1}$, then there is a bijection:
    $$\Orb(G)\xrightarrow{\sim} \mathcal{P}_{B}(2n+1) $$
    \item If $\fg = \mathfrak{sp}_{2n}$, then there is a bijection:
    $$ \Orb(G)\xrightarrow{\sim} \mathcal{P}_C(2n)$$
    \item If $\fg = \mathfrak{so}_{2n}$, then there is a surjection:
    $$\Orb(G) \twoheadrightarrow \mathcal{P}_{D}(2n) $$
    Over very even partitions, this map is two-to-one, and over all other partitions, it is a bijection.
\end{enumerate}
\end{prop}

    In what follows, we write $\OO_\mathbf{p}$ for the partition corresponding to $\mathbf{p}$ under the map above. For $\mathbf{p}$ very even we denote the corresponding preimages by $\OO_\mathbf{p}^I$ and $\OO_\mathbf{p}^{II}$. If we claim a statement that is true for both $\OO_\mathbf{p}^I$ and $\OO_\mathbf{p}^{II}$, we often omit the numerals and refer to the orbit as $\OO_\mathbf{p}$.

    To classify nilpotent orbits of exceptional simple groups we use the Bala-Carter classification. These ideas are also incredibly useful for the classical groups, we recall them below. 

    Let $L\subset G$ be a Levi subgroup containing the torus $T$, and $\OO_L\subset \fl^*$ be a nilpotent orbit. The nilpotent orbit $\OO=G\cdot \OO_L\subset \fg^*$ is called \emph{saturated} from the pair $(L, \OO_L)$ and is denoted by $\Sat_L^G \OO_L$. An orbit that cannot be saturated from any proper Levi $L\subset G$ is called \emph{distinguished}. 
    \begin{prop}\label{BC saturation} \cite{BalaCarter1976}, \cite[Thm. 8.1.1]{CM}
        Let $\OO\subset \fg^*$ be a nilpotent orbit. There is a unique up to $G$-conjugation pair $(L, \OO_L)$ of a Levi subgroup $L\subset G$, and a distinguished orbit $\OO_L\subset \fl^*$, such that $\OO=\Sat_L^G \OO_L$.
    \end{prop}

    \subsection{Nilpotent orbit covers and birational induction}\label{sect_nilp_orbits_covers_and_bir_ind}

    Recall that $L\subset G$ is a Levi subgroup, and $\OO_L\subset \fl^*$ is a nilpotent orbit. Pick a parabolic group $P\subset G$ with a Levi factor $L$ and Lie algebra $\mathfrak{p}$. Let $\mathfrak{p}^{\perp}=(\mathfrak{g}/\mathfrak{p})^* \subset \mathfrak{g}^*$. Form a $G$-equivariant fiber bundle $G\times^P (\overline{\OO}_L\times \fp^\perp)$ over $\mathcal{P}=G/P$ and note that the left action of $G$ induces a $G$-equivariant moment map 
    \begin{equation}\label{G_equiv_moment_ind_P}
    G\times^P(\overline{\OO}_L\times \fp^\perp)\to \fg^* \qquad (g, \xi) \mapsto g\xi.
    \end{equation}

    By \cite{LS}, the image of (\ref{G_equiv_moment_ind_P}) is the closure of a single nilpotent $G$-orbit $\OO$, and the orbit $\OO$ is independent of the choice of parabolic $P$. The correspondence associating to an orbit $\OO_L$ the orbit $\OO$ is called \emph{Lusztig-Spaltenstein induction}. We say that $\OO$ is induced from $(L, \OO_L)$ and denote it by $\Ind_L^G \OO_L$. We note that the map $G\times^P(\overline{\OO}_L\times \fp^\perp)\to \overline{\OO}$ is finite, but may not be birational. We also remark that an analog of \cref{BC saturation} doesn't hold for the Lusztig-Spaltenstein induction. To fix these issues we work with birational induction of nilpotent orbit covers.

    A \emph{nilpotent cover} is a finite \'{e}tale $G$-equivariant connected cover of a nilpotent coadjoint orbit. Naturally, if $\OO\subset \fg^*$ is a nilpotent orbit, and $\chi\in \OO$, the nilpotent covers of $\OO$ are in a natural correspondence with the subgroups of $\pi_1^G(\OO)=Z_G(\chi)/Z_G(\chi)^\circ$. 

    Recall that $L\subset G$ is a Levi, and let $\widetilde{\OO}_L$ be a nilpotent cover of $\OO_L\subset \fg^*$. Similarly to the above, we form a fiber bundle $G\times^P(\Spec(\CC[\widetilde{\OO}_L])\times \fp^\perp)$, and it admits a moment map $\pi\colon G\times^P(\Spec(\CC[\widetilde{\OO}_L])\times \fp^\perp)\to \fg^*$ with the image of $\pi$ being the closure of the orbit $\OO=\Ind_L^G \OO_L$. Let $\widetilde{\OO}=\pi^{-1}(\OO)$. Then $\widetilde{\OO}$ is connected, and therefore a nilpotent cover of $\OO$. We say that $\widetilde{\OO}$ is \emph{birationally induced} from $(L, \widetilde{\OO}_L)$ and denote it by $\Bind_L^G \widetilde{\OO}_L$. By \cite[Proposition 2.4.1(i)]{LMBM}, $\widetilde{\OO}$ is independent of the choice of the parabolic $P$. A nilpotent cover $\widetilde{\OO}$ that cannot be induced from a proper Levi subgroup $L$ is called \emph{birationally rigid}. 
    \begin{prop}\label{bir ind to bir rigid}\cite[Proposition 2.4.1(iii)]{LMBM}
        For any nilpotent cover $\widetilde{\OO}$ there is a unique up to $G$-conjugation pair $(L,\widetilde{\OO}_L)$ of a Levi subgroup $L\subset G$ and a birationally rigid nilpotent cover $\widetilde{\OO}_L$, such that $\widetilde{\OO}=\Bind_L^G \widetilde{\OO}_L$.
    \end{prop}

    \subsection{Parabolic Slodowy varieties}\label{sec_parabolic_slodowy}
        \subsubsection{Partial flag varieties and their deformations}
    Recall that $P \subset G$ is a parabolic subgroup containing $T$ with a Levi factor $L$, and $\mathcal{P}=G/P$ is the partial flag variety (we will sometimes identify it with the space of parabolic subgroups $P' \subset G$ that are conjugate to $P$). Consider the symplectic variety $T^*\mathcal{P}$ that can be canonically identified with $G \times^P \mathfrak{p}^{\perp}$. We have the natural morphism  $\pi_{\mathfrak{p}}\colon T^*\mathcal{P} \rightarrow \overline{\mathbb{O}}$ (c.f. (\ref{G_equiv_moment_ind_P}) for $\mathbb{O}_{L}=\{0\}$).

    \begin{rmk}
    As we have mentioned in Section \ref{sect_nilp_orbits_covers_and_bir_ind}, the morphism $\pi_{\mathfrak{p}}$ is finite but not birational in general. On the other hand, it factors as $T^*\mathcal{P} \rightarrow \on{Spec}(\CC[T^*\mathcal{P}]) \rightarrow \overline{\mathbb{O}}$, and the first morphism is birational. Moreover, $\on{Spec}(\CC[T^*\mathcal{P}])$ is nothing else but $\on{Spec}\CC[\widetilde{\mathbb{O}}]$ for the cover $\widetilde{\mathbb{O}}=\Bind_L^G \{0\}$ of the orbit $\mathbb{O}$. 
    \end{rmk}
    
    Let $X=\Spec(\CC[\widetilde{\mathbb{O}}])$. The Namikawa space $\fh_X$ by definition equals to $H^2(T^*\cP, \CC)\simeq H^2(\cP, \CC)$. The latter is identified with $\mathfrak{X}(\mathfrak{l})=(\mathfrak{l}/[\mathfrak{l},\mathfrak{l}])^*=(\mathfrak{p}/[\mathfrak{p},\mathfrak{p}])^*$. 
    The universal deformation of $T^*\mathcal{P}$ over the base $\mathfrak{X}(\mathfrak{l})$ is given as follows:
    \begin{equation*}
    T^*_{\mathfrak{X}(\mathfrak{l})}\mathcal{P} = G \times^P [\mathfrak{p},\mathfrak{p}]^{\perp} \rightarrow \mathfrak{X}(\mathfrak{l}),~(g,\xi) \mapsto \xi|_{\mathfrak{p}}.
    \end{equation*}
    For $P=B$, $\mathfrak{X}(\mathfrak{h})=\mathfrak{h}^*$,  $T^*_{\mathfrak{h}^*}\mathcal{B}$ is the well-known {\emph{Grothendieck-Springer resolution}} that will be denoted by $\widetilde{\mathfrak{g}}^*$.

    For future applications, we will need a description of ample line bundles on $T^*\mathcal{P}$. Let $\Delta^\vee$ be the set of coroots of $\mathfrak{g}$, equivalently, the set of roots of $\mathfrak{g}^\vee$.
    Let  $\Delta^\vee_{\mathfrak{p}} \subset \Delta^\vee$ be the subset consisting of $\beta^\vee$ such that $\mathfrak{g}^\vee_{\beta^\vee} \subset \mathfrak{p}^\vee$.
    We will say that $\la \in \mathfrak{h}^*$ is {\it  $\mathfrak{p}$-antidominant} if $\langle \al,\beta^{\vee} \rangle \notin \BZ_{> 0}$ for any $\beta^\vee \in \Delta_{\mathfrak{p}}^\vee$.
	We say that  $\la \in \mathfrak{h}^*$ is $\mathfrak{l}$-{\em{regular}} if $\langle \al,\beta^{\vee} \rangle \neq 0$ for any $\beta^\vee \in \Delta^\vee \setminus \Delta_{\mathfrak{l}}^\vee$. 
	We say that $\la \in \mathfrak{h}^*$ is {\em{integral}} if $\langle\la,\beta^{\vee}\rangle \in \BZ$ for any $\beta^\vee \in \Delta^\vee$. 

    To every {\emph{integral}} $\la \in \mathfrak{X}(\mathfrak{l})$ we can associate the corresponding $P$-character to be denoted by the same symbol $\la\colon P \rightarrow \mathbb{C}^\times$. Let $\mathbb{C}_\la$ be the corresponding one-dimensional representation of $P$. Let $\mathcal{O}_{\mathcal{P}}(\la):= G \times^P \mathbb{C}_\la$ be the induced line bundle.

    The following lemma is standard. 
    \begin{lemma}
    Line bundle $\mathcal{O}_{\mathcal{P}}(\la)$ is ample iff $\la$ is $\mathfrak{p}$-antidominant and $\mathfrak{l}$-regular.
    \end{lemma}

    For the future applications in this paper, it is useful to know the fiber $T^*_{\xi}\cP$ of the universal deformation of $T^*\cP$ over the point $\xi\in \fX(\fl)$. Let $M=Z_G(\xi)$ be a Levi subgroup of $G$ containing $L$, and let $P_M=P\cap M$ and $\cP_M$ be the corresponding parabolic subgroup and partial flag variety for $M$ respectively. We claim the following.
    \begin{lemma}\label{fiber of deformation of parabolic}
        There is an isomorphism $T^*_{\xi}\cP\simeq G\times^M (\xi\times T^*\cP_M)$. In particular, for $\mathfrak{l}$-regular $\xi\in \fX(\fl)$ we have $T^*_{\xi}\cP\simeq G/L=G\xi$.
    \end{lemma}
    \begin{proof}
        Let $Q\supset P$ be a parabolic subgroup of $G$ with Levi factor $\fm$, and let $U\subset Q$ be the unipotent radical. Let $\fX(\fl)^\circ\subset \fX(\fl)$ be the subset of all $\eta\in \fX(\fl)$ such that $Z_G(\eta)\subset M$. We have 
        $$T^*_{\fX(\fl)^\circ}\cP=G\times^P(\fX(\fl)^\circ\times \fp^\perp)\simeq G\times^Q(\fX(\fl)^\circ\times T^*\cP_M\times \fq^\perp)$$
        We claim that $G\times^Q(\fX(\fl)^\circ\times T^*\cP_M\times \fq^\perp)\simeq G\times^M(\fX(\fl)^\circ\times T^*\cP_M)$. Consider the tautological map $\tau: G\times^M(\fX(\fl)^\circ\times T^*\cP_M)\to G\times^Q(\fX(\fl)^\circ\times T^*\cP_M\times \fq^\perp)$ given by $(g, \eta, x)\mapsto (g, \eta, x)$. We have $G\times^M(\fX(\fl)^\circ\times T^*\cP_M)\simeq G\times^Q(U\times \fX(\fl)^\circ\times T^*\cP_M)$. The isomorphism $U\times \fX(\fl)^\circ\times T^*\cP_M\simeq \fX(\fl)^\circ\times T^*\cP_M\times \fq^\perp$ given by $(u,\eta,x)\mapsto u.(\eta, x)$ implies that $\tau$ is an isomorphism. Specializing to the point $\xi\in \fX(\fl)^\circ$, we get the desired statement.
    \end{proof}
    

    Note that we have the closed embedding 
    $T^*_{\mathfrak{X}(\mathfrak{l})}\mathcal{P} \hookrightarrow \mathcal{P} \times \mathfrak{g}^*$ given by $(g,\xi) \mapsto (gPg^{-1},g\xi)$ and inducing the identification: 
    \begin{equation*}
    T^*_{\mathfrak{X}(\mathfrak{l})}\mathcal{P}  \iso \{(\mathfrak{p}',\xi')\,|\, \mathfrak{p}' \in \mathcal{P},\, \xi' \in \mathfrak{g}^*,\,\xi'|_{[\mathfrak{p}',\mathfrak{p}']}=0\}.
    \end{equation*}
    The map $T^*_{\mathfrak{X}(\mathfrak{l})}\mathcal{P} \to \fX(\fl)$ constructed above, and the composition $T^*_{\mathfrak{X}(\mathfrak{l})}\mathcal{P} \hookrightarrow \mathcal{P} \times \mathfrak{g}^* \to \fg^*$ give rise to the map $\widetilde{\pi}_\fp: T^*_{\mathfrak{X}(\mathfrak{l})}\mathcal{P}\to \fg^*\times_{\fh^*/W} \fX(\fl)$ that will be important in what follows.
    
    We finish the discussion of partial flag varieties with the discussion of the group $Z_X$. As mentioned in \cref{warning disconnected group}, the group $Z_X$ may be disconnected. For our purposes, it is enough to consider the Lie algebra $\fz_X$ of $Z_X$. Recall that $X=\Spec(\CC[\widetilde{\OO}])$ for a $G$-equivariant cover $\widetilde{\OO}$ of a nilpotent orbit $\OO\subset \cN$. The Lie algebra $\fz_X$ in this case was studied by Brylinski and Kostant in \cite{BrylinskiKostant1994}, which led to the concept of \emph{shared orbits}, corresponding to the case when $\fz_X$ is not equal to $\fg$. All shared orbits were classified in recent work \cite{FJLS2023} of Fu, Juteau, Levy, and Sommers. We summarize the results of these two papers in our context as follows.

    \begin{prop}\label{automotphisms of partial flag}
        Let $P\subset G$ be a parabolic subgroup, and $X=\Spec(\CC[T^*\cP])$. Let $L$ be the Levi factor of $P$. Then $\fz_X\not \simeq \fg$ if and only if $\fg\simeq \fs\fo(2n+1)$ and $\fl\simeq \fg\fl(n)$. In this case $\fz_X\simeq \fs\fo(2n+2)$. 
    \end{prop}

    \subsubsection{Slodowy varieties} Let $G$, $\cN$ and $\OO$ be as above. Pick a Borel subgroup $B \subset G$ and let $\mathcal{B}:=G/B$. Fix a nondegenerate $\mathfrak{g}$-invariant form $(-,-)$ on $\fg$. Pick a point $\chi\in \OO$, and let $e\in \fg$ be the nilpotent element such that $\chi=(e, -)$. Choose an $\fs\fl_2$-triple $(e,h,f)$, and let $S(e):=e+\Ker(\ad f)$. Let $\fg=\bigoplus_{i\in \ZZ} \fg(i)$ be the decomposition into the weight spaces of $[h,-] \curvearrowright \mathfrak{g}$, and consider the Kazhdan $\CC^\times$-action on $\fg$ given on $x\in \fg(i)$ by $t\cdot x=t^{-2+i}x$. This action restricts to the action on $S(e)$, contracting it onto the point $e$ as $t \rightarrow \infty$. We will denote by $S(\chi)$ the image of $S(e)$ under the isomorphism $\fg\simeq \fg^*$ given by the form $(-,-)$ above. 
    Variety $S(\chi)$ is called the {\em{Slodowy slice}} to the nilpotent orbit $\mathbb{O}$ at point $\chi$.
    
    \begin{rmk}
    After the identification $\mathfrak{g} \simeq \mathfrak{g}^*$, the action above is given by $t \cdot \eta = t^{-2-i}\eta$ for $\eta \in \left(\mathfrak{g}(i)\right )^*\subset \fg^*$. 
    \end{rmk}

    We have the natural morphism $\mathfrak{g}^* \rightarrow \mathfrak{g}^*/\!/G$ induced by the embedding $\CC[\mathfrak{g}^*]^{G} \subset \CC[\mathfrak{g}^*]$. Recall also that by the Chevalley restriction theorem, $\mathfrak{g}^*/\!/G \simeq \mathfrak{h}^*/W$, where $W=N(T)/T$ is the Weyl group of $(G,T)$. So, we obtain the morphism $\mathfrak{g}^* \rightarrow \mathfrak{h}^*/W$ inducing the morphism $S(\chi) \rightarrow \mathfrak{h}^*/W$.

    Set $S_{\mathfrak{h}^*}(\chi) := S(\chi) \times_{\mathfrak{h}^*/W} \mathfrak{h}^* \subset \mathfrak{g}^* \times_{\mathfrak{h}^*/W} \mathfrak{h}^*$ and let $\widetilde{S}_{\mathfrak{h}^*}(\chi) \subset \widetilde{\mathfrak{g}}^*$ be the preimage of $S(\chi) \times_{\mathfrak{h}^*/W} \mathfrak{h}^*$ under the Grothendieck-Springer morphism $\widetilde{\mathfrak{g}}^* \rightarrow \mathfrak{g}^* \times_{\mathfrak{h}^*/W} \mathfrak{h}^*$. We will also consider the fiber of $S_{\mathfrak{h}^*}(\chi) \rightarrow \mathfrak{h}^*$ over zero to be denoted by $\widetilde{S}(\chi)$ and called {\emph{Slodowy variety}}. 
    Note that  $\widetilde{S}(\chi)$ is nothing else but the preimage of $S(\chi) \cap \mathcal{N}$ under the Springer morphism $\pi_{\mathfrak{b}}\colon T^*\mathcal{B} \rightarrow \mathcal{N}$. In particular, $\widetilde{S}(\chi)\to S(\chi)\cap \cN$ is a symplectic resolution. Let $\mathcal{B}_\chi \subset \widetilde{S}(\chi)$ be the preimage  $\pi_{\mathfrak{b}}^{-1}(\chi)$, it will be called the {\emph{Springer fiber}} (corresponding to $\chi$). The natural projection map $\widetilde{S}_{\mathfrak{h}^*}(\chi) \rightarrow \mathfrak{h}^*$ is the graded Poisson deformation of $\widetilde{S}(\chi)$ over $\mathfrak{h}^*$. 
   
\begin{rmk}\label{universal deformation of Slodowy slice}
    In most cases $H^2(\widetilde{S}(\chi), \CC)\simeq \fh^*$, and  $\widetilde{S}_{\mathfrak{h}^*}(\chi) \rightarrow \mathfrak{h}^*$ is in fact the universal Poisson deformation of $\widetilde{S}(\chi)$. When the isomorphism above does not hold, either $\chi$ is a regular nilpotent element, or the Slodowy variety $\widetilde{S}(\chi)$ can be realized as the Slodowy variety $\widetilde{S}(\chi')$ for a nilpotent element $\chi'\in \fg'^*$ for some simple Lie algebra $\fg'^*$, such that $\rk \fg'^* > \rk \fg$, and $H^2(\widetilde{S}(\chi'), \CC)\simeq \fh'^*$. We refer interested readers to \cite{LehnNamikawaSorger} for more details.
\end{rmk}

    We will need an alternative definition of $S(\chi)$. Namely, we will realize it as a Hamiltonian reduction of $\mathfrak{g}^*$ by a certain unipotent subgroup $U_\ell \subset G$ depending on some vector space $\ell \subset \mathfrak{g}(-1)$ (main reference is \cite{GanGinzburg}, note that our $U_\ell$ is their $M_\ell$).

    There is a a nondegenerate skew-symmetric bilinear form  on $\mathfrak{g}(-1)$ sending a pair $(x,y) \in \mathfrak{g}(-1) \times \mathfrak{g}(-1)$ to $\chi([x,y])$. Fix an isotropic subspace $\ell \subset \mathfrak{g}(-1)$, and let 
    \begin{equation*}
    \mathfrak{u}_\ell=\ell \oplus \bigoplus_{i \leqslant -2}  \mathfrak{g}(i),
    ~\mathfrak{n}_{\ell}=\ell^{\perp} \oplus \bigoplus_{i \leqslant -1}  \mathfrak{g}(i).
    \end{equation*}
    \begin{rmk}
     Note that for $\ell=0$ we have $u_{0}=\mathfrak{u}=\bigoplus_{i \leqslant -2}  \mathfrak{g}(i)$, and $\mathfrak{n}_{0}=\mathfrak{n}=\bigoplus_{i \leqslant -1}  \mathfrak{g}(i)$. For a Lagrangian $\ell$, we have $\mathfrak{u}_{\ell}=\mathfrak{n}_{\ell}$.   
    \end{rmk}

    Let $U_{\ell}, N_{\ell} \subset G$ be subgroups of $G$ with Lie algebras $\mathfrak{u}_{\ell}, \mathfrak{n}_{\ell}$ respectively; $U_\ell$ naturally acts on $\mathfrak{g}^*$, $\widetilde{\mathfrak{g}}^*$ and we can consider moment maps
    $\mu_{\mathfrak{g}^*}\colon \mathfrak{g}^* \rightarrow \mathfrak{u}_{\ell}^*,\, \mu_{\widetilde{\mathfrak{g}}^*}\colon \widetilde{\mathfrak{g}}^* \rightarrow  \mathfrak{u}_{\ell}^*,\, \mu_{T^*\mathcal{B}}\colon T^*\mathcal{B} \rightarrow \mathfrak{u}_\ell^*$.

    \begin{prop}\label{prop_realiz_S_as_Ham_red}
    The coadjoint action maps 
    \begin{equation*}
    N_\ell \times S(\chi) \iso \mu^{-1}_{\mathfrak{g}^*}(\chi),\,  N_\ell \times \widetilde{S}_{\mathfrak{h}}(\chi) \iso \mu^{-1}_{\widetilde{\mathfrak{g}}^*}(\chi),\, N_\ell \times \widetilde{S}(\chi) \iso \mu_{T^*\mathcal{B}}^{-1}(\chi)
    \end{equation*}
    are isomorphisms inducing identifications:
    \begin{equation*}
    S(\chi) \iso \mu_{\mathfrak{g}^*}^{-1}(\chi)/N_\ell,\, \widetilde{S}_{\mathfrak{h}^*}(\chi) \iso \mu_{\widetilde{\mathfrak{g}}^*}^{-1}(\chi)/N_\ell,\, \widetilde{S}(\chi) \iso \mu_{T^*\mathcal{B}}^{-1}(\chi)/N_\ell. 
    \end{equation*}
    \end{prop}
    \begin{proof}
    The same argument as in \cite[Lemma 2.1]{GanGinzburg} works, see also  \cite{Losev_isofquant}.
    \end{proof}

\begin{rmk}\label{rem_ham_real_of_parab_S}
   For a Lagrangian $\ell$, Proposition \ref{prop_realiz_S_as_Ham_red} realizes $S(\chi)$, $\widetilde{S}_{\mathfrak{h}^*}(\chi)$ as Hamiltonian reductions of $\mathfrak{g}^*$, $\widetilde{\mathfrak{g}}^*$ by $U_\ell$ at $\chi|_{\mathfrak{u}_\ell} \in \mathfrak{u}_\ell^*$, i.e., $S(\chi)=\mathfrak{g}/\!/\!/_{\chi}U_{\ell}$, $\widetilde{S}_{\mathfrak{h}^*}(\chi)=\widetilde{\mathfrak{g}}^*/\!/\!/\!_{\chi} U_{\ell}$, $\widetilde{S}(\chi)=T^*\mathcal{B}/\!/\!/_\chi U_{\ell}$. 
\end{rmk}

    \subsubsection{Parabolic Slodowy varieties}\label{parabolic Slodowy subsection}
    Let us now define {\emph{parabolic}} Slodowy varieties. They will depend on a choice of parabolic $T \subset P \subset G$. Recall  morphisms 
    $\pi_{\mathfrak{p}} \colon T^*\mathcal{P} \rightarrow \mathcal{N}$, $\widetilde{\pi}_{\mathfrak{p}}\colon T^*_{\mathfrak{X}(\mathfrak{l})}\mathcal{P} \rightarrow \mathfrak{g}^* \times_{\mathfrak{h}^*/W} \mathfrak{X}(\mathfrak{l})$.

    We define 
    \begin{equation*}
    \widetilde{S}(\chi,\mathfrak{p})=\pi_{\mathfrak{p}}^{-1}(S(\chi) \cap \mathcal{N}), \quad \widetilde{S}_{\mathfrak{X}(\mathfrak{l})}(\chi,\mathfrak{p}):=\widetilde{\pi}_{\mathfrak{p}}^{-1}(S_{\mathfrak{h}^*}(\chi) \times_{\mathfrak{h}^*} \mathfrak{X}(\mathfrak{p})).
    \end{equation*}
    We will call $\widetilde{S}(\chi,\mathfrak{p})$ a {\em{parabolic}} Slodowy variety (called $S3$ variety in \cite{w}).

    The following proposition is completely analogous to Proposition \ref{prop_realiz_S_as_Ham_red} above.
    Consider the natural action of $U_{\ell}$ on $T^*\cP$, $T^*_{\mathfrak{X}(\mathfrak{l})}\mathcal{P}$ and the corresponding moment maps $\mu_{T^*\mathcal{P}}\colon T^*\mathcal{P} \rightarrow \mathfrak{u}_{\ell}^*,\, \mu_{T^*_{\mathfrak{X}(\mathfrak{l})}\mathcal{P}}\colon T^*_{\mathfrak{X}(\mathfrak{l})}\mathcal{P} \rightarrow \mathfrak{u}_{\ell}^*$.

    \begin{prop}\label{prop_coadj_N_iso}
    The coadjoint action maps
    \begin{equation*}
    N_\ell \times \widetilde{S}_{\mathfrak{X}(\mathfrak{l})}(\chi,\mathfrak{p}) \iso \mu_{T^*_{\mathfrak{X}(\mathfrak{l})}\mathcal{P}}^{-1}(\chi), \quad N_\ell \times \widetilde{S}(\chi,\mathfrak{p}) \iso \mu^{-1}_{T^*\mathcal{P}}(\chi)
    \end{equation*}
    are isomorphisms inducing identifications:
    \begin{equation*}
    \widetilde{S}_{\mathfrak{X}(\mathfrak{l})}(\chi,\mathfrak{p}) \iso \mu_{T^*_{\mathfrak{X}(\mathfrak{l})}\mathcal{P}}^{-1}(\chi)/N_\ell, \quad \widetilde{S}(\chi,\mathfrak{p}) \iso \mu^{-1}_{T^*\mathcal{P}}(\chi)/N_\ell.
    \end{equation*}
    \end{prop}

\begin{lemma}\label{lem_fiber_parab_slod_gen}
For $\la \in \mathfrak{X}(\mathfrak{l})$ such that $Z_{\mathfrak{g}}(\la)=\mathfrak{l}$, the natural morphism $\widetilde{S}_{\la}(\chi,\mathfrak{p}) \ra S(\chi) \cap \mathbb{O}_{\la}$ is an isomorphism.
\end{lemma}
\begin{proof}
Recall that $\widetilde{S}_{\la}(\chi,\mathfrak{p})$ is the preimage of $S(\chi)$ under the morphism $T^*_\la\mathcal{P} \rightarrow \mathfrak{g}^*$. It follows from Lemma \ref{fiber of deformation of parabolic} that this morphism factors through the isomorphism $T^*_\la \mathcal{P} \iso {\mathbb{O}}_\la$. The desired follows.
\end{proof}

Set $S(\chi,\mathfrak{p}):=\on{Spec}\mathbb{C}[\widetilde{S}(\chi,\mathfrak{p})]$. The natural morphism $\widetilde{S}(\chi,\mathfrak{p}) \rightarrow S(\chi,\mathfrak{p})$ is a symplectic resolution. The contracting $\mathbb{C}^\times$-action on $S(\chi)$ induces the action on $\widetilde{S}(\chi,\mathfrak{p})=\pi_{\mathfrak{p}}^{-1}(S(\chi) \cap \mathcal{N})$ which, in turn, induces the action of $\mathbb{C}^\times$ on $S(\chi,\mathfrak{p})$.

\begin{lemma}
The $\mathbb{C}^\times$-action as above contracts every connected component of $S(\chi,\mathfrak{p})$ to the unique fixed point. 
\end{lemma}
\begin{proof}
Let us first of all note that the $\mathbb{C}^\times$-action contracts $\widetilde{S}(\chi,\mathfrak{p})$ to $\widetilde{S}(\chi,\mathfrak{p})^{\mathbb{C}^\times}$ (use that $\pi_{\mathfrak{p}}^{-1}(\chi)$ is proper and that we already know that $S(\chi)$ is conical). It follows that the algebra $\mathbb{C}[\widetilde{S}(\chi,\mathfrak{p}]$ is non-negatively graded. It remains to check that $\on{dim}\mathbb{C}[\widetilde{S}(\chi,\mathfrak{p}]_0 = |\on{Comp}\widetilde{S}(\chi,\mathfrak{p})|$. Pick  a degree $0$ function $f\colon \widetilde{S}(\chi,\mathfrak{p}) \rightarrow \mathbb{A}^1$. Consider the proper variety $\widetilde{S}(\chi,\mathfrak{p})^{\mathbb{C}^\times}$. Function $f$ is constant on every connected component of $\widetilde{S}(\chi,\mathfrak{p})^{\mathbb{C}^\times}$. Moreover, $f$ has degree $0$, so it must be constant on the attractor (with respect to the $\mathbb{C}^\times$-action) to every connected component of $\widetilde{S}(\chi,\mathfrak{p})^{\mathbb{C}^\times}$. We conclude that the image of $f$ is a finite set of points of $\mathbb{A}^1$. Every connected component of $\widetilde{S}(\chi,\mathfrak{p})$ is irreducible, so its image under $f$ is also irreducible, hence, a point.
\end{proof}

\begin{rmk}\label{disconnected}
    We note that $S(\chi, \fp)$ may have several connected components. Let us list two examples. The relation between them is explained in \cref{disconnected together}.
    \begin{itemize}
        \item[1)] Let $\fg=\fs\fp(6)$, pick $\chi$ corresponding to the partition $(3,3)$, and let $\fp$ be a parabolic with Levi factor $\fl=\fg\fl(1)^2\times \fs\fp(2)\subset \fs\fp(6)$. Then $S(\chi, \fp)$ is a disjoint union of two Kleinian singularities of type $A_1$.
        \item[2)] Let $\fg=\fs\fo(7)$, pick $\chi$ corresponding to the partition $(3,1,1,1,1)$, and let $\fp$ be a parabolic with Levi factor $\fl=\fg\fl(3)\subset \fs\fo(7)$. Then $S(\chi, \fp)$ is a disjoint union of two Kleinian singularities of type $A_1$.
    \end{itemize}
\end{rmk}

Let us now pick a pair of parabolic subalgebras $\mathfrak{p}, \mathfrak{q} \subset \mathfrak{g}$ sharing the same Borel $\mathfrak{b}$. containing the Cartan subalgebra $\mathfrak{h}$. Let $\mathfrak{m} \subset \mathfrak{q}$ be the standard Levi subalgebra of $\mathfrak{q}$. Let $\chi_{\mathfrak{m}} \in \mathfrak{m}^*$ be a regular nilpotent in $\mathfrak{m}^*$ that lies in $\mathfrak{b}^*$. For future use, we introduce a new notation for the parabolic Slodowy variety to $\chi_{\fm}$: 
\begin{equation*}
\widetilde{S}(\mathfrak{q},\mathfrak{p})=\widetilde{S}(\chi_{\mathfrak{m}},\mathfrak{p}).
\end{equation*}

\subsection{Quantizations of $X$ and $Y$} 
For a conical Poisson scheme $Z$ a {\it graded quantization} of the structure sheaf  $\mathcal{O}_{Z}$ is (see for example \cite[Section 2.2]{Losevwallcrossing}): 
 
 \begin{itemize}

 \item  a sheaf of $\CC[\hbar]$-algebras $\mathcal{D}_{\hbar}$ in conical topology on $Z$ equipped with a $\CC^\times$-action by algebra automorphisms such that $t \cdot \hbar = t^2 \hbar$,

 \item an isomorphism $\iota \colon \mathcal{D}_\hbar/(\hbar) \iso \CO_Z$ of sheaves

 \end{itemize} 
 such that:

	\begin{itemize}
		
		\item  $\mathcal{D}_\hbar$ is flat over $\CC[\hbar]$,
		
        \item $[\mathcal{D}_\hbar,\mathcal{D}_\hbar] \subset \hbar \mathcal{D}_\hbar$. This produces a Poisson bracket on $\mathcal{D}_\hbar/(\hbar)$ given by $\{a+\hbar\calD_\hbar, b+\hbar\calD_\hbar \}=\frac{[a,b]}{\hbar}+\hbar\calD_\hbar$.

        \item $\iota$ is a graded Poisson isomorphism.
		
	\end{itemize}
    We note that for affine scheme $Z$ taking the global sections of $\calD_\hbar$ we get the definition of a graded quantization of $A=\CC[Z]$. Let $X$ be a conical symplectic singularity, and assume that it admits a symplectic resolution $Y\to X$. In this section, we discuss graded quantizations of $\CC[X]$ and $\cO_{Y}$.
    
   The graded quantizations of $Y$ are parametrized by the vector space $\mathfrak{h}_X$ via the so called {\em{period map}} $\on{Per}$ from the isomorphism classes of quantizations to $\mathfrak{h}_X$ (see~\cite{BK},~\cite[Section~2.3]{Losev_isofquant}). We will denote the quantization that corresponds to $\la \in \mathfrak{h}_X$ by $\mathcal{D}_{\hbar, \la}$.
   Let $Y_{\mathfrak{h}_X} \ra \mathfrak{h}_X$ be the universal Poisson deformation of $Y$, see \cref{universal Poisson deformation}. We can talk about graded formal quantizations of $Y_{\mathfrak{h}_X}$ that are now required to be sheaves of $\CC^\times$-graded $\CC[\mathfrak{h}_X][\hbar]$-algebras, where the $\CC^\times$-action on $\mathfrak{h}_X$ is given by $t \cdot p = t^{-2}p$. It was shown in \cite[Section 6.2]{BK} that $\CO_{Y_{\mathfrak{h}_X}}$ admits a canonical graded quantization to be denoted by $\mathcal{D}_{\hbar, \mathfrak{h}_X}$. It satisfies the following property: its specialization to $\la \in \mathfrak{h}_X$ coincides with $\mathcal{D}_{\hbar, \la}$.

   \begin{warning}
       In this section, we talk about the graded quantizations, while all the provided references consider graded \emph{formal} quantizations. To get a graded quantization from a graded formal quantization, we consider the subsheaf of $\CC^\times$-finite elements. The inverse map is given by considering the completion along the ideal generated by $\hbar$. 
   \end{warning}
 We set $\mathcal{A}_{\hbar, \la}=\on{\Gamma}(Y,\mathcal{D}_{\hbar, \la})$. By the Grauert-Riemenschneider theorem, $R\Gamma(\mathcal{O}_{Y})=\mathbb{C}[X]$. From here one deduces that the algebra $\CA_{\hbar, \la}$ is a graded quantization of $\Gamma(Y,\CO_{Y})=\CC[X]$. Similarly, the algebra $\CA_{\hbar, \mathfrak{h}_X}=\Gamma(Y_{\mathfrak{h}_X},\mathcal{A}_{\hbar, \mathfrak{h}_X})$ is a graded quantization of $\CC[X_{\mathfrak{h}_X}]$ over $\CC[\fh_X]$.

In what follows, we often specialize graded quantizations to the parameter $\hbar=1$. To describe the resulting object, for a conical Poisson scheme $Z$, we can similarly define a \emph{filtered quantzation} of the sheaf $\cO_Z$ to be:

 \begin{itemize}

 \item  a sheaf of $\ZZ$-filtered algebras $\mathcal{D}$ in conical topology on $Z$,

 \item an isomorphism $\iota \colon \gr \mathcal{D} \iso \CO_Z$ of sheaves

 \end{itemize} 
 such that:

	\begin{itemize}
		
		\item the filtration $\calD_{\le i}$ on $\mathcal{D}$ is complete and separated, 
		
        \item $[\mathcal{D}_{\le i},\mathcal{D}_{\le j}] \subset \mathcal{D}_{\le i+j-2}$. This produces a Poisson bracket on $\gr \mathcal{D}$ given by
        \begin{equation*}
        \{a+\calD_{\le i-1}, b+\calD_{\le j-1}\}=[a,b]+\calD_{\le i+j-3}.
        \end{equation*}

        \item $\iota$ is a graded Poisson isomorphism.
		
	\end{itemize}

For a conical Poisson algebra $A$, we recover the definition of a filtered quantization of $A$ by taking the global sections of a filtered quantization of $\cO_{\Spec(A)}$. The filtered quantizations $\mathcal{D}_\la, \mathcal{D}_{\mathfrak{h}_X}, \CA_\la, \CA_{\mathfrak{h}_X}$ of our interest are obtained as the specializations of the graded quantizations $\mathcal{D}_{\hbar, \la}, \mathcal{D}_{\hbar, \mathfrak{h}_X}, \CA_{\hbar, \la}, \CA_{\hbar, \mathfrak{h}_X}$ at $\hbar=1$.


Let us also mention that for a vector space $\mathfrak{a}$ mapping linearly to $\mathfrak{h}_{X}$, we can consider the corresponding base changes $X_{\mathfrak{a}}$, $Y_{\mathfrak{a}}$ and the corresponding quantizations $\mathcal{D}_{\hbar,\mathfrak{a}}$, $\mathcal{A}_{\hbar,\mathfrak{a}}$. We will only deal with $\mathfrak{a}$ such that the following assumption holds.

\begin{assumption}\label{ass_resol_a}
There exists a nonempty open subset $U \subset \mathfrak{a}$ such that the morphism $Y_{U} \rightarrow X_{U}$ becomes an isomorphism. 
\end{assumption}

Recall (see Proposition \ref{namikawa's result_descr_of_h_sing} above) that in \cite[Main Theorem]{Namikawa_bir} Namikawa proved that the set $\mathfrak{h}_X^{\mathrm{sing}}$ of $\lambda \in \mathfrak{h}_X$ for which the map $Y_\la \rightarrow X_\la$ fails to be an isomorphism is the union of certain hyperplanes. He also proves that the complement to these hyperplanes contains the ample cone of $Y$ (again see Proposition \ref{namikawa's result_descr_of_h_sing} above, also compare with \cite[Lemma 2.5]{Kaledin2008}). In particular, we see that Assumption \ref{ass_resol_a} holds if the following assumption holds. 

\begin{assumption}\label{ass_ample_chern_a}
The image of $\mathfrak{a}$ in $\mathfrak{h}_X$ contains the Chern class $\chi \in \mathfrak{h}_{X,\mathbb{Z}}$ of an ample line bundle $\mathcal{L} \in \on{Pic}(Y)$.
\end{assumption}

It follows from Lemma \ref{lem_fiber_parab_slod_gen} that the Assumption \ref{ass_ample_chern_a} (hence, the Assumption \ref{ass_resol_a}) is satisfied for the family $\widetilde{S}_{\mathfrak{X}(\mathfrak{l})}(\chi,\mathfrak{p}) \rightarrow S_{\mathfrak{X}(\mathfrak{l})}(\chi,\mathfrak{p})$. 

\begin{lemma}
The Assumption \ref{ass_ample_chern_a} holds for $X=\widetilde{S}(\chi,\mathfrak{p})$, $Y=S(\chi,\mathfrak{p})$, and $\mathfrak{a}=\mathfrak{X}(\mathfrak{l})=H^2(T^*\mathcal{P},\mathbb{C})$ mapping naturally to $\mathfrak{h}_X=H^2(\widetilde{S}(\chi,\mathfrak{p}),\mathbb{C})$. 
\end{lemma}
\begin{proof}
Recall that $\widetilde{S}(\chi,\mathfrak{p})$ is a closed subvariety of $T^*\mathcal{P}$ so every ample line bundle on $T^*\mathcal{P}$ restricts to an ample line bundle on $\widetilde{S}(\chi,\mathfrak{p})$. The claim follows. 
\end{proof}

\subsubsection{Quantizations of $T^*\mathcal{P}$}\label{quant_G_mod_P}
Fix a parabolic subgroup $T \subset B \subset P \subset G$ and recall $\mathcal{P}=G/P$.
In this subsection, we describe quantizations of $T^*\mathcal{P}$. 

Let $\mathfrak{p}=\operatorname{Lie}P$, and let $\mathfrak{u}_{\mathfrak{p}}$ be the unipotent radical, $\mathfrak{l}=\mathfrak{p}/\mathfrak{u}_{\mathfrak{p}}$ be the Levi quotient. There is a natural isomorphism $H^2(T^*\mathcal{P},\CC) \simeq H^2(\mathcal{P},\CC) \simeq \mathfrak{X}(\mathfrak{l})$, where $\mathfrak{X}(\mathfrak{l})=(\mathfrak{l}/[\mathfrak{l},\mathfrak{l}])^*$. 
  Note that the map $\fh\to \fl\to \fl/[\fl,\fl]$ induces a natural embedding $\mathfrak{X}(\mathfrak{l}) \subset \mathfrak{h}^*$.

Recall that by construction $\mathcal{D}_{\hbar,\mathfrak{X}(\mathfrak{l})}(T^*\mathcal{P})$ is the sheaf on $T^*_{\mathfrak{X}(\mathfrak{l})}\mathcal{P}= G \times_P [\mathfrak{p},\mathfrak{p}]^{\perp}$. The cotangent bundle $T^*(G/[P,P])\simeq G\times^{[P,P]} [\fp,\fp]^\perp$ admits a tautological map $p$ to $ T^*_{\mathfrak{X}(\mathfrak{l})}\mathcal{P}$ given by $[(g, \xi)] \mapsto [(g, \xi)]$. Moreover, $p\colon T^*(G/[P,P]) \rightarrow T^*_{\mathfrak{X}(\mathfrak{l})}\mathcal{P}$ is a principal $P/[P,P]$-bundle. Let $\mathcal{D}_{\hbar}(G/[P,P])$ be the sheaf of homogeneous differential operators on $G/[P,P]$, and $\underline{\mathcal{D}}_{\hbar}(G/[P,P])$ be the microlocalization of $\mathcal{D}_{\hbar}(G/[P,P])$ to a quantization of $T^*(G/[P,P])$ (see, for example, \cite{Ginzburg1986} for the details on the microlocalization). Now, 
\begin{equation*}
\mathcal{D}_{\hbar,\mathfrak{X}(\mathfrak{l})}(T^*\mathcal{P}) = p_*(\underline{\mathcal{D}}_{\hbar}(G/[P,P]))^{P/[P,P]}.
\end{equation*}

\begin{rmk}
For $P=B$, the definition above tells us that     $\mathcal{D}_{\hbar,\mathfrak{h}^*}(T^*\mathcal{B})$ is the sheaf of $T$-invariants in $\underline{\mathcal{D}}_{\hbar}(T^*(G/U))$.
\end{rmk}

Algebra $\mathcal{A}_{\hbar,\mathfrak{X}(\mathfrak{l})}(T^*\mathcal{P})$ is by definition equal to global sections of $\mathcal{D}_{\hbar,\mathfrak{X}(\mathfrak{l})}(T^*\mathcal{P})$. By the construction above,  $\mathcal{A}_{\hbar,\mathfrak{X}(\mathfrak{l})}(T^*\mathcal{P})\simeq \Gamma(\mathcal{P},\mathcal{D}_\hbar(G/[P,P]))^{P/[P,P]}$. The map $\CC[\mathfrak{X}(\mathfrak{l}),\hbar] \rightarrow \mathcal{A}_{\hbar,\mathfrak{X}(\mathfrak{l})}(T^*\mathcal{P})$ is the quantum comoment map for the {\emph{right}} action of $P/[P,P]$ on $G/[P,P]$.

Consider the quantum comoment map for the left action of $G$ on $G/[P,P]$. It induces the homomorphism $\mathcal{U}_\hbar(\mathfrak{g}) \rightarrow \mathcal{A}_{\hbar,\mathfrak{X}(\mathfrak{l})}(T^*\mathcal{P})$ from the homogenized universal enveloping algebra of $\mathfrak{g}$. Let $Z_\hbar \subset \mathcal{U}_\hbar(\mathfrak{g})$ be the center. Let $\HC_\hbar\colon Z_\hbar \iso \CC[\mathfrak{h}^*,\hbar]^{W}$ be the homogeneous version of the Harish-Chandra isomorphism, where the action of $W$ on $\CC[\mathfrak{h}^*,\hbar]$ is given by: 
\begin{equation}\label{rho_shifted_action_hbar}
w.s(\la,\hbar_0)=s(w^{-1}(\la+\hbar_0\rho_{\mathfrak{g}} )-\hbar_0\rho_{\mathfrak{g}},\hbar_0)
\end{equation}
for $\la \in \mathfrak{h}^*$, $\hbar_0\in \CC$, $s\in \CC[\mathfrak{h}^*,\hbar]$ and $\rho_{\mathfrak{g}}$ being the half sum of positive roots w.r.t. the fixed Borel $B$. Using the identification $HC_\hbar$, we obtain the homomorphism $Z_\hbar \rightarrow \CC[\mathfrak{X}(\mathfrak{l}),\hbar]$ given by the composition $Z_\hbar \iso \CC[\mathfrak{h}^*,\hbar]^{W_G} \hookrightarrow \CC[\mathfrak{h}^*,\hbar] \twoheadrightarrow \CC[\mathfrak{X}(\mathfrak{l}),\hbar]$, where the latter map is induced by the embedding $\mathfrak{X}(\mathfrak{l}) \hookrightarrow \mathfrak{h}^*$. 

\begin{lemma}\label{compat_two_comoment}
$(a)$ The following diagram commutes
\begin{equation*}
\xymatrix{ Z_\hbar \ar[d] \ar[r] & \CC[\mathfrak{X}(\mathfrak{l}),\hbar] \ar[d] \\
 \mathcal{U}_\hbar(\mathfrak{g}) \ar[r] & \CA_{\hbar,\mathfrak{X}(\mathfrak{p})}(T^*\mathcal{P}) }
\end{equation*}


$(b)$ For $P=B$, the induced homomorphism $\mathcal{U}_\hbar(\mathfrak{g}) \otimes_{Z_\hbar} \CC[\mathfrak{h}^*,\hbar] \rightarrow \mathcal{A}_{\hbar,\mathfrak{h}^*}(T^*\mathcal{B})$ is an isomorphism. 
\end{lemma}
\begin{proof}
Part $(a)$ follows from \cite[proof of Lemma 5.2.1 (3)]{Losev_isofquant}. Part $(b)$ follows from \cite[Lemma 5.2.1 (4)]{Losev_isofquant}.
\end{proof}

Setting $\CA_{\hbar,\mathfrak{X}(\mathfrak{l})}(T^*\mathcal{B})=\CA_{\hbar,\mathfrak{h}^*}(T^*\mathcal{B}) \otimes_{\CC[\mathfrak{h}^*,\hbar]} \CC[\mathfrak{X}(\mathfrak{l}),\hbar]$ and using Lemma \ref{compat_two_comoment}, we obtain the homomorphism of graded $\CC[\mathfrak{X}(\mathfrak{l}),\hbar]$-algebras: 
\begin{equation*}
\Phi_{\mathfrak{p}}\colon \CA_{\hbar,\mathfrak{X}(\mathfrak{l})}(T^*\mathcal{B}) \rightarrow \CA_{\hbar,\mathfrak{X}(\mathfrak{l})}(T^*\mathcal{P}).
\end{equation*}
For $\hbar=0$, the homomorphism above identifies with the pull back homomorphism for the natural morphism $G \times^{P} [\mathfrak{p},\mathfrak{p}]^{\perp} \rightarrow \mathfrak{g}^* \times_{\fh^*/{W}} \mathfrak{X}(\mathfrak{l})$.

 \begin{rmk}\label{surj_Phi_hbar_nonzero}
In type $A$, the homomorphism  $\Phi_{\mathfrak{p}}$ is surjective (see \cite[Lemma 5.2.1 (5)]{Losev_isofquant}).
The homomorphism $\Phi_{\mathfrak{p}}$ is not surjective in general, but it is surjective for Zariski generic $\la \in \mathfrak{X}(\mathfrak{l})$. 
It will also be surjective after setting $\hbar=1$ and restricting to {\emph{integral}} $\la$ such that $\widetilde{\lambda}$ is $\mathfrak{p}$-antidominant and $\mathfrak{l}$-regular 
(see~\cite[Theorem 3.8]{BoBr}).

Similar statements about surjectivity of $\Phi_{\mathfrak{p}}$ has appeared in \cite[Corollary 3.9]{BoBr} and \cite[Remark 5.2.2]{Losev_isofquant}. Unfortunately, the results stated may not be precise. In particular, one should not expect $\Phi_{\mathfrak{p}}$ is surjective for all integral $\lambda\in \mathfrak{X}(\mathfrak{l})$. A counterexample is the following.

Consider $\fg= \fso_5$ and $\fl= \fgl_2\subset \fg$. Then $\fl/[\fl,\fl]$ is one-dimensional. Let $\lambda$ be the linear function on $\fl$ given by $\lambda(X)= \on{tr}(X)$. This choice corresponds to the canonical quantization in \cite[Section 5.1]{LMBM}, in particular $\lambda=\rho_\fg- \rho_\fl$. Now, for a parabolic  $P$ that contains $L$, consider the map $\pi_{\mathfrak{p}}\colon T^*\mathcal{P} \rightarrow \fg^*$. The image of $\pi_{\mathfrak{p}}$ is the closure of $\OO\subset \fg^*$, the subregular nilpotent orbit of $\fso_5$. On the other hand, $\widetilde{\OO}:= \pi_{\mathfrak{p}}^{-1}(\OO)$ is a double cover of $\OO$. Let $\Pi$ be $\Aut_\OO(\widetilde{\OO})$. From \cite[Corollary 6.5.6]{LMBM}, the image of $\Phi_{\fp}$ is $(\cA_{\hbar, \lambda}(T^*\cP))^{\Pi}\subset \cA_{\hbar, \lambda}(T^*\cP)$. The action of $\Pi$ on $\cA_{\hbar, \lambda}(T^*\cP)$ is lifted from the non-trivial action of $\Spec(\CC[T^*\cP])=\Spec(\CC[\widetilde{\OO}])$, and $\gr (\cA_{\hbar, \lambda}(T^*\cP))^\Pi\simeq \Spec(\CC[\OO]).$ In particular, $\Phi_\fp$ is not surjective. 
 \end{rmk}

We have described the quantizations $\mathcal{D}_{\hbar,\mathfrak{X}(\mathfrak{l})}(T^*\mathcal{P})$,  $\mathcal{A}_{\hbar,\mathfrak{X}(\mathfrak{l})}(T^*\mathcal{P})$. It is {\emph{not}} true that the specialization of $\mathcal{D}_{\hbar,\mathfrak{X}(\mathfrak{l})}(T^*\mathcal{P})$ at $\la \in \mathfrak{X}(\mathfrak{l})$ has period $\la$. For that to be true, we have to {\em{modify}} the $\CC[\mathfrak{X}(\mathfrak{l}),\hbar]$-algebra structure on $\mathcal{D}_{\hbar,\mathfrak{X}(\mathfrak{l})}(T^*\mathcal{P})$ as follows.
Let $\rho_{\mathfrak{u}_{\mathfrak{p}}}$ be the half sum of all $T$-weights of $\mathfrak{u}_{\mathfrak{p}}$. Recall that we can identify $\mathfrak{X}(\mathfrak{l})^*=\mathfrak{l}/[\mathfrak{l},\mathfrak{l}]$ with the center $\mathfrak{z}(\mathfrak{l})$.
Let $\iota\colon \mathbb{C}[\mathfrak{X}(\mathfrak{l}),\hbar]=S^{\bullet}(\mathfrak{z}(\mathfrak{l}))[\hbar] \rightarrow \mathcal{A}_{\hbar,\mathfrak{X}(\mathfrak{l})}(T^*\mathcal{P})$ be the quantum comoment map for the right action of $P/[P,P]$ on $G/[P,P]$).
Consider the map $\iota_{\mathfrak{p}} \colon \mathbb{C}[\mathfrak{X}(\mathfrak{l}),\hbar] \rightarrow \mathcal{A}_{\hbar,\mathfrak{X}(\mathfrak{l})}(T^*\mathcal{P})$ defined as the composition of  $\iota$ with the automorphism of $\mathbb{C}[\mathfrak{X}(\mathfrak{l}),\hbar]$ induced by the isomorphism $\mathfrak{X}(\mathfrak{l}) \oplus \mathbb{C} \iso \mathfrak{X}(\mathfrak{l}) \oplus \mathbb{C}$ given by $(\lambda,\hbar_0) \mapsto (\lambda-\hbar_0\rho_{\mathfrak{u}_{\mathfrak{p}}},\hbar_0)$.
From now on, we will consider the $\CC[\mathfrak{X}(\mathfrak{l}),\hbar]$-module structure on $\mathcal{D}_{\hbar,\mathfrak{X}(\mathfrak{l})}(T^*\mathcal{P})$, $\mathcal{A}_{\hbar,\mathfrak{X}(\mathfrak{l})}(T^*\mathcal{P})$ induced by $\iota_{\mathfrak{p}}$.

\begin{rmk}
Note that $\rho_{\mathfrak{u}_{\mathfrak{p}}}=\rho_{\mathfrak{g}}-\rho_{{\mathfrak{l}}}$, where $\rho_{\mathfrak{l}}$ is the half sum of positive roots for $\mathfrak{l}$.
\end{rmk}

\begin{warning}\label{warning_shift}
Note that we have shifted the $\CC[\mathfrak{X}(\mathfrak{l}),\hbar]$-, $\CC[\mathfrak{h}^*,\hbar]$-module structures on $\mathcal{A}_{\hbar,\mathfrak{X}(\mathfrak{l})}(T^*\mathcal{P})$, $\mathcal{A}_{\hbar,\mathfrak{h}^*}(T^*\mathcal{B})$, so it is {\emph{not}} true anymore that the homomorphism $\Phi_{\mathfrak{p}}$ is $\CC[\mathfrak{X}(\mathfrak{l}),\hbar]$-linear. Consider the homomorphism $\mathbb{C}[\mathfrak{h}^*,\hbar] \rightarrow \mathbb{C}[\mathfrak{X}(\mathfrak{l}),\hbar]$ induced by the embedding $\mathfrak{X}(\mathfrak{l}) \oplus \mathbb{C} \hookrightarrow \mathfrak{h}^* \oplus \mathbb{C}$ given by $(\lambda,\hbar_0) \mapsto (\lambda-\rho_{\mathfrak{l}}\hbar_0,\hbar_0)$.  
After composing this homomorphism with $\iota_{\mathfrak{p}}$, the induced $\mathbb{C}[\mathfrak{h}^*,\hbar]$-action on $\mathcal{A}_{\hbar,\mathfrak{X}(\mathfrak{l})}(T^*\mathcal{P})$ makes the homomorphism $\Phi_{\mathfrak{p}}$ $\mathbb{C}[\mathfrak{h}^*,\hbar]$-linear.
In particular, for $\la \in \mathfrak{X}(\mathfrak{l})$, specialization of $\Phi_{\mathfrak{p}}$ at $(\la, \hbar_0)$ induces the homomorphism  $\mathcal{A}_{\hbar_0,\la-\hbar_0\rho_{{\mathfrak{l}}}}(T^*\mathcal{B}) \rightarrow \mathcal{A}_{\hbar_0,\la}(T^*\mathcal{P})$.
\end{warning}


For $\la \in \mathfrak{X}(\mathfrak{l})$ let $\mathcal{D}_{\la}(\cP)$ be the sheaf of $\la+\rho_{\mathfrak{u}_{\mathfrak{p}}}$-twisted differential operators on $\cP$, and let $\mathcal{D}_{\hbar, \la}(\cP)$ be its homogeneous version. We denote by $D_\la(\mathcal{P})$, $D_{\hbar,\la}(\mathcal{P})$ the global sections of $\mathcal{D}_\la(\mathcal{P})$, $\mathcal{D}_{\hbar,\la}(\mathcal{P})$ respectively.
\begin{prop}\label{quant_flag}
	Fix $\la \in \mathfrak{X}(\mathfrak{l})$. Then the microlocalization $\underline{\mathcal{D}}_\la(\mathcal{P})$ (see~\cite[Section~4.1]{BPW} and references therein) of
	$\mathcal{D}_\la(\mathcal{P})$ to $T^*\mathcal{P}$ coincides with $\mathcal{D}_{1,\la}(T^*\mathcal{P})$ and is the filtered quantization of $X=T^*\mathcal{P}$ 
	with period $\la$. In particular, $\mathcal{A}_{\la}(T^*\mathcal{P})=D_\la(\mathcal{P})$. 
	\end{prop}
	\begin{proof}
 It follows from the definition that $\mathcal{D}_{1,\la}(T^*\mathcal{P})$ is the microlocalization of $\mathcal{D}_\la(\mathcal{P})$ to $T^*\mathcal{P}$.
	It follows from~\cite[Proposition~4.4]{BPW} that the period of $\mathcal{D}_\la(T^*\mathcal{P})$ equals  $\la+\rho_{\mathfrak{u}_{\mathfrak{p}}}-\frac{c_1(T^*\mathcal{P})}{2}$. Since $T^*\mathcal{P}$ is isomorphic to the induced vector bundle $G \times^{P} \mathfrak{p}^{\perp}$, and $\mathfrak{p}^{\perp}=(\mathfrak{g}/\mathfrak{p})^* \simeq \mathfrak{u}_{\mathfrak{p}}$ (as $T$-modules), we have $c_1(T^*(\mathcal{P}))=2\rho_{\mathfrak{u}_{\mathfrak{p}}}$. 
	\end{proof}

\begin{rmk}
Proposition \ref{quant_flag} also follows from \cite[Theorem 5.4.1]{Losev_isofquant}.    
\end{rmk}

 For $\la \in \mathfrak{X}(\mathfrak{l})$, let $\widetilde{\la}=\la-\rho_{\mathfrak{l}}$. We fix the notations
 \begin{equation*}
     \mathcal{U}_{\widetilde{\la}}=\mathcal{U}_{\widetilde{\la}}(\mathfrak{g})=D_{\widetilde{\la}}(\mathcal{B}), \quad	\mathcal{U}_{\hbar,\widetilde{\la}}=\mathcal{U}_{\hbar,\widetilde{\la}}(\mathfrak{g})=D_{\hbar,\widetilde{\la}}(\mathcal{B}).
 \end{equation*} 

	\begin{rmk}\label{rmk twisting}
	{\emph{The algebra $\mathcal{U}_{\hbar,\widetilde{\la}}$ can be described as follows. Recall that $\mathcal{U}_{\hbar}(\mathfrak{g})$ is the homogenized universal enveloping algebra of $\mathfrak{g}$, and $Z_\hbar \subset \mathcal{U}_{\hbar}(\mathfrak{g})$ is the center of $\mathcal{U}_{\hbar}(\mathfrak{g})$.
	Consider the {\emph{twisted}} Harish-Chandra isomorphism $\widetilde{HC}_{\hbar}\colon Z_\hbar \iso \CC[\mathfrak{h}^*]^{W} \otimes \CC[\hbar]$, where the action of $W$ is {\emph{standard}}.
	For $\la \in \mathfrak{h}^*$
	we define a central character $\mu_{\widetilde{\la}}\colon Z_{\hbar} \ra \CC[\hbar]$ by $\mu_{\widetilde{\la}}(z)=\widetilde{HC}_{\hbar}(z)(\widetilde{\la})$.
	Then $\mathcal{U}_{\hbar,\widetilde{\la}}$ is isomorphic to the central reduction of $\mathcal{U}_\hbar$ via $\mu_{\widetilde{\la}}$.
	Note also that for any $w \in W$, $\mathcal{U}_{\hbar,\la}=\mathcal{U}_{\hbar, w(\la)}$}}. 
	\end{rmk}


	\subsubsection{Quantizations of parabolic Slodowy varieties}\label{quant_slod_var}
    We fix a nilpotent element $e \in \mathfrak{g}$.
	Consider the corresponding parabolic Slodowy variety $\widetilde{S}(\chi,\mathfrak{p})$ (see Section \ref{sec_parabolic_slodowy}).
	Recall that $\widetilde{S}(\chi,\mathfrak{p})$ can be obtained as a Hamiltonian reduction of $T^*\mathcal{P}/\!/\!/_\chi U_\ell$ (see Proposition \ref{prop_realiz_S_as_Ham_red} and Remark \ref{rem_ham_real_of_parab_S}).
    Recall also that $\widetilde{S}_{\mathfrak{X}(\mathfrak{l})}(\chi,\mathfrak{p})=T^*_{\mathfrak{X}(\mathfrak{l})}\mathcal{P}/\!/\!/_\chi U_\ell$ is the base change of the universal deformation of $\widetilde{S}(\chi,\mathfrak{p})$ via the pullback map $p\colon \mathfrak{X}(\mathfrak{l})=H^2(T^*\mathcal{P},\CC) \rightarrow H^2(\widetilde{S}(\chi,\mathfrak{p}),\CC)$.

 \begin{rmk}\label{map_cohom}
	{\em{Note that in type $A$ the map $p$ is surjective ({\em{see}}~\cite[Section 2]{bo}). In general, we do not know when the morphism $p$ is surjective.}}   
	\end{rmk}

    We then consider quantum Hamiltonian reductions of the natural quantizations of $T^*_{\mathfrak{X}(\mathfrak{l})}\mathcal{P}$ and $\CC[T^*_{\mathfrak{X}(\mathfrak{l})}\mathcal{P}]$, see \cite[Section 3.3, 3.4]{Losev_isofquant} for generalities on quantum Hamiltonian reduction. Define
   \begin{equation*}
\mathcal{D}_{\hbar,\mathfrak{X}(\mathfrak{l})}(\widetilde{S}(\chi,\mathfrak{p}))=\mathcal{D}_{\hbar,\mathfrak{X}(\mathfrak{l})}(T^*\mathcal{P})/\!/\!/_{\chi} U_\ell, \quad  \mathcal{D}_{\mathfrak{X}(\mathfrak{l})}(\widetilde{S}(\chi,\mathfrak{p}))=\mathcal{D}_{\mathfrak{X}(\mathfrak{l})}(T^*\mathcal{P})/\!/\!/_{\chi} U_\ell,
   \end{equation*}
\begin{equation*}
\mathcal{A}_{\hbar,\mathfrak{X}(\mathfrak{l})}(\widetilde{S}(\chi,\mathfrak{p}))=\mathcal{A}_{\hbar,\mathfrak{X}(\mathfrak{l})}(T^*\mathcal{P})/\!/\!/_{\chi} U_\ell, \quad  \mathcal{A}_{\mathfrak{X}(\mathfrak{l})}(\widetilde{S}(e,\mathfrak{p}))=\mathcal{A}_{\mathfrak{X}(\mathfrak{l})}(T^*\mathcal{P})/\!/\!/_{\chi} U_\ell,
\end{equation*}
\begin{equation*}
\cW_{\hbar,\la}(\chi,\mathfrak{p}) = D_{\hbar,\lambda}(\mathcal{P})/\!/\!/_\chi U_{\ell}, \quad \cW_{\la}(\chi,\mathfrak{p}) = D_\lambda(\mathcal{P})/\!/\!/_\chi U_{\ell}.
\end{equation*}

\begin{rmk}
It is straightforward to check (using that the assumptions of \cite[Lemma 3.1.1]{Losev_isofquant} are satisfied in our situation) that we have the canonical identification: 
\begin{equation*}
\mathcal{D}_{\hbar,\mathfrak{X}(\mathfrak{l})}(\widetilde{S}(\chi,\mathfrak{p}))/(\hbar-1) \iso \mathcal{D}_{\mathfrak{X}(\mathfrak{l})}(\widetilde{S}(\chi,\mathfrak{p}))
\end{equation*}
and similar identifications  for $\CA_{\mathfrak{X}(\mathfrak{l})}(\widetilde{S}(\chi,\mathfrak{p}))$,  and $\mathcal{W}_{\la}(\chi,\mathfrak{p})$.
\end{rmk}

The grading on the algebras above comes from the $\CC^\times$-action on $T^*\mathcal{P}$ corresponding to the cocharacter $\CC^\times \rightarrow T \times \CC^\times$ given by $t \mapsto (h(t),t)$ (recall that $h \in \mathfrak{g}$ is the Cartan element of the $\mathfrak{sl}_2$-triple for $e$, abusing the notation we denote by $h\colon \CC^\times \rightarrow T$ the cocharacter corresponding to $h \in \mathfrak{h}$).

   The following proposition is standard.  
   \begin{prop}\label{prop_quant_W_alg}
   $(a)$ The sheaf $\mathcal{D}_{\hbar,\mathfrak{X}(\mathfrak{l})}(\widetilde{S}(\chi,\mathfrak{p}))$ quantizes $\widetilde{S}_{\mathfrak{X}(\mathfrak{l})}(\chi,\mathfrak{p})$ over $\mathfrak{X}(\mathfrak{l})$.
   
   $(b)$ We have:
   \begin{equation*}
\mathcal{A}_{\hbar,\mathfrak{X}(\mathfrak{l})}(\widetilde{S}(\chi,\mathfrak{p}))=\Gamma(\widetilde{S}(\chi,\mathfrak{p}),\mathcal{D}_{\hbar,\mathfrak{X}(\mathfrak{l})}(\widetilde{S}_{\mathfrak{X}(\mathfrak{l})}(\chi,\mathfrak{p}))),
\end{equation*} 
\begin{equation*}
\mathcal{A}_{\hbar,\la}(\widetilde{S}(\chi,\mathfrak{p})) = \mathcal{W}_{\hbar,\la}(\chi,\mathfrak{p}), \quad \mathcal{A}_{\la}(\widetilde{S}(\chi,\mathfrak{p})) = \mathcal{W}_{\la}(\chi,\mathfrak{p}).
\end{equation*}

   $(c)$ For $\la \in \mathfrak{X}(\mathfrak{l})$, the specialization $\mathcal{D}_{\hbar,\la}(\widetilde{S}(\chi,\mathfrak{p}))$ is the quantization of $\widetilde{S}(\chi,\mathfrak{p})$ with period $p(\la)$.
   \end{prop}
\begin{proof}
Part $(a)$ is standard (see, for example, \cite[Section 3.3]{Losev_isofquant}) and follows from the freeness of the action of $U_\ell$ on the fiber of the moment map over $\chi$ (see Proposition \ref{prop_coadj_N_iso} above). 
First equality in part $(b)$ follows from the fact that the quantum Hamiltonian reduction functors intertwine the global section functors, compare with \cite[Section 3.3.4]{Losev_cacticells}. 

The identifications $\cA_{\hbar,\la}(\widetilde{S}(\chi,\mathfrak{p})) = \mathcal{W}_{\hbar,\la}(\chi,\mathfrak{p})$, $\cA_{\la}(\widetilde{S}(\chi,\mathfrak{p})) = \mathcal{W}_{\la}(\chi,\mathfrak{p})$ are also standard (see \cite[Section 3.3]{Losev_isofquant}) and follow from the fact that the assumptions of Lemma 3.3.1 in loc. cit. are satisfied in our situation (see Proposition \ref{prop_coadj_N_iso} above).

	Part $(c)$ follows from the general results on periods of quantum Hamiltonian reductions (see 
	\cite[Section 5]{Losev_isofquant}) together with Proposition \ref{quant_flag}. 
	\end{proof}

Let us give an alternative description of the algebra $\CA_{\hbar,\mathfrak{h}^*}(\widetilde{S}(\chi))$. 
Consider the finite $W$-algebra: 
\begin{equation*}
\mathcal{W}_\hbar(\chi)=\mathcal{U}_\hbar(\mathfrak{g})/\!/\!/_{\chi}U_{\ell}=(\mathcal{U}_\hbar(\mathfrak{g})/\mathcal{U}_\hbar(\mathfrak{g})\{x-\chi(x)\,|\, x \in \mathfrak{u}_{\ell}\})^{U_\ell}.
\end{equation*} 
It is well-known that the center of $\mathcal{W}_\hbar(\chi)$ identifies with $Z_\hbar$ via the natural embedding $Z_{\hbar} \hookrightarrow \mathcal{W}_\hbar(\chi)$ (see, for example,  the footnote on page 35 of \cite{Premetstabilizer} attributing it to Ginzburg). It then follows from the definitions (c.f. \cite[Section 5.1]{Losev_isofquant}) together with Lemma \ref{compat_two_comoment} that the natural maps $\mathcal{W}_\hbar(\chi) \rightarrow \mathcal{A}_{\hbar,\mathfrak{h}^*}(\widetilde{S}(\chi))$, $\CC[\mathfrak{h}^*,\hbar] \rightarrow \mathcal{A}_{\hbar,\mathfrak{h}^*}(\widetilde{S}(\chi))$ induce the isomorphism:
\begin{equation*}
\mathcal{W}_{\hbar}(\chi) \otimes_{Z_\hbar} \CC[\mathfrak{h}^*,\hbar] \iso \mathcal{A}_{\hbar,\mathfrak{h}^*}(\widetilde{S}(\chi)).
\end{equation*}

\begin{warning}
Let us emphasize that the grading on     $\mathcal{W}_{\hbar}(\chi)$ is {\emph{not}} induced by the standard grading on $\mathcal{U}_\hbar(\mathfrak{g})$, it is induced by the {\emph{Kazhdan}} grading on $\mathfrak{g}$ defined as follows. Recall the $\ad h$-eigenspace decomposition $\mathfrak{g}=\bigoplus_{i}\mathfrak{g}(i)$, then the degree of $x \in \mathfrak{g}(i)$ is $i+2$ (the degree of $\hbar$ is still $2$).
\end{warning}

Applying the quantum Hamiltonian reduction by $U_\ell$ to the homomorphism 
\begin{equation*}
\Phi_{\mathfrak{p}}\colon \CA_{\hbar,\mathfrak{X}(\mathfrak{p})}(T^*\mathcal{B}) \rightarrow \CA_{\hbar,\mathfrak{X}(\mathfrak{p})}(T^*\mathcal{P}),
\end{equation*}
we obtain the homomorphism of $\CC[\hbar]$-algebras (c.f. \cite[Lemma 5.2.1 (3)]{Losev_isofquant}):
\begin{equation*}
\Phi_{\mathfrak{p},\chi}\colon \mathcal{W}_{\hbar}(\chi) \otimes_{Z_\hbar} \CC[\mathfrak{X}(\mathfrak{l}),\hbar]=\mathcal{A}_{\hbar,\mathfrak{X}(\mathfrak{l})}(\widetilde{S}(\chi)) \rightarrow \mathcal{A}_{\hbar,\mathfrak{X}(\mathfrak{l})}(\widetilde{S}(\chi,\mathfrak{p})).
\end{equation*}

It follows from Remark \ref{surj_Phi_hbar_nonzero} that the homomorphism $\Phi_{\mathfrak{p},\chi}$ becomes surjective after setting $\hbar=1$ and specializing at a regular integral antidominant $\la \in \mathfrak{X}(\mathfrak{p})$.

\subsection{Categories $\mathcal{O}$ and localizations}
	Recall that $T_X$ is a torus acting on $Y$
	in a Hamiltonian way. From now on, we assume that the set $Y^{T_X}$ is finite. 
	Let us fix a co-character $\nu\colon \CC^\times \ra T_X$.
	Recall that a co-character $\nu$ is {\em{regular}}
	if the set $Y^{\nu}$ is finite (in other words, coincides with $Y^{T_X}$).
	We obtain the decomposition of the vector space $\mathfrak{t}_{X,\mathbb{R}}:=\on{Hom}(\CC^\times,T_X)\otimes_{\BZ} \mathbb{R}$ into the union of {\em{open chambers}} separated by the walls corresponding to non-regular co-characters $\nu$ (see Section \ref{subsub_chambers}).
	
	\subsubsection{Categories $\CO$}\label{catOcat} We fix a parameter $\la \in \mathfrak{h}_{X}$ and a regular co-character $\nu \in \on{Hom}(\CC^{\times},T_X)$. Consider the corresponding filtered quantization $\mathcal{D}_\la$ of $\CO_{Y}$. Recall that $\CA_{\la}=\Gamma(Y,\mathcal{D}_\la)$.  The action of $T_X \curvearrowright \CO_{Y}$ lifts canonically to the action $T_X \curvearrowright \mathcal{D}_\la$ and hence to the action $T_X \curvearrowright \CA_{\la}$.

	The co-character $\nu$ induces a grading $\CA_{\la}=\bigoplus_{i\in \BZ}\CA_{\la,i}$ on $\CA_\la$.
	Let 
	$
	\CA^{>0}_{\la}:=\bigoplus_{i> 0}\CA_{\la,i}.
	$
	It follows from~\cite[Proposition~3.11]{BPW} that the action of $\nu$ 
	on $\CA_\la$ is Hamiltonian, so we have a comoment map $\CC \ra \CA_{\la,0}$. Let $h \in \CA_{\la,0}$ be the image of $1 \in \CC$.
	Let $\CA_\la$-$\on{mod}$ be the category of finitely generated $\CA_\la$-modules and let $\CO_{\nu}(\CA_\la) \subset \CA_\la$-$\on{mod}$ be the full subcategory, consisting of all modules where the action of $\CA^{>0}_\la$ is locally nilpotent and the action of $\CA_{\la,0}$ is locally finite.
	
	Let $\on{Coh}(\mathcal{D}_\la)$ be the category of all coherent $\mathcal{D}_\la$-modules and $\CO_\nu(\mathcal{D}_\la) \subset \on{Coh}(\mathcal{D}_\la)$ be the  full subcategory of modules that come with a good filtration stable under $h$ and  are supported on the contracting locus 
	$Y_+$ of $\nu$, i.e., $Y_+:=\{x\in Y\,|\,\on{lim}_{t\ra 0}\nu(t)\cdot x~\on{exists}\}$.
 The set
	$\on{Irr}(\CO_\nu(\mathcal{D}_\la))$ of irreducible objects
	can be naturally identified with $Y^{T_X}$ (see~\cite[Proposition~5.17]{BPWII} and Section \ref{subsec_cartan_sub_and_irred} below). 

\subsubsection{Localizations}\label{sec_Localizations} We have the global section functor $\Gamma\colon \on{Coh}(\mathcal{D}_\la) \ra \CA_\la$-$\on{mod}$.
	We denote by $\on{Loc}\colon \CA_\la$-$\on{mod} \ra \on{Coh}(\mathcal{D}_\la)$ the left adjoint functor given by $N \mapsto \mathcal{D}_\la \otimes_{\CA_\la} N$ (see~\cite[Section~4.2]{BPW}).
	For $\la \in \mathfrak{h}_{X}$, we say that abelian localization holds for $(\la,Y)$ when the functors
	$\Gamma$ and $\on{Loc}$ are quasi-inverse equivalences. Let $\eta\in H^2(Y,\CC)=\mathfrak{h}_{X}$ be the Chern class of an ample line bundle on $Y$. It is known (see~\cite[Corollary B.1]{BPW}) that for a sufficiently large integer $M$ (depending on $\la \in \mathfrak{h}_X$) the abelian localization holds for $\la + m\chi$ for all integers $m\ge M$.
	In particular, the irreducible objects of
	$\CO_\nu(\CA_{\la+m\chi})$ are parametrized by $Y^{T_X}$. 

\subsection{Cartan subquotients ($B$-algebras)}
The main references to this section are \cite[Sections 4,5]{LosevCatO} and \cite[Section 5.1]{BPWII}. For any $\mathbb{Z}$-graded ring $\mathcal{A}$, set 
\begin{equation*}
\cC(\CA):=\CA_{0}/\sum_{i>0}\CA_{-i}\CA_{i},
\end{equation*}
ring $\cC(\CA)$ will be called the {\em{Cartan subquotient}} or a {\em{$B$-algebra}} of $\CA$.
We will apply this construction to the quantizations (polynomial, and filtered) of $\mathbb{C}[X]$, $\CC[X_{\mathfrak{h}_X}]$ and more generally $\CC[X_{\mathfrak{a}}]$ for some vector space $\mathfrak{a} \rightarrow \mathfrak{h}_X$ with $\mathbb{Z}$-grading induced by a (regular) cocharacter $\nu\colon \mathbb{C}^\times \ra T_X$ as above. If $\mathcal{A}$ is one of these quantizations, then the corresponding Cartan subquotient will be denoted by $\cC_\nu(\CA)$. We will also use the notation $\mathcal{A}^{\geqslant 0}:=\bigoplus_{i \geqslant 0}\mathcal{A}_i$. Note that there exists the natural surjection $\mathcal{A}^{\geqslant 0} \twoheadrightarrow \cC(\mathcal{A})$ induced by the surjection $\mathcal{A}^{\geqslant 0} \twoheadrightarrow \mathcal{A}_0$.

\subsubsection{Cartan subquotient as schematic fixed points}\label{sec_schem_fixed}
Let $A$ be a graded commutative algebra, and let $X:=\on{Spec}A$. The $\mathbb{Z}$-grading on $A$ defines the action $\CC^\times \curvearrowright X$. Let $X^{\CC^\times}$ be the {\emph{functor}}:
\begin{equation*}
X^{\CC^\times}\colon {\bf{Schemes}}_{\CC} \rightarrow {\bf{Sets}}
\end{equation*}
from the category ${\bf{Schemes}}_{\CC}$ of schemes over $\CC$ to the category of {\bf{Sets}} defined as follows (see \cite{fogarty}, \cite[Section 1.2]{drinfeld}):
\begin{equation*}
X^{\CC^\times}(S):=\on{Maps}^{\CC^\times}(S,Y),~S \in {\bf{Schemes}}_{\CC},
\end{equation*}
where the action of $\CC^\times$ on $S$ is trivial and $\on{Maps}^{\CC^\times}(S,Y)$ is the set of $\CC^\times$-equivariant morphisms from $S$ to $Y$.

The following proposition is well-known (see, for example, \cite[Proposition 1.4]{ksh}).

\begin{prop}
Functor $X^{\CC^\times}$ is represented by an affine scheme whose ring of functions is exactly $\cC(A)$.  
\end{prop}

Note that the functor $X^{\CC^\times}$ is well-defined even if $X$ is not affine, one can similarly define the functor $X^{T}$ for an arbitrary variety with an action of some torus $T$. The following proposition holds by \cite{iversen}.

\begin{prop}\label{prop_fixed_smooth}
Let $X$ be a smooth algebraic $T$-variety. Then  $X^{T}$ is smooth.
\end{prop}

\subsubsection{Sheaf version of Cartan subquotient}
In this section, let $T=T_X$, and let $Y$ be a symplectic resolution of $X$. Let $X_{\mathfrak{a}}$ be a deformation of $X$ over some vector space $\mathfrak{a}$ mapping linearly to $\mathfrak{h}_X$. Let $Y_{\mathfrak{a}}$ be the corresponding deformation of $Y$. Recall that we assume that the natural morphism $Y_{\mathfrak{a}} \rightarrow X_{\mathfrak{a}}$ becomes an isomorphism after the restriction to some nonempty open subset $U \subset \mathfrak{a}$ (see Assumption \ref{ass_resol_a} above).

Following \cite[Section 5.2]{LosevCatO} one can talk about the Cartan subquotients of  {\emph{sheaves}} $\mathcal{D}$ on $Y_{\mathfrak{a}}$ (see  \cite[Proposition 5.2]{LosevCatO}), the resulting object will be a {\emph{sheaf}} of algebras on the $T$-fixed points of $Y_{\mathfrak{a}}$ (recall that $Y_{\mathfrak{a}}^T$ is smooth by Proposition \ref{prop_fixed_smooth} above).

\begin{lemma}\label{fixed_pts_deform}
There exists the unique isomorphism $\alpha: Y_{\mathfrak{a}}^{T} \simeq Y^{T} \times \mathfrak{a}$ of varieties over $\mathfrak{a}$, such that the restriction of $\alpha$ to the fibers over $0 \in \mathfrak{a}$ is the identity map.
\end{lemma}
\begin{proof}
It follows from  \cite[Section 1.2 and references therein]{Namikawa2} that there exists a $T$-equivariant identification of  $C^\infty$-manifolds over $\mathfrak{a}$ that is equal to identity being restricted to the fibers over $0$: 
\begin{equation}\label{triv_iso_fixed}
Y_{\mathfrak{a}} \iso Y \times \mathfrak{a}
\end{equation}
(the action of $T$ on the right-hand side is induced by the action of $T$ on $Y$). 

Passing to $T$-fixed points, we obtain the identification $\alpha: Y_{\mathfrak{a}}^T \iso Y^T \times \mathfrak{a}$ of $C^\infty$-manifolds over $\mathfrak{a}$. The fact that the restriction of (\ref{triv_iso_fixed}) to the fibers over $0$ is equal to identity implies that the restriction of $\alpha$ to the fibers of $0$ is identity.

It remains to check that this identification is a morphism of algebraic varieties (then the fact that it is an isomorphism follows from Zariski's main theorem). 

Indeed, note that for an arbitrary smooth variety $Z$ over $\mathbb{C}$ its connected components considered as $C^\infty$-variety coincide with its connected components considered as an algebraic variety (let $Z_i$ be a connected component of the algebraic variety $Z$, then $Z_i$ is connected as an analytic manifold by \cite[Chapter VII, § 2.2, Theorem 1]{shafarevich2}, so it is also path-connected and this implies that $Z_i$ is connected as a $C^\infty$-manifold). In particular, we see that $Y_{\mathfrak{a}}^T$ has $|Y^T|$ connected components. Let $F$ be one of them. The identification (\ref{triv_iso_fixed}) induces the isomorphism  $F \iso \{p\} \times \mathfrak{a}$ for some $p \in Y^T$. It remains to check that the identification above is algebraic. This is clear since the composition $F \hookrightarrow Y_{\mathfrak{a}}^T \rightarrow {\mathfrak{a}} \iso \{p\} \times \mathfrak{a}$ is algebraic.

\end{proof}

\begin{lemma}\label{cartan_subq_univ_deform_sh}
There are canonical isomorphisms of sheaves of algebras 
\begin{equation*}
\cC_\nu(\mathcal{D}_{\hbar,\mathfrak{a}}) \simeq \CO_{Y_{\mathfrak{a}}^{T}}[\hbar], \quad \cC_\nu(\mathcal{D}_{\hbar,\la}) \simeq \CO_{Y^{T}}[\hbar], \quad \cC_\nu(\mathcal{D}_\la) \simeq \CO_{Y^{T}}.
\end{equation*}
\end{lemma}
\begin{proof}
Since fixed $T$-points are isolated in $Y$ by our assumption, lemma follows from \cite[Lemma 5.2]{BPWII}.
\end{proof}

Using Lemmas \ref{fixed_pts_deform}, \ref{cartan_subq_univ_deform_sh} we obtain the natural homomorphism of $\mathbb{C}[\mathfrak{a},\hbar]$-algebras
\begin{equation*}
\cC_\nu(\CA_{\hbar,\mathfrak{a}}) \ra \Gamma(Y_{\mathfrak{a}}^{T},\cC_\nu(\mathcal{D}_{\hbar,\mathfrak{a}}))=\CC[Y^T] \otimes \CC[\mathfrak{a},\hbar]
\end{equation*}
that at $\la \in \mathfrak{a}$ specializes to the  natural homomorphism
\begin{equation}\label{general_comp_cart_to_sheaf}
\cC_\nu(\CA_{\hbar,\la}) \ra \Gamma(Y^T,\cC_\nu(\mathcal{D}_{\hbar,\la}))=\CC[Y^T] \otimes \CC[\hbar].
\end{equation}

\begin{rmk}\label{iso_cart_generic}
It follows from \cite[Proposition 5.3]{BPW} (see also \cite[Proposition 5.3]{LosevCatO}) that the composition (\ref{general_comp_cart_to_sheaf}) becomes $\CC[\hbar]$-linear isomorphism for Zariski generic values of $\la$. 
\end{rmk}

In what follows we are mostly interested in the case when $Y$ has finitely many $T$-fixed points. To track that, we have the following lemma.
\begin{lemma}\label{finite fixed points}
 The set $Y^{T}$ is finite if and only if $X^T$ consists of one point.
\end{lemma}
\begin{proof}
If $X^T$ is infinite, then $Y^T$ is also infinite because the map $\pi$ is proper and surjective. Assume now that $X^T$ is finite. Since $T$ commutes with the contracting action of $\CC^\times$ on $X$, $\CC^\times$ acts on $X^T$ contracting it to the unique $\CC^\times$-fixed point $e$. This implies that $X^T=\{e\}$. It remains to check that $Y^T$ is also finite. Consider the deformation $\pi_{\mathfrak{h}_X}\colon Y_{\mathfrak{h}_X} \rightarrow X_{\mathfrak{h}_X}$ over $\mathfrak{h}_X$. It follows from the graded Nakayama lemma that the algebra $\CC[X_{\mathfrak{h}_X}^{T}]$ is finitely generated over $\CC[\mathfrak{h}_X]$. Morphism $\pi_{\mathfrak{h}_X}$ is an isomorphism over generic $\la \in \mathfrak{h}_X$ so $\CC[Y_{\la}^{T}]$ is a finite-dimensional vector space for generic $\la$, hence, the set $Y_{\la}^{T}$ is finite. It remains to recall that the spaces $Y_\la$, $Y$ are diffeomorphic as $C^\infty$-manifolds with a $T$-action (cf.  \cite[Section 1, Step 2 and references therein]{Namikawa2}, so we must have $Y^T=Y_\la^T$ is finite. 
\end{proof}

\subsubsection{Cartan subquotients and irreducible objects in category $\mathcal{O}$}\label{subsec_cartan_sub_and_irred}
Let us now recall the relation between $\cC_\nu(\mathcal{A}_{\la})$ and irreducible objects in the category $\CO_\nu(\mathcal{A}_\la)$. Pick $\la$ such that the abelian localization holds and the  morphism (\ref{general_comp_cart_to_sheaf}) specialized at $\hbar=1$ is an isomorphism:
\begin{equation*}
\cC_\nu(\mathcal{A}_{\la}) \iso \Gamma(Y^T,\cC_\nu(\mathcal{D}_\la))=\bigoplus_{p \in Y^T} \CC.
\end{equation*}
 We see that the algebra $\cC_\nu(\mathcal{A}_{\la})$ is finite-dimensional semisimple and has $|Y^T|$ irreducible (one dimensional) representations labeled by the fixed points of $Y$. Let $\CC_p$ be the irreducible representation of $\cC_\nu(\mathcal{A}_{\la})$ corresponding to the point $p \in Y^T$. We can then form the {\emph{standard}} module (see \cite[Section 5.2]{BPWII}) 
\begin{equation*}
\Delta_\la(p)=\Delta(p)=\mathcal{A}_\la \otimes_{\mathcal{A}_\la^{\geqslant 0}} \CC_p \in \CO_\nu(\mathcal{A}_\la),
\end{equation*}
where the action of $\mathcal{A}_\la^{\geqslant 0} \curvearrowright \CC_p$ is induced by the natural surjection $\mathcal{A}_\la^{\geqslant 0} \twoheadrightarrow \cC_\nu(\mathcal{A}_\lambda)$.

\begin{rmk}
Similarly, one can define the {\emph{costandard}} module $\nabla(p)=\nabla_\la(p)$ as follows (see \cite[Section 5.2]{BPWII}): $\nabla(p)= (\CC_p^* \otimes_{\mathcal{A}_\la^{\geqslant 0}} \mathcal{A})^\star$, here $\star$ denotes the restricted duality.
\end{rmk}

By \cite[Proposition 5.17]{BPWII}, the category $\cO_{\nu}(\cA_\lambda)$ is highest weight, and therefore every $\Delta(p)$ has {\emph{unique}} irreducible quotient denoted by $L(p)$. Moreover, modules $\{L(p)\,|\, p \in Y^T\}$ are pairwise non-isomorphic and form the complete list of simples in $\CO_\nu(\mathcal{A}_\la)$. Recall that we have the equivalence of categories $\on{Loc}\colon \CO_\nu(\mathcal{A}_\la) \iso \CO_\nu(\mathcal{D}_\la)$. Applying $\on{Loc}$ to $L(p)$, we obtain irreducible objects in $\CO_\nu(\mathcal{D}_\la)$. In this way we obtain the  labeling $Y^T \iso \on{Irr}(\CO_\nu(\mathcal{D}_\la))$. This is exactly the labeling mentioned in Section \ref{sec_Localizations}.

Let us finally mention that every object $V \in \mathcal{O}_\nu(\mathcal{A}_\la)$ is naturally graded via the action by the commutator with  $h$ (recall that $h$ is the image of $1$ under the comoment map induced by $\nu$). 
It follows from the definitions that $\mathcal{A}_{\la,0}$ acts on each graded component. Now, the {\emph{highest weight}} component of $L(p)$ is one-dimensional and the action of $\mathcal{A}_{\la,0}$ on the corresponding one-dimensional space factors through the quotient $\cC_\nu(\mathcal{A}_\la)$ and is precisely the character of $\cC_\nu(\mathcal{A}_\la)$, corresponding to $\CC_p$. In other words, we can read off the character $\cC_\nu(\mathcal{A}_\la) \rightarrow \CC_p$ from the irreducible representation $L(p)$ as the ``highest weight'' of $L(p)$ considered as $\mathcal{A}_{\la,0}$-module. We note that we have a surjective specialization map $\cA_{h, \fa, 0}\to \cA_{1, \lambda, 0}=\cA_{\lambda, 0}$, and therefore each element $x\in \cA_{h, \fa, 0}$ acts on $L(p)\in \CO_\nu(\CA_{\la})$ by a scalar.

\begin{rmk}
Note that highest weights of $\Delta(p)$, $\nabla(p)$, $L(p)$ are the same.    
\end{rmk}

\subsubsection{From algebraic to sheaf Cartan subquotients}
Recall that the algebra $\CA_{\hbar,\mathfrak{a}}$ is $\mathbb{Z}_{\geqslant 0}$-graded with $\mathfrak{a}$ and $\CC \hbar$ sitting in degree $2$. We denote the $i$-th graded piece by $\CA_{\hbar,\mathfrak{a}}^{i}$. Pick $\la \in \mathfrak{a}$ and $\hbar_0 \in \mathbb{C}$, we can consider the specialization $\CA_{(\hbar_0, \la)}$ of $\CA_{\hbar,\mathfrak{a}}$ to the point $(\hbar_0, \la) \in \CC\oplus \mathfrak{a}$. We see that for $\hbar_0 \neq 0$ the action of $\mathbb{C}^\times$ (corresponding to the $\mathbb{Z}_{\geqslant 0}$-grading above) identifies $\CA_{(\hbar_0, \la)}$ and $\CA_{(1, \hbar_0^{-1}\la)}=\CA_{\hbar_0^{-1}\la}$. Consider the composition
\begin{equation}\label{comp_cart_alg_to_sh}
\CA_{\hbar,\mathfrak{a},0} \twoheadrightarrow \cC_\nu(\CA_{\hbar,\mathfrak{a}}) \ra \Gamma(Y_{\mathfrak{a}}^{T},\cC_\nu(\mathcal{D}_{\hbar,\mathfrak{a}}))=\CC[Y^T] \otimes \CC[\mathfrak{a},\hbar].
\end{equation}
Pick $x \in \CA_{\hbar,\mathfrak{a},0}^{i}$ for an even $i$. 

\begin{prop}\label{hw_geom_vs_hw_alg} 
(a) For $\la \in \mathfrak{a}$ such that the abelian localization holds and $\hbar_0 \neq 0$, the image of $x_{(\hbar_0, \la)} \in \CA_{(\hbar_0, \la),0}$ under 
\begin{equation*}
\CA_{(\hbar_0, \la),0} \twoheadrightarrow \cC_\nu(\CA_{(\hbar_0, \la)}) \ra \Gamma(Y^T,\cC_\nu(\mathcal{D}_{(\hbar_0, \la)}))=\CC[Y^T]
\end{equation*}
is equal to the collection of scalars by which $\hbar_0^{{{i/2}}}x$ acts on the highest components of the irreducible modules in $\CO_\nu(\CA_{\hbar_0^{-2}\la})$.

(b) For $\hbar_0=0$ the composition (\ref{comp_cart_alg_to_sh}) identifies with the pull back homomorphism 
\begin{equation*}
\mathbb{C}[X_{\mathfrak{a}}]_0 \twoheadrightarrow \mathbb{C}[X_{\mathfrak{a}}^{T}] \ra \mathbb{C}[Y_{\mathfrak{a}}^T]=\mathbb{C}[Y^T] \otimes \CC[\mathfrak{a}].
\end{equation*}

(c) If $\mathfrak{a} \rightarrow \mathfrak{h}_X$ is such that the Assumption \ref{ass_ample_chern_a} 
holds, then  $(a)$ determines the image of $x$ uniquely. 
\end{prop}
\begin{proof}
Since the homomorphism (\ref{comp_cart_alg_to_sh}) is $\mathbb{C}^\times$-equivariant, it is enough to prove $(a)$ for $\hbar_0=1$. For $\hbar_0=1$ the claim follows from the definitions. Part $(b)$ is clear. To show (c), we note that any polynomial $f(t)$ is uniquely determined by its values at infinitely many points. Thus, $(c)$ follows from \cite[Corollary B.1]{BPW} by restricting to lines $\ell(x)|_{x \in \mathfrak{a}}$ of the form $\ell(x)=\{x+c\tilde{\chi}\,|\, c \in \mathbb{C}\}$, where $\tilde{\chi} \in \mathfrak{a}$ is a preimage of the first Chern class $\chi$ of an ample line bundle $\mathcal{L}$ (see the notations of Assumption \ref{ass_ample_chern_a}).
\end{proof}

Note that we are not claiming in Proposition \ref{hw_geom_vs_hw_alg} (b) that $\cC_\nu(\mathcal{A}_{\hbar,\mathfrak{a}})/(\hbar)$ is isomorphic to $\mathbb{C}[X^T_{\mathfrak{a}}]$. Although, it is clearly true that we have the identification $\mathcal{A}_{\hbar,\mathfrak{a},0}/(\hbar) = \mathbb{C}[X_{\mathfrak{a}}]_0$.

We always have a surjective homomorphism $\mathbb{C}[X^T_{\mathfrak{a}}] \twoheadrightarrow \cC_\nu(\mathcal{A}_{\hbar,\mathfrak{a}})/(\hbar)$. Moreover, 
in the following proposition 
we explain when this map is indeed an isomorphism. 
\begin{prop}\label{prop_when_cartan_hbar_commute} 
Assume that $\operatorname{dim}\mathbb{C}[X^{T}] \leqslant |Y^{T}|$. Then: 

(a) We have $\operatorname{dim}\mathbb{C}[X^{T}] = |Y^{T}|$.

(b) Cartan subquotient $\cC_\nu(\mathcal{A}_{\hbar,\mathfrak{a}})$ is a flat (hence, free) module over $\mathbb{C}[\mathfrak{a},\hbar]$ of rank $|Y^{T}|$. 

(c) The natural (surjective) morphism $\mathbb{C}[X_{\mathfrak{a}}^{T}]=\cC_\nu(\mathbb{C}[X_{\mathfrak{a}}]) \rightarrow \cC_\nu(\mathcal{A}_{\hbar,\mathfrak{a}})/(\hbar)$ is an isomorphism.
\end{prop}
\begin{proof}
It is straightforward to check that we have the surjective homomorphism $\mathbb{C}[X^{T}] \twoheadrightarrow \cC_{\nu}(\CA_{\hbar,\mathfrak{a}})/(\hbar,\mathfrak{a}^*)$. Using the inequality $\operatorname{dim}\mathbb{C}[X^{T}] \leqslant |Y^{T}|$ together with the graded Nakayama lemma, we see that it is enough to check that the fibers of $\cC_\nu(\mathcal{A}_{\hbar,\mathfrak{a}})$ over a generic point of $\mathfrak{a} \oplus \mathbb{C}$ have dimension equal to $|Y^{T}|$.

Recall that we have the homomorphism $\cC_\nu(\mathcal{A}_{\hbar,\mathfrak{a}}) \rightarrow \mathbb{C}[Y^T] \otimes \mathbb{C}[\mathfrak{a},\hbar]$. The same argument as the one in \cite[Proposition 5.3]{LosevCatO} shows that this map becomes an isomorphism for Zariski generic $(\la,\hbar_0) \in \mathfrak{a} \oplus \mathbb{C}$. This proves the claim.
\end{proof}

\section{Equivariant cohomology}\label{sec_equiv}
\subsection{Basic results} \label{basic eqcoh}
Consider a torus $T$ acting on a variety $X$ (not necessarily with finitely many fixed points). This section consists of several results on equivariant cohomology $H_{T}^*(X)$ that we often use in the paper. Our main reference will be \cite{Anderson2023EquivariantCI}.

By the definition, $H_{T}^*(X)$ is  the cohomology ring $H^*(\mathbb{E}\times^T X)$ where $\mathbb{E}$ is a certain universal principal $T$-bundle of $X$ (see \cite{Anderson2023EquivariantCI}). After the identification $T \simeq (\mathbb{C}^\times)^r$, we have $\mathbb{E} \simeq (\mathbb{A}^\infty \setminus \{0\})^r$). Consider a $T$-equivariant vector bundle $E$ over $X$, then $\mathbb{E}\times^T E$ is a vector bundle over $\mathbb{E}\times^T X$, and the equivariant Chern classes $c_i^{T}(E)$ are defined as the Chern classes of $\mathbb{E}\times^T E$.

When $X$ is a point, we have $H_{T}^*(\pt)$ is $\CC[\ft]$, the algebra of functions on the Lie algebra of $T$. The identification is constructed as follows. To every character $\lambda\colon T \rightarrow \mathbb{C}^\times$ we can associate the one dimensional vector bundle $\mathbb{C}_\la$ over $X=\on{pt}$. Then the identification $H^*_T(\on{pt}) \iso \mathbb{C}[\mathfrak{t}]$ is given by $c_1^T(\mathbb{C}_\la) \mapsto -\lambda$. More generally, every $T$-equivariant bundle over $X=\on{pt}$ is simply a $T$-module $E$. Then, $c_i^T(E) \mapsto e_i(-\chi_1,\ldots,-\chi_{\on{rk}E})$, where $\chi_1,...,\chi_{\rk E}$ is the list of characters of $E$.


In this paper, we will often consider cases with the following features.
\begin{enumerate}
    \item $T$ acts on $X$ with finitely many fixed points.
    \item $H^*_{T}(X)$ is a free module over $H^*_{T}(\pt)$.
\end{enumerate}
With these conditions, the localization theorem in equivariant cohomology (see, e.g., \cite[Section 5]{Anderson2023EquivariantCI}) implies that the pullback map $\iota^*: H_{T}^*(X)\rightarrow H_{T}^*(X^T)$ is an injection. Moreover, it becomes an isomorphism after localizing at certain multiplicative subset $S\subset \CC[\ft]$ that is generated by homogeneous polynomials. When $T\simeq \CC^\times$, we have $\CC[\ft]= \CC[t]$, the polynomial ring in one variable. In this case, we can always choose $S$ as the set generated by $t$.
\subsubsection{Specialization and pullback} \label{fixed point pullback}
Consider a torus $T$ that acts on a smooth variety $X$. Let $\ft$ be the Lie algebra of $T$, and $\lambda\in \ft$. The equivariant cohomology $H^*_{T}(X)$ is an algebra over $\CC[\ft]$. We keep the assumption that $H^*_{T}(X)$ is a free module over $H^*_{T}(\pt)$, this will always be the case in our paper.

Write $H^{*}_T(X)_\lambda$ for the specialization of $H^*_{T}(X)$ at $\lambda$. Torus $T$ acts on $X$, so we can regard $\lambda$ as a vector field on $X$ to be denoted $\la_X$. Write $X^{\lambda_X}\subset X$ for the subvariety of zeroes of this vector field. By the Berline-Vergne localization theorem (\cite[Proposition 2.1]{Berline1985}, \cite[Sections 4,5]{Atiyah1984}), we have the following isomorphism   
\begin{equation}\label{BV_localization}
H^{*}_T(X)_\lambda \iso H^*(X^{\lambda})
\end{equation}
induced by the natural map $H^*_T(X) \rightarrow H^*(X^\lambda)$. 
In general, it makes sense to consider the pullback map in equivariant cohomology. As a consequence, we have a natural map $h_X\colon H_{T}^*(\pt)$ $\rightarrow H_{T}^*(X)$, and $H_{T}^*(X)$ is a module over $H_{T}^*(\pt) = \CC[\ft]$. 
Note that the isomorphism (\ref{BV_localization}) commutes with the pullback map as follows. Consider a $T$-equivariant morphism $f\colon X\rightarrow Y$. Then the isomorphism above identifies $f_\lambda^{*}: H^{*}_T(Y)_\lambda \rightarrow H^{*}_T(X)_\lambda$ with the pullback for the fixed-point varieties $H^*(Y^{\lambda}) \rightarrow H^*(X^{\lambda})$.

A consequence is the following. Consider a $T$-equivariant morphism $f\colon X\rightarrow Y$. Assume that both $X$ and $Y$ are such that $H_T^*(X)$, $H_T^*(Y)$ are free over $H^*_T(\on{pt})$. Then $f^{*}\colon  H^{*}_T(Y) \rightarrow H^{*}_T(X)$ is surjective if and only if $f^{*}\colon H^{*}(Y) \rightarrow H^{*}(X)$ is surjective by graded Nakayama lemma. 

\subsection{Equivariant cohomology of partial flag varieties}\label{subsec_expl_descr_equiv_partial}
In this section and the next, we start with general results on flag and partial flag varieties. We then explain explicit results for the cases $G$ of classical type.

Let $G$ be a simple simply-connected Lie group of rank $n$. Let $T_G$ be a maximal torus of $G$ and let $\fh_G$ be its Lie algebra. Let $B \subset P\subset G$ be a Borel and a parabolic subgroup of $G$ that contains $T_G$. Let $M$ be the Levi subgroup of $P$. Write $W_G$ and $W_M$ for the Weyl groups of $G$ and $M$. Write $\cB$ and $\cP$ for the flag variety $G/B$ and the partial flag variety $G/P$, respectively. The torus $T_G$ naturally acts on $\cB$ and $\cP$ by multiplication from the left. 
\begin{prop}\cite[Section 5]{Kaji2015}\label{general equicoh} \leavevmode
    \begin{enumerate}
        \item There is a canonical identification $H^*_{T_G}(\cB)\simeq \CC[\ft_G]\otimes_{\CC[\ft_G]^{W_G}} \CC[\ft_G]$. In particular, the equivariant cohomology $H^*_{T_G}(\cB)$ is a free module over $H^*_{T_G}(\pt) = \CC[t_1,...,t_n]$ of rank $|W_G|$.
        \item The natural pullback $H^*_{T_G}(\cP)\rightarrow H^*_{T_G}(\cB)$ identifies $H^*_{T_G}(\cP)$ with the $W_M$-invariant subring of $H^*_{T_G}(\cB)$. In other words, $H^*_{T_G}(\cP)\simeq \CC[\fh_G]\otimes_{\CC[\fh_G]^{W_G}} \CC[\fh_G]^{W_M}$.
    \end{enumerate}
\end{prop}
A simple corollary is the following.
\begin{cor} \label{Polquotient}
    Consider a torus $T\subset T_G$ and its Lie algebra $\fh$. The equivariant cohomology ring $H^*_{T}(\cP)$ is the quotient $\CC[\fh]\otimes_{\CC[\fh_G]^{W_G}} \CC[\fh_G]^{W_M}$ of the polynomial algebra $\CC[\fh]\otimes \CC[\fh_G]^{W_M}= H^*_{T}(\pt)\otimes_\CC H^*_{M}(\pt)$.
\end{cor}

\subsubsection{}\label{descr_gen_rel_equiv_cohom_partial_via_chern}  Let us  describe the identification 
\begin{equation}\label{equiv_cohom_partial_flags_descr}
\CC[\fh]\otimes_{\CC[\fh_G]^{W_G}} \CC[\fh_G]^{W_M} \iso H^*_{T}(\cP)
\end{equation}
from Corollary \ref{Polquotient}.

For every finite dimensional representation $V$ of $M$, let $\mathcal{V}$ be the induced vector bundle $G \times^P V$ on $\mathcal{P}$ (the action of $P$ on $V$ is via the natural surjection $P \twoheadrightarrow M$). 
Let $\on{ch}_{\mathfrak{h}_G}V^*$ be the character of $V^*$ considered as an element of $\CC[\fh_G]^{W_M}$. Similarly, to every finite dimensional representation $S$ of $T$, we can associate the trivial vector bundle  $\mathcal{S}:=\mathcal{P} \times S$ with the $T$-equivariant structure induced by the action of $T$ on $S$. The character $\on{ch}_{\mathfrak{h}}S^*$ can be considered as an element of $\mathfrak{h}^* \subset \CC[\mathfrak{h}]$. The isomorphism (\ref{equiv_cohom_partial_flags_descr}) is uniquely determined as follows:
\begin{equation*}
[\on{ch}_{\mathfrak{h}}S^* \otimes \on{ch}_{\mathfrak{h}_G}V^*] \mapsto c_1^T(\mathcal{S})c^T_1(\mathcal{V})~\text{for every}~S \in \on{Rep}_{\mathrm{f.d.}}T,~V \in \on{Rep}_{\mathrm{f.d.}}M.
\end{equation*}

Note that $c_i^T(\mathcal{S}) \in H^*_T(\mathcal{P})$ is the image of the functional $x \mapsto e_i(\al_1(x),\ldots,\al_s(x))$ (considered as an element of $\CC[\mathfrak{h}] \otimes 1$), where $\al_i$ are eigenvalues of $T$ acting on $S^*$ (counted with multiplicities), we will denote this functional by $e_i(\on{ch}_{\mathfrak{h}}S^*) \otimes 1$. Similarly, $c_i^T(\mathcal{V})$ is the image of the functional $x \mapsto e_i(\beta_1,\ldots,\beta_{v})$ (considered as an element of $1 \otimes  \CC[\mathfrak{h}_{G}]^{W_M}$), where  $\al_i$ are eigenvalues of $T_G$ acting on $V^*$ (counted with multiplicities), this functional will be denoted by $1 \otimes e_i(\on{ch}_{\mathfrak{h}_G}V^*)$.

Let us describe another possible way to determine the isomorphism (\ref{equiv_cohom_partial_flags_descr}). Pick any  
collection $S_k$, $V_l$ of representations of  $T$, $M$  such that $e_i(\on{ch}S_k^*)$, $e_j(\on{ch}V_l^*)$ generate $\mathbb{C}[\mathfrak{h}]$ and $\mathbb{C}[\mathfrak{h}_G]^{W_M}$ respectively.
Then (\ref{equiv_cohom_partial_flags_descr}) is uniquely determined as follows:
\begin{equation*}
[e_i(\on{ch}_{\mathfrak{h}}S_k^*) \otimes e_j(\on{ch}_{\mathfrak{h}_G}V_l^*)] \mapsto c_i^T(\mathcal{S}_k)c_j^T(\mathcal{V}_l).
\end{equation*}

Explicitly, the kernel of the map 
\begin{equation}\label{characters_to_cherns_surj_map}
e_i(\on{ch}_{\mathfrak{h}}S_k^*) \otimes e_j(\on{ch}_{\mathfrak{h}_G}V_l^*) \mapsto c_i^T(\mathcal{S}_k)c_j^T(\mathcal{V}_l)
\end{equation} 
is generated by $e_i(\on{ch}_{\mathfrak{h}}V^*) \otimes 1 - 1 \otimes e_i(\on{ch}_{\mathfrak{h}_G}V^*)$, where $V$ runs through finite dimensional representations of $G$. Let's see that the difference as above maps to zero via the map (\ref{characters_to_cherns_surj_map}).
Indeed, for every $G$-module $V$, the corresponding induced vector bundle $\mathcal{V} = G \times^P V$ is trivial, the isomorphism $G \times^P V \iso \mathcal{P} \times V$ is given by $[(g,v)] \mapsto ([g],gv)$. This  isomorphism is $T$-equivariant w.r.t. the diagonal action of $T$ on $\mathcal{P} \times V$. It follows that the elements $e_i(\on{ch}_{\mathfrak{h}}V^*) \otimes 1$, $1 \otimes e_i(\on{ch}_{\mathfrak{h}_G}V^*)$ both map to $c_i^T(\mathcal{V})$.

Alternatively, the kernel is generated by $e_i(\on{ch}_{\mathfrak{h}}V_p^*) \otimes 1 - 1 \otimes e_i(\on{ch}_{\mathfrak{h}_G}V_p^*)$, where $\{V_p\}$ is any collection of representations of $G$ such that $e_i(\on{ch}_{\mathfrak{h}_G}V_p)$ generate $\mathbb{C}[\mathfrak{h}_G]^{W_G}$.
    
\subsection{Fixed points of partial flag varieties}
Let $T\subset G$ be a torus. The following is a general result for the $T$-fixed point varieties of partial flags. Fix a parabolic subgroup $P\subset G$. Let $\cP^{T}$ be the corresponding fixed point variety. Write $W_L$ for the Weyl group of $P$. Write $M$ for the connected component of $C_G(T)$ that contains $\Id$. Let $W_M$ and $W_G$ be the Weyl group of $M$ and $G$, respectively. 
\begin{prop}\label{fixed pt of partial flag}\cite[Proposition 3.3]{Steinberg1976}
    The variety $\cP^{T}$ parametrizes the parabolic subalgebras $\fp'\subset \fg$ that are conjugated to $\fp$ and $\fp'\cap \fl$ nonempty. It has $|W_L\setminus W_G/ W_M|$ connected components. For $[w]\in W_L\setminus W_G/ W_M$, the corresponding component is isomorphic to $M/(M\cap (w.P))$, a partial flag variety of $M$, via the map $\fp'\mapsto \fp'\cap \fl$.
    
\end{prop}
\begin{rmk}
    The intersection $(M\cap (w.P))$ depends on the choice of $w$, but for all $w$ in the same double $W_M-W_L$ coset, the corresponding varieties $M/(M\cap (w.P))$ are isomorphic.
\end{rmk}

Let $M$ be a Levi subgroup of $G$. Consider $T= Z_M(G)$ and $P= B$. The next corollary describes the variety $\cB^T$.

\begin{cor}\label{smL} \leavevmode

$(a)$  Each connected component $\cB_\alpha$ of $\cB^T$ is then isomorphic to the flag variety $\cB_M$ of $\fm$ via the map $\fb\mapsto \fb\cap \fm$.

$(b)$ $\fb_1, \fb_2 \in \cB^T$ lie in the same connected component iff they are conjugated by $M$. 

$(c)$ Group $W$ acts transitively on the set of connected components of $\cB^T$, realizing this set as $W/W_M$. 
\end{cor}
Consider $H^{*}_{T}(\cB^T)$ as a quotient of $\CC[\fh]\otimes_\CC \CC[\fh_G]$, then it makes sense to talk about the set-theoretic support of $\Spec(H^{*}_{T}(\cB^T))$ as a subset of $\fh\times \fh_G$. 
\begin{lemma} \label{fixed point flag variety}
    The set-theoretic support of $\Spec(H^{*}_{T}(\cB^T))$ consists of $|W_G|/|W_M|$ irreducible components, each isomorphic to $\fh$. 
\end{lemma}
\begin{proof}
    From \cref{smL}, we know that each component of $\cB^T$ is a flag variety that is fixed by $T$, so the set-theoretic support of each component is isomorphic $\fh$ by the description in \cref{general equicoh}. To show that two distinct components have distinct supports, we specialize to a generic element $\lambda$ in $\fh$. We have $\cB^T= \cB^\lambda$ in the notation of Section 3.1.1.

    We choose a maximal Cartan subalgebra $\fh_G$ of $G$ so that $\fh\subset \fh_G$. View $\lambda$ as a point of $\fh_G$. Let $I_{\fh_G}(\lambda)$ be the maximal ideal that corresponds to $\lambda$ in $\fh_G$. Using the isomorphism (\ref{equiv_cohom_partial_flags_descr}), $H^*_T(\cB)_\lambda$ can be rewritten as $\CC_\lambda\otimes_{\CC[\fh_G]^{W_G}}\CC[\fh_G] /I_{\fh_G}(\lambda)^{W_G}= \CC[\fh_G]/I_{\fh_G}(\lambda)^{W_G}$. The set-theoretic support of $\Spec(\CC[\fh_G]/I_{\fh_G}(\lambda)^{W_G})$ as a subscheme of $\fh_G$ is precisely the orbit $W_G.\lambda\subset \fh_G$. Now the variety $\cB^\lambda$ has precisely $|W_G.\lambda|$ connected components. Therefore, each component is supported on a single point of the orbit $W_G.\lambda$.
\end{proof}



\subsection{Explicit descriptions for  $G$ of types $ABC$} 

We discuss the cases $G$ of classical type and make explicit some results from the previous section. Assume that $G=\on{SL}_n$ and $B$ is the subgroup of upperdiagonal matrices. Let $V$ be the standard representation of $G$. Then, there exists the natural identification $\mathcal{B} \iso \mathcal{F}$, where $\mathcal{F}$ is the variety of complete flags in $V$.  Let $T$ be a torus acting linearly on $V$. For a point of $\cF$, we use the notation $F^{\bullet}= F^{1}\subset\ldots\subset F^{n}$ where $\dim F^i= i$. On $\cF$, we have the tautological line bundles $\cF^{i}$ that parameterize the pair $(F^\bullet, F^i/F^{i-1})$. It is easy to see that $\mathcal{F}^i = G \times^B \mathbb{C}_{i}$, where $\mathbb{C}_{i}$ is the one-dimensional representation of $B \twoheadrightarrow T$, corresponding to the character $T \ni \on{diag}(t_1,\ldots,t_n) \mapsto t_i^{-1}$. Write $x_i$ for $c_1^{T}(\cF^i)$. 

The following proposition is a particular case of our discussion in Section \ref{descr_gen_rel_equiv_cohom_partial_via_chern} (applied to $G=\on{SL}_n$, $p=1$ and $V_1=\mathbb{C}^n$).
\begin{prop}\cite[Proposition 4.4.1]{Anderson2023EquivariantCI}
    We have 
    $$H^{*}_{T}(\cF)= \CC[\fh][x_1,...,x_n]/(e_i(x)- c^T_{i}(V))_{i=1,...,n}$$
    in which $e_i$ is the elementary symmetric polynomial of $n$ variables $x_1,...,x_n$. 
\end{prop}
From this description, it makes sense to consider $\on{Spec}(H^{*}_{T}(\cF))$ as a subscheme of $\fh
\\\times \CC^{n}$. The following corollary is about a simple case of the set-theoretic support of $\on{Spec}(H^{*}_{T}(\cF))$.
\begin{cor} \label{single point support}
    Assume that $T\simeq (\CC^\times)^k$ acts on $V$ with constant weight $w$ (so $T$ fixes $\cF$). Then the set-theoretic support of $\on{Spec}(H^{*}_{T}(\cF))$ is the image of the embedding
    $$\ft\rightarrow \ft\times \CC^n, \quad \lambda\mapsto (\lambda, w(\la),\ldots,w(\la))$$
    where $w(\la)$ is considered as a complex number.

    In particular, the specialization of $\on{Spec}(H^{*}_{T}(\cF))$ to any point $\la \in \ft$ is set-theoretically supported on a single point. 
\end{cor}

We have similar results for orthogonal and symplectic flag varieties. Consider a symplectic (resp. orthogonal) vector space $V$ of dimensions $m=2n$ (resp. $2n+1$). We assume that the action of $T$ preserves the corresponding bilinear form. Then $T$ acts on $\cF$, the variety of maximal isotropic flags of $V$. 

A maximal isotropic flag of $V$ is again given by $F^{\bullet}= F^{1}\subset\ldots\subset F^{n}$ where $\dim F^i= i$. Similarly, we have the line bundles $\cF^i$ and the first Chern classes $x_i$, $i=1,\ldots,m$. 
The following proposition is a particular case of our discussion in Section \ref{descr_gen_rel_equiv_cohom_partial_via_chern}. Namely, for $\mathfrak{g}=\mathfrak{sp}_{2n}$ or $\mathfrak{g}=\mathfrak{so}_{2n+1}$, we take $p=1$ and $V_1=V$.
Using the fact that 
\begin{equation*}
e_{2i}(\on{ch}V)=e_{2i}(x_1,\ldots,x_{n},-x_1,\ldots,-x_{n}))=
e_i(-x_1^2,\ldots,-x_n^2)
\end{equation*}
are the generators of $\mathbb{C}[\mathfrak{t}]^{W_G}$
we obtain the following proposition.
\begin{prop}\cite[Proposition 13.3.1]{Anderson2023EquivariantCI}
    For $\mathfrak{g}$ being equal to $\mathfrak{sp}_{2n}$ or $\mathfrak{so}_{2n+1}$ we have 
    $$
    H^{*}_{T}(\cF)= \CC[\ft][x_1,...,x_n]/(e_i(-x_{1}^{2},...,-x_{n}^{2})- c_{2i}(V))_{i=1,\ldots,n}.
    $$

\end{prop}
\begin{cor} \label{single point support 2}
    As a consequence, if $T$ fixes $V$, the set-theoretic support of $\on{Spec}(H^{*}_{T}(\cF))$ is $\ft\times \{0\}\subset \ft\times \CC^n$.
\end{cor}

Next, we state similar results for partial flag varieties of classical types. Here, we consider $\fg(V)$ of classical types where the dimension of $V$ is $n$, $2n$, or $2n+1$, respectively. Consider a composition of $\mu= (\mu_1,\ldots,\mu_l)$ of $n$ (so $\mu_i\geqslant 1$ and $\mu_1+\ldots+\mu_l= n$). Write $\cF_\mu$ for the partial (isotropic) flag variety of $V$ associated with $\mu$. In other words, $\cF_\mu$ parametrizes the chain of subspaces $F^\bullet= F^1\subset... \subset F^{l-1}\subset V$ where $d_i:= \dim F^i= \mu_1+...+\mu_i$. The variety $\cF_\mu$ is isomorphic to $G(V)/P_\mu$ where $P_\mu= GL({\mu_1})\times\ldots\times GL({\mu_{l-1}})\times G'$ for $G'= GL(\mu_l)$ or $G'$ is of type $B_{\mu_l}$, $C_{\mu_l}$, or $D_{\mu_l}$. Recall that we write $\cF$ for the full flag variety, then we have a natural map $\pi_\mu: \cF\rightarrow \cF_\mu$. Write $M$ and $W_{M}$ for the Levi and the Weyl group of $P_\mu$, respectively.

\begin{prop}\cite[Corollary 4.5.4]{Anderson2023EquivariantCI} \label{cohomology of partial flag}
    The natural pullback $\pi_{\mu}^{*}\colon H^{*}_{T}(\cF_\mu)\rightarrow H^{*}_{T}(\cF)$ is an injection of $\CC[\fh]$-modules. The image of $\pi_{\mu}^{*}$ is the $W_{M}$-invariant part of $H^{*}_{T}(\cF)$. In other words, we have the following explicit description.
    \begin{enumerate}
        \item Assume $\fg$ is of type A. Over $\CC[\fh]$, the image of this map is generated by $e_i(x_{{d_j}+1},\ldots, x_{d_{j+1}})$ for $0\leqslant j\leqslant l-1$.
        \item Assume $\fg$ is of type B, C. Over $\CC[\fh]$, the image of this map is generated by $e_i(x_{{d_j}+1},\ldots, x_{d_{j+1}})$ for $0\leqslant j\leqslant l-1$ and $e_i(x^{2}_{{d_{l-1}}+1},\ldots, x^{2}_{d_{l}})$.
    \end{enumerate}

\end{prop}
Similar to the case of full flag variety, the elements $e_i(x_{{d_j}+1},..., x_{d_{j+1}})$ come from the corresponding tautological vector bundles $(F^\bullet, F^{j+1}/F^{j})$.

\subsection{Fixed points of Slodowy varieties} For a nilpotent element $\chi\in \fg^*$, we have the corresponding Springer fiber $\cB_\chi$ and the Slodowy slice $S(\chi)$. Consider a torus $T\subset G$ so that $T$ centralizes the $\fs\fl_2$ triple that we use to construct $S(\chi)$. Then $T$ naturally acts on $S(\chi)$ and $\cB_\chi$. Write $\fl$ for the centralizer of $T$ in $\fg$. We have $\chi\in \fl$, and it makes sense to consider the Springer fiber $\cB_\chi(L)$. Recall that we write $\cB_L$ for the flag variety of $L$.

\begin{prop} \label{fixed point Springer fiber}
    The connected components of $(\cB_\chi)^T$ are isomorphic to $\cB_\chi(L)$. Each connected component of $\cB^T$ contains precisely one connected component of $(\cB_\chi)^T$. The inclusion $(\cB_\chi)^T\hookrightarrow \cB^T$ can be viewed as $|W_G/W_L|$ copies of the natural inclusion $\cB_\chi(L)\hookrightarrow\cB_L$.
\end{prop}
\begin{proof}
    The variety $(\cB_\chi)^T$ parametrizes the Borel subalgebras $\fb\subset \fg$ that contain $\chi$. Because $\chi\in \fl$, we have $\chi\in \fb$ if and only if $\chi \in \mathfrak{b} \cap \mathfrak{l}$. Therefore, the intersection of $(\cB_\chi)^T$ and each connected component of $\cB^T$ is isomorphic to $\cB_\chi(L)$ via the map $\fb\mapsto \fb\cap \fl$. 
    
    The Springer fiber $\cB_\chi(L)$ is connected (see \cite{Spaltenstein}). Therefore, we have a bijection between the connected components of $(\cB_\chi)^T$ and the connected $\cB^T$. The latter is parameterized by $W_G/W_L$ by \cref{smL}. 
\end{proof}
Next, we further assume that $\chi$ is regular in $\fl$. In this case, we have $S(\chi)^T= (\cB_\chi)^T$. 
\begin{cor} \label{easy fixed point Springer fiber}
    Each connected component of $(\cB_\chi)^T$ (or $S(\chi)^T$) is a single point, and they are naturally labeled by $[w]\in W_G/W_L$.
\end{cor}

A more general result for parabolic Slodowy variety $S(\chi, \fp)$ is as follows.
\begin{prop} \label{prop fixed point parabolic Slodowy}
    $S(\chi, \fp)^T$ is a discrete set of points, these points are labeled by the free $W_L$-$W_M$ double cosets of $W_G$.
\end{prop}
\begin{proof}
    See \cref{app_torus_fixed}.
\end{proof}

\section{Conical symplectic singularities and symplectic duality}\label{sec_sympldual}
 

\subsection{Symplectic duality: basic properties}\label{sect_sympl_duality_basic} We mostly follow \cite{kmp} in this section. 
 Let $X$ be a conical symplectic singularity, and assume that the Poisson bracket on $X$ has degree two. Let $\mathbb{C}[X]^2 \subset \mathbb{C}[X]$ be the degree two component, this is a Lie algebra w.r.t. $\{\,,\,\}$. Let $Z_X$ and $T_X$ be as in \cref{conical sympl sing}. 
    
    We assume the natural identification of Lie algebras $\on{Lie}Z_X \simeq (\mathbb{C}[X]^2,\{\,,\,\})$. Note that \cite[Proposition 2.14]{Losev4} implies the identification if $\CC[X]^2$ is a reductive Lie algebra. Let $\mathfrak{t}^*_{X,\mathbb{Z}}=\on{Hom}(T_X,\mathbb{C}^\times)$. The action of $T_X \times \mathbb{C}^\times$ on $X$ induces the bigrading $\mathbb{C}[X]=\bigoplus_{\mu \in \mathfrak{t}^*_{X,\mathbb{Z}},\, k \in \mathbb{Z}_{\geqslant 0}}\mathbb{C}[X]^k_\mu$. Note that $\mathbb{C}[X]^2_0$ identifies with $\mathfrak{t}_X$.

    Let $\widetilde{X}_{\mathfrak{h}_X}$ be the deformation of $\widetilde{X}$ over $\mathfrak{h}_X=H^2(\widetilde{X}^{\mathrm{reg}},\mathbb{C})$ defined in \cref{conical sympl sing}. Let $\mathcal{D}_{\hbar, \mathfrak{h}_X}$ be the canonical graded formal quantization of $\widetilde{X}_{\mathfrak{h}_X}$, $\mathcal{A}_{\hbar, \mathfrak{h}_X}$ be the algebra of global sections, and let $\mathcal{A}_{\hbar, \mathfrak{h}_X}^2$ be the degree two component. There is an exact sequence of Lie algebras:
    \begin{equation*}
    0 \rightarrow \mathfrak{h}_X \oplus \mathbb{C}\hbar \rightarrow \mathcal{A}_{\hbar, \mathfrak{h}_X}^2 \rightarrow \mathbb{C}[X]^2 \rightarrow 0.
    \end{equation*}
Passing to the zero weight space w.r.t. $T_X$ we obtain the exact sequence:
    \begin{equation}\label{exact_seq_alg}
    0 \rightarrow \mathfrak{h}_X \oplus \mathbb{C}\hbar \rightarrow \mathcal{A}_{\hbar, 
    \mathfrak{h}_X,0}^2 \rightarrow \mathfrak{t}_X \rightarrow 0.
    \end{equation}

Let us describe another exact sequence that one can associate to a symplectic singularity $X$. Let us assume that the $\mathbb{Q}$-factorial terminalization $\widetilde{X}$ is smooth (so we will denote it by $Y$). The odd cohomology of $Y$ vanishes (\cite[Proposition 2.5]{BPW}) so we have an exact sequence: 
\begin{equation}\label{exact_seq_cohom}
0 \rightarrow \mathfrak{t}_X \oplus \mathbb{C}\hbar \rightarrow H^2_{T_X \times \mathbb{C}^\times}(Y,\CC) \rightarrow \mathfrak{h}_X \ra 0.
\end{equation}
The first map is a composition of the pull back homomorphism $H^2_{T_X \times \mathbb{C}^\times}(\on{pt},\CC)\to  H^2_{T_X \times \mathbb{C}^\times}(Y,\CC)$ and the identification $\mathfrak{t}_X \oplus \mathbb{C}\hbar\simeq H^2_{T_X \times \mathbb{C}^\times}(\on{pt},\CC)$. The second map is the restriction homomorphism $ H^2_{T_X \times \mathbb{C}^\times}(Y,\CC)\to H^2(Y,\CC)=\fh_X$.

The basic idea of symplectic duality is that often conical symplectic singularities come in dual pairs $(X, X^{\vee})$. We refer the reader to \cite{BPWII} for details on symplectic duality. Let us recall some basic properties of it. 
One prediction of symplectic duality (formulated in \cite[Section 5.1]{kmp}) is that for the dual pairs, exact sequences (\ref{exact_seq_alg}), (\ref{exact_seq_cohom}) identify with each other:  there exists a natural identification $\mathcal{A}^2_{\hbar, \mathfrak{h}_X,0} \simeq H^2_{T_{X^\vee} \times \mathbb{C}^\times}(Y^\vee,\CC)$ that induces the identifications:
\begin{equation*}
\xymatrix{
0 \ar[r] & \mathfrak{h}_X \oplus \mathbb{C}\hbar \ar[r] \ar[d]_{\simeq} & \mathcal{A}^2_{\hbar, \mathfrak{h}_X,0} \ar[r] \ar[d]_{\simeq} & \mathfrak{t}_X \ar[r]  \ar[d]_{\simeq} & 0 \\
0 \ar[r] & \mathfrak{t}_{X^\vee} \oplus \mathbb{C}\hbar \ar[r] & H^2_{T_{X^\vee} \times \mathbb{C}^\times}(Y^\vee,\CC) \ar[r] & \mathfrak{h}_{X^\vee} \ar[r] & 0.
}
\end{equation*}
In particular, we should have canonical identifications
\begin{equation}\label{ident_vect_sp_sympl_dual}
\mathfrak{h}_X \simeq \mathfrak{t}_{X^\vee},~ \mathfrak{t}_X \simeq \mathfrak{h}_{X^\vee},
\end{equation}
as was conjectured in \cite{BPWII}.

Recall now that in Section \ref{subsub_chambers} we described  stratifications of $\mathfrak{h}_X$, $\mathfrak{t}_X$ into chambers separated by walls.
Another expectation of the symplectic duality is that the identifications (\ref{ident_vect_sp_sympl_dual}) are compatible with the stratifications described in Section \ref{subsub_chambers}. Namely, they induce the isomorphisms 
\begin{equation}\label{sing_ident_vect_sp_sympl_dual}
\mathfrak{h}^{\mathrm{sing}}_X \simeq \mathfrak{t}^{\mathrm{sing}}_{X^\vee},~ \mathfrak{t}^{\mathrm{sing}}_X \simeq \mathfrak{h}^{\mathrm{sing}}_{X^\vee}.
\end{equation}

In particular, a choice of a generic cocharacter $\nu\colon \mathbb{C}^\times \rightarrow T_{X}$ determines a chamber in $\mathfrak{h}_{X,\mathbb{R}} \setminus \mathfrak{t}_{X,\mathbb{R}}^{\mathrm{sing}}$ that by Proposition \ref{namikawa's result_descr_of_h_sing}  has the form $w(\on{Amp}(\widetilde{X}^{\vee}))$ for some $\mathbb{Q}$-factorial terminalization $\widetilde{X}^\vee \rightarrow X^\vee$ and $w \in W_{X^\vee}$.

In this way, the choice of $\nu\colon \mathbb{C}^\times \rightarrow T_{X}$ determines a choice of a $\mathbb{Q}$-factorial terminalization $\widetilde{X}^\vee$ on the dual side as well as the element $w \in W_{X^\vee}$. We will say that $\nu$ {\emph{corresponds}} to $\widetilde{X}^\vee$ and will use the identification $H^2(\widetilde{X}^{\vee,\mathrm{reg}},\mathbb{C}) \iso \mathfrak{h}_{Y^\vee}$ given by $\lambda \mapsto w(\lambda)$ (i.e., $\nu \in \on{Amp}(\widetilde{X}^\vee)$ w.r.t. the identification above).


Assume now that both $X$ and $X^\vee$ admit symplectic resolutions.

\begin{rmk}\label{symplectic resolution iff one fixed point}
It is expected that $X$ has a symplectic resolution if and only if $(X^\vee)^{T_{X^\vee}}$ consists of {\emph{one}} point.    
\end{rmk}

The symplectic duality predicts the existence of a canonical identification of finite sets:
\begin{equation}\label{canonical:)_ident_fixed_points}
Y^{T_X} \iso (Y^\vee)^{T_{X^\vee}},~p \mapsto p^\vee.
\end{equation}

\subsection{Hikita-Nakajima conjectures} \label{section Hikita general}
Let $\nu\colon \CC^\times \rightarrow T_X$ be a generic cocharacter meaning that the corresponding element of $\mathfrak{t}_X$ is {\emph{not}} contained in $\mathfrak{t}_X^{\mathrm{sing}}$. Let $Y^\vee \rightarrow X^\vee$ be the symplectic resolution that corresponds to $\nu$ in the sense of Section \ref{sect_sympl_duality_basic}.
We will always assume that $\nu$ is ``deep enough'' in the ample chamber $\on{Amp}(Y^\vee)$. Note that by \cite[Corollary B.1]{BPW}, this, in particular, implies that the abelian localization holds for $(\nu,Y^\vee)$.

The following conjectures are natural to make (part (i) goes back to Hikita \cite{Hikita}, and part (ii) goes back to Nakajima \cite[Section 5.6]{Kamnitzer_symplectic}). We will call (i) {\emph{Hikita conjecture}} and (ii) will be called {\emph{equivariant Hikita-Nakajima conjecture}}. Recall that $X^{\nu(\CC^\times)}$ is the {\em{schematic}} fixed points of $X$ (see Section \ref{sec_schem_fixed}).

  \begin{conj}\label{naiveHikita}
        \begin{itemize}
            \item[(i)] There is a graded algebra isomorphism $\CC[X^{\nu(\CC^\times)}]\simeq H^*(Y^\vee)$;
            \item[(ii)] There is a graded isomorphism of $\CC[\fh_X, \hbar]$-algebras $\cC_\nu(\cA_{\hbar, \fh_X})\simeq H^*_{T_{X^\vee}\times \CC^\times}(Y^\vee)$, where $\cC_\nu(\cA_{\hbar, \fh_X})$ is the Cartan subquotient with respect to a generic one-parameter subgroup $\nu\colon \CC^\times \to T_{X}$ (corresponding to the resolution $Y^\vee$).
        \end{itemize}
    \end{conj}
    We remark that (ii) implies (i) by specifying to the point $0\in \fh_X\oplus \CC$. Moreover, setting $\hbar=0$ in (i), we obtain the identification 
    \begin{equation}\label{classical_equiv_hikita}
    \mathbb{C}[X_{\mathfrak{h}_X}^{\nu(\mathbb{C}^\times)}] \simeq H^*_{T_{X^\vee}}(Y^\vee)
    \end{equation}
    that we will call {\emph{classical equivariant Hikita-Nakajima conjecture}}.
    In what follows, we show that the conjecture as stated is not true and propose a refined version that we can prove in some cases and which implies Conjecture \ref{naiveHikita} in some cases. 
    \begin{rmk}
        Note that $H^*_{T_{X^\vee} \times \mathbb{C}^\times}(\widetilde{X}^\vee)$ is a free $H^*_{T_{X^\vee} \times \mathbb{C}^\times}(\on{pt})$-module. In particular, $H^*_{T_{X^\vee} \times \mathbb{C}^\times}(\widetilde{X}^\vee)$ is flat over $H^*_{T_{X^\vee} \times \mathbb{C}^\times}(\on{pt})$. However, that is not true for the left-hand side of (ii), a counterexample is described in \cref{counterflat}. This is the first indication that the Hikita conjecture might be false. 
    \end{rmk}

    \begin{rmk}\label{rem_when_expect_hn_coulomb}
    It is still expected that Conjecture \ref{naiveHikita} holds as stated in the following case (this is the setting in which Nakajima formulated it in \cite[Section 1(viii)]{nakajima_coulombI}). Let $(G_{\mathrm{gauge}},N)$ be a pair of a reductive (gauge) group $G_{\mathrm{gauge}}$ and a finite-dimensional $G_{\mathrm{gauge}}$-module ${\bf{N}}$. We can associate to $(G_{\mathrm{gauge}},N)$ a pair of varieties 
    \begin{equation*}
    \mathcal{M}_C=\mathcal{M}_{C}(G_{\mathrm{gauge}},{\bf{N}}),~\mathcal{M}_H=\mathcal{M}_{H}(G_{\mathrm{gauge}},{\bf{N}})
    \end{equation*}
    called Coulomb and Higgs branches of the theory defined by $(G_{\mathrm{gauge}},{\bf{N}})$ (for the definition of $\mathcal{M}_C$ see \cite[Section 3]{BFNII}, $\mathcal{M}_H$ is the Hamiltonian reduction of $T^*{\bf{N}}$ by the action of $G_{\mathrm{gauge}}$). They are expected to be symplectic dual. We note that $\cM_C$ is always a symplectic singularity (see \cite{bellamy} and \cite{weekes_sympl_sing}). Assume that $\cM_H$ is a symplectic singularity.
    Let $\mu\colon T^*{\bf{N}} \rightarrow \mathfrak{g}_{\mathrm{gauge}}^*$ be a moment map for $G_{\mathrm{gauge}} \curvearrowright T^*{\bf{N}}$ and assume that for a generic character $\xi \in (\mathfrak{g}_{\mathrm{gauge}}/[\mathfrak{g}_{\mathrm{gauge}},\mathfrak{g}_{\mathrm{gauge}}])^* \subset \mathfrak{g}^*_{\mathrm{gauge}}$, the action of $G_{\mathrm{gauge}}$ on $\mu^{-1}(\xi)$ is free, and that $\mu^{-1}(\xi)$ is smooth. Then  Conjecture \ref{naiveHikita} should hold for $X=\mathcal{M}_{C}$, $X^\vee=\mathcal{M}_{H}$. 
    \end{rmk}

  \section{Refined BVLS duality and symplectic duality for parabolic Slodowy varieties}\label{sect_refined_BVLS}
    Recall that $\fg$ is a simple Lie algebra, and $G$ is the corresponding adjoint group. Let $\fg^\vee$, $G^\vee$ be the Langlands dual Lie algebra and group, respectively. Pick the Cartan subalgebras $\fh\subset \fg$ and $\fh^\vee\subset \fg^\vee$. We start with recalling a standard Barbasch-Vogan-Lusztig-Spaltenstein (BVLS) duality. 

    Let $\OO^\vee\subset (\fg^\vee)^*$ be a nilpotent $G^\vee$-orbit, let $\chi^\vee\in \OO^\vee$, and $e^\vee\in \fg^\vee$ be such that $\chi^\vee=(e^\vee, \bullet)$. Pick an $\fs\fl_2$-triple $(e^\vee, f^\vee, h^\vee)$ such that $h^\vee\in \fh^\vee\simeq \fh^*$. Recall that the maximal ideals in $\mathcal{U}(\fg)$ are in bijection with the central characters, and $Z(\mathcal{U}(\fg))\simeq S(\fh)^W\simeq \CC[\fh^*]^{W}$. Let $I(\OO^\vee)\subset \mathcal{U}(\fg)$ be the maximal ideal with the central character $\frac{1}{2}h^\vee$. The associated graded ideal $\gr I(\OO^\vee)\subset S(\fg)$ defines a closed subset $V\subset \fg^*$, and there is a unique open $G$-orbit $\OO\subset V$. Following \cite{BarbaschVogan1985} we define the BVLS duality map $D: \cN^\vee/G^\vee \to \cN/G$ by setting $D(\OO^\vee)=\OO$. The orbits in the image of the map $D$ are called \emph{special}. A key property of the map $D$ is that it sends the Bala-Carter saturation to the Lusztig-Spaltenstein induction (see Sections \ref{subsec_saturation}, \ref{sect_nilp_orbits_covers_and_bir_ind} for the definitions). Namely, for a Levi $L^\vee\subset G^\vee$ and a nilpotent orbit $\OO_{L^\vee}^\vee\subset (\fl^\vee)^*$ we have:
    \begin{equation*}
    D(\Sat_{L^\vee}^{G^\vee} \OO_{L^\vee}^\vee)=\Ind_L^G D(\OO_{L^\vee}^\vee).
    \end{equation*}

    We proceed with the refined version of the map $D$, namely a map $\widetilde{D}$ from the set of nilpotent orbits in $(\fg^\vee)^*$ to the set of covers of nilpotent orbits in $\fg^*$. Let $(L^\vee, \OO_{L^\vee}^\vee)$ be the Bala-Carter pair attached to the orbit $\OO^\vee$. The orbit $D(\OO^\vee_{L^\vee})$ is a special orbit $\OO_L$ in $\fl^*$, and let $\widetilde{\OO}_L$ be the universal $L$-equivariant cover of $\OO_L$. We define $\widetilde{D}(\OO^\vee)=\Bind_L^G \widetilde{\OO}_L$. The following are the key properties of this map.

    \begin{prop} \label{about cover} \cite[Proposition 9.2.1]{LMBM} 
        \begin{itemize}
            \item[(i)] $\widetilde{D}(\OO^\vee)$ is a $G$-equivariant cover of $D(\OO^\vee)$.
            \item[(ii)] The map $\widetilde{D}$ is injective.
            \item[(iii)] Let $\cA$ be the canonical filtered quantization of $\CC[\widetilde{D}(\OO^\vee)]$. Note that it is endowed with a quantum comoment map $\Phi\colon \mathcal{U}(\fg)\to \cA$. Then $\Ker(\Phi) = I(\OO^\vee)$.
            \item[(iv)] For $\OO^\vee$ distinguished, the cover $\widetilde{D}(\OO^\vee)$ is birationally rigid.
           \item[(v)] \cite[Remark 1.3]{MBMYu} The Galois group of the cover $\widetilde{D}(\OO^\vee)\to D(\OO^\vee)$ is isomorphic to the Lusztig quotient $\Bar{A}(\OO^\vee)$.
        \end{itemize}
    \end{prop}


    Recall that $S(\chi^\vee)$ is the Slodowy slice to $\OO^\vee$ at the point $\chi^\vee$. Let 
    \begin{equation*}
    X^\vee=S(\chi^\vee)\cap \cN^\vee,\, X=\Spec(\CC[\widetilde{D}(\OO^\vee)]).
    \end{equation*}
    For the orbit $\OO^\vee$ let $(L^\vee, \OO_{L^\vee}^\vee)$ be the Bala-Carter pair (see Proposition \ref{BC saturation}). 
    \begin{prop}\label{prop_spaces_for_orbits_slodowy}
        The following are true:
        \begin{itemize}
            \item(i) For most orbits $\OO^\vee$ we have $\fh_{X^\vee}\simeq \fh^\vee$;
            \item(ii) For any orbit $\OO^\vee$ we have $\fz_{X^\vee}\simeq \fz_{G^\vee}(e^\vee, f^\vee, h^\vee)$, and therefore $\fz_{X^\vee}\simeq \fz(\fl^\vee)$, $T_{X^\vee}=Z ({L^\vee})^\circ$ (connected component of $1$ in the center of $L^\vee$);
            \item(iii) For most orbits $\OO^\vee$ we have $\fz_{X}\simeq \fg$, and therefore $\fz_X\simeq \fh$;
            \item(iv) For any orbit $\OO^\vee$ we have $\fh_{X}\simeq \mathfrak{X}(\mathfrak{l})$.
        \end{itemize}
    \end{prop}
    \begin{proof}
         For (i), we refer to \cref{universal deformation of Slodowy slice}. (ii) is well-known.
         For (iii), we refer to \cref{automotphisms of partial flag}. (iv) is shown in \cite[Proposition 7.2.2(i)]{LMBM}.
    \end{proof}
    A remark about ``most" orbits is in order. For (i), we refer to \cite[Theorem 1.3]{Lehn_2011}. The dual phenomena for (iii) is the one of ``shared orbits" considered in \cite{BrylinskiKostant1994}. The full list of shared orbits and shared covers is known, thanks to the results of \cite{FJLS2023}. In particular, one can check that $\widetilde{D}(\OO^\vee)$ admits a shared cover for a group of a higher rank if and only if $\OO^\vee$ is one of the exceptions for (i). 
    
    We expect that $S(\chi^\vee)\cap \cN^\vee$ and $\Spec(\CC[\widetilde{D}(\OO^\vee)])$ form a symplectic dual pair. In this paper we focus on a special case of the duality above. From now on we impose an additional assumption that the nilpotent $\chi^\vee \in \mathcal{N}^\vee$ is regular for some Levi $\mathfrak{l}^\vee \subset \mathfrak{g}^\vee$. Equivalently, this means that the corresponding orbit $\mathbb{O}^\vee = G^\vee \chi^\vee$ is the saturation of the regular nilpotent orbit $\mathbb{O}^\vee_{L^{\vee}} \subset (\mathfrak{l}^{\vee})^*$. Note that $D(\mathbb{O}^\vee_{L^\vee})=\{0\}$ is the zero orbit. We conclude that $\widetilde{D}(\mathbb{O}^\vee)=\on{Bind}_{L}^{G}\{0\}$ can be described as follows. Recall that $P \subset G$ is a parabolic with Levi $L$. Consider the natural morphism $\mu\colon T^*\mathcal{P} \rightarrow \mathcal{N}$, the image of this morphism is the closure $\overline{\mathbb{O}}$ of some orbit $\mathbb{O}$ (recall that $\mathbb{O}=\on{Ind}_L^G \{0\}$). Now, $\on{Bind}_{L}^{G}\{0\}=\widetilde{\mathbb{O}}$ is nothing else but $\mu^{-1}(\mathbb{O})$.  We see that in this case $X=\Spec(\CC[\widetilde{D}(\OO^\vee)])$ identifies with $\on{Spec}\CC[T^*\mathcal{P}]$ and so $Y=T^*\mathcal{P}$. We conclude that 
\begin{equation*}
Y^\vee=\widetilde{S}(\chi^\vee),\quad Y=T^*\mathcal{P}
\end{equation*}
are symplectic dual. Note that both of them are particular cases of {\emph{parabolic}} Slodowy varieties, namely we see that $\widetilde{S}(\mathfrak{p}^\vee,\mathfrak{b}^\vee)=\widetilde{S}(\chi^\vee)$ is dual to $T^*\mathcal{P}=\widetilde{S}(\mathfrak{b},\mathfrak{p})$.

This suggests a candidate for a symplectic dual to $\widetilde{S}(\mathfrak{p}^\vee,\mathfrak{q}^\vee)$. It is natural to conjecture that 
\begin{equation*}
Y^\vee=\widetilde{S}(\mathfrak{p}^\vee,\mathfrak{q}^\vee),\, Y=\widetilde{S}(\mathfrak{q},\mathfrak{p})
\end{equation*}
are symplectic dual.

\begin{rmk}\label{disconnected together}
    Recall from \cref{disconnected} that the varieties $Y^\vee$, $Y$ may be disconnected. In fact, the two examples presented are exactly of form $\widetilde{S}(p^\vee, q^\vee)$ and $\widetilde{S}(\fq, \fp)$ respectively for $\fp$ and $\fq$ being parabolic subalgebras with Levi factors $\fl=\fg\fl(1)^2\times \fs\fp(2) \subset \fs\fp(6)$ and $\fm=\fg\fl(3)\subset \fs\fp(6)$ respectively. We expect that $\widetilde{S}(p^\vee, q^\vee)$ and $\widetilde{S}(\fq, \fp)$ always have the same number of connected components.
\end{rmk}

\begin{warning}
Strictly speaking,  not $Y^\vee$ and $Y$ are  dual but $X^\vee=\on{Spec}\CC[\widetilde{S}(\mathfrak{p}^\vee,\mathfrak{q}^\vee)]$ should be dual to $X=\on{Spec}\CC[\widetilde{S}(\mathfrak{q},\mathfrak{p})]$.    There is a freedom in the choice of a resolution. We will later see that the natural choices are 
\begin{equation*}
Y^\vee = \widetilde{S}(\mathfrak{p}^\vee_-,\mathfrak{q}^{\vee}_-),~Y=\widetilde{S}(\mathfrak{q},\mathfrak{p}).
\end{equation*}
\end{warning}

\begin{rmk}
If $\mathfrak{g}=\mathfrak{sl}(n)$, then the symplectic duality between $\widetilde{S}(\mathfrak{p}^\vee,\mathfrak{q}^\vee),\, \widetilde{S}(\mathfrak{q},\mathfrak{p})$ is well-studied, see for example \cite[Section 10.2.2]{BPWII} (note that in this case varieties above are type $A$ quiver varieties, see \cite{maffei}). On the other hand, not much is known outside of type $A$. For example, we do not know how to describe the space $\mathfrak{h}_{X^\vee}=H^2(\widetilde{S}(\mathfrak{p}^\vee,\mathfrak{q}^\vee),\CC)$ in general. Namely, recall that for $\mathfrak{p}^\vee=\mathfrak{b}^\vee$, we have $\widetilde{S}(\mathfrak{b}^\vee,\mathfrak{q}^\vee)=T^*(G^\vee/Q^\vee)$, so $H^2(\widetilde{S}(\mathfrak{p}^\vee,\mathfrak{q}^\vee),\CC)=\mathfrak{X}(\mathfrak{m}^\vee)$. For $\mathfrak{q}^\vee=\mathfrak{b}^\vee$, in most cases we have $\mathfrak{h}_{X^\vee}=\mathfrak{h}^\vee$ by Proposition \ref{prop_spaces_for_orbits_slodowy} (ii) so in general one might guess that in some cases the natural restriction homomorphism $H^2(T^*(G^\vee/Q^\vee),\CC) \rightarrow H^2(\widetilde{S}(\mathfrak{p}^\vee,\mathfrak{q}^\vee),\CC)$ is surjective. A case by case analysis similar to \cite{LehnNamikawaSorger} can be done, but we do not know any general result on when this is the case. 

For example, if $\mathfrak{m}^\vee\simeq \fs\fp(2)\times \fg\fl(1)\times \fg\fl(1)\subset \fs\fp(6)$ and $\fl^\vee\simeq \fg\fl(3)$, the map $H^2(T^*(G^\vee/Q^\vee),\CC) \rightarrow H^2(\widetilde{S}(\mathfrak{p}^\vee,\mathfrak{q}^\vee),\CC)$ is surjective, see \cref{appendixA}. However, if we consider $\fm^\vee\simeq \fg\fl(7)\times \fg\fl(1)\times \fs\fo(5)\subset \fs\fo(21)$ and $\fl^\vee\simeq \fg\fl(3)^{\times 3}\times \fs\fo(3)\subset \fs\fo(21)$, it is easy to show that the map $H^2(T^*(G^\vee/Q^\vee),\CC) \rightarrow H^2(\widetilde{S}(\mathfrak{p}^\vee,\mathfrak{q}^\vee),\CC)$ cannot be surjective by computing dimensions of both sides.

Finally, we note that the description of the Lie algebra $\fz_X$ seems to be even more complicated. We are not aware of any general results in this direction.
\end{rmk}

The goal of this paper is to investigate the Hikita-Nakajima conjecture for the pair $\widetilde{S}(\mathfrak{p}^\vee_-,\mathfrak{q}^\vee_-)$, $\widetilde{S}(\mathfrak{q},\mathfrak{p})$. We will see that it is {\emph{not}} correct as stated (outside of type $A$) but has a natural modification that we will prove. In type $A$, our modification will imply the actual Hikita-Nakajima conjecture (see Appendix \ref{app_HN_type_A}).

In general, it would be very interesting to study the symplectic duality for the pair $\widetilde{S}(\mathfrak{p}^\vee_-,\mathfrak{q}^\vee_-)$, $\widetilde{S}(\mathfrak{q},\mathfrak{p})$. Note that although many things are not known, these varieties are very explicit and closely related to the ``classical'' Lie theory. For example, by \cite{w}, categories $\mathcal{O}$ over the quantizations of $\widetilde{S}(\mathfrak{q},\mathfrak{p})$ are equivalent to {\emph{parabolic singular}} categories $\mathcal{O}$. One of the predictions of symplectic duality is that the categories $\mathcal{O}$ for dual varieties are {\emph{Koszul dual}} (see \cite[Section 10.1]{BPWII}). For the pair as above, this statement indeed follows from \cite{bak_koszul} and \cite{w}. 

Let us also emphasize that there is an interesting combinatorics associated with the varieties $\widetilde{S}(\mathfrak{q},\mathfrak{p})$. See, for example, the description of the torus fixed points of $\widetilde{S}(\mathfrak{q},\mathfrak{p})$ given in Appendix \ref{app_torus_fixed}. In type $A$, it gives an alternative (more ``Lie theoretic'') way of thinking about the fixed points of the type $A$ Nakajima quiver varieties (compare with \cite[Section 3.4]{dinkins_elliptic_of_type_A}). The explicit relation between two descriptions of fixed points is a part of the work in progress \cite{losev_krylov_wc_for_gieseker}. We discuss the expected combinatorial trace of the Hikita-Nakajima conjecture in \cref{combinatorics}.

\section{Hikita conjecture for nilpotent orbits: additional structures and counterexamples}\label{sect_counter}

In this section, we recall additional structures that we have on both sides of the classical Hikita-Nakajima conjecture for $X^\vee=S(\chi^\vee) \cap \mathcal{N}^\vee$ and $X=\Spec(\CC[\widetilde{D}(\OO^\vee)])$. While they share many similarities, we observe the differences that will subsequently provide counterexamples to the original Hikita conjecture and lead to its modification that we propose in \cref{equivariant conjecture section}. To simplify the notations, we denote by $H^*_{?}(\bullet)$, $H_*^{?}(\bullet)$ the equivariant cohomology and Borel-Moore homology with coefficients in $\CC$ respectively.

\subsection{Cohomological side}
Recall that  $T_{X^\vee}=T_{\chi^\vee}\simeq T_{e^\vee} \subset T^\vee$ is a maximal torus in $Z_{e^\vee}$, the pointwise centralizer of $\{e^\vee, h^\vee, f^{\vee}\}$ in $G^\vee$. Torus $T_{\chi^\vee}$ acts naturally on $\widetilde{S}(\chi^\vee)$ and $\cB_{\chi^\vee}$. Write $L^\vee$ for the connected component of the centralizer of $T_{\chi^\vee}$ in $G^\vee$. Let $\mathfrak{h}_{\chi^\vee}=\on{Lie}T_{\chi^\vee}$. At the cohomological side of the Hikita-Nakajima conjecture we deal with the $\CC[\mathfrak{h}_{\chi^\vee}]$-algebra $H^*_{T_{\chi^\vee}}(\widetilde{S}(\chi^\vee))$. The $\CC[\mathfrak{h}_{\chi^\vee}]$-module $H^*_{T_{e^\vee}}(\widetilde{S}(e^\vee))$ has the natural structure of a module over the Weyl group $W$ of $G^\vee$ (via the standard ``Springer'' action). Moreover, this action of $W$ can be extended to an action of $\widetilde{W}= S^\bullet\mathfrak{h}^{\vee} \# W$ as follows.

Consider the identification $H^2_{T_{\chi^\vee}}(T^*\mathcal{B}^\vee,\CC) \iso \mathfrak{h} \oplus \mathfrak{h}_{\chi^\vee}$, then the element $\xi \in \mathfrak{h}$ acts on $H^*_{T_{\chi^\vee}}(\widetilde{S}(\chi^\vee))$ via the multiplication by the image of $\xi \in H^2_{T_{\chi^\vee}}(T^*\mathcal{B}^\vee,\CC)$ under the restriction homomorphism $H^2_{T_{\chi^\vee}}(T^*\mathcal{B}^\vee) \rightarrow H^2_{T_{\chi^\vee}}(\widetilde{S}(\chi^\vee))$.

\begin{rmk}
More generally, one can consider the convolution algebra $H^{G^\vee \times \CC^\times}_{*}(\widetilde{\mathcal{N}}^\vee \times_{\mathcal{N}^\vee} \widetilde{\mathcal{N}}^\vee)$ that is isomorphic to the {\emph{degenerate affine Hecke algebra}} $H$. Recall that $H$ is the graded algebra over $\CC[\hbar]$, that specializes  to $\widetilde{W}$ for $\hbar=0$ (in particular, we have the identification $H^{G^\vee}_{*}(\widetilde{\mathcal{N}}^\vee \times_{\mathcal{N}^\vee} \widetilde{\mathcal{N}}^\vee) \simeq \widetilde{W}$). The convolution algebra  $H^{G^\vee \times \CC^\times}_{*}(\widetilde{\mathcal{N}}^\vee \times_{\mathcal{N}^\vee} \widetilde{\mathcal{N}}^\vee)$ acts naturally on 
\begin{equation*}
H_*^{T_{\chi^\vee} \times \CC^\times}(\widetilde{S}(\chi^\vee)) \simeq H^*_{T_{\chi^\vee}\times \CC^\times}(\widetilde{S}(\chi^\vee))=H^*_{T_{\chi^\vee}\times \CC^\times}(\mathcal{B}_{\chi^\vee}),
\end{equation*} 
where the first isomorphism is given by the Poincare duality.
Similarly, algebra $H$ acts naturally on $H_*^{T_{\chi^\vee} \times \CC^\times}(\mathcal{B}_{\chi^\vee})$, and modules $H^*_{T_{\chi^\vee}\times \CC^\times}(\mathcal{B}_{\chi^\vee})$, $H_*^{T_{\chi^\vee} \times \CC^\times}(\mathcal{B}_{\chi^\vee})$ are dual to each other.
\end{rmk}

Let us now describe the irreducible modules over $\widetilde{W}$ geometrically (more generally, the analogous description of irreducible $H$-modules is given in \cite[Theorem 1.15]{lusztig_cuspidalIII}). 


Irreducible $\widetilde{W}$-modules are parametrized by conjugacy classes of triples $(\chi^\vee,x,\rho)$, where $\chi^\vee \in \mathcal{N}^\vee$ is a nilpotent element, $x \in \mathfrak{z}_{\mathfrak{g}^{\vee}}(\chi^\vee)$ is semisimple, and $\rho$ is an irreducible representation of $A(e^\vee,x):=Z_{G^\vee}(e^\vee,x)/Z_{G^\vee}(e^\vee,x)^0$ that appears in $H^*(\cB_{\chi^\vee}^{x})$. Namely, we can assume that $x \in \mathfrak{h}_{\chi^\vee}$ and consider the specialization $H^*_{T_{\chi^\vee}}(\widetilde{S}_{\chi^\vee})_{x}=H^*(\widetilde{S}_{\chi^\vee}^{x})$ (see Section \ref{fixed point pullback}). Let $\mathfrak{l}_x^\vee=\mathfrak{z}_{\mathfrak{g}^\vee}(x)$. We note that if $x$ is generic then $\mathfrak{l}_x^\vee=Z_{\mathfrak{g}^\vee}(T_{\chi^\vee})=\mathfrak{l}^\vee$.

It is easy to see (compare with \cite[Proposition 6.2]{KATO1983193}),  that we have isomorphisms of $A(e^\vee,x) \times \widetilde{W}$-modules (induced by the natural pushforward maps):
\begin{equation*}
\on{Ind}_{\widetilde{W}_{L_x}}^{\widetilde{W}}H_*(\mathcal{B}_{L_{x}^\vee,e^\vee}) \iso H_*(\mathcal{B}_{e^\vee}^x),~ \on{Ind}_{\widetilde{W}_{L_x}}^{\widetilde{W}}H_*(\widetilde{S}_{L^\vee_x}(e^{\vee})) \iso H_*(\widetilde{S}({e^\vee})^{x}).
\end{equation*}
It induces the isomorphism of $\widetilde{W}$-modules:
\begin{equation}\label{iso_induction_Kato}
\on{Ind}_{\widetilde{W}_{L_x}}^{\widetilde{W}}H_*(\mathcal{B}_{L_{x}^\vee,e^\vee})_\rho \iso H_*(\mathcal{B}_{e^\vee}^x)_{\rho},~ \on{Ind}_{\widetilde{W}_{L_x}}^{\widetilde{W}}H_*(\widetilde{S}_{L^\vee_x}(e^{\vee}))_{\rho} \iso H_*(\widetilde{S}(e^\vee)^{x})_{\rho},
\end{equation}
where $\bullet_{\rho}$ is the multiplicity space of $\rho$.

Module  $H_*(\widetilde{S}^{x}(e^\vee))_{\rho}$ is called the {\emph{costandard}} module over $\widetilde{W}$ corresponding to $(e^\vee,x,\rho)$. The dual module $H_*(\mathcal{B}_{e^\vee}^{x})_{\rho}$ is called the {\emph{standard}} module. The image of the natural morphism $H_*(\mathcal{B}_{e^\vee}^{x})_{\rho} \rightarrow H_*(\widetilde{S}(e^\vee)^x)_{\rho}$ is the unique irreducible quotient of  $H_*(\mathcal{B}_{e^\vee}^{x})_{\rho}$ and is the unique irreducible submodule of $H_*(\widetilde{S}(e^\vee)^{x})_{\rho}$ (compare with \cite[Theorem 8.1.13]{Chriss-Ginzburg}).
Let us describe this irreducible module explicitly. It follows from (\ref{iso_induction_Kato}) that the module we are dealing with is nothing else but 
\begin{equation*}
\on{Ind}_{\widetilde{W}_{L_x}}^{\widetilde{W}}(\on{Im}(H_*(\mathcal{B}_{L_x^\vee,e^\vee})_{\rho} \rightarrow H_*(\widetilde{S}_{L_x^\vee}(e^\vee))_{\rho}))=\on{Ind}_{\widetilde{W}_{L_x}}^{\widetilde{W}}H^{\mathrm{top}}(\mathcal{B}_{L_x^\vee,e^\vee})_{\rho}.
\end{equation*}
To see that, recall that $H_{*}(\mathcal{B}_{L_x^\vee,e^\vee})$ lives in degrees $0,1,\ldots,2d$, where $d=\on{dim}\mathcal{B}_{L_x^\vee,e^\vee}$ while $H_{*}(\widetilde{S}_{L_x^\vee}(e^\vee))\simeq H^{4d-*}(\widetilde{S}_{L_x^\vee}(e^\vee))\simeq H^{4d-*}(\mathcal{B}_{L_x^\vee,e^\vee})$ lives in degrees $2d,2d+1,\ldots,4d$ so we see that the only term that survives is the degree $2d$ term i.e. the image of $H_{2d}(\mathcal{B}_{L_x^\vee,e^\vee}) \rightarrow H_{2d}(\widetilde{S}_{L_x^\vee}(e^\vee))\simeq H^{\mathrm{top}}(\mathcal{B}_{L_x^\vee,e^\vee})$. It remains to note that the latter map is an isomorphism.

\begin{rmk}
Note that the module $\on{Ind}_{\widetilde{W}_{L_x}}^{\widetilde{W}}H^{\mathrm{top}}(\mathcal{B}_{L_x^\vee,e^\vee})_{\rho}$ is indeed irreducible: it follows from the Springer theory that     $H^{\mathrm{top}}(\mathcal{B}_{L_x^\vee,e^\vee})_{\rho}$ is an irreducible $W_{L_x}$-module. Note also that the action of $\mathfrak{h}$ on $H^{\mathrm{top}}(\mathcal{B}_{L_x^\vee,e^\vee})_{\rho}$ is via $x$ (i.e., an element $\xi \in \mathfrak{h}$ acts via the multiplication by $\xi(x)$, here we identify $\mathfrak{h}$ with $(\mathfrak{h}^\vee)^*$ and use the embedding $\mathfrak{h}_{e^\vee} \subset \mathfrak{h}^{\vee}$). Note now that the stabilizer of $x$ in $W$ is precisely $W_{L_x}$, so we conclude from the Mackey theory that $\on{Ind}_{\widetilde{W}_{L_x}}^{\widetilde{W}}H^{\mathrm{top}}(\mathcal{B}_{L_x^\vee,e^\vee})_{\rho}$ is irreducible.
\end{rmk}




So, the cohomological side $H^*_{T_{e^\vee}}(\widetilde{S}(e^\vee))$  of the classical equivariant Hikita-Nakajima conjecture is a $\widetilde{W}$-module over $S^\bullet\mathfrak{h}_{e^\vee}$ that is free over $S^\bullet\mathfrak{h}_{e^\vee}$. Moreover, its specialization at every $x \in \mathfrak{h}_{e^\vee}$ decomposes as the direct sum $H^*(\widetilde{S}(e^\vee)^x) = \bigoplus_{\rho \in \on{Irrep}(A(e^\vee,x))} \rho \otimes H^*(\widetilde{S}(e^\vee)^x)_\rho$. 
Taking $x=0$, we see that that if the action of $A(e^\vee)$ on $H^*(\widetilde{S}(e^\vee))$ is {\em{nontrivial}} then module $H^*(\widetilde{S}(e^\vee))$ is {\emph{not}} cyclic.

Assume now that $e^\vee$ is {\em{regular}} in $\mathfrak{l}^\vee$, the Lie algebra of $L^\vee$. Let $\mathfrak{h}_{e^\vee}^{\mathrm{reg}} \subset \mathfrak{h}_{e^\vee}$ be the set of $\la \in \mathfrak{h}_{e^\vee}=\mathfrak{z}(\mathfrak{l})^*$ such that $Z_{\mathfrak{g}}(\la)=\mathfrak{l}$, i.e., $\langle \la,\beta^\vee \rangle \neq 0$ for any $\beta \in \Delta_{\mathfrak{p}} \setminus \Delta_{\mathfrak{l}}$. The module $H^*_{T_{e^\vee}}(\widetilde{S}(e^\vee))|_x=H^*(\widetilde{S}(e^\vee)^x)=H^{\mathrm{top}}(\widetilde{S}(e^\vee)^x)$ is irreducible for $x \in \mathfrak{h}_{e^\vee}^{\mathrm{reg}}$ and is  isomorphic to $\on{Ind}_{\widetilde{W}_{L}}^{\widetilde{W}} \mathbb{C}_x$. Moreover, it is cyclic over $S^\bullet\mathfrak{h}$ and is generated by $[1] \in H^*_{T_{e^\vee}}(\widetilde{S}(e^\vee))|_x$ in this case. In other words, we have an isomorphism  of $\widetilde{W}$-modules over $\mathfrak{h}_{e^\vee}^{\mathrm{reg}}$:
\begin{equation*}
H^*_{T_{e^\vee}}(\widetilde{S}(e^\vee))|_{\mathfrak{h}_{e^\vee}^{\mathrm{reg}}} \simeq \on{Ind}_{\widetilde{W}_L}^{\widetilde{W}}\CC[\mathfrak{h}_{e^\vee}^{\mathrm{reg}}],
\end{equation*}
and $H^*_{T_{e^\vee}}(\widetilde{S}(e^\vee))|_{\mathfrak{h}_{e^\vee}^{\mathrm{reg}}}$ is generated by $1$ over $S^\bullet\mathfrak{h} \otimes \CC[\mathfrak{h}_{e^\vee}^{\mathrm{reg}}]$.







\subsection{Fixed points side}\label{fixed points subsection}
Let $\widetilde{X}$ be a $\QQ$-factorial terminalization of $\operatorname{Spec}(\CC[\widetilde{D}(\mathbb{O}^\vee)])$. The action of the torus $T$ on $\operatorname{Spec}(\CC[\widetilde{D}(\mathbb{O}^\vee)])$ extends to an action on the deformation $\operatorname{Spec}(\CC[\widetilde{D}(\mathbb{O}^\vee)])_{\mathfrak{h}_{e^\vee}}$ over $\fh_{e^\vee}$. Moreover, the map $\operatorname{Spec}(\CC[\widetilde{D}(\mathbb{O}^\vee)])_{\mathfrak{h}_{e^\vee}}\to \fh_{e^\vee}$ is $T$-equivariant. Note that the action of $T$ on $\fh_{e^\vee}$ is identified with the action of $T$ on $H^2(\widetilde{X}^{reg}, \CC)$ that is trivial. Thus, we have a natural homomorphism \begin{equation*}
(\operatorname{Spec}(\CC[\widetilde{D}(\mathbb{O}^\vee)])_{\mathfrak{h}_{e^\vee}})^{T} \rightarrow \mathfrak{h}_{e^\vee}.
\end{equation*}

The pullback gives an action of the algebra $\CC[\fh^*]$ on $\mathbb{C}[(\operatorname{Spec}(\CC[\widetilde{D}(\mathbb{O}^\vee)])_{\mathfrak{h}_{e^\vee}})^{T}]$. It naturally extends to an action of the algebra $\widetilde{W}=S^\bullet\mathfrak{h} \# W$, the action of $W$ is induced by the adjoint action of $N(T) \subset G$.

Module $\mathbb{C}[(\operatorname{Spec}(\CC[\widetilde{D}(\mathbb{O}^\vee)])_{\mathfrak{h}_{e^\vee}})^{T}]$ may contain a torsion considered as a module over $\mathbb{C}[\mathfrak{h}_{e^\vee}]$ (in particular, it may not be flat over $\mathfrak{h}_{e^\vee}$ in general), see Example \ref{counterflat} below. To elaborate on that we want to explicitly state the "flatness conditions". 

\begin{flat*}
    The algebra $\CC[X_{\fh_X}^T]$ is flat over $\CC[\fh_X]$.
\end{flat*}

Assume that $X$ has a unique $T$-fixed point, and therefore the algebra $\CC[X_{\fh_X}^T]$ is finitely generated over $\CC[\fh_X]$. Then for any point $\eta\in \fh_X$ the algebra $\CC[X_{\eta}^T]$ is finite-dimensional. We can state a weaker version of the condition above that is easier to verify.

\begin{weakflatt*}
     Let $\eta_0\in \fh_X$ be a generic point. For any point $\eta\in \fh_X$, we have
     \begin{equation} \label{weakflat}
         \dim \CC[X_{\eta}^T] = \dim \CC[X_{\eta_0}^T].\tag{$\spadesuit$}
     \end{equation}
\end{weakflatt*}

We remark that if $X$ has a unique fixed point, the Flatness condition automatically implies the Weak flatness condition. In particular, this always happens when $X=\Spec({\CC}[\widetilde{\OO}^\vee])$. The unique fixed $T$-point in $\overline{\OO}^\vee$ is the point $0$, and therefore $X$ has finitely many fixed points. Arguing as in \cref{finite fixed points}, it implies that $X$ has a unique fixed point. And thus ${\CC}[\widetilde{\OO}^\vee]$ is a local algebra.

Moreover, when $\widetilde{D}({\mathbb{O}}^\vee)$ is an actual nilpotent orbit (not a cover) and its closure in $\mathfrak{g}^*$ is normal, ${\CC}[\widetilde{\OO}^\vee]$ is generated by $1$ as a $\CC[\fg^*]$-module, and thus $\CC[\Spec({\CC}[\widetilde{\OO}^\vee])^T]$ is generated by $1$ as a $S^\bullet\fh$-module. Note also that $\mathbb{C}[(\operatorname{Spec}(\CC[\widetilde{D}(\mathbb{O}^\vee))])^{T}]$ is indecomposable already over $S^\bullet \mathfrak{h}$ in this case (any local ring is indecomposable as a module over itself).

Assume now that $e^\vee$ is regular in some Levi $\mathfrak{l}^\vee$. As we discussed, $\operatorname{Spec}(\CC[\widetilde{D}(\mathbb{O}^\vee)])_{\mathfrak{h}_{e^\vee}}$ is nothing else but $\on{Spec}\CC[T^*_{\mathfrak{X}(\mathfrak{l})}\mathcal{P}]$ in this case (recall that we have the natural identification $\mathfrak{X}(\mathfrak{l})=\mathfrak{h}_{e^\vee}$). Moreover, the morphism $T^*_{\mathfrak{X}(\mathfrak{l})}\mathcal{P} \rightarrow \on{Spec}\CC[T^*_{\mathfrak{X}(\mathfrak{l})}\mathcal{P}]$ of schemes over $\mathfrak{X}(\mathfrak{l})$ becomes an isomorphism over $\mathfrak{X}(\mathfrak{l})^{\mathrm{reg}}=\mathfrak{h}_{e^\vee}^{\mathrm{reg}}$ (recal that $\mathfrak{h}_{e^\vee}^{\mathrm{reg}}$ consists of $\la$ s.t. $Z_{\mathfrak{g}}(\la)=\mathfrak{l}$). 
In other words, the restriction  $\operatorname{Spec}(\CC[\widetilde{D}(\mathbb{O}^\vee)])_{\mathfrak{h}_{e^\vee}^{\mathrm{reg}}}$ identifies with $G \times^L \mathfrak{h}_{e^\vee}^{\mathrm{reg}}$, see \cref{fiber of deformation of parabolic}, and is a closed subvariety of $\mathfrak{g}^* \times_{\mathfrak{h}^*/W} \mathfrak{h}_{e^\vee}^{\mathrm{reg}}$ via the map $[(g,x)] \mapsto gx$. It is then easy to see that after passing to $T$-fixed points, the corresponding $\widetilde{W}$-module is nothing else but $\on{Ind}_{\widetilde{W}_L}^{\widetilde{W}}\CC[\mathfrak{X}(\mathfrak{l})^{\mathrm{reg}}]$ (in this case this is a simple computation, but it will also follow from the results of Section \ref{Section_weak_HN_for_parab_lodowy_proof}).


\subsection{Cohomological vs fixed points side}
In the Hikita-Nakajima conjecture we are comparing two algebras over $\CC[\mathfrak{h}_{e^\vee}]$: $\mathbb{C}[(\operatorname{Spec}(\CC[\widetilde{D}(\mathbb{O}^\vee)])_{\mathfrak{h}_{e^\vee}})^{T}]$ and $H^*_{T_{e^\vee}}(\widetilde{S}(e^\vee))$. As we saw in the previous sections, both of them are $\widetilde{W}$-modules. It is natural to expect that the (conjectural) identification of these algebras should be compatible with the $\widetilde{W}$-module structures. However, from the discussion in previous sections it is clear that these modules are not isomorphic in general. 

The differences are: 
\begin{enumerate}
    \item\label{issue_1} Module $H^*_{T_{e^\vee}}(\widetilde{S}(e^\vee))$ is {\emph{always}} free over $\CC[\mathfrak{h}_{e^\vee}]$ while $\CC[(\operatorname{Spec}(\CC[\widetilde{D}(\mathbb{O}^\vee)])_{\mathfrak{h}_{e^\vee}})^{T}]$ might have torsion (see Example \ref{counterflat}).

   \item\label{issue_2} If $\widetilde{D}(\mathbb{O}^\vee)$ is an {\emph{orbit}} with {\emph{normal}} closure then the module $\CC[(\operatorname{Spec}(\CC[\widetilde{D}(\mathbb{O}^\vee)]))^{T}]$ is {\emph{cyclic}} indecomposable over $\CC[\mathfrak{h}^*]$ with generator $1$ but $H^*(\mathcal{B}_{e^\vee})$ is not indecomposable in general (see Example \ref{gl3 in B}).
\end{enumerate}

So, we see that one should not expect, in general, that the algebras above will be isomorphic. 
On the other hand, If $e^\vee$ is regular in $\mathfrak{l}^\vee$, then the restrictions of $\CC[\mathfrak{h}_{e^\vee}]$-modules $\mathbb{C}[(\operatorname{Spec}(\CC[\widetilde{D}(\mathbb{O}^\vee)])_{\mathfrak{h}_{e^\vee}})^{\mathbb{C}^\times}]$, $H^*_{T_{e^\vee}}(\widetilde{S}(e^\vee))$ to $\mathfrak{X}(\mathfrak{l})^{\mathrm{reg}}$ are {\emph{isomorphic}} as $\widetilde{W}$-modules and are moreover generated by $1$ over $\CC[\mathfrak{h}^*] \otimes \CC[\mathfrak{h}_{e^\vee}^{\mathrm{reg}}]$. Thus, they are isomorphic as {\emph{algebras}} and the $\widetilde{W}$-action on them is uniquely determined by the action of $\CC[\mathfrak{h}^*]$.
So, the natural guess for the modification of the original Hikita-Nakajima conjecture would be as follows.

Instead of considering the whole $\widetilde{W}$-modules as above, we can consider their $\widetilde{W} \otimes \CC[\mathfrak{h}_{e^\vee}]$-submodules generated by the unit elements $1$ (note that this does not change restrictions of our modules to $\mathfrak{h}_{e^\vee}^{\mathrm{reg}}$ but affects restrictions to other points of $\mathfrak{h}_{e^\vee}$). Since the action of $W$ on $1$ is trivial, this simply reduces to considering images:
\begin{equation*}
\on{Im}(\CC[\mathfrak{h}^*] \otimes \CC[\mathfrak{h}_{e^\vee}] \rightarrow H^*_{T_{e^\vee}}(\widetilde{S}(e^\vee))),~\on{Im}(\CC[\mathfrak{h}] \otimes \CC[\mathfrak{h}_{e^\vee}] \rightarrow \CC(\operatorname{Spec}(\CC[\widetilde{D}(\mathbb{O}^\vee)])_{\mathfrak{h}_{e^\vee}})^{T}).
\end{equation*}
\begin{warning}
Note that taking images does not commute with passing to fibers. 
\end{warning}

This resolves the issue (\ref{issue_2}) (both of the modules are now $\CC[\mathfrak{h}^*] \otimes \CC[\mathfrak{h}_{e^\vee}]$-cyclic by the definition) but may not resolve the issue $(\ref{issue_1})$ (namely, the second $\CC[\mathfrak{h}_{e^\vee}]$-module may not be torsion free over $\CC[\mathfrak{h}_{e^\vee}]$, see Example \ref{counterflat}). To resolve this, we use the natural pull back homomorphism $\CC[\operatorname{Spec}(\CC[\widetilde{D}(\mathbb{O}^\vee)])_{\mathfrak{h}_{e^\vee}}]^{T}) \rightarrow \CC[(T^*_{\mathfrak{X}(\mathfrak{l})}\mathcal{P})^{T}]$ and replace the second algebra by:
\begin{equation*}
\on{Im}(\CC[\mathfrak{h}^*] \otimes \CC[\mathfrak{h}_{e^\vee}] \rightarrow \CC[(T^*_{\mathfrak{X}(\mathfrak{l})}\mathcal{P})^{T}].
\end{equation*}
Concluding, our modification of the Hikita-Nakajima conjecture in this case claims that there exists the  isomorphism of $\CC[\mathfrak{h}^*] \otimes \CC[\mathfrak{h}_{e^\vee}]$-algebras:
\begin{equation*}
 \on{Im}(\CC[\mathfrak{h}^*] \otimes \CC[\mathfrak{h}_{e^\vee}] \rightarrow H^*_{T_{e^\vee}}(\widetilde{S}(e^\vee))) \simeq  \on{Im}(\CC[\mathfrak{h}^*] \otimes \CC[\mathfrak{h}_{e^\vee}] \rightarrow \CC[(T^*_{\mathfrak{X}(\mathfrak{l})}\mathcal{P})^{T}]
\end{equation*}
We note that if $\widetilde{D}(\mathbb{O}^\vee)$ is an orbit with normal closure and $\CC[\operatorname{Spec}(\CC[\widetilde{D}(\mathbb{O}^\vee)])_{\mathfrak{h}_{e^\vee}}^{T}]$ is flat over $\CC[\fh_{e^\vee}]$ (for example $\fg$ is of type A), then the algebra on the RHS is precisely $\CC[\operatorname{Spec}(\CC[\widetilde{D}(\mathbb{O}^\vee)])_{\mathfrak{h}_{e^\vee}}^{T}]$. In particular, if $\mathfrak{g}=\mathfrak{sl}_n$, then our statement recovers the actual Hikita-Nakajima conjecture, so it gives an alternative transparent ``Lie theoretic'' proof of this statement (proved before by Alex Weekes in his PhD thesis, see also \cite{KamnitzerTingleyWebsterWeeksYacobi} using realizations of the varieties above as slices in the affine Grassmannians). See Appendix \ref{app_HN_type_A} for details. 

\begin{rmk}
Note that we are not discussing the equivariant Hikita-Nakajima conjecture in this Section. Recall that the RHS of this conjecture is $H^*_{T_{e^\vee} \times \mathbb{C}^\times}(\widetilde{S}(e^\vee))$ which is a module over the degenerate affine Hecke algebra ${\bf{H}}$ (see \cite[Section 0.1]{lusztig_cuspidalI}) that is a deformation of $\widetilde{W}$ over $\mathbb{C}[\hbar]$. The LHS of the Hikita conjecture is $\mathsf{C}_\nu(\mathcal{A}_{\hbar,\mathfrak{h}_{e^\vee}}(\operatorname{Spec}(\CC[\widetilde{D}(\mathbb{O}^\vee)])))$. We {\emph{do not know}} how to construct the action of ${\bf{H}}$  on the latter.
\end{rmk}

\subsection{Counterexamples to the Hikita-Nakajima conjectures} The goal of this section is to give counterexamples to the original Hikita conjecture for a pair $\widetilde{S}(e^\vee)$, $\on{Spec}\mathbb{C}[\widetilde{D}(\mathbb{O}^\vee)]$ supporting our observations in the previous section. We will see that counterexamples already exist in the ``best possible setting'', namely, even if we put the following restrictions on $e^\vee$.
\begin{enumerate}
    \item\label{ass_1} The element $e^\vee$ is regular in the Lie algebra of some Levi subgroup $L^\vee \subset G^\vee$. 
    \item\label{ass_2}  The refined BVLS duality $\widetilde{D}$ sends the orbit ${\mathbb{O}}^\vee=G^\vee. e^\vee$ to the orbit of a nilpotent element $e$ (i.e,  $e$ is birationally induced from $0$ orbit in $L$). Furthermore, we assume that the closure $\overline{G.e}$ is normal.
\end{enumerate}

In this setting, the classical Hikita conjecture predicts an isomorphism of algebra
\begin{equation} 
    \CC[\overline{G.e}\cap \fh]\simeq H^*(\spr).
\end{equation}



Recall that we fixed the identification $\mathfrak{g}^* \simeq \mathfrak{g}$. Consider a generic element $x \in \mathfrak{z}(\mathfrak{l})$. Following \cite[Proposition 4]{conjugacyclass}, we describe a quotient of the algebra $\CC[\overline{G.e}\cap \fh]$ using $x$.

Let $J'$ be the defining ideal of $G.x$ in $\fg$. We have a natural grading of $J'$ coming from $\CC[\fg]$. Then $J_e= \gr J'$ is the defining ideal of $\overline{G.e}$ in $\fg$. Now let $I'$ be the ideal generated by $J'$ and the coordinate functions $\bigoplus_{\al \in \Delta} \mathfrak{g}_\al$ of the non-diagonal entries of $\fg$. And write $I_e$ for the defining ideal of $\overline{G.e}\cap \fh$ in the algebra $\CC[\fh]$. Then we have $I_e\subset \gr I'$. Thus, we get a surjective $\widetilde{W}$-equivariant map of algebras $\CC[\overline{G.e}\cap \fh] \twoheadrightarrow  \CC[\fh]/ \gr I'$.
 
We have the following proposition.
 \begin{prop}  \label{equivalent statements}
    In the setting above, the classical Hikita conjecture is equivalent to two statements.
 \begin{enumerate}
     \item The quotient morphism $\CC[\overline{G.e}\cap \fh] \twoheadrightarrow  \CC[\fh]/ \gr I'$ is an isomorphism. 
     \item We have an isomorphism of rings $H^*(\spr)\simeq \CC[\fh]/ \gr I'$.
 \end{enumerate}   
 \end{prop}
 
\begin{proof}
    If both conditions hold, then $\CC[\overline{G.e}\cap \fh]\simeq  \CC[\fh]/ \gr I'\simeq H^*(\spr)$. On the other hand, assuming that the classical Hikita conjecture is true, we have $\CC[\overline{G.e}\cap \fh]\simeq H^*(\spr)$. Then we must have $\dim \CC[\overline{G.e}\cap \fh]=$ $\dim H^*(\spr)$ $= |W/W_{L^\vee}|$. This is the dimension of $\CC[\fh]/\on{gr} I'$. Therefore, it follows that the surjection in (1) of \cref{equivalent statements} is an isomorphism.
\end{proof} 

We remark that the proof implies that (1) is equivalent to the weak flatness condition \ref{weakflat}.


Let us now give several examples in types B-C in which at least one of the statements in \cref{equivalent statements} fails to hold. In other words, the classical Hikita conjecture does not apply in these cases.

In these examples, one can check that the elements $e$ we choose are birationally Richardson by following \cite[Section 3]{Fu_2003}. The normality of orbit closures follows from the table at the end of \cite{Kraft-Procesi}.

The following example showcases a situation in which Statement 1 in \cref{equivalent statements} is true, but Statement 2 is not.
\begin{example} \label{gl3 in B}
Let $G= Sp(6)$ and $G^\vee= SO(7)$, they have the same Weyl group $W= S_3\rtimes (\cyclic{2})^{\oplus 3}$. Let $e^\vee\in G^\vee$ be a nilpotent element with the associated partition $(3,3,1)$. The refined BVLS duality sends the orbit of $e^\vee$ in $\fs\fo(7)$ to the orbit of $e$, a nilpotent element with the associated partition $(2,2,2)$ in $\fs\fp(6)$.

\textbf{The function side}. In this case, we have $L^\vee$ and $L$ are isomorphic to $GL(3)$. We choose an embedding $\fs\fp(6)\hookrightarrow \fg\fl(6)$ such that we can realize $\fh$ as the set of diagonal matrices diag$(x_1,x_2,x_3, -x_3,-x_2,-x_1)$.

A generic element $x\in Z_\fg(\fl)$ has the form diag$(t,t,t,-t,-t,-t)$ for $t\neq 0$. Fix such a $t$, identify $\CC[\fh]$ with $\CC[x_1,x_2,x_3]$. The defining ideal $I'$ of the orbit $W.x\subset \fh$ is generated by the elements $(x_i^2- t^2)_{i=1,2,3}$. As a consequence, $\gr I'$ is $(x_i^2)_{i=1,2,3}$. Furthermore, we claim that $I_e= \gr I'$ in this case. Since the partition of $e$ is $(2,2,2)$, any element $X$ in the closure $\overline{G.e}$ has the property $X^2=0$. So $x_i^2\in I_e$ for $1\leqslant i\leqslant 3$, and $I_e= \gr I'$; Statement 1 in \cref{equivalent statements} is verified. The algebra $\CC[\fh]/ I_e$ is isomorphic to $\CC[x_1,x_2, x_3]/(x_1^{2}, x_2^{2}, x_3^{2})$.

\textbf{The cohomological side}. We first compute the Betti numbers of $\spr$ using the algorithms provided in \cite[Section 8]{Kim2018}. The partition of $e^\vee$ is $(3,3,1)$, so the component group of $C_G(e)$ is isomorphic to $\cyclic{2}$. This group acts on each $H^{2i}(\spr)$ with two irreducible characters, the trivial character and the sign character. Let $n_i^{+}$ (resp. $n_i^{-}$) denote the multiplicity of the trivial (resp. sign) character in $H^{2i}(\spr)$. In \cite[Proposition 8.2]{Kim2018}, the closed formulas of $n_i$ for the partition $(2k+1, 2k+1, 1)$ are worked out as follows
$$n_i^+= {2k+ 1 \choose i}, n_i^{-}= {2k+1\choose i-2} \quad \text{for} \quad 0\leqslant 2i\leqslant 2k+3.$$
Write $d_i$ for the dimension of $H^{2i}(\spr)$. In our case, we obtain $d_0= {3\choose 0} = 1$, $d_1= {3\choose 1}= 3$, and $d_2= {3\choose 2}+ {3\choose 1}= 4$. Alternatively, these numbers can be obtained from a geometric description of the Springer fiber $\spr$ as follows.

Realize $G^\vee$ as $SO(V)$ where $V=$Span$(v_1,...,v_7)$, and the symmetric bilinear form is given by $\langle v_i, v_j\rangle= (-1)^{i+1}\delta_{i, 8-j}$. The action of $e^\vee$ on $V$ is given by $e^\vee(v_4)= 0$ and $e^\vee(v_i)= v_{i-1}$ if $i\neq 4$. An element of $\spr$ is a $e^\vee$-stable flag $V^1\subset V^2\subset V^3$ of isotropic subspaces of $V$. We will use a sequence of vectors $(u_1, u_2, u_3)$ to represent the flag with $V^i=$ Span $(v_1,...,v_i)$ for $1\leqslant i\leqslant 3$. The Springer fiber $\spr$ has four irreducible components $X_1, X_2, X_3$ and $X_4$, each $X_i$ is isomorphic to a $\PP^1$-bundle over $\PP^1$. In terms of vectors, they can be described as follows.

The component $X_1$ consists of the flags of types $V^2=$Span$(v_1,v_5)$ and $v_3$ is an isotropic vector of Span$(v_2, v_4, v_6)$. The component $X_2$ consists of the flags of types $(u_1, u_2, u_3)$ such that $u_2$ has the form $av_1+ bv_5+ u$ with $u\neq 0$ isotropic in Span$(v_2,v_4,v_6)$, and $u_1= eu_2$, while $u_3\in$ Span$(v_1,v_5)$ (we require $u_3\neq u_1$). The component $X_3$ consists of the flags of types $(v_1, av_2+ bv_5, x(av_3+ bv_6)+ yv_5)$, and the component $X_4$ consists of the flags of types $(v_5, bv_6+ av_1, x(av_2+ bv_7)+ yv_1)$. The component group $\cyclic{2}$ preserves $X_1$ and $X_2$, while permutes $X_3$ and $X_4$. For the intersection, $X_1$ and $X_2$ are disjoint, $X_3$ and $X_4$ are disjoint, and the intersections $X_1\cap X_3, X_1\cap X_4, X_2\cap X_3, X_2\cap X_4 $ are all isomorphic to $\PP^1$.

\textbf{Comparison of the two algebras}. The socle of $\CC[x_1,x_2, x_3]/(x_1^{2}, x_2^{2}, x_3^{2})$ is one-dimensional, generated by the element $x_1x_2x_3$. The socle of $H^*(\spr)$ contains $H^4(\spr)$, which has dimension $4$. Thus, the two algebras are not isomorphic; the second statement in \cref{equivalent statements} does not hold.
\end{example}

\begin{example}\label{counterflat}(Failure of flatness condition (\ref{weakflat})).
Let $G= Sp(6)$ and $G^\vee= SO(7)$, they have the same Weyl group $W= S_3\rtimes (\cyclic{2})^{\oplus 3}$. Let $e^\vee\in G^\vee$ be a nilpotent element with the associated partition $(3,2,2)$. The refined BVLS duality sends the orbit of $e^\vee$ in $\fs\fo(7)$ to the orbit of $e$, a nilpotent element with the associated partition $(3,3)$ in $\fs\fp(6)$.

In this case, $L^\vee$ is isomorphic to $SO(3)\times GL(2)$ and $L$ is isomorphic to $Sp(2)\times GL(2)$. Here $W_L\simeq W_{L^\vee}\simeq (\cyclic{2})^{\oplus 2}$. Hence, the dimension of the quotient $\CC[\fh]/ \gr I$ is $|W|/|W_L|= \frac{3!\times 2^3}{2^2}= 12$. However, by the table at the end of \cite{conjugacyclass}, we see that the dimension of $\CC[\overline{G.e}\cap \fh]$ is $13$. Thus, the flatness conjecture does not hold in this case.

We give the readers a quick explanation for this failure of flatness. Let $e'\in \fsp(6)$ be a nilpotent element with associated partition $(2,2,2)$, this is the case we have considered in \cref{gl3 in B}. Hence, we know that $I_{e'}$ is generated by $(x_i^{2})_{i=1,2,3}$ where $x_i$ are the coordinate functions of $\fh$. Now, it is clear that $e'\in \overline{G.e}$, so $I_{e'}\supset I_e$. In particular, we deduce that $x_1x_2x_3\notin I_e$ because $x_1x_2x_3\notin I_{e'}$. On the other hand, in the notation of \cref{gl3 in B}, a generic element $x\in \mathfrak{z}_\fg(\fl)$ has the form diag$(t,t,0,0,-t,-t)$. This means that in the orbit $W.x\subset \fh$, we always have $x_i=0$ for some $i\in \{1,2,3\}$. In other words, $x_1x_2x_3\in \gr I'$, so $\gr I'\supsetneq I_e$. Statement 1 of \cref{equivalent statements} fails here. 
\end{example}
\begin{rmk}
In work in progress, Finkelberg, Hanany, and Nakajima realize this example as Higgs and Coulomb branches $X^\vee=\mathcal{M}_{H}$, $X=\mathcal{M}_{C}$ of a certain orthosymplectic quiver gauge theory.  So, one should not expect in general that if $\widetilde{\mathcal{M}}_{H}$ is a resolution of the Higgs branch of some $3$-dimensional $\cN=4$ supersymmetric gauge theory, then the algebra $H^*(\widetilde{\mathcal{M}}_{H})$ will be isomorphic to the algebra of schematic fixed point $\CC[\mathcal{M}_{C}^{T_{\mathcal{M}_C}}]$ (see Remark \ref{rem_when_expect_hn_coulomb} above for the statement that might be true).
\end{rmk}

A straightforward generalization of this phenomenon for flatness condition is the following.
\begin{prop} \label{sp flatness}Consider $\fg= \fsp(2n)$. 
\begin{enumerate}
    \item If $\fl= \fgl(n)$, then (\ref{weakflat}) holds.
    \item If $\fl$ contains a nontrivial factor $\fsp_{2m}$, then (\ref{weakflat}) does not hold.
\end{enumerate}
\end{prop}
\begin{proof}
    Consider $\fl= \fgl(n)$. A generic element $x\in Z_\fg(\fl)$ has the form diag$(t,...,t,-t,...,-t)$ for $t\neq 0$. Fix such a $t$, identify $\CC[\fh]$ with $\CC[x_1,...,x_n]$. The defining ideal $I'$ of the orbit $W.x\subset \fh$ is generated by the elements $(x_i^2- t^2)_{1\leqslant i\leqslant n}$. As a consequence, $\gr I'$ is $(x_i^2)$. Furthermore, we claim that $I_e= \gr I'$ in this case. Since the partition of $e$ is $(2^{n})$, any element $X$ in the closure $\overline{G.e}$ has the property $X^2=0$. So $x_i^2\in I_e$ for $1\leqslant i\leqslant n$, and $I_e= \gr I'$. Statement 1 in \cref{equivalent statements} is verified. The algebra $\CC[\fh]/ I_e$ is isomorphic to $\CC[x_1,..., x_n]/(x_1^{2},..., x_n^{2})$.

    Consider $\fl= \fsp(2a)\times\prod_{i=1}^{k} \fgl(b_i)$ for some $a> 0$. A generic element $x\in Z_\fg(\fl)$ has the following property. In $b_1+...+b_k+ 1$ arbitrary coordinates of $x$, there is at least one $0$. Hence, we have $x_1x_2...x_{b_1+...+b_k+1}\in \gr I'$. Then, to prove the proposition, we just need to show that $x_1x_2...x_{b_1+...+b_k+1}$ does not belong to $I_e$. 
    
    Write $c_1$ for $b_1+...+b_k$. As $a>0$, we have $c_1+1\leqslant n$. Consider a matrix realization $\fsp({2n})= \fsp(V)$ where $V=\CC^{2n}$. We choose a basis $v_1,...,v_{2n}$ of $V$ such that the symplectic form is given by $(J_s,...,J_s)$ where $J_s= \begin{pmatrix}
0 & 1 \\
-1 & 0 
\end{pmatrix}$. Write $V'$ for Span$(v_1,...,v_{2c_1+2})$. We then have an embedding 
$$i_{c_1}: \fsp({2c_1+2})= \fsp(V')\hookrightarrow \fsp(V)= \fsp({2n}).$$
We write $\OO_\fp$ for the orbit with partition $\fp$. Consider the nilpotent orbit $\OO_{(2^{c_1+1})}$ in $\fsp({2c_1+2})$ and its image under $i_{c_1}$. We see that 
$$i_{c_1}(\overline{\OO_{(2^{c_1+1})}})\subset \overline{\OO_{(2^{c_1+1}, 1^{2n-2c_1-2})}}\subset \overline{G.e}.$$ 
This composition of embedding induces a surjection on the level of scheme theoretic intersection $$\CC[\fh\cap \overline{G.e}]\twoheadrightarrow \CC[\fh'\cap \overline{\OO_{(2^{c_1+1})}}].$$
We see that the element $x_1...x_{c_1+1}$ has a nonzero image in the quotient $\CC[\fh'\cap \overline{\OO_{(2^{c_1+1})}}]$ thanks to the result of part (1). Hence, it is nonzero in $\CC[\fh\cap \overline{G.e}]$. 
\end{proof}
\begin{rmk} \label{sp is bad}
    We have discussed the case $\fl\subset \fg= \fsp(2n)$ for $\fl$ contains a non-trivial factor $\fsp$. If $\fl$ is of the form $\prod_{i=1}^{k}\fgl({a_i})$, it was claimed in \cite[Theorem 3]{Tanisaki} that Condition 1 in \cref{equivalent statements} always holds. The statement of this theorem is correct, but the proof in \cite{Tanisaki} contains a mistake (\cite[Lemma 7]{Tanisaki} is not correct, e.g., for $\fl= \fgl(3)\times \fgl(1)\subset \fs\fp(8)$). The correct proof is given in \cite[Sections 2,3]{hoang2}. 
\end{rmk}

In the two examples above, we have considered $G^\vee$ of type $B$ and $G$ of type $C$. The next example demonstrates this phenomenon for $G^\vee$ of type $C$ and $G$ of type $B$. In particular, Condition 2 in \cref{equivalent statements} fails (Condition 1 is true, but we do not give proof here). The method to show that the two algebras are not isomorphic is similar to \cref{gl3 in B}. 

\begin{example} \label{gl3gl1}
Let $G= SO(9)$ and $G^\vee= Sp(8)$. Let $e^\vee\in G^\vee$ be a nilpotent element with the associated partition $(4,2,2)$. The refined BVLS duality sends the orbit of $e^\vee$ in $\fs\fp(8)$ to the orbit of $e$, a nilpotent element with the associated partition $(3,3,1,1,1)$ in $\fs\fo(9)$.

\textbf{The function side}. In this case, we have $L^\vee \simeq Sp(4)\times GL(2)$ and $L\simeq SO(5)\times GL(2)$. We choose an embedding $\fs\fo(9)\hookrightarrow \fg\fl(9)$ such that we can realize $\fh$ inside $\fh'$ as the set of diagonal matrices diag$(x_1,...,x_4,0,-x_4,...,-x_1)$.

A generic element $x\in Z_\fg(\fl)$ has the form diag$(t,t,0,0,0,0,0,-t,-t)$ for $t\neq 0$. Fix such a $t$, identify $\CC[\fh]$ with $\CC[x_1,x_2,x_3,x_4]$. The orbit $W.x$ in $\fh$ consists of $24$ points $(a,b,c,d)$ such that two of the coordinates are 0, and the other two are $t$ or $-t$. It is straightforward to see that the following elements belong to $I'$ 
$$\{(x_i^2- t^2)x_i\}_{i=1,2,3,4}; 2t^2- \sum_{i=1}^4 x_i^{2}; \{x_{ijk}\}_{1\leqslant i<j<k \leqslant 4};  $$
Let $I''$ be the ideal generated by these polynomials. Next, one can check that the dimension of $\CC[\fh]/I''$ (equals to the dimension of $\CC[\fh]/\gr I''$)  is $24=|W/W_L|$. A basis of $\CC[x_1,x_2,x_3,x_4]/\gr I''$ is given by $1, x_1, x_2, x_3, x_4, x_1^{2}, x_2^{2}, x_3^{2},$ $ \{x_ix_j\}_{1\leqslant i<j \leqslant 4},$ $ x_1x_2^{2}, x_1x_3^{2}, x_2x_3^{2}, x_2x_4^{2},$ $ x_3x_4^{2}, x_3x_1^{2},$ $ x_4x_1^{2}, x_4x_2^{2},$ $ x_1^{2}x_2^{2}, x_2^{2}x_3^{2}$. Therefore, $I'=I''$ and the socle of $\CC[x_1,x_2,x_3,x_4]/\gr I'$ has dimension $2$.

\textbf{The cohomological side}. The dimension of the Springer fiber $\spr$ is $3$, so the grading does not match the function side since we have elements of degree $4$ in $\CC[\fh]/I_e$. Let $d_i$ be the dimension of $H^{2i}(\spr)$. By the recursive formula in \cite[Theorem 6.1]{Kim2018}, we see that the Betti numbers of $\spr$ are $d_0= 1, d_1= 4, d_2=8 ,d_3= 11$. The socle of this algebra has dimension at least $11$. Hence, the two algebras are not isomorphic.

\end{example}

\section{Localizations, categories $\mathcal{O}$ and Cartan subquotients: case of parabolic $W$-algebras}\label{sec_hw_parab}

This section is rather technical and will be used in the proof of the main Theorem in Section \ref{Section_weak_HN_for_parab_lodowy_proof}.
In this section, we describe the homomorphism (\ref{comp_cart_alg_to_sh}) explicitly for $Y=\widetilde{S}(\mathfrak{q},\mathfrak{p})$ (see Lemma \ref{lemma_hw_compute_parabolicW}). Additionally, we will discuss the abelian localization theorem for parabolic $W$-algebras.


\subsection{Localization theorem for parabolic $W$-algebras}\label{sec_loc_thm_parab_W}
In this section we describe certain $\la \in \mathfrak{X}(\mathfrak{l})$ for which the abelian localization holds for $(p(\la),\widetilde{S}(\chi,\mathfrak{p}))$ (recall that $p\colon H^2(T^*\mathcal{P},\CC) \rightarrow H^2(\widetilde{S}(\chi,\mathfrak{p}),\CC)$ is the restriction map). This is well-known to experts, but we include the proof for completeness (note that we are not assuming that $e$ is regular in Levi in this section).

We start with a description of (some) parameters for which the abelian localization holds for $(\lambda,T^*\mathcal{P})$. We fix a Borel subalgebra $T \subset B \subset P$.
To $\la \in \mathfrak{X}(\mathfrak{l})$ we associate $\widetilde{\la}:=\la-\rho_{\mathfrak{b}}(\mathfrak{l})$.


	\begin{prop}\label{loc_flags} 
Abelian localization holds for $\la \in \mathfrak{X}(\mathfrak{l})=H^2(T^*\mathcal{P},\CC)$ if $\widetilde{\la}$ is 
$\mathfrak{p}$-antidominant and $\mathfrak{l}$-regular.
	\end{prop}
	
\begin{proof}
	The claim is proved completely analogically to~\cite[Th\'eor$\grave{\on{e}}$me principal]{bb} (c.f.~\cite[Theorem 5.1]{bak}). 
\end{proof}

\begin{warning}
We are {\emph{not}} claiming that the parameters as in Proposition \ref{loc_flags} are the only parameters for which the abelian localization holds. 
It also might be true that every $\mathfrak{l}$-regular $\widetilde{\la}$ is automatically $\mathfrak{p}$-antidominant (this is certainly the case for $\mathfrak{p}=\mathfrak{b}$).
\end{warning}

\begin{rmk}
Note that if we fix another $T \subset B' \subset P$, then the corresponding element $\widetilde{\la}' := \la -\rho_{\mathfrak{b}'}(\mathfrak{l})$ will be $\mathfrak{p}$-antidominant $\mathfrak{l}$-regular iff $\widetilde{\la}$ is $\mathfrak{p}$-antidominant $\mathfrak{l}$-regular. Indeed, $B'=wBw^{-1}$ for some $w \in W_L$ so $\widetilde{\la}'=w(\widetilde{\la})$. Now the claim follows from the fact that $w(\Delta_{\mathfrak{p}})=\Delta_{\mathfrak{p}}$.
\end{rmk}

	
\begin{rmk}
Note that the condition that $\widetilde{\la}=\la-\rho_{\mathfrak{b}}(\mathfrak{l})$ is $\mathfrak{p}$-antidominant is equivalent to the condition that  $\widetilde{\la}$ is $\mathfrak{b}$-antidominant (use that $\langle \la,\beta^\vee\rangle = 0$ for $\beta^\vee \in \Delta_{\mathfrak{l}}^\vee$). Similarly, $\widetilde{\la}$ is $\mathfrak{l}$-regular iff it is regular.   
\end{rmk}

    \begin{prop}\label{loc_slodowy} Pick $\la \in \mathfrak{X}(\mathfrak{l})$. Let $p\colon H^2(T^*\mathcal{P},\CC) \ra H^2(\widetilde{S}(\chi,\mathfrak{p}),\CC)$ be the restriction map. If the abelian localization holds for $(\la,T^*\mathcal{P})$ then the abelian localization holds for $(p(\la),\widetilde{S}(e,\mathfrak{p}))$. 
	\end{prop}

 \begin{proof}
    We use the same approach as in~\cite[Propositions~3.7,~3.8]{Losev_cacticells}, see also ~\cite{Gi} and~\cite{dk}. Recall (see Proposition \ref{prop_quant_W_alg} (b))
 that $\Gamma(\widetilde{S}(\chi,\mathfrak{p}),\mathcal{D}_{\la}(\widetilde{S}(e,\mathfrak{p})))$ identifies with the 
 parabolic $W$-algebra $\mathcal{W}_{\lambda}(\chi,\mathfrak{p})$. 

    Assume that the abelian localization holds for $(\la,T^*\mathcal{P})$. 
    Let us denote by $\on{Coh}^{\mathfrak{u}_{\ell},\chi}(\mathcal{D}_{\la}(T^*\mathcal{P}))$ the category of twisted $(\mathfrak{u}_{\ell},\chi)$-equivariant coherent 
    $\mathcal{D}_{\la}(T^*\mathcal{P})$-modules and by $\CU_{\la}(\mathfrak{g},\mathfrak{p})$-$\on{mod}^{\mathfrak{u}_{\ell},\chi}$ the category of twisted $(\mathfrak{u}_{\ell},\chi)$-equivariant finitely generated $\mathcal{U}_{\la}(\mathfrak{g},\mathfrak{p})$-modules, see \cite[Section 4]{Kashiwara_Dmod} for the definitions. 

    We have the following  equivalences (the second one is the parabolic version of the Skryabin equivalence): 
    \begin{equation}\label{G/Q_to_W_mod}
    \on{Coh}^{\mathfrak{u}_{\ell},\chi}(\mathcal{D}_\la(T^*\mathcal{P})) \iso \on{Coh}(\mathcal{D}_{\la}(\widetilde{S}(e,\mathfrak{p}))),\,  
    U_{\la}(\mathfrak{g},\mathfrak{p})\text{-}\on{mod}^{\mathfrak{u}_{\ell},\chi} \iso
    \mathcal{W}_{\la}(\chi,\mathfrak{p})\text{-}\on{mod}
    \end{equation}
    given by taking $\mathfrak{u}_\ell$-semiinvariants with character $\chi$. Both of these isomorphisms are consequences of the fact that the action of $U_\ell$ on the affine variety $\mu_{T^*\mathcal{P}}^{-1}(\chi)$ is free. For the first isomorphism, see~\cite[Lemma~5.3]{dk} or \cite[Proposition~2.8]{kr}. The second isomorphism can be proved analogously  to~\cite[Section 6.2]{GanGinzburg}, see also~\cite{sk}. 
    


   It is clear  that for $M \in \on{Coh}^{\mathfrak{u}_\ell,\chi}(\mathcal{D}_{\la}(T^*\mathcal{P}))$, we have
    \begin{equation}\label{commute_glob_invar}
    \Gamma(M)^{(\mathfrak{u}_\ell,\chi)}=\Gamma(M^{(\mathfrak{u}_\ell,\chi)}).   
    \end{equation} 
    
    We note that the global section functor $\Gamma: \on{Coh}(\mathcal{D}_{\la}(T^*\mathcal{P})) \to  \CU_{\la}(\mathfrak{g},\mathfrak{p})$ and localization functor $\on{Loc}: \CU_{\la}(\mathfrak{g},\mathfrak{p})\to \on{Coh}(\mathcal{D}_{\la}(T^*\mathcal{P}))$ naturally lift to the functors between categories of respective twisted equivariant modules. 
    
    Since the abelian localization holds for $(\la,T^*\mathcal{P})$, these functors are quasi-inverse on the categories of non-equivariant modules. One can see that the twisted equivariant module structure on $\cM\in \on{Coh}(\mathcal{D}_{\la}(T^*\mathcal{P}))$ coincides with the induced structure on $\Loc(\Gamma(\cM))$, and therefore the global section functor induces an equivalence 
    \begin{equation*}
    \Gamma\colon \on{Coh}^{\mathfrak{u}_\ell,\chi}(\mathcal{D}_{\la}(T^*\mathcal{P})) \iso  \CU_{\la}(\mathfrak{g},\mathfrak{p})\text{-}\on{mod}^{\mathfrak{u}_\ell,\chi}.    
    \end{equation*}
    It now follows from~(\ref{G/Q_to_W_mod}),~(\ref{commute_glob_invar}) that the localization holds for $(p(\la),\widetilde{S}(\chi,\mathfrak{p}))$.
\end{proof}





\subsection{Cartan subquotients and highest weights for finite $W$-algebras}
The results of this section follow from \cite{BGK}, the only difference is that the authors of loc. cit. work with filtered algebras although we work with algebras over $\mathbb{C}[\hbar]$. Another reference is \cite{LosevStructureO}.

Recall that $M \subset Q$ is the standard Levi, and let $e=e_{\mathfrak{m}} \in \mathfrak{m}$ be the regular nilpotent that lies in $\mathfrak{b}$. Recall that $\chi \in \mathfrak{g}^*$ the linear functional $(e,-)$.  Let $Z_M \subset M$ be the connected component of the center of $M$ and $\mathfrak{z}(\mathfrak{m})=\on{Lie}Z_M \subset \mathfrak{m}$. Pick a generic cocharacter $\nu\colon \CC^\times \rightarrow Z_M$ such that $\nu$ is dominant w.r.t. $Q$ (i.e., $\langle 
\al,\nu\rangle \geqslant 0$ for every $\alpha \in \Delta_{\mathfrak{q}}$).  Let $\widetilde{\nu}\colon \CC^\times \ra T$ be a generic $B$-dominant cocharacter.

\subsubsection{Definition of $\mathcal{W}_\hbar(\chi)$ via non-linear Lie algebras}
Results of this section can be found in \cite[Section 2.2]{BGK}. 

Recall that $(e,h,f)$ form an $\mathfrak{sl}_2$-triple. 
Consider the grading $\mathfrak{g}=\bigoplus_{i \in \mathbb{Z}}\mathfrak{g}(i)$ induced by the adjoint action of $h$ and set $\mathfrak{r}:=\bigoplus_{i \geqslant 0}\mathfrak{g}(i)$, this is a Lie subalgebra of $\mathfrak{g}$.

Set $\mathfrak{k} = \mathfrak{g}(-1)$ and let $\mathfrak{k}^{\mathrm{ne}}$ be a {\emph{non-linear Lie algebra}} (c.f. \cite[beginning of Section 2.2]{BGK}) with the non-linear Lie bracket given by $[x^{\mathrm{ne}},y^{\mathrm{ne}}]=\langle \chi,e_{\mathfrak{m}},[x,y]\rangle$, where we denote by $x^{\mathrm{ne}}, y^{\mathrm{ne}} \in \mathfrak{k}$ the elements corresponding to $x,y \in \mathfrak{k}$.
Consider the direct sums  $\widetilde{\mathfrak{r}} = \mathfrak{r} \oplus \mathfrak{k}^{\mathrm{ne}}$,  $\widetilde{\mathfrak{g}} = \mathfrak{g} \oplus \mathfrak{k}^{\mathrm{ne}}$ equipped with the natural structures of the non-linear Lie algebra.

For a vector space $V$, let $T_{\hbar}(V):=(\bigoplus_{k=0}^\infty V^{\otimes k}) \otimes \mathbb{C}[\hbar]$ be the corresponding tensor algebra over $\mathbb{C}[\hbar]$. 
Set 
\begin{equation*}
\mathcal{U}_\hbar(\widetilde{\mathfrak{g}}) := T_\hbar(\widetilde{\mathfrak{g}})/M_\hbar(\widetilde{\mathfrak{g}}),~
\mathcal{U}_\hbar(\widetilde{\mathfrak{r}}) := T_\hbar(\widetilde{\mathfrak{r}})/M_\hbar(\widetilde{\mathfrak{r}}),
\end{equation*}
where $M_\hbar(\widetilde{\mathfrak{g}})$ is the two-sided ideal in $T_\hbar(\widetilde{\mathfrak{g}})$ generated by $a \otimes b - b \otimes a -\hbar[a,b]$, $a,b \in \widetilde{\mathfrak{g}}$ and similarly for $M_\hbar(\widetilde{\mathfrak{r}})$.

We define the {\emph{Kazhdan grading}} on  $\mathcal{U}_\hbar(\widetilde{\mathfrak{g}})$, $\mathcal{U}_\hbar(\widetilde{\mathfrak{r}})$ as follows:
$
\on{deg}\mathfrak{g}(i)=i+2,~\on{deg}\mathfrak{k}^{\mathrm{ne}}=1,~\operatorname{deg}\hbar=2$.

Let $J_{\hbar} \subset \mathcal{U}_\hbar(\widetilde{\mathfrak{r}})$ be the left ideal generated by  the elements $\{x-\chi(x)-x^{\mathrm{ne}},\,|\, x \in \mathfrak{n}\}$ (note that this ideal is {\emph{homogeneous}} w.r.t. the Kazhdan grading). The PBW theorem for $\mathcal{U}_\hbar(\widetilde{\mathfrak{g}})$ gives the projection 
$\operatorname{Pr}_{\hbar}\colon \mathcal{U}_\hbar(\widetilde{\mathfrak{g}}) \twoheadrightarrow J_\hbar$. Let $I_\hbar \subset \mathcal{U}_{\hbar}(\mathfrak{g})$ be the left ideal generated by $\{x-\chi(x)\,|\, x \in \mathfrak{u}\}$.

Set 
\begin{equation*}
\mathcal{U}_{\hbar}(\mathfrak{g},\chi) \subset \mathcal{U}_\hbar(\widetilde{\mathfrak{r}}) = \{u \in \mathcal{U}_\hbar(\widetilde{\mathfrak{r}}),\,|\, \operatorname{Pr}([x-x^{\mathrm{ne}},u])=0,~\text{for all}~x \in \mathfrak{n}\}.
\end{equation*}

Set $Q_\hbar=\mathcal{U}_\hbar(\mathfrak{g})/I_\hbar$ and recall that $\mathcal{W}_\hbar(\chi)=Q_\hbar^{\mathfrak{n}}$, where $\mathfrak{n}=\bigoplus_{i \leqslant -1}\mathfrak{g}(i)$. The action of $\mathfrak{g}$ on $Q_\hbar$ extends to the action of $\widetilde{\mathfrak{g}}$ via the formula $x^{\mathrm{ne}}(u+I_\hbar)=ux+I_\hbar$. \cite[Theorem 2.4]{BGK} implies the following.

\begin{prop}\label{alter_def_W}
Subspace $\mathcal{U}_\hbar(\mathfrak{g},\chi) \subset \mathcal{U}_\hbar(\widetilde{\mathfrak{r}})$ is a subalgebra and the map 
$
\mathcal{U}_\hbar(\mathfrak{g},\chi) \rightarrow \mathcal{W}_\hbar(\chi),~x \mapsto x(1+I_\hbar)
$
is the isomorphism of algebras. 
\end{prop}

\subsubsection{Explicit description of the Cartan subquotient of $\mathcal{W}_{\hbar}(\chi)$}
 In \cite{BGK} authors described the Cartan subquotient $\cC_\nu(\mathcal{W}(\chi))$. Let us formulate the homogeneous version of their result. For simplicity, we assume that $e=e_{\mathfrak{m}}$ is regular in $\mathfrak{m}$ as this is the only case we care about in the paper.
 Let $\gamma \in \mathfrak{h}^*$ be as in \cite[Section 4.1]{BGK} it follows from \cite[Lemma 4.1]{BGK} that $\gamma$ extends (uniquely) to the character of $\mathfrak{r}_0$. Let $S_{-\gamma}\colon \mathcal{U}_\hbar(\mathfrak{r}_0) \iso \mathcal{U}_\hbar(\mathfrak{r}_0)$ be the automorphism given by $x \mapsto x-\hbar^{\frac{\operatorname{deg}x}{2}}\gamma(x)$, $x \in \mathfrak{r}$ (note that the Kazhdan grading becomes {\emph{even}} being restricted to $\mathfrak{r}_0 \subset \mathfrak{m}$).

 Cocharacter $\nu\colon \mathbb{C}^\times \rightarrow Z_M$ defines the grading on $\mathcal{U}_{\hbar}(\widetilde{\mathfrak{r}})$ and on its subalgebra $\mathcal{U}_{\hbar}(\mathfrak{g},\chi)$. Let $\mathcal{U}_{\hbar}(\mathfrak{g},\chi)_0 \subset \mathcal{U}_{\hbar}(\mathfrak{g},\chi)_0$ be the degree zero components. Note that $\mathfrak{g}_0 = \mathfrak{m}$ so $\mathfrak{r}_0 = \mathfrak{r} \cap \mathfrak{m}$. Moreover, note that $\mathfrak{k} \cap \mathfrak{m} = \mathfrak{m}(-1) = 0$ since the $h$-grading on $\mathfrak{m}$ is {\emph{even}}. It follows that $\widetilde{\mathfrak{r}}_0 = \mathfrak{r}_0$. The following lemma is a standard corollary.

\begin{lemma}
The natural embedding $\mathcal{U}_\hbar(\mathfrak{r}_0) \hookrightarrow \mathcal{U}_\hbar(\widetilde{\mathfrak{r}})_0$ induces the identification: 
\begin{equation}\label{Cartan_subq_descr}
\mathcal{U}_\hbar(\mathfrak{r}_0) \iso \cC_\nu(\mathcal{U}_\hbar(\widetilde{\mathfrak{r}})).
\end{equation}
\end{lemma}

Let $\pi\colon \mathcal{U}_\hbar(\widetilde{\mathfrak{r}})_0 \twoheadrightarrow \mathcal{U}_\hbar(\mathfrak{r}_0)$ be the composition of the surjection $\mathcal{U}_\hbar(\widetilde{\mathfrak{r}})_0 \twoheadrightarrow \cC_\nu(\mathcal{U}_\hbar(\widetilde{\mathfrak{r}}))$ and the inverse to the isomorphism (\ref{Cartan_subq_descr}).

Recall that $Z_\hbar \subset \mathcal{U}_\hbar(\mathfrak{g})$ is the center. 
The natural embedding $Z_\hbar \hookrightarrow \mathcal{W}_\hbar(\chi)$ induces the embedding $Z_\hbar \hookrightarrow \cC_\nu(\mathcal{W}_\hbar(\chi))$. We want to understand the algebra $\cC_{\nu}(U_{\hbar}(\fg, \chi))$. For that, we need the following standard lemma. Let $M$, $N$ be nonnegatively graded modules over the graded ring $R=\mathbb{C}[x_1,\ldots,x_r]$,  $\operatorname{deg}x_i=2$. Assume  that $\operatorname{dim}M_i < \infty$ for every $i \in \mathbb{Z}_{\geqslant 0}$. We also assume that $N$ is flat ($=$ graded free) over $R$.
\begin{lemma}\label{lemma_iso_enough_for_hbar_zero}
 Let $\varphi\colon M \rightarrow N$ be a graded homomorphism. Then  $\varphi$ is an isomorphism iff its specialization to $0 \in \operatorname{Spec}R$ is an isomorphism.
\end{lemma}
\begin{proof}

Let us first of all show that the morphism $\varphi$ is surjective. Set $\mathfrak{m}=(x_1,\ldots,x_r)$. We have an exact sequence $M \rightarrow N \rightarrow \operatorname{coker}\varphi$. Tensoring it by $R/\mathfrak{m}$ we obtain an exact sequence $M/\mathfrak{m}M \rightarrow N/\mathfrak{m}N \rightarrow \operatorname{coker}\varphi/\mathfrak{m}\operatorname{coker}\varphi$. The first map in this sequence is an isomorphism that implies $\operatorname{coker}\varphi = \mathfrak{m}\operatorname{coker}\varphi$. Note now that $\operatorname{coker}\varphi$ is nonnegative graded and the generators of $\mathfrak{m}$ have positive degrees. This implies that $\operatorname{coker}\varphi=0$. 

Now, consider the exact sequence $0 \rightarrow \operatorname{ker}\varphi \rightarrow M \rightarrow N \rightarrow 0$. Module $N$ is flat so after tensoring by $R/\mathfrak{m}$ we obtain the exact sequence $0 \rightarrow \operatorname{ker}\varphi/\mathfrak{m}\operatorname{ker}\varphi \rightarrow  M/\mathfrak{m}M \rightarrow N/\mathfrak{m}N \rightarrow 0$ implying that $\operatorname{ker}\varphi/\mathfrak{m}\operatorname{ker}\varphi=0$ so $\operatorname{ker}\varphi=0$.
\end{proof}

\begin{prop}\label{prop_cartan_subq_W_in_U_realiz}
The map $S_{-\gamma}\circ \pi \colon \mathcal{U}_{\hbar}(\widetilde{\mathfrak{r}})_0 \rightarrow \mathcal{U}_{\hbar}(\mathfrak{r}_0) \subset \mathcal{U}_\hbar(\mathfrak{m})$ induces the surjective map $\mathcal{U}_\hbar(\mathfrak{g},\chi)_0 \twoheadrightarrow Z(\mathcal{U}_\hbar(\mathfrak{m}))$ that, in its turn, induces the isomorphism $\cC_\nu(\mathcal{U}_\hbar(\mathfrak{g},\chi)) \iso Z(\mathcal{U}_\hbar(\mathfrak{m}))$. 
\end{prop}
\begin{proof}
The homogeneous version of the argument used in the proof of \cite[Theorem 4.3]{BGK} shows that the map $S_{-\gamma} \circ \pi$ induces the homomorphism $\cC_\nu(\mathcal{U}_\hbar(\mathfrak{g},\chi)) \ra Z(\mathcal{U}_\hbar(\mathfrak{m}))$ of graded $\mathbb{C}[\hbar]$-algebras. 
Clearly, $Z(\mathcal{U}_\hbar(\mathfrak{m})) = \mathbb{C}[\mathfrak{h}^*,\hbar]^{W_M}$ is flat over $Z_\hbar = \mathbb{C}[\mathfrak{h}^*,\hbar]^{W}$. 
So, by Lemma \ref{lemma_iso_enough_for_hbar_zero}, in order to check that $\cC_\nu(\mathcal{U}_\hbar(\mathfrak{g},\chi)) \ra Z(\mathcal{U}_\hbar(\mathfrak{m}))$ is an isomorphism it is enough to show that it becomes an isomorphism after the specialization $\hbar=0$. 

Let us identify:
\begin{equation*}
\mathcal{U}_\hbar(\mathfrak{g},\chi) \iso \mathcal{W}_\hbar(\chi)~\text{(see Proposition \ref{alter_def_W})},~\widetilde{HC}_{\hbar}\colon Z(\mathcal{U}_\hbar(\mathfrak{m})) \iso \mathbb{C}[\mathfrak{h}^*]^{W_M} \otimes \mathbb{C}[\hbar].
\end{equation*}
It remains to check that the $\hbar=0$ specialization of the morphism 
\begin{equation}\label{local_our_morph}
\cC_\nu(\mathcal{W}_\hbar(\chi)) 
 \simeq \cC_\nu(\mathcal{U}_\hbar(\mathfrak{g},\chi)) \rightarrow Z(\mathcal{U}_\hbar(\mathfrak{m})) \simeq \mathbb{C}[\mathfrak{h}^*]^{W_M} \otimes \mathbb{C}[\hbar]
\end{equation}
is equal to the isomorphism $\CC[S(\chi)^{Z_M}] \simeq \CC[\mathfrak{h}^*]^{W_M}$ induced by the natural identification $S(\chi)^{Z_M}=S(\chi,\mathfrak{m}) \iso \mathfrak{m}^*/\!/M$.

Recall that the morphism (\ref{local_our_morph}) is induced by the collection of maps:
\begin{equation}\label{maps_that_induce}
\mathcal{W}_{\hbar}(\chi)_0 \leftarrow \mathcal{U}_{\hbar}(\mathfrak{g},\chi)_0 \hookrightarrow \mathcal{U}_{\hbar}(\widetilde{\mathfrak{r}})_0 \rightarrow \mathcal{U}_{\hbar}(\mathfrak{m}) \supset Z(\mathcal{U}_{\hbar}(\mathfrak{m})).
\end{equation}

 Recall also that 
 \begin{equation*}
 \mathcal{W}_{\hbar}(\chi)/(\hbar) = \CC[(\chi+(\mathfrak{g}/\mathfrak{u})^*)/\!/N],~\mathcal{U}_{\hbar}(\mathfrak{g},\chi)/(\hbar) = \CC[\widetilde{\mathfrak{r}}^*/\!/N],
 \end{equation*}
 \begin{equation*}
\mathcal{U}_\hbar(\widetilde{\mathfrak{r}})/(\hbar)=\CC[\widetilde{\mathfrak{r}}^*],~\mathcal{U}_{\hbar}(\mathfrak{m})/(\hbar)=\CC[\mathfrak{m}^*],~Z(\mathcal{U}_{\hbar}(\mathfrak{m}))/(\hbar)=\CC[\mathfrak{m}^*/\!/M].
\end{equation*} 
We will also use the identification $S(\chi) \iso (\chi+(\mathfrak{g}/\mathfrak{u})^*)/\!/N$ induced by the embedding $\chi+\mathfrak{z}_{\mathfrak{g}}(f) \subset \chi+(\mathfrak{g}/\mathfrak{u})^*$ (we identify $\mathfrak{z}_{\mathfrak{g}}(f)$ with its image in $\mathfrak{g}^*$).

After specializing to $\hbar=0$ the maps in (\ref{maps_that_induce})  become:
\begin{equation*}
\CC[\chi+\mathfrak{z}_{\mathfrak{g}}(f)]_0 \leftarrow \CC[(\chi+(\mathfrak{g}/\mathfrak{u})^*)/\!/N] \leftarrow \CC[\widetilde{\mathfrak{r}}^*/\!/N]_0 \hookrightarrow \CC[\widetilde{\mathfrak{r}}^*]_0 \rightarrow \CC[\mathfrak{m}^*] \supset \CC[\mathfrak{m}^*]^M, 
\end{equation*}
where the map $\CC[\widetilde{\mathfrak{r}}^*]_0 \twoheadrightarrow \CC[\mathfrak{m}^*]$ is induced by the projection $\widetilde{\mathfrak{r}} \twoheadrightarrow \widetilde{\mathfrak{r}}_0=\mathfrak{r}_0$ composed with the embedding $\mathfrak{r}_0 \hookrightarrow \mathfrak{m}$ and all other morphisms are obvious.
It follows that if $s \in \CC[\widetilde{\mathfrak{r}}^*]$ is invariant w.r.t. both $N$ and $Z_M$, then the induced surjection $\CC[\chi + \mathfrak{z}_{\mathfrak{g}}(f)] \twoheadrightarrow \CC[\mathfrak{m}^*]^M$ sends $s|_{\chi + \mathfrak{z}_{\mathfrak{g}}(f)}$ to $s|_{\mathfrak{m}^*}$. Being restricted to $\chi|_{\mathfrak{m}} + \mathfrak{z}_{\mathfrak{m}}(f)$, it then sends $s|_{(\chi|_{\mathfrak{m}} + \mathfrak{z}_{\mathfrak{m}}(f))}$ to $s|_{\mathfrak{m}^*}$. Since $s|_{(\chi|_{\mathfrak{m}} + \mathfrak{z}_{\mathfrak{m}}(f))}$ is the restriction of $s|_{\mathfrak{m}^*}$ to $\chi|_{\mathfrak{m}} + \mathfrak{z}_{\mathfrak{m}}(f)$, we are done.
\end{proof}

Combining Proposition \ref{prop_cartan_subq_W_in_U_realiz} with the identification $\mathcal{U}_\hbar(\mathfrak{g},\chi) \iso \mathcal{W}_\hbar(\chi)$ described in Proposition \ref{alter_def_W} together with the identification $\widetilde{HC}_{\hbar}\colon Z(\mathcal{U}_\hbar(\mathfrak{m})) \iso \mathbb{C}[\mathfrak{h}^*]^{W_M} \otimes \mathbb{C}[\hbar]$, we obtain the isomorphism: 
\begin{equation}\label{cartan_subq_of_W_eq}
\cC_\nu(\mathcal{W}_\hbar(\chi)) \simeq \mathbb{C}[\mathfrak{h}^*]^{W_M} \otimes \mathbb{C}[\hbar].
\end{equation}

\begin{example*}
Let's consider the example $\chi=0$.   Then, $\mathfrak{m}=\mathfrak{h}$, $\mathcal{W}_\hbar(0)=\mathcal{U}_\hbar(\mathfrak{g})$ but the identification $\mathbb{C}[\mathfrak{h}^*] \otimes \mathbb{C}[\hbar] \simeq \cC_\nu(\mathcal{U}_\hbar(\mathfrak{g}))$ is {\emph{not}} just induced by the natural embedding $\mathcal{U}_\hbar(\mathfrak{h}) \hookrightarrow \mathcal{U}_\hbar(\mathfrak{g})$. It is induced by this embedding composed with the twist $\mathbb{C}[\mathfrak{h}^*,\hbar] \iso \mathbb{C}[\mathfrak{h}^*,\hbar]$ given by a pull back w.r.t. the morphism $(\lambda,\hbar_0) \mapsto (\lambda-\hbar_0\rho_{\mathfrak{g}},\hbar_0)$. We need this shift because we want the center $Z_\hbar$ to identify with $\mathbb{C}[\mathfrak{h}^*]^{W_G} \otimes \mathbb{C}[\hbar]$ after the identifications above.  
\end{example*}


The following lemma follows from the proof of Proposition \ref{prop_cartan_subq_W_in_U_realiz} above.
\begin{lemma}\label{induced_ident_hbar_0}
For $\hbar=0$, the isomorphism (\ref{cartan_subq_of_W_eq}): $\CC[S(\chi)^{Z_M}] \simeq \CC[\mathfrak{h}^*]^{W_M}$ is induced by the natural isomorphism $S(\chi_{\mathfrak{m}})^{Z_M}=S(\chi,\mathfrak{m}) \iso \mathfrak{m}^*/\!/M$.
\end{lemma}

\subsubsection{Highest weights for $\mathcal{W}_\hbar(\chi)$}
In this section, we describe the homomorphism (\ref{comp_cart_alg_to_sh}) for $Y=\widetilde{S}(\mathfrak{q},\mathfrak{b})$.

\begin{prop}\label{BGK_hw_theory}
The induced embedding 
\begin{equation*}
\CC[\mathfrak{h}^*,\hbar]^{W_G} \xrightarrow{(HC_\hbar)^{-1}} Z_\hbar \hookrightarrow \cC_\nu(\mathcal{W}_\hbar(\chi)) \simeq \mathcal{W}_\hbar(\chi,\mathfrak{m}) \iso \CC[\mathfrak{h}^*]^{W_M} \otimes \mathbb{C}[\hbar]
\end{equation*}
is given by $s \mapsto \widetilde{s}$, where $\widetilde{s}(\la,\hbar_0)=s(\la-\hbar_0\rho_{\mathfrak{g}},\hbar_0)$ for $\la \in \mathfrak{h}^*$, $\hbar_0 \in \mathbb{C}$. In particular, the image of this embedding is equal to $\CC[\mathfrak{h}^*]^{W_G} \otimes \mathbb{C}[\hbar]$.
\end{prop}
\begin{proof}
Follows from \cite[Theorem 4.7 and Lemma 5.1]{BGK}: from their results, it follows that the composition $\CC[\mathfrak{h}^*]^{W_G} \otimes \CC[\hbar] \xrightarrow{(\widetilde{HC}_\hbar)^{-1}} Z_\hbar \hookrightarrow \cC_\nu(\mathcal{W}(\chi)) \simeq \mathcal{W}_\hbar(\chi,\mathfrak{m}) \iso \CC[\mathfrak{h}^*,\hbar]^{W_M}$ sends $s$ to $s$, 
it remains to note that the composition $\widetilde{HC}_{\hbar} \circ (HC_\hbar)^{-1}$ is given by $s \mapsto \widetilde{s}$. 
\end{proof}

We are now ready to describe the composition (\ref{comp_cart_alg_to_sh}) in the case of finite $W$-algebras.
Recall from the discussion in  \cref{quant_slod_var} the identification $\mathcal{A}_{\hbar,\mathfrak{h}^*}(\widetilde{S}(\chi_{\mathfrak{m}})) \simeq \mathcal{W}_\hbar(\chi_{\mathfrak{m}}) \otimes_{Z_\hbar} \CC[\mathfrak{h}^*,\hbar]$. 
Under this identification the $\rho_{\mathfrak{g}}$-twisted action of $\CC[\mathfrak{h}^*,\hbar]$ on $\mathcal{A}_{\hbar,\mathfrak{h}^*}(\widetilde{S}(\chi_{\mathfrak{m}}))$ corresponds to the {\emph{standard}} action of $\CC[\mathfrak{h}^*,\hbar]$ on the tensor product above.



\begin{prop}\label{cartan_subq_w}
After the identification (\ref{cartan_subq_of_W_eq}) combined with the bijection 
\begin{equation*}
\widetilde{S}(\chi_{\mathfrak{m}})^{Z_M}=\cB_{\chi_{\mathfrak{m}}}^{Z_M}= {}^{M}W
\end{equation*}
from \cref{easy fixed point Springer fiber}, the homomorphism 
\begin{equation}\label{cartan_to_fixed_slodowy}
\CC[\mathfrak{h}^*,\hbar]^{W_M} \otimes_{Z_\hbar} \CC[\mathfrak{h}^*,\hbar] =\cC_\nu(\mathcal{A}_{\hbar,\mathfrak{h}^*}(\widetilde{S}(\chi_{\mathfrak{m}}))) \rightarrow \Gamma(\widetilde{S}_{\mathfrak{h}^*}(\chi_{\mathfrak{m}}),\mathcal{D}_{\hbar,\mathfrak{h}^*}(\widetilde{S}(\chi_{\mathfrak{m}})))=\bigoplus_{w \in {}^{M}W} \CC[\mathfrak{h}^*,\hbar]
\end{equation}
is given by $f \otimes g \mapsto (f(w\la,\hbar) g)_{w \in {}^MW}$.
\end{prop}
\begin{proof}
Let us first of all note that by Remark \ref{iso_cart_generic} and Proposition \ref{BGK_hw_theory} together with \cref{hw_geom_vs_hw_alg} (a), for fixed $\hbar=\hbar_0$ and a Zariski generic $\la \in \mathfrak{h}^*$, morphism (\ref{cartan_to_fixed_slodowy}) is an isomorphism and should send $s \otimes g$ to the collection of functions $s(w\la,\hbar_0)g(\la,\hbar_0)$ parametrized by  $w \in W_M \backslash W$. From the continuity (see also Proposition \ref{hw_geom_vs_hw_alg} (c)), we conclude that this holds for every $\lambda$, so it remains to identify this labeling with the labeling by $\widetilde{S}(e_{\mathfrak{m}})^{Z_M}$.
This can be done after the specialization at $\hbar=0$. So, using \cref{induced_ident_hbar_0}, we see that it remains to describe the pull back homomorphism
\begin{equation*}
\CC[S(e_{\mathfrak{m}})^{Z_M} \times_{\mathfrak{g}^*/W_G} \mathfrak{h}^*] \rightarrow \CC[\widetilde{S}_{\mathfrak{h}^*}(e_{\mathfrak{m}})^{Z_M}].
\end{equation*}
The fact that it is indeed given by $s \otimes g \mapsto (s(w\la)g)_{w}$  follows from Proposition \ref{descr_fixed_parabolic}: recall that fixed points are pairs $(m_{w\la}w \cdot \mathfrak{p},x_{w\la})$ and the value of $s \in \CC[\mathfrak{h}^*]^{W_M}$ at the corresponding fixed point is equal to $s(x_{w\la})=s(w\la)$. 
\end{proof}

\subsection{Highest weights for finite parabolic $W$-algebras}
Recall from \cref{quant_slod_var} the isomorphism $\mathcal{A}_{\hbar,\mathfrak{h}^*}(\widetilde{S}(\chi_{\mathfrak{m}}))\simeq\mathcal{W}_{\hbar}(\chi_{\mathfrak{m}}) \otimes_{Z_\hbar} \CC[\mathfrak{h}^*,\hbar]$  and the homomorphism $\Phi_{\mathfrak{p}, \chi_\fm}\colon \mathcal{A}_{\hbar,\mathfrak{h}^*}(\widetilde{S}(\chi_{\mathfrak{m}}))\to \mathcal{A}_{\hbar,\mathfrak{h}^*}(\widetilde{S}(\fq, \fp))$ 
It induces the homomorphism: 
\begin{equation*}
\CC[\mathfrak{h}^*,\hbar]^{W_M} \otimes_{Z_\hbar} \CC[\mathfrak{h}^*,\hbar]\simeq\cC_\nu(\mathcal{W}_{\hbar}(\chi_{\mathfrak{m}})) \otimes_{Z_\hbar} \CC[\mathfrak{h}^*,\hbar]\simeq \cC_\nu(\mathcal{A}_{\hbar,\mathfrak{h}^*}(\widetilde{S}(\chi_{\mathfrak{m}}))) \rightarrow \cC_\nu(\mathcal{A}_{\hbar,\mathfrak{X}(\mathfrak{l})}(\widetilde{S}(\mathfrak{q},\mathfrak{p})))
\end{equation*}
that we call the \emph{Cartan comoment map}.

The goal of this section is to describe the composition: 
\begin{multline}\label{comp_parab}
\CC[\mathfrak{h}^*,\hbar]^{W_M} \otimes_{Z_\hbar} \CC[\mathfrak{h}^*,\hbar] \rightarrow \cC_\nu(\mathcal{A}_{\hbar,\mathfrak{X}(\mathfrak{l})}(\widetilde{S}(\mathfrak{q},\mathfrak{p}))) \rightarrow \\ \rightarrow \Gamma(\widetilde{S}_{\mathfrak{X}(\mathfrak{l})}(\mathfrak{q},\mathfrak{p}),\mathcal{D}_{\hbar,\mathfrak{X}(\mathfrak{l})}(\widetilde{S}(\mathfrak{q},\mathfrak{p})))=\bigoplus_{w \in {}^{M}(W/W_L)} \CC[\mathfrak{X}(\mathfrak{l}),\hbar].
\end{multline}


\begin{lemma}\label{lemma_hw_compute_parabolicW}
The homomorphism (\ref{comp_parab}) is given by:  
\begin{equation*}
s \otimes g \mapsto ((\la,\hbar_0) \mapsto s(w\la,\hbar_0) g(\la+\hbar_0\rho_{{\mathfrak{l}}},\hbar_0))_{w \in {}^M(W/W_L)},
\end{equation*}
where $(\lambda,\hbar_0) \in \mathfrak{X}(\mathfrak{l}) \oplus \mathbb{C}$.
\end{lemma}

\begin{proof}
For $\hbar=1$, the abelian localization holds for integral $\la$ such that $\widetilde{\lambda}$ is  $\mathfrak{p}$-antidominant regular (see Propositions \ref{loc_flags}, \ref{loc_slodowy}). Moreover, the homomorphism 
\begin{equation*}
\Phi_{\mathfrak{p},\chi_{\mathfrak{m}}}\colon \mathcal{A}_{1,\widetilde{\la}}(\widetilde{S}(\chi_{\mathfrak{m}})) \twoheadrightarrow \mathcal{A}_{1,\la}(\widetilde{S}(\mathfrak{q},\mathfrak{p}))
\end{equation*}
is surjective for such $\lambda$ (see \cite[Theorem 3.8]{BoBr}). It follows that every irreducible $\mathcal{A}_{1,\la}(\widetilde{S}(\mathfrak{q},\mathfrak{p}))$-module is also irreducible over $\mathcal{A}_{1,\widetilde{\la}}(\widetilde{S}(\chi_{\mathfrak{m}}))$.

Consider now the category $\CO_\nu(\mathcal{A}_{1,\widetilde{\la}}(\widetilde{S}(\chi_{\mathfrak{m}})))$ and recall that the irreducible object corresponding to $w \in {}^MW$ has the highest weight equal to $w(\widetilde{\la})-\rho_{\mathfrak{g}}$. It follows from \cite[Theorem 8]{w} that the category $\CO_\nu(\mathcal{A}_{1,\widetilde{\la}}(\widetilde{S}(\chi_{\mathfrak{m}})))$ is equivalent to the BGG category $\mathcal{O}$ for $\mathcal{U}_{\la_{\mathfrak{m}}}(\mathfrak{g})$, where $\la_{\mathfrak{m}} \in \mathfrak{h}^*$ is a $\mathfrak{b}$-antidominant weight with stabilizer $W_M$. Under this equivalence,
$\CO_\nu(\mathcal{A}_{1,\la}(\widetilde{S}(\mathfrak{q},\mathfrak{p}))) \subset \CO_\nu(\mathcal{A}_{1,\widetilde{\la}}(\widetilde{S}(\chi_{\mathfrak{m}})))$ identifies with the
subcategory consisting of objects that are locally finite for the action of $\mathfrak{p}$ (i.e., we consider the {\emph{singular}} block of the {\em{parabolic}} BGG category $\mathcal{O}$, see for example \cite[Section 9]{hum_bgg_O}).
It follows from the definitions that the equivalence above sends a simple object of $\CO_\nu(\mathcal{A}_{1,\widetilde{\la}}(\widetilde{S}(e_{\mathfrak{m}})))$ corresponding to $w \in {}^MW$ to the simple of the BGG category $\mathcal{O}$ for $\mathcal{U}_{\la}(\mathfrak{g})$ with the highest weight $\widetilde{\nu}=w^{-1}(\la) - \rho_{\mathfrak{g}}$ (note that the map $w \mapsto w^{-1}$ indeed defines the bijection ${}^MW \iso W/W_M$). It is a standard result (compare with \cite[Section 9.2]{hum_bgg_O} and \cite[Section 2]{gruber}) that a simple of $\mathcal{O}_{\widetilde{\nu}}(\mathcal{U}_\la)$ labeled by $wW_M$ is $\mathfrak{p}$-integrable iff $w$ is longest in $W_Lw$. We conclude that the simple objects in $\CO_\nu(\mathcal{A}_{1,\la}(\widetilde{S}(\mathfrak{q},\mathfrak{p}))) \subset \CO_\nu(\mathcal{A}_{1,\la}(\widetilde{S}(e_{\mathfrak{m}})))$ are labeled by ${}^M(W/W_L) \hookrightarrow {}^MW$, where the embedding sends $wW_L \in {}^M(W/W_L)$ to the longest representative in $wW_L$.

We then conclude from Proposition \ref{cartan_subq_w} 
that the homomorphism (\ref{comp_parab}) sends $s \otimes g$ to a collection of functions $(\lambda,\hbar_0) \mapsto s(w\la,\hbar_0) g(\la +\hbar_0\rho_{{\mathfrak{l}}},\hbar_0)$ parametrized by ${}^M(W/W_L)$ (shift by $\hbar_0\rho_{\mathfrak{l}}$ appears because the homomorphism $\Phi_{\mathfrak{p},e_{\mathfrak{m}}}$ is {\emph{not}} $\CC[\mathfrak{X}(\mathfrak{l}),\hbar]$-linear, see Warning \ref{warning_shift}). To determine the labeling, we restrict to $\hbar=0$. As in the proof of Proposition \ref{cartan_subq_w}, the claim then follows from Proposition  \ref{descr_fixed_parabolic} together with Corollary \ref{separating_cor_parabolic}.
\end{proof}


\section{Equivariant Hikita-Nakajima conjecture: general approach}\label{equivariant conjecture section}
In this section, we describe the general approach to the Hikita-Nakajima conjecture that we propose (following \cite{ksh}) as well as its relation to a certain statement proposed by Bullimore, Dimofte, Gaiotto, Hilburn, and Kim in \cite{votices_vermas}. We also briefly discuss a quantum version of this conjecture proposed by Kamnitzer, Proudfoot, and McBreen in \cite{kmp} and mention how it fits into the approach. Let us emphasize that the approach is highly based on the ideas of Bellamy, Braverman, Kamnitzer, Losev, Nakajima, Tingley, Webster, Weekes, Yacobi, and their co-authors.

\subsection{Large polynomial algebra} Let $X$, $X^\vee$ be a pair of symplectically dual varieties $X$, $X^\vee$, and assume that both $X$ and $X^\vee$ admit symplectic resolutions $Y$ and $Y^\vee$ respectively. 

\begin{warning}
Recall that in the statement of Hikita-Nakajima conjecture, we only assume that $X^\vee$ has a symplectic resolution.  
\end{warning}

	
Our goal is to compare the $\CC[\mathfrak{t}_{X^\vee},\hbar]$-algebras: 
\begin{equation}\label{our_alg}
H^*_{T_{X^\vee} \times \CC^\times}(Y^\vee), \qquad \cC_\nu(\mathcal{A}_{\hbar,\mathfrak{h}_{X}}(X)).
\end{equation}

Let us first note that both algebras should carry an additional structure. Namely, recall (see Section \ref{sect_sympl_duality_basic}) that the symplectic duality predicts identification of vector spaces $\mathcal{A}^2_{\hbar,\mathfrak{h}_{X},0} \simeq H^2_{T_{X^\vee}\times \CC^\times}(Y^\vee)$. 
Moreover, note that the elements of these vector spaces considered as elements of algebras (\ref{our_alg}) commute with each other. For $H^2_{T_{X^\vee}\times \CC^\times}(Y^\vee)$ the claim is clear, and for $\mathcal{A}^2_{\hbar,\mathfrak{h}_{X},0}$ it follows from the exact sequence (\ref{exact_seq_alg}). Thus, both algebras are not only $\CC[\mathfrak{t}_{X^\vee},\hbar]$-algebras but also modules over the bigger polynomial algebra $S^{\bullet}(\mathcal{A}^2_{\hbar,\mathfrak{h}_{X},0})=S^{\bullet}(H^2_{T_{X^\vee}\times \CC^\times}(Y^\vee))$.

In fact, in many cases, there is an even bigger polynomial algebras acting on both of the algebras above. For example, assume that $Y^\vee=\widetilde{\mathcal{M}}_H$ is a Nakajima quiver variety corresponding to some quiver $I$, and let $X=\mathcal{M}_C$ be the corresponding Coulomb branch, and $A$ be a flavor torus. Then $\mathcal{A}_{\hbar,\mathfrak{a}}=H_*^{(G_I[[t]] \rtimes \CC^\times) \times A}(\mathcal{R})$ is naturally a module over the polynomial algebra $B:=H^*_{G_I \times A \times \CC^{\times}}(\on{pt})$ (the subalgebra $B \cdot 1\subset \mathcal{A}_{\hbar,\mathfrak{a}}$ is sometimes called the ``Cartan'' subalgebra of $\mathcal{A}_{\hbar,\mathfrak{a}}$, see e.g. \cite[Definition 3.16]{BFNII}). The same algebra $B$ acts naturally on $H^*_{A \times \CC^\times}(\widetilde{\mathcal{M}}_H)=H^*_{G_I \times A \times \CC^\times}(\mu^{-1}(0)^{\mathrm{st}})$ via the natural restriction homomorphism (which is surjective by the results of McGerty and Nevins, see \cite{kirwan_surj_quiv}):
\begin{equation*}
B=H^*_{G_I \times A\times \CC^\times}(\mu^{-1}(0)) \twoheadrightarrow H^*_{G_I \times A \times \CC^\times}(\mu^{-1}(0)^{\mathrm{st}})=H^*_{A \times \CC^{\times}}(\widetilde{\mathcal{M}}_H).
\end{equation*}

\begin{rmk}
The appearance of the algebra $B$ on both sides of the symplectic duality and its importance for the Hikita-Nakajima conjecture was already observed in \cite[Section 5.7]{HKW}.  
\end{rmk}

Let us list another example which is relevant to the subject of the paper. Let  $Y^\vee=\widetilde{S}(\mathfrak{p}^\vee,\mathfrak{q}^\vee)$ and $X=\on{Spec}\CC[\widetilde{S}(\mathfrak{q},\mathfrak{p})]$. Then $H^*_{T_{e^\vee} \times \CC^\times}(\widetilde{S}(\mathfrak{p}^\vee,\mathfrak{q}^\vee))$ is a module over $H^*_{T_{e^\vee} \times \CC^\times}(T^*\mathcal{Q}^\vee)$ via the restriction map. The latter algebra is a quotient of the polynomial algebra $B=H^*_{T_{e^\vee} \times \CC^\times}(\on{pt}) \otimes H^*_{L}(\on{pt})$ (see \cref{Polquotient}). Now, the morphism $\Psi_{\chi, \fp}\colon \mathcal{A}_{\hbar,\mathfrak{X}(\mathfrak{l})}(\widetilde{S}(\mathfrak{q},\mathfrak{b})) \rightarrow \mathcal{A}_{\hbar,\mathfrak{X}(\mathfrak{l})}(\widetilde{S}(\mathfrak{q},\mathfrak{p}))$  constructed in Section \ref{quant_slod_var}  realizes $\cC_\nu(\mathcal{A}_{\hbar,\mathfrak{X}(\mathfrak{l})}(\widetilde{S}(\mathfrak{q},\mathfrak{p})))$ as a module over $\cC_\nu(\mathcal{A}_{\hbar,\mathfrak{X}(\mathfrak{l})}(\widetilde{S}(\mathfrak{q},\mathfrak{b})))=H^*_{T_{e^\vee}}(T^*\mathcal{Q}^\vee)$. So, the algebra $B$ indeed acts on both sides of the picture. 

\begin{rmk}
Note that in the Coulomb branch case, $B \cdot 1$  is a maximal commutative subalgebra of $\mathcal{A}_{\hbar,\mathfrak{h}_X}$ and it defines the integrable system on $\mathcal{M}_C$. We do not know if this should be a general phenomenon. In the case of Slodowy varieties, one might wonder if the algebra $B \cdot 1$ should correspond to some commutative subalgebra in the $W$-algebra. See also a recent preprint \cite{Tu_integr_syst_min_nilp}.
\end{rmk}

\subsection{Main conjecture} 
So, let $B$ be some polynomial algebra that acts naturally on both of our algebras. Acting on $1$, we obtain the homomorphisms: 
\begin{equation*}
B \xrightarrow{a} H^*_{T_{X^\vee} \times \CC^\times}(Y^\vee),~ B \xrightarrow{b} \cC_\nu(\mathcal{A}_{\hbar,\mathfrak{h}_X}(X)).
\end{equation*}
Recall that we have the {\emph{natural}} identification 
\begin{equation*}
H^*_{T_{X^\vee} \times \CC^\times}((Y^\vee)^{T_{X^\vee}}) \simeq \Gamma(Y_{\mathfrak{h}_X}^{T_X},\cC_\nu(\mathcal{D}_{\hbar,\mathfrak{h}_X}(Y)))).
\end{equation*}
induced by the bijection $(Y^\vee)^{T_{X^\vee}} \simeq Y^{T_X}$ together with the isomorphism $\mathfrak{t}_{X^\vee} \simeq \mathfrak{h}_X$ (see Section \ref{sect_sympl_duality_basic}). 
Note that both of the algebras above are simply $\bigoplus_{p \in Y^{T_X}}\CC[\mathfrak{h}_{X},\hbar]$.

\subsubsection{Statement of the main conjecture}
The following Conjecture should be compared with \cite[Conjecture 5.26]{HKW} and \cite[Conjecrture 8.10]{KamnitzerTingleyWebsterWeeksYacobi}.

\begin{conj}\label{conj_most_general_weak_Hikita}
The following diagram is commutative: 

\begin{equation}\label{min_comm_diag}
\xymatrix{ &  H^*_{T_{X^\vee} \times \CC^\times}(Y^\vee) \ar[r] &  H^*_{T_{X^\vee} \times \CC^\times}((Y^\vee)^{T_{X^\vee}}) \ar[d]_{\simeq} \\
B \ar[ru]^{a} \ar[rd]^{b} & &  \bigoplus_{p \in Y^{T_X}}\CC[\mathfrak{h}_{X},\hbar] \\
& \cC_\nu(\mathcal{A}_{\hbar,\mathfrak{h}_X}(X)) \ar[r] & \ar[u]^{\simeq}  \Gamma(Y_{\mathfrak{h}_X}^{T_X},\cC_\nu(\mathcal{D}_{\hbar,\mathfrak{h}_X}(Y))))
}
\end{equation}

In particular, there exists the isomorphism of $B$-algebras: 
\begin{equation} \label{main identification}
\on{Im}(B \rightarrow H^*_{T_{Y^\vee} \times \CC^\times}(Y^\vee)) \simeq \on{Im}(B \rightarrow \cC_\nu(\mathcal{A}_{\hbar,\mathfrak{h}_X}(X)) \rightarrow \Gamma(Y_{\mathfrak{h}_X}^{T_X},\cC_\nu(\mathcal{D}_{\hbar,\mathfrak{h}_X}(Y)))).
\end{equation}
\end{conj}

\begin{warning}
Let us mention an important technical detail: for (\ref{min_comm_diag}) to be commutative, one should {\emph{normalize}} the map $a$ by saying that $\hbar$ goes to $\frac{1}{2}\hbar$ (see, for example, \cite[Conjecture 8.9]{KamnitzerTingleyWebsterWeeksYacobi} or Theorem \ref{main_th_weak_hikita} below). This seemingly minor detail is crucial for applications to representation theory (see, for example, definition of $I({\mathbb{O}}^\vee)$ in Section \ref{sect_refined_BVLS}).
	
\end{warning}

\begin{rmk}\label{form_conj_as_solving}
Let us mention one possible way to think about Conjecture \ref{conj_most_general_weak_Hikita}. We can consider $H^*_{T_{Y^\vee} \times \CC^\times}(Y^\vee)$, $\cC_\nu(\mathcal{A}_{\hbar,\mathfrak{h}_X}(X)$ as $\CC[\mathfrak{h}_X,\hbar]$-modules of rank $|(Y^\vee)^{T_{X^\vee}}|=|Y^{T_X}|$.
The algebra $B$ acts on both of these modules by $\CC[\mathfrak{h}_X,\hbar]$-linear endomorphisms. Let $\mathbb{K}$ be the algebraic closure of the field of fractions $\on{Frac}(\CC[\mathfrak{h}_X,\hbar])$. Base changing to $\mathbb{K}$, our modules become $|Y^{T_X}|$-dimensional vector spaces over $\mathbb{K}$, and they decompose into the direct sum of one-dimensional modules over $B_{\mathbb{K}}:=B \otimes_{\CC[\mathfrak{h}_X,\hbar]} \mathbb{K}$ parametrized by certain characters of $B$ (in other words, we diagonalize the action of $B_{\mathbb{K}}$ on these modules). Conjecture \ref{conj_most_general_weak_Hikita} gives a concrete recipe to find these characters and claims that they should be the same.  
\end{rmk}

\subsubsection{Main conjecture in the context of localization algebras}

In \cite[Definition 2.1]{localiz_alg}, authors introduced a notion of a {\emph{localization algebra}}, let us recall their definition. 

\begin{definition} A 
localization algebra is a quadruple $\mathcal{Z}=(U,Z,\mathcal{I},h)$, where $U$ is a finite-dimensional complex vector space, $Z$ is a finitely generated graded $\on{Sym}U$-algebra, $\mathcal{I}$ is a finite set, and 
\begin{equation*}
h\colon Z \rightarrow \bigoplus_{p \in \mathcal{I}}\on{Sym}U
\end{equation*}
is a homomorphism of $\on{Sym}U$-algebras. If the kernel and cokernel are torsion $\on{Sym}U$-modules, then we call $\mathcal{Z}$ strong. If $Z$ is free of rank $|\mathcal{I}|$ as a $\on{Sym}U$-module, we call $\mathcal{Z}$ free.  
\end{definition}

We observe that algebras $H^*_{T_{X^\vee} \times \CC^\times}(Y^\vee)$, $\cC_\nu(\mathcal{A}_{\hbar,\mathfrak{h}_{X}}(X))$ have {\emph{natural}} structures of localization algebras. Indeed, for $Z=H^*_{T_{X^\vee} \times \CC^\times}(Y^\vee)$, as was already observed in \cite[Example 2.2]{localiz_alg}, we should take $U=\on{Lie}(T_{X^\vee} \times \mathbb{C}^\times)^*$, $\mathcal{I}=(Y^\vee)^{T_{X^\vee}}$ and $h=\iota^*$, where $\iota$ is the embedding $(Y^\vee)^{T_{X^\vee}} \hookrightarrow Y^\vee$. This is an example of a {\emph{strong}} and {\emph{free}} localization algebra. Now, for $Z=\cC_\nu(\mathcal{A}_{\hbar,\mathfrak{h}_{X}}(X))$, we take $U=(\mathfrak{h}_X \oplus \mathbb{C})^*$, $\mathcal{I}=Y^{T_{Y}}$, and $h$ being the natural map $\cC_\nu(\mathcal{A}_{\hbar,\mathfrak{h}_X}(X)) \rightarrow \Gamma(Y_{\mathfrak{h}_X}^{T_X},\cC_\nu(\mathcal{D}_{\hbar,\mathfrak{h}_X}(Y))))$. This is an example of a {\emph{strong}} localization algebra. 

Note now that our Conjecture \ref{conj_most_general_weak_Hikita} basically claims that the localization algebras structures above should be {\emph{compatible}}. 
In this sense, Conjecture \ref{conj_most_general_weak_Hikita} should be considered in the framework of the duality of  ``localization algebras'', see \cite[Section 10.6]{BPWII}.

Let us mention that in \cite[Section 10.6]{BPWII}, \cite{localiz_alg} authors study another interesting example of the localization algebra, namely the (universal) deformation $Z(\widetilde{E})$ of the center of the {\emph{geometric}} category $\mathcal{O}_\nu(\mathcal{D}_\la)$ for $Y^\vee$. In \cite[Conjecture 10.32]{BPWII} authors conjecture that for $\lambda$ such that $\mathcal{O}_\nu(\mathcal{D}_\la)$ is {\emph{indecomposable}}, localization algebras $Z(\widetilde{E})$, $H^*_{T_{X^\vee}}(Y^\vee)|_{\lambda}$ should be isomorphic. We {\emph{do not know}} how to generalize this statement to arbitrary $\lambda$.

\subsection{Relation to \cite{votices_vermas} and quantum Hikita-Nakajima conjecture}

\subsubsection{Relation to physics} Let us now explain the motivation for Conjecture \ref{conj_most_general_weak_Hikita} coming from the paper \cite{votices_vermas} by Bullimore, Dimofte, Gaiotto, Hilburn, and Kim.  We restrict ourselves to the case $Y^\vee=\widetilde{\mathcal{M}}_H$, $X=\mathcal{M}_C$ for a quiver gauge theory. Recall that $A$ is a flavor torus. 
Pick a $A$-fixed point $p \in \widetilde{\mathcal{M}}_H^{A}$.
In \cite{votices_vermas}, authors conjecture that (roughly speaking) the quantization $\mathcal{A}_{\hbar,\mathfrak{a}}(\mathcal{M}_C)$ should act on the equivariant  cohomology $H^*_{A \times \CC^\times}(\on{QMaps}_{p}(\mathbb{P}^1,\widetilde{\mathcal{M}}_H))$ of the space of based quasi-maps sending $\infty$ to $p$ and the corresponding module should be the universal {\emph{point}} module $\Theta_{\hbar,\mathfrak{a}}(p^{\vee})$ over the quantized Coulomb branch. In general, it is not clear how to define the action of $\mathcal{A}_{\hbar,\mathfrak{a}}(\mathcal{M}_C)$ on the space above (see \cite{tamagni} for the case of $\mathfrak{sl}_2$). On the other hand, the action of $B=H^*_{G_I \times A \times \CC^\times}(\on{pt})$ is easy to describe. Namely, consider the evaluation at $0$ morphism $\on{ev}_0\colon \on{QMaps}_{p}(\mathbb{P}^1,\widetilde{\mathcal{M}}_H) \rightarrow \mu^{-1}(0)/G_{I}$, then the element $\tau$ should act via the multiplication by $\on{ev}^*(\tau)$ (compare with \cite{l}).

Now, the natural grading on the point module should correspond to the decomposition of the space  $\on{QMaps}_{p}(\mathbb{P}^1,\widetilde{\mathcal{M}}_H)$ via the degree of the quasi-map, let us denote the degree $d$ component by $\on{QMaps}_{p}^d(\mathbb{P}^1,\widetilde{\mathcal{M}}_H)$. So, the highest weight component is nothing else but the cohomology of $\on{QMaps}_{p}^0(\mathbb{P}^1,\widetilde{\mathcal{M}}_H):=\{p\}$. We conclude that the action of $\tau \in B$ on the highest weight component $H^*_{A \times \CC^\times}(p)=\CC[\mathfrak{a},\hbar]$ is given by the multiplication by $\iota_{p}^*\tau$. This is precisely what Conjecture \ref{conj_most_general_weak_Hikita} claims.

\subsubsection{Relation to the quantum Hikita-Nakajima conjecture} \label{quantum Dmod}
As we see from the discussion above, Conjecture \ref{conj_most_general_weak_Hikita} is a shadow of the (conjectural) realization of the universal point module over the Coulomb branch via the space of based quasi-maps to the Higgs branch. This statement should, in particular, imply that characters of (universal) point modules (considered as functions on $B$) coincide with the $\hbar=q$ specializations of (normalized) vertex functions {\em{with descendants}} (introduced by Okounkov in \cite[Section 7.2]{O}) restricted to the corresponding fixed points.
It turns out that this statement is precisely the quantum analog of Conjecture \ref{conj_most_general_weak_Hikita} and can be used to prove the quantum Hikita conjecture in some cases (the quantum Hikita conjecture was formulated in \cite{kmp}). Namely, normalized vertex functions
with descendants are {\emph{solutions}} of the ``PSZ'' quantum $D$-module while characters of point modules are solutions of the $D$-module of graded traces (compare with Remark \ref{form_conj_as_solving}). By identifying solutions, we obtain identifications of the corresponding $D$-modules.
This is the joint work in progress of Hunter Dinkins, Ivan Karpov, and the second author. The relation of the quantum Hikita conjecture with the action of quantized Coulomb branches on the cohomology of the quasi-maps spaces was already observed in \cite[Remark 1.10]{HKW}. We are grateful to Joel Kamnitzer for pointing this out.

\subsection{Refined Hikita-Nakjima conjecture vs Hikita-Nakajima conjecture}\label{weak_vs_original} The natural question is: when Conjecture \ref{conj_most_general_weak_Hikita} implies the actual Hikita-Nakajima conjecture? Clearly, this is the case when morphisms $a,b$ are surjective and $\cC_\nu(\mathcal{A}_{\hbar,\mathfrak{h}_X}(X))$ is flat over $\CC[\mathfrak{h}_{X},\hbar]$ (note that $H^*_{T_{X^\vee} \times \CC^\times}(Y^\vee)$ is always free over $H^*_{T_{X^\vee} \times \CC^\times}(\on{pt})$). 
The following claims hold.

\begin{itemize}
    \item[(i)] Morphisms $a,b$ are surjective if $Y^\vee$ is a Nakajima quiver variety and $X$ is the corresponding Coulomb branch (surjectivity of $a$ follows from \cite[Corollary 1.5]{kirwan_surj_quiv}, for the surjectivity of $b$ see \cite[Proposition 8.7]{ksh} and references therein).
    \item[(ii)] Module $\cC_\nu(\mathcal{A}_{\hbar,\mathfrak{h}_X}(X))$ is flat over $\CC[\mathfrak{h}_{X},\hbar]$ when $X$ is a Coulomb branch of type $ADE$ quiver corresponding to a framing vector $(w_i)$  and a dimension vector $(v_i)$ such that $w_i \neq 0$ implies that $\omega_i$ is minuscule and $\mu=\sum_{i}w_i\omega_i - \sum_i v_i\al_i$ is dominant (follows from \cite[Section 8.2]{KamnitzerTingleyWebsterWeeksYacobi}). This also holds when $X$ is a Coulomb branch of the Jordan quiver (see \cite[Appendix A]{ksh}). Finally, we have flatness for $X=S(e^\vee)$ (see Proposition \ref{BGK_hw_theory}).
    \item[(iii)] Morphism $a$ is in general {\emph{not}} surjective for $Y^\vee=\widetilde{S}(e^\vee)$, see Example \ref{gl3 in B}
    \item[(iv)] Morphism $b$ is in general {\emph{not}} surjective. For example, consider $e^\vee\in \fsp({2n})$ that corresponds to a $2$-row partition. Then $\widetilde{D}(\OO^\vee)$ is a double cover of an orbit $\OO\subset \fsp_{2n+1}$ that has the form $(3,2,...,2,1,...,1)$. This double cover can be realized as another nilpotent orbit in $\fso_{2n+2}$ (\cite[Proposition 2.12]{FJLS2023}). One can check that the map $b$ is not surjective. In this case, it makes more sense to replace the polynomial algebra $B$ that comes from $\fso({2n+1})$ with the one that comes from $\fso_{2n+2}$.
    \item[(v)] Module $\cC_\nu(\mathcal{A}_{\hbar,\mathfrak{h}_X}(X))$ is in general {\emph{not}} flat over $\CC[\mathfrak{h}_{X},\hbar]$ for $X=\overline{\mathbb{O}}$, see Example \ref{counterflat}.
\end{itemize}

In general, we expect that $\cC_\nu(\mathcal{A}_{\hbar,\mathfrak{h}_X}(X))$ is flat over $\CC[\mathfrak{h}_{X},\hbar]$ for any {\emph{quiver gauge theory}} (see \cite[Conjecture 8.3]{ksh}).
From the above, we already know that Conjecture \ref{conj_most_general_weak_Hikita} implies the Hikita-Nakajima conjecture when $X$ is a Coulomb branch as in (ii). Proof of Conjecture \ref{conj_most_general_weak_Hikita} in this particular case can be deduced from the results of \cite{KamnitzerTingleyWebsterWeeksYacobi}, \cite{catO_yang}. 

Overall, Conjecture \ref{conj_most_general_weak_Hikita} is a replacement of the 
Hikita-Nakajima conjecture. 



\begin{rmk}\label{rem_general_conj}
Let us now drop the assumption that $X$ has a resolution of singularities. Recall that $\widetilde{X} \rightarrow X$ is the $\QQ$-factorial terminalization corresponding to the choice of $\nu\colon \CC^\times \rightarrow T_{X^\vee}$ on the dual side. We still conjecture that there exists an isomorphism of algebras 
\begin{multline*}
\on{Im}(S^\bullet(H^2_{T_{X^\vee} \times \CC^\times}(Y^\vee)) \rightarrow H^*_{T_{X^\vee} \times \CC^\times}(Y^\vee)) \simeq \\
 \simeq \on{Im}(S^\bullet(\mathcal{A}^2_{\hbar,\mathfrak{h}_X,0}(X)) \rightarrow \cC_\nu(\mathcal{A}_{\hbar,\mathfrak{h}_X}(X)) \rightarrow \Gamma(\widetilde{X}^{T_{X}},\mathcal{D}_{\hbar,\mathfrak{h}_X}(\widetilde{X}))).
\end{multline*}
A more refined version of this conjecture for $X^\vee=S(e^\vee)$, $X=\on{Spec}\CC[\widetilde{D}(\mathbb{O}^\vee)]$ is given in Section \ref{sec_more_gen_settings}.
\end{rmk}

\section{Refined equivariant Hikita-Nakajima conjecture for $\widetilde{S}(\mathfrak{q},\mathfrak{p})$}\label{Section_weak_HN_for_parab_lodowy_proof}  

The goal of this section is to prove Conjecture \ref{conj_most_general_weak_Hikita} for the pair 
\begin{equation*}
Y^\vee=\widetilde{S}(\mathfrak{p}^\vee,\mathfrak{q}^\vee),~Y=\widetilde{S}(\mathfrak{q},\mathfrak{p}),
\end{equation*}
where $\mathfrak{q}^\vee, \mathfrak{p}^\vee \subset \mathfrak{g}^\vee$ are determined by $\mathfrak{p}$, $\mathfrak{q}$ respectively using the natural bijection between roots of $\mathfrak{g}$ and roots of $\mathfrak{g}^\vee$.

\subsection{Restriction to fixed points  for $T^*\mathcal{Q}^\vee$ and $\widetilde{S}(\mathfrak{p}^{\vee},\mathfrak{q}^\vee)$}
Recall the restriction map $p\colon \mathfrak{X}(\mathfrak{m})=H^2(T^*\mathcal{Q}_-^\vee) \rightarrow H^2(\widetilde{S}(\mathfrak{p}_-^\vee,\mathfrak{q}_-^\vee))$.
By results of Section \ref{sec_loc_thm_parab_W}, abelian localization theorem holds for $(p(\la),\widetilde{S}(\mathfrak{p}^\vee,\mathfrak{q}^\vee))$ if $\la\in \fX(\fm)$ is $\mathfrak{q}^\vee$-antidominant.
So, on the dual side, we should pick generic $\nu\colon \CC^\times \rightarrow Z_M$ that is $Q$-antidominant.  By Proposition \ref{descr_fixed_parabolic}, torus fixed points of $Y$ are in the natural bijection with $(W_M\backslash W/W_L)^{\mathrm{free}}$, and the fixed points of $Y^\vee$ are in bijection with $(W_L\backslash W/W_M)^{\mathrm{free}}$. The bijection of (\ref{canonical:)_ident_fixed_points}) is given by: \begin{equation*}
{\bf{i}}\colon Y^{T_X} \iso (Y^\vee)^{T_{X^\vee}},\, [w] \mapsto [w^{-1}].
\end{equation*}

Our goal is to compare the algebras $H^*_{Z_{L^\vee} \times \CC^\times}(\widetilde{S}(\mathfrak{p}^\vee,\mathfrak{q}^\vee))$, $\cC_{\nu}(\mathcal{A}_{\hbar,\mathfrak{X}(\mathfrak{l})}(\widetilde{S}(\mathfrak{q},\mathfrak{p})))$.

Without losing the generality, we can assume that $G^\vee$ is simply connected. It follows that the derived subgroup of $M^\vee$ is simply connected. Recall that the action of $\CC^\times$ on $T^*\mathcal{Q}^\vee$ given by $t \cdot (\mathfrak{q}',x)=(\mathfrak{q}',t^{-2}x)$.

As in Section \ref{subsec_expl_descr_equiv_partial} for a finite dimensional $M^\vee$-module $V$, let $\mathcal{V}$ be the induced vector bundle $G^\vee \times^{Q^\vee} V \rightarrow \mathcal{Q}^{\vee}$. Slightly abusing notations, we will denote by the same symbol the pullback of $\mathcal{V}$ to $T^*\mathcal{Q}^\vee$.
Note that $\mathcal{V}$ has a natural $T^\vee \times \CC^\times$-equivariant structure (actually, even $G^\vee \times \CC^\times$-equivariant structure). 
Let us describe the algebra $H^*_{T^\vee \times \CC^\times}(T^*\mathcal{Q}^\vee)$. Recall that by the Thom isomorphism, the pullback induces the isomorphism $H^*_{T^\vee \times \CC^\times}(\mathcal{Q}^\vee) \iso H^*_{T^\vee \times \CC^\times}(T^*\mathcal{Q}^\vee)$.

It then follows from  \cref{cohomology of partial flag} that 
\begin{equation}\label{ident_cohom_parab_pol_alg}
H^*_{T^\vee \times \CC^\times}(T^*\mathcal{Q}^\vee) \iso  \CC[\mathfrak{h}^*,\hbar]^{W_M} \otimes_{\CC[\mathfrak{h}^*,\hbar]^W} \CC[\mathfrak{h}^*,\hbar], 
\end{equation}
sending $c_i(\mathcal{V})$ to $e_i(\la_1,\ldots,\la_n)$, where $\la_1,\ldots,\la_n \in \mathfrak{h}=(\mathfrak{h}^\vee)^*$ are $\mathfrak{h}^\vee$-weights (with multiplicities) that appear in $V^*$. Let 
\begin{equation*}
a\colon \CC[\mathfrak{h}^*,\hbar]^{W_M} \otimes \CC[\mathfrak{h}^*,\hbar] \twoheadrightarrow \CC[\mathfrak{h}^*,\hbar]^{W_M} \otimes_{\CC[\mathfrak{h}^*,\hbar]^W} \CC[\mathfrak{h}^*,\hbar]
\end{equation*}
be given by $s(\la,\hbar) \otimes g(\la,\hbar) \mapsto [s(\la,\frac{1}{2}\hbar) \otimes g(\la,\frac{1}{2}\hbar)]$.

\begin{lemma}\label{restriction_fixed_point_parabolic}
The restriction homomorphism 
\begin{equation}\label{restr_homom_formula}
\CC[\mathfrak{h}^*,\hbar]^{W_M} \otimes_{\CC[\mathfrak{h}^*,\hbar]^W} \CC[\mathfrak{h}^*,\hbar]=H^*_{T^\vee \times \CC^\times}(\mathcal{Q}^\vee) \rightarrow H^*_{T^{\vee} \times \CC^\times}((\mathcal{Q}^\vee)^{T^{\vee}})=\bigoplus_{[w] \in W/W_{M}}\CC[\mathfrak{h}^*,\hbar]
\end{equation}
is given by $s \otimes g \mapsto ((w \cdot s)g)_{[w] \in W/W_M}$.
\end{lemma}
\begin{proof}
The restriction homomorphism is clearly $\CC[\mathfrak{h}^*,\hbar]$-linear, so we only need to compute the image of $c_i(\mathcal{V}) \otimes 1$. Consider the natural morphism $\mathcal{B}^\vee \rightarrow \mathcal{Q}^\vee$,
the pull back of $\mathcal{V}$ to $\mathcal{B}^\vee$ decomposes into the direct sum of line bundles, so it is enough to describe the character by which $T^\vee$ acts on the fiber of a line bundle $\mathcal{O}_{\mathcal{B}^\vee}(-\la)=G^\vee \times_{B^{\vee}} \mathbb{C}_{-\la}$ ($\la\colon T^\vee \rightarrow \CC^\times$) restricted to a  $T^\vee$-fixed point $wB^\vee/B^\vee$.

Fiber of $\CO_{\mathcal{B}^\vee}(-\la)$ at $wB^\vee/B^\vee$ identifies with  $wB^\vee \times_{B^\vee} \CC_{-\la}$ so for $t \in T^\vee$ and $c \in \CC_{-\la}$ we have: 
\begin{equation*}
t \cdot [w,c]=[tw,c]=[w(w^{-1}(t)),c]=[w,w^{-1}(t^{-1})\cdot c]=\la(w^{-1}(t))[w,c].
\end{equation*}
So the $T^\vee$-character of $\CO_{\mathcal{B}^\vee}(-\la)|_{wB^\vee/B^\vee}$ , considered as an element of $(\mathfrak{h}^\vee)^*$ is given by $\xi \mapsto \langle w^{-1}(\xi),\la\rangle=\langle \xi,w(\la)\rangle$ so is equal to $w(\la)$. 

It follows that $c_i(\mathcal{V})$ maps to $e_i(w\la_1,\ldots,w\la_n)$, where $\la_i$ are weights of $V^*$. Recall also that the identification (\ref{ident_cohom_parab_pol_alg}) is given by $c_i(\mathcal{V}) \mapsto e_i(\la_1,\ldots,\la_n)$. So, we finally conclude that the map (\ref{restr_homom_formula}) sends $e_i(\la_1,\ldots,\la_n) \otimes 1$ to the collection of functions $(e_i(w\la_1,\ldots,w\la_n))_{[w] \in W/W_M}$. The claim follows.
\end{proof}

We have the embedding $\widetilde{S}(\mathfrak{p}^\vee,\mathfrak{q}^\vee) \subset T^*\mathcal{Q}^\vee$.
Recall that the $\CC^\times$-action on $\widetilde{S}(\mathfrak{p}^\vee,\mathfrak{q}^\vee)$ is induced by the cocharacter $\CC^\times \rightarrow T^\vee \times \CC^\times$, $t \mapsto (2\rho_{\mathfrak{l}}(t),t)$. 


It still makes sense to consider a restriction (and forgetting the equivariance) homomorphism 
\begin{equation*}
\on{res}\colon H^*_{T^\vee \times \CC^\times}(T^*\mathcal{Q}^\vee) \rightarrow H^*_{Z_{L^\vee} \times \CC^\times}(\widetilde{S}(\mathfrak{p}^\vee,\mathfrak{q}^\vee))\end{equation*} 
but it will {\emph{not}} be $\CC[\mathfrak{h}^*,\hbar]$-linear, the twist by $2\hbar\rho_{\mathfrak{l}}$ appears (compare with Warning \ref{warning_shift}).
Namely, consider the homomorphism $u\colon Z_{L^\vee} \times \mathbb{C}^\times \rightarrow T^\vee \times \mathbb{C}^\times$ given by $(g,t) \mapsto (2\rho_{\mathfrak{l}}(t)g,t)$. Let $\iota\colon \widetilde{S}(\mathfrak{p}^\vee,\mathfrak{q}^\vee) \hookrightarrow T^*\mathcal{Q}^\vee$ be the embedding. 
It follows from the definitions that  for $p \in Z_{L^\vee} \times \mathbb{C}^\times$ and $x \in \widetilde{S}(\mathfrak{p}^\vee,\mathfrak{q}^\vee) $ we have $\iota(p \cdot x)=a(p) \cdot \iota(x)$. It follows that the homomorphism $\on{res}$ is indeed well-defined.   
Let 
\begin{equation*}
j\colon \mathbb{C}[\mathfrak{h}^*,\hbar] \twoheadrightarrow  \mathbb{C}[\mathfrak{X}(\mathfrak{l}),\hbar]
\end{equation*}
be the map given by 
\begin{equation*}
f \mapsto ((\lambda,\hbar_0) \mapsto f(\lambda+2\hbar_0\rho_l,\hbar_0)).
\end{equation*}
The following lemma holds by the definitons.

\begin{lemma}\label{lemma_twisting_restriction}
For $x \in H^*_{T^\vee \times \mathbb{C}^\times}(T^*\mathcal{Q}^\vee)$     and $f \in \mathbb{C}[\mathfrak{h}^*,\hbar]=H^*_{T^\vee \times \mathbb{C}^*}(\on{pt})$, we have 
\begin{equation*}
\on{res}(fx)=j(f)\on{res}(x).
\end{equation*}
\end{lemma}
\begin{proof}
Passing to the Lie algebras level, the homomorphism $u$ induces the map 
\begin{equation}\label{map_derivative_of_u}
\mathfrak{z}(\mathfrak{l}^\vee) \oplus \mathbb{C} \rightarrow \mathfrak{h}^\vee \oplus \mathbb{C},\,(\la,\hbar_0) \mapsto (\la+2\hbar_0\rho_\mathfrak{l},\hbar_0).
\end{equation}
It remains to note that $j$ is nothing else but the pullback homomorphism induced by the morphism (\ref{map_derivative_of_u}). 
\end{proof}


\begin{cor}\label{cor_descr_comp_for_slodowy_side}
The composition
\begin{multline}\label{restr_fixed_parab_slodowy_formula}
\CC[\mathfrak{h}^*,\hbar]^{W_M} \otimes_{\CC[\mathfrak{h}^*,\hbar]^W} \CC[\mathfrak{h}^*,\hbar]=H^*_{T^\vee \times \CC^\times}(T^*\mathcal{Q}^\vee) \rightarrow H^*_{Z_{L^\vee} \times \CC^\times}(\widetilde{S}(\mathfrak{p}^\vee,\mathfrak{q}^\vee)) \rightarrow \\
\rightarrow H^*_{Z_{L^\vee} \times \CC^\times}(\widetilde{S}(\mathfrak{p}^\vee,\mathfrak{q}^\vee)^{Z_{L^\vee}})=\bigoplus_{[w] \in (W_L\backslash W/W_{M})^{\mathrm{free}}}\CC[\mathfrak{X}(\mathfrak{l}),\hbar]
\end{multline}
is given by $[s \otimes g] \mapsto ((\la,\hbar_0) \mapsto ((w \cdot s)g)(\la+2\hbar_0\rho_{\mathfrak{l}},\hbar_0))_{[w] \in (W_L\backslash W/W_{M})^{\mathrm{free}}}$. 
\end{cor}
\begin{proof}
Let us, first of all, compute the image of $[s \otimes 1]$ under (\ref{restr_fixed_parab_slodowy_formula}). It follows from \cref{descr_fixed_parabolic} that we have the natural embedding $\widetilde{S}(\mathfrak{p}^\vee,\mathfrak{q}^\vee)^{Z_{L^\vee}} \subset (\mathcal{Q}^\vee)^{T^\vee}$.
Recall also that the embedding $\widetilde{S}(\mathfrak{p}^\vee,\mathfrak{q}^\vee) \subset T^*\mathcal{Q}^\vee$ is $Z_{L^\vee} \times \mathbb{C}^\times$-equivariant w.r.t. the homomorphism $u$.
Then it from Lemma \ref{restriction_fixed_point_parabolic} that the image of $[s \otimes 1]$ is equal to 
\begin{equation*}
((\la,\hbar_0) \mapsto (w \cdot s)(\la+2\hbar_0\rho_{\mathfrak{l}},\hbar_0)_{[w] \in (W_L \backslash W/ W_M)^{\mathrm{free}}}.
\end{equation*}
Now, to determine the image of $[s \otimes g]$, it is enough to apply Lemma \ref{lemma_twisting_restriction}.
\end{proof}

\subsection{Main result}

\begin{prop}\label{main_prop_weak_hikita}
The following diagram is commutative (recall that the homomorphism $a$ sends $\hbar$ to $\frac{1}{2}\hbar$, and the identification of $\widetilde{S}(\mathfrak{p}^\vee,\mathfrak{q}^{\vee})^{Z_{L^\vee}}$ with $(W_M \backslash W/W_L)^{\mathrm{free}}$ is via ${\bf{i}}$):
\begin{equation*}
\xymatrix{ H^*_{T^\vee \times \CC^\times}(T^*\mathcal{Q}^\vee) \ar[r]^{\on{res}}  & H^*_{Z_{L^\vee} \times \CC^\times}(\widetilde{S}(\mathfrak{p}^\vee,\mathfrak{q}^\vee)) \ar[r]^{\on{res}} &  H^*_{Z_{L^\vee} \times \CC^\times}(\widetilde{S}(\mathfrak{p}^\vee,\mathfrak{q}^\vee)^{Z_{L^\vee}}) \ar[d]_{\simeq} \\
\CC[\mathfrak{h}^*,\hbar]^{W_M} \otimes \CC[\mathfrak{h}^*,\hbar] \ar[u]^{a} \ar[d]_{b} &  &  \bigoplus_{[w] \in (W_M \backslash W/W_L)^{\mathrm{free}}}\CC[\mathfrak{X}(\mathfrak{l}),\hbar] \\
\cC_\nu(\mathcal{A}_{\hbar,\mathfrak{X}(\mathfrak{l})}(\widetilde{S}(\mathfrak{q},\mathfrak{b}))) \ar[r]^{\cC_\nu(\Phi)} & \cC_\nu(\mathcal{A}_{\hbar,\mathfrak{X}(\mathfrak{l})}(\widetilde{S}(\mathfrak{q},\mathfrak{p}))) \ar[r] & \ar[u]^{\simeq}  \Gamma(\widetilde{S}_{\mathfrak{X}(\mathfrak{l})}(\mathfrak{q},\mathfrak{p})^{T_{e}},\cC_\nu(\mathcal{D}_{\hbar,\mathfrak{X}(\mathfrak{l})}(\widetilde{S}(\mathfrak{q},\mathfrak{p}))))
}
\end{equation*}
\end{prop}
\begin{proof}
Directly follows from Lemma \ref{lemma_hw_compute_parabolicW},  and Corollary \ref{cor_descr_comp_for_slodowy_side}.
\end{proof}

As an immediate corollary of Proposition \ref{main_prop_weak_hikita}, we finally conclude.
\begin{theorem}\label{main_th_weak_hikita}
We have an isomorphism of graded $\mathbb{C}[\mathfrak{h}^*,\hbar]^{W_M} \otimes \mathbb{C}[\mathfrak{h}^*,\hbar]$-algebras:
\begin{multline*}
\on{Im}(H^*_{T^\vee \times \mathbb{C}^\times}T^*\mathcal{Q}^\vee \rightarrow H^*_{Z_{L^\vee} \times \mathbb{C}^\times}(\widetilde{S}(\mathfrak{p}^\vee,\mathfrak{q}^\vee))) \simeq \\
\simeq \on{Im}(\mathbb{C}[\mathfrak{h}^*,\hbar]^{W_M} \otimes \mathbb{C}[\mathfrak{h}^*,\hbar] \rightarrow \cC_\nu(\mathcal{A}_{\hbar,\mathfrak{X}(\mathfrak{l})}(\widetilde{S}(\mathfrak{q},\mathfrak{p}))) \rightarrow \cC_\nu(\mathcal{D}_{\hbar,\mathfrak{X}(\mathfrak{l})}(\widetilde{S}(\mathfrak{q},\mathfrak{p}))))
\end{multline*}
\end{theorem}
\begin{proof}
The claim follows from Proposition \ref{main_prop_weak_hikita} and the fact that the restriction homomorphism $H^*_{Z_{L^\vee} \times \mathbb{C}^\times}(\widetilde{S}(\mathfrak{p}^\vee,\mathfrak{q}^\vee)) \rightarrow H^*_{Z_{L^\vee} \times \mathbb{C}^\times}(\widetilde{S}(\mathfrak{p}^\vee,\mathfrak{q}^\vee)^{Z_{L^\vee}})$ is injective (use Lemma \ref{lemma_equiv_cohom_resol_free}).
\end{proof}


\section{Corollaries of refined Hikita, conjectures, and further directions}\label{sec_cor_conj_further}
One of the goals of this section is to discuss the Hikita-Nakajima conjecture for nilpotent orbits and Slodowy slices at various equivariant parameters. In other words, we consider $\fq^\vee= \fb^\vee$ in Section 8. Recall that we discussed when the actual Hikita-Nakajima conjecture follows from our result after \cref{quantum Dmod}. A set of sufficient conditions includes that two morphisms $a,b$ in \cref{conj_most_general_weak_Hikita} are surjective and $\cC_\nu(\mathcal{A}_{\hbar,\mathfrak{h}_X}(X))$ is flat over $\CC[\mathfrak{h}_{X},\hbar]$. Moreover, they are also necessary conditions for \cref{Hikita_intro} in the case $\fq^\vee= \fb^\vee$ (\cite[Proposition 1.5]{hoang2}). Hence, in this section, we focus on these conditions.


It is noted that $b$ is always surjective if we assume that the dual variety is a nilpotent orbit (instead of a nontrivial cover) with normal closure. 
All the algebras involved in the diagram in \cref{main_th_weak_hikita} are algebras over $\CC[\fh_X,\hbar]$. 
We say that an equivariant parameter $(\lambda, t)\in \fh_X\oplus \CC\hbar$ is \emph{good} if $a|_{\lambda, t}$, $b|_{\lambda,t}$ are surjective 
and a condition (\ref{weakflat}) holds for $\cC_\nu(\mathcal{A}_{\lambda, t}(X))$, i.e. $\dim \cC_\nu(\mathcal{A}_{\lambda, t}(X))=\dim \cC_\nu(\mathcal{A}_{\lambda_0, t_0}(X))$ for a generic $(\lambda_0, t_0)$. If $(\lambda,t)$ is good, the Hikita-Nakajima conjecture holds at this parameter.

Let $\fg^\vee$ be a classical simple Lie algebra. When $\fg$ is of type A, the Hikita-Nakajima conjecture holds for the pair $\widetilde{S}(\fq,\fp), \widetilde{S}(\fp,\fq)$. Details are given in \cref{app_HN_type_A}. In this section, we assume that $\fg$ is of type B, C or D. We focus on the equivariant parameters $(\lambda, t)= (0,0)$ and $(\lambda,t)= (0,1)$. If $(0,0)$ is good for a pair $(X^\vee, X)$, then the classical Hikita conjecture holds. Moreover, the Hikita-Nakajima conjecture will also hold for this pair thanks to the graded Nakayama lemma (\cref{basic eqcoh}). The parameter $(0,1)$ is related to the canonical quantization (see \cite[Section 5]{LMBM}) and has representation theoretic meaning. We also discuss a combinatorial application of our refined Hikita at this parameter in \cref{combinatorics}. In \cref{sec_more_gen_settings} we discuss more general settings and give a heuristic explanation why the formulation of the Hikita conjecture needs further modifications. 

\subsection{Classical Hikita conjecture for Slodowy slices and nilpotent orbits}
Recall that $e^\vee\in \fg^\vee$ is a nilpotent element such that $e^\vee$ is regular in some Levi $\fl^\vee\subset \fg^\vee$. Then $X^\vee$ is $S(\chi^\vee)$ and $X$ is the induced variety $\Bind_{L}^{G}\{0\}$ (see Section 2.2 for the definition). Then $T_{X^\vee}$ is the torus $T_{e^\vee}$. 

First, we focus on the case when $\widetilde{D}(\OO^\vee)$ is a nilpotent orbit (not a cover) with a normal closure. These conditions automatically imply that the map $b$ in \cref{main identification} is surjective. Now, we consider the two other sufficient conditions for the Hikita-Nakajima conjecture.

In this setting, the map $a$ in \cref{main_th_weak_hikita} factors through the pullback in equivariant cohomology
\begin{equation*} \label{equivariant surjectivity} \tag{*}
    i^*_{T_{e^\vee}\times \CC^\times}\colon H^*_{T_{e^\vee}\times \CC^\times}(\cB^\vee)\rightarrow H^*_{T_{e^\vee}\times \CC^\times}(\spr).
\end{equation*}
Hence, the surjectivity of $a|_{(0,0)}$ is equivalent to the surjectivity of the pullback $i^{*}\colon H^*(\cB^\vee)\rightarrow H^*(\spr)$. This condition often appears in the literature; see, e.g., \cite[Theorem 1.2]{Kumar2012}, \cite{Carrell_2017}. Unfortunately, it is rarely satisfied when $\fg^\vee$ is not of type A. 

Let $\mathbf{p}^\vee$ be the partition of $e^\vee$. Let $C_{G^\vee}(e^\vee)$ be the centralizer of $e^\vee$ in $G^\vee$. Let $A^\vee = C_{G^\vee}(e^\vee)/C^{o}_{G^\vee}(e^\vee)$ be the corresponding component group. The image of $i^*$ is in $H^*(\spr)^{A^\vee}$. Hence, if $A^\vee$ acts on $\spr$ nontrivially, $i^*$ is not surjective.

We first explain when the action of $A^\vee$ on $H^*(\spr)$ is trivial. More details on the surjectivity of $i^*$ are discussed in \cref{cohomological surjectivity}. Next, we restrict ourselves to the cases where $A^\vee$ acts trivially on $H^*(\spr)$ and discuss the flatness of $\cC_\nu(\mathcal{A}_{\hbar,\mathfrak{h}_X}(X))$. We conclude with \cref{cases for classical Hikita} that list certain cases where the classical Hikita conjecture (hence, the Nakajima-Hikita conjecture) is true.

\subsubsection{Surjectivity of the map $a$}
Let $\mathbf{p}^\vee$ be the partition of $e^\vee$. The criteria for $\mathbf{p}^\vee$ in classical types are specified in \cref{prop:orbitstopartitions}. We have the following lemma.
\begin{lemma} \label{when A acts trivially}
    Consider $\fg^\vee$ of types B,C, and D. Let $(G^\vee)^{ad}$ be the corresponding adjoint group. Then 
     $A^\vee$ acts trivially on $H^*(\spr)$ if only if $(A^\vee)^{ad}$ is trivial. In particular, the partitions $\mathbf{p}^\vee$ are as follows.
    \begin{enumerate}
        \item $\mathbf{p}^\vee \in \mathcal{P}_{C}(2n)$, $\mathbf{p}^\vee$ has at most one even member, and if it has one, the multiplicity of this member is odd.
        \item $\mathbf{p}^\vee\in \mathcal{P}_{B}(2n+1)$, $\mathbf{p}^\vee$ has exactly one odd member.
        \item $\mathbf{p}^\vee\in \mathcal{P}_{D}(2n)$, $\mathbf{p}^\vee$ has at most two distinct odd members, and if it has two distinct odd members, their multiplicities are odd.
    \end{enumerate}
\end{lemma}
\begin{proof}
    This lemma is a direct consequence of the explicit Springer correspondence (see, e.g., \cite[Section 13.3]{carter1993finite}) and a result by Shoji. From \cite[Theorem 2.5]{Shoji1983}, a character of $A^\vee$ appears in $H^*(\spr)$ if and only if it appears in $H^{top}(\spr)$. We explain which characters of $A^\vee$ appears in $H^{top}(\spr)$ for $\fg^\vee$ of type C, other types are similar. 
    
    The following exposition follows from \cite[Section 13.3]{carter1993finite}. From $\mathbf{p}$, we obtain a certain symbol $S_\mathbf{p}$ that has two rows, each consisting of non-negative integers. This symbol $S_\mathbf{p}$ has the following property.
    \begin{itemize}
        \item[(*)] The difference between any two distinct entries in the same row is at least $2$.
    \end{itemize}
    
    The characters of $A^\vee$ in $H^{top}(\spr)$ correspond to the permutations of $S_\mathbf{p}$ that satisfy (*). Hence, for $A^\vee$ to act trivially on $H^{top}(\spr)$, we want $S_\mathbf{p}$ to have no other such permutation. The nonzero entries that only appear once in $S_\mathbf{p}$ form a set $S$. Write $S$ as the disjoint union of maximal intervals of the form $i,i+1,...,j$ for $i<j$. Because $S_\mathbf{p}$ has property (*), the odd entries of such an interval belong to one row of $S_\mathbf{p}$, and the even entries belong to the other row. As a consequence, the condition that $S_\mathbf{p}$ does not have nontrivial permutations that satisfy (*) is equivalent to the condition that $S$ itself is an interval of odd lengths or $S$ is empty. From \cite[Section 13.3]{carter1993finite}, these conditions are equivalent to $\mathbf{p}^\vee$ having precisely one even member of odd multiplicity or $\mathbf{p}^\vee$ having no even member.

\end{proof}
\begin{rmk} \label{trivial action does not mean surj} 
    Even when $A^\vee$ acts trivially on $H^*(\spr)$, one should not expect the pullback $i^*$ to be surjective in general. Details are given in \cref{cohomological surjectivity}.
\end{rmk}

\subsubsection{Flatness of $\cC_\nu(\mathcal{A}_{\hbar,\mathfrak{h}_X}(X))$}
In this subsection we restrict ourselves to the case where the partition $\mathbf{p}^\vee$ of $e^\vee$ satisfies the conditions in \cref{when A acts trivially}. The next lemma describes partitions $\mathbf{p}^\vee$ such that $\widetilde{\OO}=\widetilde{D}(G^\vee.e^\vee)$ is an orbit in $\fg^*$ with normal closure. Recall that we assume that $e^\vee$ is regular in $\fl^\vee$, and therefore $\widetilde{\OO}= \Bind_L^{G}(\{0\})$.

\begin{lemma} \label{best condition}
    For $\mathbf{p}^\vee$ in \cref{when A acts trivially}, the partitions of the following forms satisfy that $\widetilde{D}(\OO^\vee)$ is a nilpotent orbit with normal closure. 
    \begin{enumerate}
        \item $\mathbf{p}^\vee \in \mathcal{P}_{C}(2n)$, $\mathbf{p}^\vee$ is of the form $((2a)^{d_0}, 2b_1+1\geqslant...\geqslant 2b_k+1)$ where $d_0$ odd and $a\geqslant b_1$.
        \item $\mathbf{p}^\vee\in \mathcal{P}_{B}(2n+1)$, $\mathbf{p}^\vee$ is of the form $((2a+1)^{d_0}, 2b_1\geqslant ... \geqslant 2b_k)$ where $d_0$ odd and $a\leqslant b_k$.
        \item $\mathbf{p}^\vee\in \mathcal{P}_{D}(2n)$, we have two cases:
        \begin{enumerate}
            \item $\mathbf{p}^\vee$ is of the form $((2a+1)^{d_0}, 2b_1\geqslant...\geqslant 2b_k, 1^{d})$ where $d_0, d$ odd and $a+1 \geqslant b_1$.
            \item $\mathbf{p}^\vee$ consists only of even members and has at most two distinct members.
        \end{enumerate}        
    \end{enumerate}
    The partitions above may not be written in the standard order.
\end{lemma}
\begin{proof}
    We prove the lemma for the case $\fg^\vee= \fsp({2n})$, the other cases are similar. First, for the partition $\mathbf{p}^\vee$ in (1) of \cref{when A acts trivially}, the group $(A^\vee)^{ad}$ is trivial. It follows that the Lusztig quotient $\Bar{A}(\OO^\vee)$ is trivial. Recall from \cref{about cover}(v) that $\bar{A}(\OO^\vee)$ is isomorphic to the Galois group of the cover $\widetilde{D}(\OO^\vee)\rightarrow D(\OO^\vee)$. Therefore, $\widetilde{D}(\OO^\vee)= D(\OO^\vee)$. 

    Next, we show that the closure of $D(\OO^\vee)$ is normal if and only if $a\geqslant b_1$. Let $\mathbf{p}$ be the partition of $D(\OO^\vee)$. Write $e(\mathbf{p}^\vee)$ for $((2a)^{d_0}, 2b_1+1\geqslant...\geqslant 2b_k+1,1)$. Then $\mathbf{p}$ is $B(e(\mathbf{p}^\vee)^t)$ in which B denotes the B-collapse (for more details, see, e.g., \cite[Section 3.2]{LMBM}). 
    
    If $a\geqslant b_1$, the above process gives $p=$ $(2a+1, (2a)^{d_0-1}, 2b_1+1\geqslant...\geqslant 2b_k+1)^t$. Then all members of $\mathbf{p}$ are odd, so the closure of $D(\OO^\vee)$ is normal (see, e.g., \cite[Section 3]{Wong}, \cite{Kraft-Procesi}).

    Consider $a< b_1$. Let $2l_1$ be the multiplicity of $b_1$ in $\mathbf{p}^\vee$. If $a=0$, then $e(\mathbf{p}^\vee)^t$ is already a partition of type B. If $a>0$, then $e(\mathbf{p}^\vee)^t$ is of the form $(y_1,y_2,...,y_{2b_1+1})$ in which $y_1=y_2+ 1$ is even, $y_{2j+1}=y_{2j+2}$ odd for $j\leqslant a-1$, $y_{2a+1} =y_{2a+2}+ 1$ is even. Hence, the B-collapse of $e(\mathbf{p}^\vee)^t$ becomes $(y_1-1,y_2,y_3,y_4,$$...,y_{2a+1}, y_{2a+2}+1, y_{2a+3},$$...,y_{2b_1+1})$. In either case, the first $2l_1$ columns of the Young diagram of $\mathbf{p}$ has length $2b_1+1$, and the $2l+1$-th column has length $max(2b_{l+1}+1, 2a+1)< 2b_1+1$. Consequently, the closure of $D(\OO^\vee)$ is not normal (see, e.g. \cite[Theorem 3.3]{Wong}, \cite{Kraft-Procesi}). 
\end{proof}
\subsubsection{Summary}
We have seen in \cref{sp flatness} that the flatness condition of \cref{fixed points subsection} fails in the case $\fl$ is a Levi subalgebra of $\fsp({2n})$ that contains a nontrivial $\fsp$ factor. Therefore, it fails for $\mathbf{p}^\vee$ in (2) of \cref{best condition} when $a>0$. When $a=0$, a necessary condition for $i^*$ to be surjective is that $\mathbf{p}^\vee$ takes the form $(2^{2l},1^{2n+1-4l})$ (\cref{surjectivity B C D}). The situation is slightly less restrictive for the other two cases. In the following, we list the cases where we expect the parameter $(0,0)$ to be good.
\begin{conj} \label{cases for classical Hikita}
    Let $e^\vee$ be a regular nilpotent element in a Levi $\fl^\vee\subset \fg^\vee$. The Hikita-Nakajima conjecture holds for the pair $(X^\vee= S(\chi^\vee)$, $X= \Spec(\CC[\Bind_{L}^{G}\{0\}]))$ in the following cases. 
    \begin{enumerate}
        \item $G= Sp(2n)$, $L$ takes the form $\prod_{i=1}^{l} GL({2})\times \prod_{j=1}^{n-l} GL({1})$.
        \item $G= SO(2n+1)$, $L$ takes the form $SO(3)\times \prod_{i=1}^{l} GL({2})\times \prod_{j=1}^{n-l-1} GL({1})$, $SO(5)\times \prod GL(3)$, or $SO(2l+1)\times \prod_{j=1}^{n-l} GL({1})$.
        \item $G= SO(2n)$, $L$ takes the form $SO(4)\times \prod_{i=1}^{l} GL({2})\times \prod_{j=1}^{n-l-2} GL({1})$, $SO(2)\times \prod_{i=1}^{l} GL({2})\times \prod_{j=1}^{n-l-1} GL({1})$, $SO(2l)\times \prod_{j=1}^{n-l} GL({1})$, of $L= GL(n)$ if $n$ is even.
    \end{enumerate}
\end{conj}
We do not claim that these are all the cases where the Hikita-Nakajima conjecture holds for this type of pairs. Let us briefly explain the situations for the cases in \cref{cases for classical Hikita}. The surjectivity of the pullback map is proved for $L= SO(2l+1)\times \prod_{j=1}^{n-l} GL(1)$ in \cref{b=1}, other cases are predicted in \cref{surj conjecture}. The weak flatness condition (\ref{weakflat}) holds for the case $G= Sp(2n)$, and $L= \prod_{i=1}^{l} GL(2)\times \prod_{j=1}^{n-l} GL(1)$ (see \cref{sp is bad}). The weak flatness condition (\ref{weakflat}) for $G$ is of type B, D and other statements related to this conjecture are studied in \cite{hoang2}.
 
\subsection{Other parameters}\label{canonical parameter section}
This section discusses Hikita-Nakajima conjecture for nilpotent orbits and Slodowy slices at $(\lambda, t)\neq (0,0)$. We recall the following basic result in equivariant cohomology.

Consider a torus $T$ and a $T$-equivariant morphism $f: X\rightarrow Y$. Let $\ft$ be the Lie algebra of $T$, and $\lambda\in \ft$, $\lambda\neq 0$. The map $f^*_{T}: H^{*}_T(Y)\rightarrow H^{*}_T(X)$ is a morphism of $\CC[\ft]$-algebras. From the discussion in \cref{fixed point pullback}, the specialization of $f_{T}^*$ at $\lambda$ can be identified with $H^*(Y^{\lambda}) \rightarrow H^*(X^{\lambda})$ where $X^\la$ and $Y^\la$ are the fixed point subvarieties of the vector fields induced by $\lambda$.

Let $(\lambda, t)\in \fh_X\oplus \CC\hbar$ be a general parameter. If $t\neq 0$, we can replace $(\lambda, t)$ by $(\lambda/t, 1)$, see the discussion before \cref{comp_cart_alg_to_sh}. Therefore, we can assume that $t$ is either $0$ or $1$. 

We first consider $(\lambda, 0)$ for $\lambda\in \fh_X$, $\lambda\neq 0$. We explain how to reduce this case to the parameter $(0,0)$ for a certain (smaller) pair $(X'^\vee\subset (\fg'^\vee)^*, X')$ where $\fg'^\vee$ is of the same type as $\fg^\vee$. Our reduction argument uses the result that the Nakajima-Hikita conjecture holds in type A (\cref{app_HN_type_A}).

Write $(\cB^\vee)^\la$ and $\spr^\la$ for the fixed point varieties of $\cB^\vee$ and $\spr$ under the action of the vector field induced by $\lambda\in \fh_X= \ft_{e^\vee}$. At $(\lambda,0)$, the pullback map in (\ref{equivariant surjectivity}) becomes
$$i^{*}_{(\lambda, 0)}: H^*((\cB^\vee)^{\lambda})\rightarrow H^*((\spr)^{\lambda}).$$

Let $L^{\vee}_\lambda\subset G^\vee$ be the centralizer of $\lambda$ and write $\fl^{\vee}_\lambda$ for its Lie algebra. Since $\la\in \ft_{e^\vee}$, we have $\fl^{\vee}_\lambda \supset \fl^{\vee}$. Consider $e^\vee$ as a nilpotent element of $\fl^{\vee}_\lambda$. By \cref{fixed point Springer fiber}, the question about the surjectivity of $i^*_{(\lambda, 0)}$ for $e^\vee\in \fg^\vee$ reduces to the question about surjectivity of $i^*_{(0,0)}$ for $e^\vee\in \fl^{\vee}_\lambda$. 

Next, the Levi subalgebra $\fl^{\vee}_\lambda\subset \fg^\vee$ has a decomposition of the form $\fg'^{\vee}\times \prod_{i=1}^{m} \fgl_{\lambda_i}$ where $\fg'^{\vee}$ is of the same type as $\fg^\vee$. Correspondingly, we can decompose $e^\vee$ as $(e'^{\vee},e^{\vee}_1,...,e^{\vee}_m)$ for $e'^{\vee}\subset \fg'^{\vee}$ and $e^{\vee}_i\subset \fgl_{\lambda_i}$. Note that $e'^{\vee}$ is regular in $\fl'^\vee = \fl^\vee \cap \fg'^\vee$. Let $\cB_i$ be the flag variety for $\fg\fl_i$. Note that all the pullback maps $H^*(\cB_i)\rightarrow H^*(\cB_{e^{\vee}_i})$ are surjective. Therefore, we have the following proposition.

\begin{prop} \label{non gl factor coh}
    The specialized map $i^{*}_{\lambda, 0}$ is surjective if and only if $(i')^*\colon H^*((\cB')^\vee)\rightarrow H^*(\cB_{(e^\vee)'})$ is surjective. 
\end{prop}
A similar phenomenon appears on the side of $X\subset \fg^*$. Write $\fl_\lambda$ and $\fg'$ for the Langlands dual of $\fl^{\vee}_\lambda$ and $\fg'^\vee$ respectively. Let $L_\lambda$ and $G'\subset L_\lambda$ be the corresponding subgroups. In the setting of Section 9.1, $X$ is a nilpotent orbit with normal closure in $\fg^*$. Recall that we have the deformation $\pi: X_{\fh_X}\rightarrow \fh_X\simeq \fh_{e^{\vee}}$. The algebra $\CC[\pi^{-1}(\lambda)]$ is isomorphic to the algebra of functions on a coadjoint orbit $\OO'_{\lambda}$ of $\fg^*$ (see proof of \cite[Proposition 4.5]{Losev4}). We describe the orbit $\OO'_{\lambda}$.  Under an identification $\ft_{e^\vee} =\fh_X\subset \fh^*$, we obtain an element $c_\lambda\in \fh^*$ that corresponds to $\lambda\in \ft_{e^\vee}$. Then the orbit $\OO'_{\lambda}$ can be realized as $G\times^{L_\lambda} (c_\lambda+ \Bind_L^{L_\lambda} \{0\})$. 

View the orbit $\OO_{L_\lambda}:= \Bind_L^{L_\lambda} \{0\}$ as $\Tilde{D}(L^\vee_\lambda. e^\vee)$. Write $\fl_\lambda$ as $\fg'\times \prod_{i=1}^{m} \fgl_{\lambda_i}$. This expression is compatible with the refined BVLS duality, so we can decompose the orbit $\OO_{L_\lambda}$ as $\Tilde{D}(e'^{\vee})\times \Tilde{D}(e^{\vee}_1)\times \ldots\times \Tilde{D}(e^{\vee}_m)$. Because orbit closures are normal in $\fgl_{\lambda_i}$, the normality of $\overline{\OO}_{L_\lambda}\subset \fl_\lambda^*$ depends only on the normality of $\overline{\Tilde{D}(e'^{\vee})}$ in $(\fg')^*$. 

Assume that $e^\vee$ has partition $\fp^\vee$ as in \cref{best condition}. 
Then the partition $\fp'^\vee$ of $e'^\vee$ is a subpartition of $\fp^\vee$, hence satisfies the condition in \cref{best condition}.
Hence, the closure of the orbit $\widetilde{D}(e'^{\vee})$ in $(\fg')^*$ is normal, and thus $\overline{\OO}_{L_\lambda}$ is normal.
This implies the normality of $\OO'_{\lambda}$ (see \cite[Section 1]{Hesselink1979}). Thus, the weak flatness condition (\ref{weakflat}) at $(\lambda, 0)$ becomes the weak flatness condition at $(0,0)$ for $e\in L_\lambda$. In conclusion, we have the following proposition.

\begin{prop} \label{reduce_to_0}
    The Hikita-Nakajima conjecture is true for the pair $(X^\vee= S({\chi^\vee}), X= \Bind_{L}^{G}(\{0\}))$ at a parameter $(\lambda, 0)$ if and only if it is true for the pair $(X^\vee= S{(\chi'^\vee)}, X= \Bind_{L'}^{G'}(\{0\}))$ at $(0,0)$.
\end{prop}
In particular, consider $\lambda\in \fh_{e^{\vee}}$ such that the projection of $\lambda$ to each factor $\fgl$ of $\fl^\vee$ is nonzero. Then $e'^\vee$ is regular in $\fg'^{\vee}$, and we obtain the following corollary. 
\begin{cor}
    For such $\lambda$, the Hikita-Nakajima conjecture for $(X^\vee= S_{\chi^\vee}, X= \Bind_{L}^{G}(\{0\}))$ holds at $(\lambda, 0)$. 
\end{cor}
\begin{rmk}
    For $\lambda\neq 0$, the validity of Hikita-Nakajima conjecture at $(\lambda, 0)$ does not imply its validity at $(\lambda, t)$, $t\neq 0$.
\end{rmk}

Next, we consider the Hikita-Nakajima conjecture at $(\lambda, t)= (0,1)\in \fh_X\oplus \CC\hbar$. At this parameter, the pullback map in (\ref{equivariant surjectivity}) becomes 
$$i^{*}_{(0,1)}: H^*((\cB^\vee)^{\CC^\times})\rightarrow H^*((\spr)^{\CC^\times}).$$
Regarding this map, we have the following lemma and a more general conjecture.
\begin{lemma} \label{sujective Ctimes}
    The triviality of the action of the component group is sufficient for the surjectivity of this pullback map. In other words, $i^{*}_{(0,1)}$ is surjective for $\mathbf{p}^\vee$ described in \cref{when A acts trivially}.
\end{lemma}
\begin{proof}[Sketch of proof] The proof is technical and we are not going to use the result later on, so we give a sketch. Assume $\fg$ is of type C and $\mathbf{p}^\vee$ takes the form $((2a)^{d_0}, 2b_1+1\geqslant...\geqslant 2b_k+1)$ where $d_0$ odd if $a\neq 0$. Each connected component of $(\spr)^{\CC^\times}$ is then isomorphic to the product $(\cB_{e_0})^{\CC^\times}\times (\cB_{e_1})^{\CC^\times}$ where $e_0$ and $e_1$ have partitions $((2a)^{d_0})$ and $(2b_1+1\geqslant...\geqslant 2b_k+1)$, respectively (see \cref{fixed point Springer fiber}). Hence, it is enough to work with the case where the members of $\mathbf{p}^\vee$ have the same parity.

\textit{Step 1}. Consider the case where $\mathbf{p}^\vee$ has all parts equal. Let $k$ be the number of parts in $\mathbf{p}^\vee$. Each connected component of $(\spr)^{\CC^\times}$ is a tower of projective bundles over an orthogonal (or symplectic) partial flag variety of $\CC^k$ (\cite[Section 8.3]{hoang2024}). A symplectic partial flag variety is again a tower of projective bundles over a point. An orthogonal partial flag variety of $\CC^k$ can be viewed as a tower of odd quadric over a point. The result then follows from the fact that the cohomology rings (with complex coefficients) of odd quadrics and projective spaces are generated by their second degree components (\cite{EG94}).

\textit{Step 2}. Consider the case where $\mathbf{p}^\vee$ consists of only odd members. Let $\mathbf{p}^\vee$ be $(2b_1+1,...,2b_k+1)$ for $b_1\geqslant b_2\geqslant...\geqslant b_k\geqslant 0$. We prove the surjectivity of $i^{*}_{(0,1)}$ by induction on the number of distinct parts of $\mathbf{p}^\vee$. The base case that $\mathbf{p}^\vee$ has equal parts was proved in the previous step. Similarly to \cref{inductive non-surjectivity}, the surjectivity of $i^{*}_{(0,1)}$ for $\mathbf{p}^\vee$ would follow from the surjectivity of $i^{*}_{(0,1)}$ for $(\mathbf{p}^\vee)^i$ where $(\mathbf{p}^\vee)^i$ is obtained from $\mathbf{p}^\vee$ by replacing a pair $(2b_i+1= 2b_{i+1}+1)$ by $(2b_i+3= 2b_{i+1}+3)$ for some $1\leqslant i\leqslant k-1$. After a finite number of such steps, we will arrive at the case of a partition with fewer distinct parts, and the surjectivity follows from the induction hypothesis.
    
\end{proof}
    
 Moreover, we conjecture the following.
\begin{conj} \label{surject h=1}
   The image of $i^{*}_{\CC^\times}$ is the $A^\vee$-invariant part of $H^*((\spr)^{\CC^\times})$.
\end{conj}
We make some comments about the parameters $(\lambda, 1)$ for general $\lambda$. First, consider the vector field induced by $(\la,1)$ and the fixed point varieties $(\cB^\vee)^{(\la,1)}$, $(\spr)^{(\la,1)}$. At $(\lambda,1)$, the pullback map in (\ref{equivariant surjectivity}) becomes
$$i^{*}_{(\lambda, 1)}: H^*((\cB^\vee)^{(\la,1)})\rightarrow H^*((\spr)^{(\la,1)}).$$
Recall that we write $T_{\lambda}^\vee$ for the torus with Lie algebra $\CC(\lambda,0)$. Outside of a $\ZZ$-lattice $\fh_{{e}^\vee}$, the variety $H^*((\cB^\vee)^{(\la,1)})$ is the simultaneous fixed point of $T_{\lambda}^\vee$ and $\CC^\times$. In this case, we have a result similar to \cref{reduce_to_0}.   

At the parameters $\lambda\in \fh_{{e}^\vee}$ where $(\lambda, 1)$ is not generic in $\CC\lambda\oplus \CC\hbar$, the behavior of the map $i^{*}_{(\lambda, 1)}$ may exhibit interesting features. For example, $i^{*}_{(\lambda, 1)}$ may not be surjective even when both $i^{*}_{(0, 1)}$ and $i^{*}_{(\lambda, 0)}$ are surjective. As an application, in \cref{cohomological surjectivity}, we use a certain parameter $(\lambda, 1)$ to prove the failure of the surjectivity of $i^{*}_{(0,0)}$ (see \cref{3-row}). 

\subsection{A combinatorial application: Springer fibers and a left cell}\label{combinatorics}
In this section, we explain a combinatorial application of our refined Hikita conjecture. Our aim is to produce a bijection between two sets of somewhat different natures. The first set consists of the element in a left cell in a (integral) Weyl group. This set appears as the parametrization of simple modules in the BGG category $\cO$ that are annihilated by a certain primitive ideal. The second set describes the $A^\vee$-orbits of irreducible components of certain fixed point loci of the Springer fibers.  

We specialize all the algebras in the diagram of \cref{main_prop_weak_hikita} to $(\lambda, t)= (\lambda,1)\in \fh_X\oplus \CC\hbar$ and obtain two maps of algebras.

    $$\psi_\lambda: \CC[\fh^*]\rightarrow H_{T_{e^\vee}\times \CC^\times}^*(\cB^{\vee})_{(\lambda, 1)}\rightarrow H_{T_{e^\vee}\times \CC^\times}^*(\spr)_{(\lambda, 1)}\rightarrow \oplus_{[w]\in W/W_{L}} ,$$
    $$\Psi_\lambda: \CC[\fh^\vee]\rightarrow \cC_\nu(\cA_{1,\lambda}(T^*\cB^\vee))\xrightarrow{b_{(\lambda,1)}} \cC_\nu(\cA_{1,\lambda}(T^*\cP))\rightarrow \oplus_{[w]\in W/W_{L}}.$$
From \cref{main_prop_weak_hikita} we can identify the sources, targets, and the images of these two maps. Therefore, we obtain an isomorphism between the two algebras $B_{1,\lambda}= \on{Im}( \Psi_\lambda)$ and $B_{2,\lambda}= \on{Im}(\psi_\lambda)$.

The algebra $B_{1,\lambda}$ is determined by the image of $\fh\subset \CC[\fh^*]$. For each $[w]\in W/W_{L}$, let $\CC_{[w]}$ denote the corresponding summand in $\oplus_{[w]\in W/W_{L}} \CC)$. By the proof of \cref{lemma_hw_compute_parabolicW}, the map $\fh\rightarrow \CC_{[w]}$ records the highest weight of a simple $\cC_\nu(\cA_{1,\lambda}(T^*\cP))$-modules. View $\Spec(B_{1,\lambda})$ as a subscheme of $\fh^*$, then $\Spec(B_{1,\lambda})$ is supported on a finite set $\Supp_\lambda$.


Next, we relate $\Supp_0$ to a fixed point variety of the Springer fiber $\spr$. Let $Y_{(0,1)}$ be the set of $A^\vee$-orbits of the connected components of $(\spr)^{\CC^\times}$. 
\begin{prop} \label{support and component}
    The isomorphism $B_{2,0}\simeq B_{1,0}$ from \cref{main_prop_weak_hikita} induces a bijection between the two sets $\Supp_0$ and $Y_{(0,1)}$. 
\end{prop}
\begin{proof}
    First, we view $\Supp_0$ as the set-theoretic support of $\Spec(B_{2,0})$. Here, $\Spec(B_{2,0})$ is naturally a subscheme of $\fh^\vee$. Similarly to Section \ref{canonical parameter section}, we have 
$$B_{2,0}= \on{Im} \left (\CC[\fh^\vee]\rightarrow H^*((\cB^\vee)^{\CC^\times})\xrightarrow{i^*_{0, 1}} H^*((\spr)^{ \CC^\times})\xrightarrow{j^*_{0}} \oplus_{[w]\in W/W_{L}} \CC\right ).$$

    Let $X_{\alpha}$ be a connected component of $\spr^{\CC^\times}$. Let $\cF_\alpha$ be the corresponding connected component of $(\cB^\vee)^{\CC^\times}$ that contains $X_\alpha$. From \cref{fixed point flag variety}, we know that $\Spec(H^*(\cF_\alpha))$ is set-theoretically supported on a single point and for two components $\cF_\alpha\neq \cF_\beta$ of $(\cB^\vee)^{\CC^\times}$, the corresponding supports are different. 
    
    From \cite[Section 3.7]{DLP} and \cite[Section 2.1]{DLP}, we deduce that $\cF_\alpha= \cF_\beta$ if and only if $X_\alpha$ and $X_\beta$ are in the same $A^\vee$-orbit. Write $\Con((\cB^\vee)^{\CC^\times}))$ for the set of connected components of $(\cB^\vee)^{\CC^\times}$. We have an identification
    \begin{equation*} \label{components to component} \tag{**}
        Y_{(0,1)}= \{\cF^{\alpha}\in \Con((\cB^\vee)^{\CC^\times}), i^*_{0,1}(H^*(\cF^{\alpha}))\neq 0\}.
    \end{equation*}

    
   For any nonzero quotient $B_\alpha$ of $H^*(\cF_\alpha)$, the scheme $\Spec(B_\alpha)$ is supported at a single point. Hence, $\Supp_0$, the support of $B_{2,0}$ is identified with $\{\cF\in \Con((\cB^\vee)^{\CC^\times}), j_0^{*}\circ i^*_{0,1}(H^*(\cF))\neq 0\}$. Now, recall that the map $H^*((\spr)^{\CC^\times})\xrightarrow{j^*_{0}} \bigoplus_{[w]\in W/W_{L}} \CC)$ is the pullback $H^*((\spr)^{\CC^\times})\rightarrow H^*((\spr)^{T_{e^\vee}})$. Each connected component $X_\alpha$ of $(\spr)^{\CC^\times}$ is smooth and projective. Hence, $X_\alpha$ contains at least one point fixed by $T_e{^{\vee}}$. Therefore, $j_0^{*}\circ i^*_{0,1}(H^*(\cF))\neq 0$ is equivalent to $i^*_{0,1}(H^*(\cF))\neq 0$. Combining with (\ref{components to component}), we obtain the desired bijection.    
\end{proof}

\begin{rmk}
    For a general $\lambda\in \fh_X$, let $T^\vee_{\lambda,1}\subset T_{e^\vee}\times \CC^\times$ be the torus with Lie algebra $\CC(\lambda, 1)$. The action of $A^\vee$ and the action of $T^\vee_{\lambda,1}$ do not commute. Hence, we need to define the set $Y_{(\lambda, 1)}$ in another way.

    Consider an equivalence relation on the set $\Con((\spr)^{T^\vee_{\lambda,1}})$ defined as follows. We say $X_\alpha$ and $X_\beta$ belong to the same class if $X_\alpha$ and $X_\beta$ both lie in a connected component $\cF_\gamma$ of $\Con((\cB^\vee)^{T^\vee_{\lambda,1}})$. Let $Y_{(\lambda, 1)}$ be the set of equivalence class in $\Con((\spr)^{T^\vee_{\lambda,1}})$ with respect to this relation. Then the proof \cref{support and component} works similarly, and we have a bijection between $Y_{(\lambda,1)}$ and $\Supp_\lambda$.

\end{rmk}

\begin{rmk}
    Note that we have obtained the bijections between $\Supp_\lambda$ and $Y_{(\lambda,1)}$ without computing either set. This is a notable strength of our deformation approach since only explicit computation is needed at the generic $(\lambda,t)$ for the proof of \cref{main_prop_weak_hikita}. On the other hand, these bijections at certain special parameters $(\lambda, 1)$ have interesting nontrivial combinatorial interpretations. In the rest of this section, we discuss a corollary of \cref{support and component}.
\end{rmk}

Assume $\widetilde{D}(\OO^\vee)$ is an orbit in $\fg$ with normal closure. Then the map $\cC_\nu(\cA_{1,\lambda}(T^*\cB^\vee))\xrightarrow{b_{(\lambda,1)}} \cC_\nu(\cA_{1,\lambda}(T^*\cP))$ is surjective. Hence, the set $\Supp_{\lambda}$ is the set of the highest weights of the simple modules of $\cC_\nu(\cA_{1,\lambda}(T^*\cP))$ in the corresponding category $\cO_\nu$. 

Write $I_{\fp,\lambda}$ for the kernel of the map $b_{(\lambda,1)}$ above. From the proof of \cref{lemma_hw_compute_parabolicW}, the simples of $\cC_\nu(\cA_{1,\lambda}(T^*\cP))$ are in bijection with the simples annihilated by $I_{\fp,\lambda}$ of the BGG category $\cO$ for $\cU_\lambda(\fg)$. These simple modules are parameterized by a certain subset $c_{P,\lambda}\subset W$. In particular, each element $w\in c_{P,\lambda}$ corresponds to the simple module $L(w.\frac{1}{2}(\rho_\fl-\rho))$ of $\cU(\fg)$. Here, the action of $w$ is the $\rho$-shifted action. For $\lambda= 0$, we have the following corollary.

\begin{cor} \label{combi bijection}
    When $\widetilde{D}(\OO^\vee)$ is an orbit in $\fg$ with normal closure, we have a bijection $Y_{(0, 1)}\simeq c_{P,0}$.
\end{cor}

\begin{rmk}
    

    Assume $e^\vee$ is even. The bijection in \cref{combi bijection} can be made more precise as follows. The set $Y_{(0,1)}$ of $A^\vee$-orbits of connected components of $(\spr)^{\CC^\times}$ admits a parametrization by the cosets $W_{\fp^\vee} w$ such that $P^{\vee}wB^\vee/B^\vee\cap \spr$ nonempty. If we pick out the longest elements in those cosets, they are precisely the elements that belong to $c_{P,0}$. 

    Recall that \cref{combi bijection} requires the hypothesis that $\widetilde{\OO}= \widetilde{D}(\OO^\vee)$ is the orbit $\OO= D(\OO^\vee)$ of $\fg$. If we drop this assumption, the map $\cC_\nu(\cA_{1,0}(T^*\cB^\vee))\xrightarrow{b_{(0,1)}} \cC_\nu(\cA_{1,0}(T^*\cP))$ may not be surjective. Nevertheless, the map  $\cC_\nu(\cA_{1,0}(T^*\cB^\vee))\xrightarrow{b_{(0,1)}} \cC_\nu(\cA_{1,0}(T^*\cP))\rightarrow \oplus_{[w]\in W/W_{L}}\CC$ still records the highest weights of certain simple $(\cA_{1,0}(T^*\cP))^\Pi$-modules where $\Pi=\on{Aut}_\OO(\widetilde{\OO})$ (see \cref{surj_Phi_hbar_nonzero}). Therefore, we have an injection $Y_{(0,1)}= \Supp_0\hookrightarrow c_{P,0}$. 
    
    An explicit combinatorial construction of this injection (as well as the generalization of this story to the parabolic setting) will be presented in another work.
\end{rmk}

Let us finish this section by mentioning an interesting relation of the combinatorial bijection obtained in \ref{support and component} and the results and predictions of a recent preprint \cite{Peng_Xie_Yan_Mirror} (see also \cite{yi_ya}). In loc. cit, the authors propose a mirror symmetry between  vertex operator algebras arising (via the $4d/2d$ correspondence) from some $4d$ $\mathcal{N}=2$ superconformal field theories and the Coulomb branches of  $3d$ theories obtained by compactifying the $4d$ theories above on the circle. 
One of the main conjectures of the paper (see point 1 of the introduction in loc. cit.) claims that there should be a bijection between irreducible objects in category $\mathcal{O}$ for certain affine vertex algebras and $\mathbb{C}^\times$-fixed points of certain affine Springer fibers. This looks similar to the bijection \ref{support and component}. One might hope that
 there is an analog of the equivariant Hikita-Nakajima conjecture in this setting that would imply and give a conceptual explanation of the above conjecture.


\subsection{More general settings}\label{sec_more_gen_settings} 
\subsubsection{Distinguished case} So far, we have been considering the Nakajima-Hikita conjecture for the pair $X^\vee= S_{e^\vee}$ and $X=$ $\Spec(\CC[\widetilde{\OO}])$ for $e^\vee$ regular in some Levi and $\widetilde{\OO}=\widetilde{D}(\OO^\vee)$. In this section, we propose a conjecture for the ``opposite" setting. In particular, we assume that $e^\vee$ is a distinguished nilpotent element of $\fg^\vee$. In this setting, the torus $T_{X^\vee}$ is trivial. Recall that by \cref{about cover} (v) $\Aut_\OO (\widetilde{\OO})\simeq \bar{A}(\OO^\vee)$. Thus, the following two important conditions are relaxed.
\begin{enumerate}
    \item $X=\on{Spec}\CC[\widetilde{D}(\OO^\vee)]$ no longer admits a symplectic resolution.
    \item The Lusztig quotient $\Bar{A}(\OO^\vee)$ is rarely nontrivial. In other words, we no longer assume that $\widetilde{D}(\OO^\vee)= D(\OO^\vee)$.
\end{enumerate}
Recall that we write $A^\vee$ for the component group of $C_{G^\vee}(e^\vee)$. Let $A'$ be the kernel of the map $A^\vee\twoheadrightarrow \Bar{A}(\OO^\vee)$. We expect the following:
    \begin{conj}\label{target distinguished}
        Assume that $e^\vee\in \fg^\vee$ is a distinguished nilpotent element, and let $X=\Spec(\CC[\widetilde{D}(\OO^\vee)])$. Then there is an isomorphism of $\CC[\hbar^{\pm 1}]$-algebras
        
    \begin{equation*} 
        H^{*}_{\CC^\times}(\spr)^{A'} \otimes_{\mathbb{C}[\hbar]} \mathbb{C}[\hbar^{\pm 1}] \simeq \cC_\nu(\cA_\hbar(X)) \otimes_{\mathbb{C}[\hbar]} \mathbb{C}[\hbar^{\pm 1}] 
    \end{equation*}
that is compatible with the action of $\Bar{A}(\OO^\vee)$ on both sides.  
    \end{conj}

    \begin{warning}
    We only propose this statement for distinguished $e^\vee$. For example, the statement in \cref{target distinguished} is wrong when $e^\vee$ is regular in some Levi $\fl^\vee$, $\Bar{A}(\OO^\vee)$ trivial, and $A^\vee$ non-trivial. In this case, the right side of \cref{target distinguished} is a module over $\CC[\fh^\pm]$ of generic rank $|W/W_L|$ while the rank of the left side is strictly smaller.
\end{warning}

Note that Conjecture \ref{target distinguished} is equivalent to the existence of $\Bar{A}(\OO^\vee)$-equivariant isomorphism of algebras: 
\begin{equation}\label{identi_conj_dist_hbar_1}
    H^*(\mathcal{B}_{e^\vee}^{\mathbb{C}^\times})^{A'} \simeq \cC_\nu(\mathcal{A}_\hbar(X)).
    \end{equation}
    Passing to $\Bar{A}(\OO^\vee)$-invariants, we then expect the identification: 
    \begin{equation}\label{identi_Galois_fixed}
    H^*(\mathcal{B}_{e^\vee}^{\mathbb{C}^\times})^{A^\vee} \simeq \cC_\nu(\mathcal{A}(X))^{\Bar{A}(\OO^\vee)}.
    \end{equation}
    Actually, we expect that the identification (\ref{identi_Galois_fixed}) holds for {\emph{arbitrary}} nilpotent $e^\vee$, where on the RHS we consider the $\Bar{A}(\OO^\vee)$-invariants of the Cartan subquotient of the quantization of $X$ with period $0$. 
    Moreover, we expect that both of the algebras should be quotients of $\mathbb{C}[\mathfrak{h}^*]$, so the identification (\ref{identi_Galois_fixed}) would reduce to the equality kernels.
    It actually follows from \cite[Corollary 6.5.6]{LMBM} together with \cite[Theorem 8.5.1]{LMBM} and \cite[Theorem 5.0.1]{MBMat} that the natural homomorphism $\mathbb{C}[\mathfrak{h}^*] \twoheadrightarrow \cC_\nu(\mathcal{A}(X)^{\Bar{A}(\OO^\vee)})$ is {\emph{surjective}}. Moreover, we have the following lemma.
    \begin{lemma} \label{invariantC to Cinvariant}
        There is a canonical surjective homomorphism $\cC_\nu(\mathcal{A}(X)^{\Bar{A}(\OO^\vee)})\twoheadrightarrow \cC_\nu(\mathcal{A}(X))^{\Bar{A}(\OO^\vee)}$.
    \end{lemma}
    \begin{proof}
    It follows from the definitions that the action of $\Bar{A}(\OO^\vee)$ on $\mathcal{A}(X)$ commutes with the $T_X$-action. Moreover, the action of $\Bar{A}(\OO^\vee)$ on $\mathcal{A}(X)$  is locally finite. 
    We conclude that: 
    \begin{equation}\label{fixed_cartan}
    \cC_\nu(\mathcal{A}(X))^{\Bar{A}(\OO^\vee)} = \Big(\mathcal{A}(X)_0/\sum_{i>0}\mathcal{A}(X)_{-i}\mathcal{A}(X)_{i}\Big)^{\Bar{A}(\OO^\vee)} = \mathcal{A}(X)_0^{\Bar{A}(\OO^\vee)}/\Big(\sum_{i>0}\mathcal{A}(X)_{-i}\mathcal{A}(X)_{i}\Big)^{\Bar{A}(\OO^\vee)} 
    \end{equation}
    \begin{equation}\label{cartan_of_fixed}
    \cC_\nu(\mathcal{A}(X)^{\Bar{A}(\OO^\vee)}) = \mathcal{A}(X)^{\Bar{A}(\OO^\vee)}_0/\sum_{i>0}\mathcal{A}(X)^{\Bar{A}(\OO^\vee)}_{-i}\mathcal{A}(X)^{\Bar{A}(\OO^\vee)}_{i}.
    \end{equation}
    Clearly, (\ref{cartan_of_fixed}) surjects onto (\ref{fixed_cartan}).
    \end{proof}
    
    As a consequence, we always have a surjection $\mathbb{C}[\mathfrak{h}^*] \twoheadrightarrow \cC_\nu(\mathcal{A}(X))^{\Bar{A}(\OO^\vee)}$. On the cohomological side, we also expect that the restriction homomorphism $H^*(\mathcal{B}^{\mathbb{C}^\times}) \rightarrow H^*(\spr^{\mathbb{C}^\times})^{A^\vee}$ is always surjective (\cref{surject h=1}). An evidence is given by \cref{sujective Ctimes}. 

\begin{warning}\label{warning_nonsurj_A_inv} 
It is crucial that we are passing to the $\mathbb{C}^\times$-fixed points. It is in general {\emph{not}} true that the restriction homomorphism $H^*(\mathcal{B}) \rightarrow H^*(\mathcal{B}_{e^\vee})^{A^\vee}$ is surjective. 

This phenomenon has already appeared in the case where $A^\vee$ acts trivially in $H^*(\mathcal{B}_{e^\vee})$ as follows. In \cref{best condition}, we have described the condition for $A^\vee$ to act trivially on $H^*(\mathcal{B}_{e^\vee})$ and $\Tilde{D}(\OO^\vee)$ is an orbit with normal closure. On the other hand, the necessary condition for the pullback $H^*(\mathcal{B}) \rightarrow H^*(\mathcal{B}_{e^\vee})$ to be surjective is much stricter, see \cref{surjectivity} and \cref{surjectivity B C D}. 


\end{warning}

\begin{rmk}\label{compare simples}
    \cref{target distinguished}, or equivalently the $\bar{A}(\OO^\vee)$-isomorphism (\ref{identi_conj_dist_hbar_1}), implies a bijection between the two $\bar{A}(\OO^\vee)$-sets $\Con(\spr^{\CC^{\times}})/A'$ and the set of irreducible modules of $\cA_\hbar(X)$. When $e^\vee$ is an even nilpotent element, we can identify $\Con(\spr^{\CC^{\times}})/A'$ with $\text{Irr}(\spr)/A'$ where $\text{Irr}(\spr)$ is the set of irreducible components of the Springer fiber $\spr$. This set $\text{Irr}(\spr)/A'$ is conjectured to be the $\bar{A}(\OO^\vee)$-set that parametrizes finite dimensional simple modules with integral central character of $\mathcal{W}_{e^\vee}$, the W-algebra attached to $e^\vee$ (\cite[Conjecture 1.2]{hoang2024actioncomponentgroupsirreducible}). In loc. cit., this conjecture has been checked for exceptional $\fg^\vee$. Therefore, it is desirable to establish a connection between the simple modules of $\cA_\hbar(X)$ and the modules of $\mathcal{W}_{e^\vee}$. This is a topic of ongoing work by the authors. 
\end{rmk}

We now return to the case where $e^\vee$ is distinguished. Note that in Conjecture \ref{target distinguished} we tensor both sides by $\mathbb{C}[\hbar^{\pm 1}]$. One still may expect that the isomorphism in this conjecture extends to $\hbar=0$. That would imply the identification 
\begin{equation}\label{wrong_identification}
H^*(\mathcal{B}_{e^\vee})^{A^\vee} \simeq \cC_\nu(\mathcal{A}_\hbar(X))^{\Bar{A}(\OO^\vee)}/(\hbar).
\end{equation}
The RHS is a quotient of $\mathbb{C}[X^{T}]^{\Bar{A}(\OO^\vee)}$ that, in its turn, is a quotient of $\mathbb{C}[(X/\Bar{A}(\OO^\vee))^T]$ (\cref{invariantC to Cinvariant}). Recall that $X=\operatorname{Spec}(\mathbb{C}[\widetilde{D}(\mathbb{O}^\vee)])$ and $\Bar{A}(\OO^\vee)=\operatorname{Aut}_{D(\mathbb{O}^\vee)}(\widetilde{D}(\mathbb{O}^\vee))$. Assume also that the closure $\overline{D(\mathbb{O}^\vee)}$ is normal.
We have:  
\begin{equation*}
(X/\Bar{A}(\OO^\vee))^T = (\operatorname{Spec}\mathbb{C}[\widetilde{D}(\mathbb{O}^\vee)]/{\Bar{A}(\OO^\vee)})^T = \operatorname{Spec}\mathbb{C}[D(\mathbb{O}^\vee)]^T = \overline{D(\mathbb{O}^\vee)} \cap \mathfrak{h}^*.
\end{equation*}

It follows that the natural morphism $\mathbb{C}[\mathfrak{h}^*] \rightarrow \cC_\nu(\mathcal{A}_\hbar(X))^{\Bar{A}(\OO^\vee)}/(\hbar)$ is surjective. So, the RHS of (\ref{wrong_identification}) is the quotient of $\mathbb{C}[\mathfrak{h}^*]$ (after we assume the normality of $\overline{D(\mathbb{O}^\vee)}$). On the cohomological side, we expect the restriction homomorphism $H^*(\mathcal{B}) \to H^*(\mathcal{B}_e)^{A^\vee}$ to be surjective for certain (but not all) $e^\vee$ distinguished. Two particular cases where (\ref{wrong_identification}) holds are mentioned in Section 10.4.3.


In general, even for distinguished $e^\vee$, we do not expect that the specialization of $ H^{*}_{\CC^\times}(\spr)^{A^\vee}$ to $\hbar=0$ is isomorphic to the specialization of $\cC_\nu(\mathcal{A}_\hbar(X))$. However, regardless of $e^\vee$ being distinguished or not, one still might expect that the {\emph{images}}
\begin{equation}\label{images_canonical_param}
\operatorname{Im}(H^*_{\mathbb{C}^\times}(\mathcal{B}) \rightarrow H^*_{\mathbb{C}^\times}(\mathcal{B}_e)),~\operatorname{Im}(\mathbb{C}[\mathfrak{h}^*,\hbar] \rightarrow \cC_\nu(\mathcal{A}_\hbar)^{\Bar{A}(\OO^\vee)})
\end{equation}
are isomorphic. We note that $\operatorname{Im}(H^*_{\mathbb{C}^\times}(\mathcal{B}) \rightarrow H^*_{\mathbb{C}^\times}(\mathcal{B}_e)) \subset H^*_{\mathbb{C}^\times}(\mathcal{B}_e)^{A^\vee})$. 

For $\hbar \neq 0$, homomorphisms in (\ref{images_canonical_param}) should be surjective (the second one is known to be true by \cite[Corollary 6.5.6]{LMBM}), 
so the identification of images should imply the conjectural identification (\ref{identi_Galois_fixed}) above.



\subsubsection{Case of arbitrary $e^\vee$}
Let us now assume that $e^\vee$ is arbitrary.
Assume that $Y^\vee=\widetilde{S}(e^\vee)$, $X=\on{Spec}\CC[\widetilde{D}(\mathbb{O}^\vee)]$ and let $\widetilde{X}$ be a $\mathbb{Q}$-factorial terminalization of $X$ (described in \cite[Section 7.2]{LMBM}, see also \cite{Namikawa2022}). 
Let $\widetilde{X}_{\mathfrak{h}_{e^\vee}}$ be the deformation of $\widetilde{X}$ over $\mathfrak{h}_{e^\vee}$.
Explicitly, we have:
\begin{equation*}
\widetilde{X} = G \times^P (\operatorname{Spec}\mathbb{C}[\widetilde{\mathbb{O}}_L] \times \mathfrak{p}^\perp),~\widetilde{X}_{\mathfrak{h}_{e^\vee}} = G \times^P (\operatorname{Spec}\mathbb{C}[\widetilde{\mathbb{O}}_L] \times [\mathfrak{p},\mathfrak{p}]^\perp).
\end{equation*}
    Set $X_L:=\operatorname{Spec}\mathbb{C}[\widetilde{\mathbb{O}}_L]$.

For a module $M$ over $\mathbb{C}[\mathfrak{h}_{e^\vee},\hbar]$, we will denote by $M_{\hbar}$ the  $\mathbb{C}[\mathfrak{h}_{e^\vee},\hbar^{\pm 1}]$ -module $M \otimes_{\mathbb{C}[\mathfrak{h}_{e^\vee},\hbar]} 
\mathbb{C}[\mathfrak{h}_{e^\vee},\hbar^{\pm 1}]$. Recall that $L^\vee$ is the connected component of $1 \in Z_{G^\vee}(T_{e^\vee})$ and recall that $e^\vee$ is {\emph{distinguished}} in $\mathfrak{l}^\vee=\on{Lie}L^\vee$. Let $A_{L^\vee}$ be the component group of $Z_{L^\vee}(e^\vee)$. The group $A_{L^\vee}$ acts  naturally on $H^*_{T_{e^\vee} \times \CC^\times}(\mathcal{B}_{e^\vee})$. 
We make the following general conjecture (compare with Remark \ref{rem_general_conj}). 
\begin{conj} \label{image distinguished}
    We have an isomorphism of algebras:
    \begin{multline} 
    \on{Im}\left(H^{*}_{T_{e^\vee} \times \CC^\times}(T^*\cB^\vee) \rightarrow H^{*}_{T_{e^\vee} \times \CC^\times}(\spr^{T_{e^\vee} \times \CC^\times})_{\hbar}^{A'_{L^\vee}}  \right) \simeq 
    \on{Im} \left( \CC[\fh, \hbar] \otimes \CC[\mathfrak{h}_{e^\vee}] \rightarrow \Gamma(\widetilde{X}_{\mathfrak{h}_{e^\vee}}^{T_X},\cC_\nu(\mathcal{D}_{\hbar,\mathfrak{h}_{e^\vee}}(\widetilde{X})))_{\hbar}\right).
    \end{multline}
\end{conj}

\begin{rmk}\label{lemma_reduce_im_distinguished}
    Let us explain how  Conjectures \ref{target distinguished} and \ref{image distinguished} are related. Assume that the Conjecture (\ref{target distinguished}) holds for $(e^\vee,L^\vee)$. We expect the following isomorphism:
\begin{equation}\label{version_conj_dist}
H^{*}_{T_{e^\vee} \times \CC^\times}(\spr^{T_{e^\vee} \times \CC^\times})^{A'_{L^\vee}}_{\hbar} \simeq \Gamma(\widetilde{X}^{T_X}_{\mathfrak{h}_{e^\vee}},\cC_\nu(\mathcal{D}_{\hbar,\mathfrak{h}_{e^\vee}}(\widetilde{X})))_{\hbar}.
\end{equation}
Assuming Conjecture  (\ref{target distinguished}) holds, we expect that the identification of Conjecture \ref{image distinguished} is compatible with the isomorphism (\ref{version_conj_dist}). 
\end{rmk}

Let us finish this section mentioning that we {\emph{do not}} have any conjecture describing $\cC_\nu(\mathcal{A}_{\hbar,\mathfrak{X}}(X))$ for $X=\operatorname{Spec}(\mathbb{C}[\widetilde{D}(\mathbb{O}^\vee)])$. Informally, the reason is the following.

As we have already seen, the specialization of $\cC_\nu(\mathcal{A}_{\hbar,\mathfrak{h}_X}(X))$ to $\hbar=1$, $\lambda=0$ should be isomorphic to $H^*(\mathcal{B}_{e^\vee}^{\mathbb{C}^\times})^{A'}$. Note also that for $\hbar=1$ and $\la$ being generic, the fiber $\cC_\nu(\mathcal{A}_{\hbar,\mathfrak{X}}(X))_{(\la,1)}$ should be isomorphic to $(H^*(\mathcal{B}_{e^\vee}^{T_{e} \times \mathbb{C}^\times}))^{A_L'}$.
So, to be able to describe the whole family $\cC_\nu(\mathcal{A}_{\hbar,\mathfrak{h}_X}(X))$ via $\mathcal{B}_{e^\vee}$, we need some ``cohomology theory'' that captures the $A^\vee$-action on $\mathcal{B}_{e^\vee}$. It is not clear to us what the right candidate should be.  

Note that Conjecture \ref{image distinguished} bypasses this problem by considering the image of the equivariant cohomology of $\mathcal{B}$, this allows to neglect the actions of $A^\vee$ and $\Bar{A}_{L}(\OO^\vee)$ on the targets.

\subsubsection{Evidence for the conjectures}
\cref{target distinguished} is closer to the original Nakajima-Hikita conjecture in the sense that we keep the algebra $\cC_\nu(\cA(\widetilde{D}(\OO^\vee)))$ on the right side. 
On the other hand, 
\cref{image distinguished} is a direct generalization of \cref{conj_most_general_weak_Hikita} for $X^\vee=S_{e^\vee}$, $X=\Spec(\CC[\widetilde{D}(\OO^\vee)])$. Next, we explain some evidence for \cref{target distinguished} and \cref{image distinguished} when $e^\vee$ is distinguished.

Consider the case where $\fg^\vee$ is of type D and the partition of $e^\vee$ has two distinct parts. In this case, both $A^\vee$ and $\Bar{A}(\OO^\vee)$ are trivial. The identification \ref{wrong_identification} becomes the classical Hikita conjecture, which is true (see \cite[Proposition 6.1]{hoang2}). A ``twin'' case is where $\fg$ is of type $C$ and the partition of $e^\vee$ has two distinct parts. We now have $A^\vee= \Bar{A}(\OO^\vee)= \cyclic{2}$, and $\widetilde{D}(\OO^\vee)$ is a double cover of ${D}(\OO^\vee)$. The identification \ref{wrong_identification} is true by \cite[Remark 5.5]{hoang2}. We also expect \cref{image distinguished} to hold in this setting, the case where $e^\vee$ is subregular in type D has been proved in \cite[Theorem 1.3]{chen2024}.

\begin{rmk}
    We note that by \cite{Henderson2014}, in this case, $X^\vee$ can be studied as a quiver variety. The Coulomb branch of the corresponding quiver gauge theory is a slice in affine Grassmannian. We expect it to be isomorphic to the closure of the corresponding type $D$ orbit, similar to \cite{Mirkovi2022}.
\end{rmk}





\appendix
\section{Concrete examples of weight computations}\label{appendixA}
Assume, for simplicity, that $\fq\subset \fg$ is a Borel subalgebra. The statement of \cref{main_prop_weak_hikita} implies the coincidence of certain weights of $\fh$. Namely, consider the two different compositions $\CC[\fh^*, \hbar]^{W_M}\otimes_{\CC[\hbar]} \CC[\fh^*, \hbar]\to \bigoplus_{[w] \in W/W_L}\CC[\mathfrak{X}(\mathfrak{l}),\hbar] $. Specializing to the parameters $(\lambda, \hbar_0)\in \fX(\fl)\oplus \CC$ we get a collection of one-dimensional representations of $\CC[\fh^*]$. Using the top path in the diagram, $\chi\in \fh^*$ goes to the equivariant Chern class of the restriction of the line bundle $\fL(\chi)$ on $T^*\cB$ to $(\widetilde{S}(\chi^\vee)^{Z_{L^\vee}}$. Using the bottom path, we compute the image of $\chi$ under the Cartan comoment map. In the example below we explicitly compute the sets of weights obtained both ways and confirm that they coincide.

    Let $G^\vee$ be $\on{PGL}_4$. Consider the nilpotent element $e^\vee$ of $\fg^\vee$ corresponding to the partition $\mathbf{p}^\vee=(3,1)$ and $L^\vee=(\on{GL}_1 \times \on{GL}_3)/\mathbb{C}^\times$. Consider the map 
    \begin{equation}\label{equation A2}
    H_{{T^\vee}\times \CC^\times}^*(\cB^\vee)\rightarrow H_{T_{e^\vee}\times \CC^\times}^*(\cB^\vee)\rightarrow H_{T_{e^\vee}\times \CC^\times}^{*}(\spr) \rightarrow H_{T_{e^\vee}\times \CC^\times}^{*}(\spr^{T_{e^\vee}})= \oplus_{W/W_L}\CC[a,\hbar].
    \end{equation} 
    Here $T_{e^\vee} \simeq \CC^\times$, via $\CC^\times \ni t \mapsto [\on{diag}(t,1,1,1)]$. Let $a$ be the natural coordinate of $\CC[\ft_{e^\vee}]$.
    Let us explicitly describe $\mathcal{B}_{e^\vee}$. It has three irreducible components $C_1, C_2, C_3$, each of them is isomorphic to $\mathbb{P}^1$.  
    Namely, let $v_1,v_2,v_3,v_4$ be the standard basis in $\mathbb{C}^4$. We have 
    \begin{equation*}
    e^\vee\colon v_1 \mapsto 0,~ v_4 \mapsto v_3 \mapsto v_2 \mapsto 0.
    \end{equation*}
    Then, 
    \begin{equation*}
    C_1 = \{\langle p_1v_2+q_1v_1\rangle \subset \langle v_1,v_2\rangle \subset \langle v_1,v_2,v_3\rangle \subset \mathbb{C}^4\,|\, [p_1:q_1] \in \mathbb{P}^1\},
    \end{equation*}
    \begin{equation*}
    C_2 = \{\langle v_2\rangle \subset  \langle v_2, p_2v_3+q_2v_1\rangle \subset \langle v_1,v_2,v_3\rangle \subset \mathbb{C}^4\,|\, [p_2:q_2] \in \mathbb{P}^1\},
    \end{equation*}
    \begin{equation*}
    C_3 = \{\langle v_2 \rangle \subset \langle v_2, v_3 \rangle \subset \langle v_2,v_3,p_3v_4+q_3v_1\rangle \subset \mathbb{C}^4\,|\, [p_3:q_3] \in \mathbb{P}^1\}
    \end{equation*}
    The action of $T_{e^\vee} \times \mathbb{C}^\times=\mathbb{C}^\times\times \mathbb{C}^\times$ is the following:
    \begin{equation}\label{act_on_B_e}
    (a,t) \cdot [p_1:q_1]=[t^2p_1:aq_1],~(a,t) \cdot [p_2:q_2] = [p_2:aq_2],~(a,t) \cdot [p_3:q_3] = [t^{-2}p_3:aq_3].
    \end{equation}

    We have four $T_{e^\vee} \times \mathbb{C}^\times$-fixed points:
\begin{equation*}
o_1=[0_1:1_1],~o_2=[1_1:0_1]=[0_2:1_2],~o_3=[1_2:0_2]=[0_3:1_3],~o_4=[1_3:0_3].
\end{equation*}
The identification with ${}^{L^\vee}W$ is as follows:
\begin{equation*}
o_1 \mapsto 1=:w_1,~o_2 \mapsto (12)=:w_2,~o_3 \mapsto (123)=:w_3,~o_4 \mapsto (1234)=:w_4. 
\end{equation*}

Consider the line bundle $\mathcal{L}_i$ over $\mathcal{B}$ which fiber  over a flag $\{0\}=F_0 \subset F_1 \subset F_2 \subset F_3 \subset F_4=\mathbb{C}^4$ is $F_i/F_{i-1}$. Note that $\mathcal{L}_i$ is {\emph{not}} $\on{PGL}_4$-equivariant. On the other hand, $\mathcal{L}_i \otimes \mathcal{L}_{i+1}^{-1}$ is $\on{PGL}_4$-equivariant. 

 Similar to above, let $\CC_i$ be the one-dimensional representation of $\CC^\times$, such that $t\cdot v = t^i v$. We see that as representations of $T_e\times \CC^{\times}$ we have the following. 
 \begin{equation*}
 \mathcal{L}_1 \otimes \mathcal{L}_2^{-1}|_{o_1} = \CC_{1} \otimes \CC_{-2}, \,  \mathcal{L}_1 \otimes \mathcal{L}_2^{-1}|_{o_2} = \CC_{-1} \otimes \CC_{2}, \, \mathcal{L}_1 \otimes \mathcal{L}_2^{-1}|_{o_3} = \CC_{0} \otimes \CC_{2}, \,\mathcal{L}_1 \otimes \mathcal{L}_2^{-1}|_{o_4} = \CC_{0} \otimes \CC_{2},
 \end{equation*} 
  \begin{equation*}
 \mathcal{L}_2 \otimes \mathcal{L}_3^{-1}|_{o_1} = \CC_{0} \otimes \CC_{2}, \,  \mathcal{L}_2 \otimes \mathcal{L}_3^{-1}|_{o_2} = \CC_{1} \otimes \CC_{0}, \, \mathcal{L}_2 \otimes \mathcal{L}_3^{-1}|_{o_3} = \CC_{-1} \otimes \CC_{0}, \,\mathcal{L}_2 \otimes \mathcal{L}_3^{-1}|_{o_4} = \CC_{0} \otimes \CC_{2},
 \end{equation*} 
  \begin{equation*}
 \mathcal{L}_3 \otimes \mathcal{L}_4^{-1}|_{o_1} = \CC_{0} \otimes \CC_{2}, \,  \mathcal{L}_3 \otimes \mathcal{L}_{4}^{-1}|_{o_2} = \CC_{0} \otimes \CC_{2}, \, \mathcal{L}_3 \otimes \mathcal{L}_4^{-1}|_{o_3} = \CC_{1} \otimes \CC_{2}, \,\mathcal{L}_3 \otimes \mathcal{L}_4^{-1}|_{o_4} = \CC_{-1} \otimes \CC_{-2},
 \end{equation*}

We conclude that 
\begin{equation*}
e_{1}=e_{11}-e_{22} \mapsto (a-2\hbar,-a+2\hbar,2\hbar,2\hbar),
\end{equation*}
\begin{equation*}
 e_{2}=e_{22}-e_{33} \mapsto (2\hbar,a,-a,2\hbar),
\end{equation*}
\begin{equation*}
e_{3}=e_{33}-e_{44}  \mapsto (2\hbar,2\hbar,a+2\hbar,-a-2\hbar).
\end{equation*}

    Let us now describe the Cartan comoment map. Recall that the dual variety is $T^*(\on{SL}_4/P)$, where $P$ is the parabolic corresponding to the Levi $L=(\on{GL}_1 \times \on{GL}_3) \cap \on{SL}_4$. Clearly, $\on{SL}_4/P \simeq \mathbb{P}^3$ and we can describe quantizations of $\mathbb{C}[T^*\mathbb{P}^3]$ explicitly. The quantization $\mathcal{A}_{\hbar,\mathfrak{X}(\mathfrak{l})}(T^*(\on{SL}_4/P))$ is $D(\on{SL}_4/[P,P])^{P/[P,P]}$. Explicitly, we have an isomorphism 
    \begin{equation*}
    \on{SL}_4/[P,P] \iso \mathbb{A}^4 \setminus \{0\}
    \end{equation*}
    given by the action on $v_1 \in \mathbb{C}^4$. Then $P/[P,P]$ identifies with $\mathbb{C}^\times$ acting on $\mathbb{A}^4 \setminus \{0\}$ on the right via $(r_1,r_2,r_3,r_4) \cdot t = (r_1t,r_2t,r_3t,r_4t)$,

    Vasya's version: the explicit identification $P/[P,P] \iso \mathbb{C}^\times$ is given by $[(a_0,b_0,b_0,b_0)] \mapsto a_0$ and so the identification $\mathfrak{z}_{\mathfrak{l}}=\mathfrak{p}/[\mathfrak{p},\mathfrak{p}] \simeq \on{Lie}\mathbb{C}^\times = \mathbb{C}$ is given by $(a_0,b_0,b_0,b_0) \mapsto a_0$. So, 
    \begin{equation*}
    \mathcal{A}_{\hbar,\mathfrak{X}(\mathfrak{l})}(T^*(\on{SL}_4/P)) = (D_{\hbar}(\mathbb{A}^4))_0,    \end{equation*}
    where the degree zero part is taken w.r.t. to the grading $\on{deg}r_i=-1$, $\on{deg}\partial_{r_i}=1$.

    The identification $H^2(\mathbb{P}^3,\mathbb{C}) \simeq (\mathfrak{p}/[\mathfrak{p},\mathfrak{p}])^*=\mathfrak{z}(\mathfrak{l})^*$ is given by 
    \begin{equation*}
    c_1(\mathcal{O}_{\mathbb{P}^3}(-1)) \mapsto ((a_0,b_0,b_0,b_0) \mapsto a_0)=a.
    \end{equation*}
    Note now that the period of the microlocalization of the sheaf of differential operators $\mathcal{D}(\mathbb{P}^3)$ to $T^*\mathbb{P}^3$ equals 
    \begin{equation*}
    -c_1(T^*\mathbb{P}^3)/2=-c_1(\mathcal{O}_{\mathbb{P}^3}(-4))/2=-2a.
    \end{equation*}
    Note that 
 \begin{equation*}
 \rho_{\mathfrak{sl}_4}=\Big\langle\Big(\frac{3}{2},\frac{1}{2},-\frac{1}{2},-\frac{3}{2}\Big),-\Big\rangle, \,\rho_{\mathfrak{l}}=\Big\langle\Big(0,1,0,-1\Big),-\Big\rangle,\,\text{so}\,\rho_{\mathfrak{sl}_4}-\rho_{\mathfrak{l}}=\Big\langle\Big(\frac{3}{2},-\frac{1}{2},-\frac{1}{2},-\frac{1}{2}\Big),-\Big\rangle,
 \end{equation*} 
 i.e., it sends $(a_0,b_0,b_0,b_0)$ to $\frac{3}{2}(a_0-b_0)=2a_0$ so is indeed equal to $2a$ as desired (c.f. Proposition \ref{quant_flag}).

    In particular, we see that the {\emph{twisted}} action of $\mathbb{C}[\mathfrak{X}(\mathfrak{l}),\hbar]=\mathbb{C}[a,b,\hbar]/(a+3b)$ is induced by the homomorphism 
    \begin{equation*}
    1 \otimes a \mapsto -r_1\partial_{r_1}-r_2\partial_{r_2}-r_{3}\partial_{r_3}-r_4\partial_{r_4}-2\hbar,
    \end{equation*}
    \begin{equation*}
    1 \otimes b \mapsto \frac{1}{3}(r_1\partial_{r_1}+r_2\partial_{r_2}+r_3\partial_{r_3}+r_4\partial_{r_4})+\frac{2}{3}\hbar.
    \end{equation*}
    So, for $\hbar=1$, the specialization of $(D_{1}(\mathbb{A}^4))_0$ to $a=a_0 \in \mathfrak{X}(\mathfrak{l})$ are global sections of the quantization of $\mathcal{O}_{T^*(\on{SL}_4/P)}$ with period equal to $a_0$. Note also that the abelian localization holds for $\mathcal{A}_{1,a_0}(T^*(\on{SL}_4/P))$ for $a_0$ being small enough negative integer number.

    The map $\mathcal{U}_{\hbar}(\mathfrak{sl}_4) \rightarrow (D_{\hbar}(\mathbb{A}^4))_0$ is induced by the map $\mathcal{U}_{\hbar}(\mathfrak{gl}_4) \rightarrow (D_{\hbar}(\mathbb{A}^4))_0$ given by:
    \begin{equation*}
    e_{ij} \mapsto -r_j\partial_{r_i}.
    \end{equation*}
    In particular, every module over $(D_{\hbar}(\mathbb{A}^4))_0$ can be considered as an $\mathcal{U}_{\hbar}(\mathfrak{gl}_4)$-module.

    We have four families of highest weight modules over $D_{\hbar}(\mathbb{A}^4)_0$:
\begin{equation*}
r_4^{\frac{c}{\hbar}}\Big(D_\hbar(\mathbb{A}^4)_{(r_4)}/D_\hbar(\mathbb{A}^4)_{(r_4)}(\partial_{r_1},\partial_{r_2},\partial_{r_3},\partial_{r_4})\Big)_0=r_4^{\frac{c}{\hbar}}\mathbb{C}[r_1/r_4,r_2/r_4,r_3/r_4,\hbar],
\end{equation*}
\begin{equation*}
r_3^{\frac{c}{\hbar}+1}\Big(D_\hbar(\mathbb{A}^4)_{(r_3)}/D_\hbar(\mathbb{A}^4)_{(r_3)}(\partial_{r_1},\partial_{r_2},\partial_{r_3},r_4)\Big)_0=r_3^{\frac{c}{\hbar}+1}\mathbb{C}[r_1/r_3,r_2/r_3,r_3\partial_{r_4},\hbar], 
\end{equation*}
\begin{equation*}
r_2^{\frac{c}{\hbar}+2}\Big(D_\hbar(\mathbb{A}^4)_{(r_2)}/D_\hbar(\mathbb{A}^4)_{(r_2)}(\partial_{r_1},\partial_{r_2},r_3,r_4)\Big)_0=r_2^{\frac{c}{\hbar}+2}\mathbb{C}[r_1/r_2,r_2\partial_{r_3},r_2\partial_{r_4},\hbar], 
\end{equation*}
\begin{equation*}
r_1^{\frac{c}{\hbar}+3}\Big(D_\hbar(\mathbb{A}^4)_{(r_1)}/D_\hbar(\mathbb{A}^4)_{(r_1)}(\partial_{r_1},r_2,r_3,r_4)\Big)_0=r_1^{\frac{c}{\hbar}+3}\mathbb{C}[r_1\partial_{r_2},r_1\partial_{r_3},r_1\partial_{r_4},\hbar].
\end{equation*}
The action of the Euler vector field $r_1\partial_{r_1}+r_2\partial_{r_2}+r_3\partial_{r_3}+r_4\partial_{r_4}$ on all of them is via the multiplication by $c$.

We see that the highest weights (w.r.t. $\mathfrak{h}_{\mathfrak{gl}_4} \subset \mathfrak{gl}_4$) are: 
\begin{equation*}
(0,0,0,-c),\, (0,0,-\hbar-c,\hbar),\, (0,-c-2\hbar,\hbar,\hbar),\, (-3\hbar-c,\hbar,\hbar,\hbar).
\end{equation*}

We have $\rho_{\mathfrak{sl}_4}=\Big(\frac{3}{2},\frac{1}{2},-\frac{1}{2},-\frac{3}{2}\Big)$.      The $\hbar\rho_{\mathfrak{sl}_4}$-shifts of the weights above are:
\begin{equation*}
\Big(\frac{3}{2}\hbar,\frac{1}{2}\hbar,-\frac{1}{2}\hbar,-\frac{3}{2}\hbar-c\Big),\, \Big(\frac{3}{2}\hbar,\frac{1}{2}\hbar,-\frac{3}{2}\hbar-c,-\frac{1}{2}\hbar\Big),\, \Big(\frac{3}{2}\hbar,-\frac{3}{2}\hbar-c,\frac{1}{2}\hbar,-\frac{1}{2}\hbar\Big),\, \Big(-\frac{3}{2}\hbar-c,\frac{3}{2}\hbar,\frac{1}{2}\hbar,-\frac{1}{2}\hbar\Big)
\end{equation*}
and they indeed lie in the same $S_4$-orbit. Moreover, after our identifications, the labels in $W/W_L=S_4/(S_1 \times S_3)$ for our modules are representatives of 
\begin{equation*}
(1432),\, (132),\,(12)\,,1
\end{equation*}
respectively.

After twisting, we see that the highest weights of the modules over the quantization with period $b=b_0$ (equivalently, $a=-3b_0$) labeled by $[1], [(12)], [(132)], [(1432)] \in W/W_L$ respectively correspond to $c=-a-2\hbar$ and are equal to:
\begin{equation*}
(a-\hbar,\hbar,\hbar,\hbar),\, (0,a,\hbar,\hbar),\, (0,0,a+\hbar,\hbar),\, (0,0,0,a+2\hbar).
\end{equation*}

After the $\rho_{\mathfrak{sl}_4}$-shift we get:
\begin{equation*}
\Big(a+\frac{1}{2}\hbar,\frac{3}{2}\hbar,\frac{1}{2}\hbar,-\frac{1}{2}\hbar\Big),\, \Big(\frac{3}{2}\hbar,a+\frac{1}{2}\hbar,\frac{1}{2}\hbar,-\frac{1}{2}\hbar\Big),\, \Big(\frac{3}{2}\hbar,\frac{1}{2}\hbar,a+\frac{1}{2}\hbar,-\frac{1}{2}\hbar\Big),\, \Big(\frac{3}{2}\hbar,\frac{1}{2}\hbar,-\frac{1}{2}\hbar,a+\frac{1}{2}\hbar\Big).
\end{equation*}

We conclude that the map
\begin{equation}\label{comp_RHS_HN}
\mathbb{C}[\mathfrak{h}^*_{\mathfrak{gl}_4},\hbar] \otimes_{\mathbb{C}[\hbar]} \mathbb{C}[\mathfrak{h}^*_{\mathfrak{gl}_4},\hbar] \rightarrow \cC_\nu(\mathcal{A}_{\mathfrak{X}(\mathfrak{l})}(T^*(\on{SL}_4/P))) \rightarrow \Big(\mathbb{C}[a,b,\hbar]/(a+3b)\Big)^{\oplus 4}
\end{equation}
is given by 
\begin{equation*}
e_{11} \otimes 1 \mapsto (a+\frac{1}{2}\hbar,\frac{3}{2}\hbar,\frac{3}{2}\hbar,\frac{3}{2}\hbar),
\end{equation*}
\begin{equation*}
e_{22} \otimes 1 \mapsto (\frac{3}{2}\hbar,a+\frac{1}{2}\hbar,\frac{1}{2}\hbar,\frac{1}{2}\hbar)
\end{equation*}
\begin{equation*}
e_{33} \otimes 1 \mapsto (\frac{1}{2}\hbar,\frac{1}{2}\hbar,a+\frac{1}{2}\hbar,-\frac{1}{2}\hbar),
\end{equation*}
\begin{equation*}
e_{44} \mapsto (a+\frac{1}{2}\hbar,-\frac{1}{2}\hbar,-\frac{1}{2}\hbar,-\frac{1}{2}\hbar).
\end{equation*}


Restricting to $\mathfrak{sl}_4$ and replacing $\hbar$ by $2\hbar$ we conclude that 
\begin{equation*}
e_1=e_{11}-e_{22} \mapsto (a-2\hbar,2\hbar-a,2\hbar,2\hbar),
\end{equation*}
\begin{equation*}
e_2=e_{22}-e_{33} \mapsto (2\hbar,a,-a,2\hbar),   
\end{equation*}
\begin{equation*}
e_3=e_{33}-e_{44} \mapsto (2\hbar,2\hbar,a+2\hbar,-a-2\hbar).
\end{equation*}
We see that the computations match.

Finally, note that for $\hbar=1$ and $a=0$, the category $\mathcal{O}$ for $\mathcal{A}_0(T^*(\on{SL}_4/P))$ contains {\emph{three}} simple objects, their twisted highest weights are equal to:
\begin{equation*}
\Big(\frac{1}{2},\frac{3}{2},\frac{1}{2},-\frac{1}{2}\Big),\,\Big(\frac{3}{2},\frac{1}{2},\frac{1}{2},-\frac{1}{2}\Big),\, \Big(\frac{3}{2},\frac{1}{2},-\frac{1}{2},\frac{1}{2}\Big).
\end{equation*}
Considering $\mathcal{A}_0(T^*(\on{SL}_4/P))$ as a quotient of $\mathcal{U}(\mathfrak{sl}_4)$ we see that the labels of the irreducibles above are 
\begin{equation*}
(142),\, (14),\, (134)
\end{equation*}
and these three elements indeed form one 
Kazhdan-Lusztig cell in $S_4$, let us denote this cell by $c$. On the dual side, we consider the set of connected components of $\mathcal{B}_{e^\vee}$. It follows from (\ref{act_on_B_e}) that 
\begin{equation*}
\mathcal{B}_{e^\vee}^{\mathbb{C}^\times} = \{o_1\} \sqcup C_2 \sqcup \{o_4\}
\end{equation*}
and the bijection $\mathcal{B}_{e^\vee}^{\mathbb{C}^\times} \simeq c$ is given by:
\begin{equation*}
\{o_1\}  \mapsto (142),\, C_2 \mapsto (14),\, \{o_4\} \mapsto (134).
\end{equation*}


\section{Surjectivity of the pullback map} \label{cohomological surjectivity}
Let $G^\vee =Sp({2n})$, $\fg^\vee = \fsp({2n})$. Let $e^\vee$ be a nilpotent element in $\fg^\vee$. Let $\pv$ be the associated partition of $e^\vee$. We write $\spr$ or $\cB_{\pv}$ for the Springer fiber of $e$, and write $\cB^\vee$ (or $\cB_{2n}$) for the flag variety of $\fg^\vee$.  Consider the pullback map $i_{\pv}^*:H^*(\cB^\vee)\rightarrow H^*(\spr)$. The goal of this appendix is to show that $i_{\pv}^*$ is surjective only if $\pv$ satisfies certain restrictive conditions. 

Let $Z^{\vee}_G(e^\vee)$ be the centralizer of $e^\vee$ in $G^\vee$. The map $i_{\pv}^*$ is $Z^{\vee}_G(e^\vee)$-equivariant, and the action of $Z^{\vee}_G(e^\vee)$ on $H^*(\spr)$ is trivial. Therefore, if $Z^{\vee}_G(e^\vee)$ acts nontrivially on $H^*(\spr)$, then $i_\pv^{*}$ is not surjective. Thus, we restrict our attention to elements $e^\vee$ such that $Z^{\vee}_G(e^\vee)$ acts trivially on $H^*(\spr)$. In particular, $\pv$ has at most one even part, and the multiplicity of this even part is odd (or $0$ if all parts of $\pv$ are odd), see \cref{when A acts trivially}. 

Let $\pv= (2a, 2b_1+1,...,2b_k+1)$, where the given partition may not be in standard order. We only require $b_1\geqslant b_2\geqslant... \geqslant b_k$ and $a\geqslant 0$. The main result of this appendix is as follows.
\begin{theorem}\label{surjectivity} \leavevmode
    \begin{enumerate}
        \item If $a<b_1$ or $a-2\geqslant b_i\geqslant 1$ for some $1\leqslant i\leqslant k$, then the map $i_{\pv}^*$ is not surjective.        
        \item If $i_\pv^{*}$ is surjective, the partition $\pv$ is of the form $((2a+1)^{2d_1},(2a)^{2d_2+1},(2a-1)^{2d_3}, 1^{2d_4})$ for some $a, d_1,...,d_4\geqslant 0$. 
    \end{enumerate}
\end{theorem}

We note that part (2) is a direct implication of (1). We have analogous to part (2) results for $\fg^\vee$ of types $B$ and $D$.
\begin{theorem}\label{surjectivity B C D} \leavevmode
\begin{enumerate}
    \item Consider $\fg^\vee= \fso({2n+1})$. If $i_\pv^{*}$ is surjective, the partition $\pv$ is of the form $((2a+2)^{2d_1}, (2a+1)^{2d_2+1}, (2a)^{2d_3})$ for some $a, d_1, d_2, d_3\geqslant 0$.
    \item Consider $\fg^\vee= \fso({2n})$. If $i_\pv^{*}$ is surjective, the partition $\pv$ takes one of the following forms
    \begin{itemize}
        \item $((2a+2)^{2d_1}, (2a+1)^{2d_2}),$ 
        \item $(2a+1)^{2d_1}, (2a)^{2d_2})$,
        \item $((2a+3)^{2d_1+1}, (2a+2)^{2d_2}, (2a+1)^{2d_3+1}),$
        \item $((2a+1)^{2d_1+1},(2b+1)^{2d_2+1}),$
        \item $((2a+1)^{2d_1+1},(2a)^{2d_2}, 2^{2d_3},1^{2d_4+1}),$
    \end{itemize}
    for some $a, b, d_1, d_2, d_3, d_4\geqslant 0$.
\end{enumerate}
\end{theorem}
We prove \cref{surjectivity}, the proofs for type B and D are similar. Our argument combines techniques in equivariant cohomology and a modified version of \cite[Section 4.4]{surjectiveargument}.
\begin{proof}[Proof of \cref{surjectivity}]
    The proof consists of several steps as follows.
    \begin{enumerate}
        \item Step 1: We first reduce the consideration to the case where $\fp$ has $2$ or $3$ parts. This is explained in \cref{reduce to 3}.
        \item Step 2: The map $i^{*}_{(2a,2b+1,2b+1)}: H^*(\cB^\vee)\rightarrow H^*(\cB_{(2a,2b+1,2b+1)})$ is not surjective for $0\leqslant a< b$. (Part (1) of \cref{inductive non-surjectivity} )
        \item Step 3: The map $i^{*}_{(2a,2b+1,2b+1)}: H^*(\cB^\vee)\rightarrow H^*(\cB_{(2a,2b+1,2b+1)})$ is not surjective for $a\geqslant b+2$. This is deduced from Part (2) of \cref{inductive non-surjectivity} and \cref{3-row}.
    \end{enumerate}
    These results mean that for $i_{\pv}^*$ to be surjective, we must have $a-b_i= 0$ or $a-b_i= 1$ for $b_i> 0$. In the following, we provide the proof of the lemmas that we have mentioned.
\end{proof}

\begin{lemma} \label{reduce to 3}
    Consider $\pv= (2a, 2b_1+1,...,2b_k+1)$. If the map $i_{\pv}^*:H^*(\cB^\vee)\rightarrow H^*(\spr)$ is surjective, then $i_{(2a, 2b_i+1,2b_i+1)}^*$ is surjective for $1\leqslant i\leqslant k$.
\end{lemma}
\begin{proof}
    Note that for $\pv= (2a, 2b_1+1,...,2b_k+1)$, $e$ is regular in a Levi subalgebra of $\fg$, $\fl= \fsp({2a})\times \prod_{i=1}^{k} \fgl({2b_i+1})$. We prove the lemma for $i=1$, the other cases are similar. 
    
    Consider $\fl'\subset \fg$ so that $\fl'\simeq \fsp({2a+2(2b_1+1)})\times \prod_{i=2}^{k} \fgl({2b_i+1})$ and $\fl\subset \fl'$. Let $T_1'\subset L\subset G$ be the connected centralizer of $\fl'$, then $T_1'$ is a torus. Consider a generic cocharacter $T_1=\CC^\times \hookrightarrow T_1'$. We have natural actions of $T_1$ on $\cB^\vee$ and $\spr$. The fixed-point locus of this $T_1$ action on $\spr$ is a disjoint union of isomorphic copies of $\cB_{(2a,2b_1+1,2b_1+1)}$ (see \cref{fixed point Springer fiber}).

    Consider the $T_1$-equivariant cohomologies of $\cB^\vee$ and $\spr$. They are modules over $H_{T_1}^*(\pt)=\CC[t_1]$. Write $i^{T_1, *}_{\pv}$ for the pullback map $H_{T_1}^*(\cB^\vee)\rightarrow H_{T_1}^*(\spr)$. If $i_{\pv}^*$ is surjective, then $i^{T_1, *}_{\pv}$ is surjective. Specializing $t_1$ to $1$, we get that $i^{*}_1: H({\cB^\vee}^{T_1})\rightarrow H(\spr^{T_1})$ is surjective. Restricted to each irreducible component of $\spr^{T_1}$, the map $i^{*}_1$ is $i^{*}_{(2a,2b_1+1, 2b_1+1)}$ (see \cref{non gl factor coh}); therefore, the lemma follows.
\end{proof}

For $a=0$, $\pv= (n,n)$ for $n>1$ odd, by \cite[Proposition 8.1]{Lehn_2011}, $\dim H^2(\spr)= n+1> \dim H^2(\cB^\vee)$, so the map $i_{\pv}^*$ is not surjective. Next, we consider the case where $\pv$ has $3$ parts. 

Let $\pv= (2a, 2b+1, 2b+1)$, we do not require $a>b$ here. We realize $Sp({2n})$ as $Sp(V_\pv)$ where $V_\pv$ is a symplectic vector space. Choose a basis $v_1, v_2,...v_{2n}$ of $V_\pv$ such that the symplectic form is given by 
\begin{itemize}
    \item For $1\leqslant i<j \leqslant 2a$, $\langle v_i, v_j \rangle= (-1)^{i}$ if $i+j= 2a+1$, and $0$ if $i+j\neq 2a+1$.
    \item For $2a+1\leqslant i< j \leqslant 2n$, $\langle v_i, v_j \rangle= (-1)^{i}$ if $i+j= 2a+1+ 2n$, and $0$ if $i+j\neq 2a+1+2n$.
\end{itemize}

 With respect to this basis, $e$ is a matrix having $3$ Jordan blocks of sizes $2a, 2b+1, 2b+1$. We have that the kernel of $e$ is $W=$Span$(v_1, v_{2a+1}, v_{2a+2b+2})$. We have a projection $\pi: \cB^\vee\rightarrow \PP(V_\pv)$ sending a flag $U^\bullet$ to $U^1$. The restriction of $\pi$ to $\spr$ gives a projection $\pi_{e^\vee}: \spr\rightarrow \PP(W)\subset \PP(V_\pv)$. 

\begin{prop} \label{inductive non-surjectivity}\leavevmode
\begin{enumerate}
    \item For $1\leqslant a<b$, the map $i^{*}_{(2a,2b+1,2b+1)}: H^*(\cB^\vee)\rightarrow H^*(\cB_{(2a,2b+1,2b+1)})$ is not surjective.
    \item Consider $a>b\geqslant 1$. Then $i^{*}_{(2a,2b+1,2b+1)}$ is surjective only if $i^{*}_{(2a-2,2b+1,2b+1)}: H^*(\cB_{2n-2})\rightarrow H^*(\cB_{(2a-2,2b+1,2b+1)})$ is surjective.
\end{enumerate}
   
\end{prop}
\begin{proof}
   \textbf{Part 1}. Write $W_1$ for the vector space Span$(v_{2a+1}, v_{2a+2b+2})$. The projection $\pi_{e^\vee}: \spr\rightarrow \PP(W)$ is a locally trivial fibration with the fibers as follows (see e.g. \cite[Section 2.1]{Shoji1983} or \cite[Section 5]{Kim2018})
   \begin{itemize}
       \item If $U^1\in \PP(W_1)$, then $\pi_{e}^{-1}(U^1)$ is isomorphic to $\cB_{(2a, 2b, 2b)}$.
       \item If $U^1\notin \PP(W_1)$, then $\pi_{e}^{-1}(U^1)$ is isomorphic to $\cB_{(2a-2, 2b+1, 2b+1)}$.
   \end{itemize}   
   Now from \cite[Theorem 3.9]{DLP}, the two varieties $\cB_{(2a+1, 2b-2, 2b-2)}$ and $\cB_{(2a-1, 2b, 2b)}$ admit affine pavings. Consider $\cB_{(2a, 2b, 2b)}$ as a subvariety of $\cB_{(2a,2b+1,2b+1)}$ parametrizing the flags having $U^1= \CC v_{2a+1}$. Write $s_{2a+1}$ for this embedding. It satisfies the condition that the complement $\cB_{(2a,2b+1,2b+1)}\setminus s_{2a+1}(\cB_{(2a, 2b, 2b)})$ admits an affine paving. Hence, the pullback map $s_{2a+1}^*: H^*(\cB_{(2a,2b+1,2b+1)})\rightarrow H^*(\cB_{(2a, 2b, 2b)})$ is surjective (see the argument in \cite{Hotta1977} for type A).
   Next, write $\cB_{2n-2}$ for the flag variety of $\fsp_{2n-2}$.  Similarly to the above, we have an embedding $s_{2a+1}: \cB_{2n-2}\rightarrow \cB_{2n}$. Then $s_{2a+1}$ induces a surjective map on the level of cohomology. Therefore, we have the following commutative diagram.
   \begin{center}
            \begin{tikzcd}
H^*(\cB_{2n}) \arrow{d}{i^{*}_{(2a,2b+1,2b+1)}} \arrow{r}{s_{2a+1}^*} & H^*(\cB_{2n-2}) \arrow{d}{i^{*}_{(2a,2b,2b)}} \\
H^*(\cB_{(2a,2b+1,2b+1)}) \arrow{r}{s_{2a+1}^*} & H^*(\cB_{(2a,2b,2b)})
            \end{tikzcd}
            \end{center}
    In the above diagram, the surjectivity of $i^{*}_{(2a,2b+1,2b+1)}$ would imply the surjectivity of $i^*_{(2a,2b,2b)}$. However, for $a< b$, the pullback map $i^*_{(2a,2b,2b)}$ is not surjective due to the non-triviality of the action of the component group on $H^*(\cB_{(2a,2b,2b)})$. Therefore, $i^{*}_{(2a,2b+1,2b+1)}$ is not surjective.

    \textbf{Part 2}. The details are similar to Part 1.  In this case, the projection $\pi_{e^\vee}: \spr\rightarrow \PP(W)$ is a locally trivial fibration with the fibers as follows.
   \begin{itemize}
       \item If $U^1=[\CC v_1]$, then $\pi^{-1}(U^1)$ is isomorphic to $\cB_{(2a-2, 2b+1, 2b+1)}$.
       \item If $U^1\in \PP(W)\setminus [\CC v_1]$, then $\pi^{-1}(U^1)$ is isomorphic to $\cB_{(2a, 2b, 2b)}$.
   \end{itemize} 
   Let $s_1$ be the embedding $\cB_{2n-2}\rightarrow \cB_{2n}$ parametrizing the flags with $U^1=\CC v_1$. Then the conclusion of Part 2 is immediate from the following commutative diagram, where both horizontal maps are surjective.
   \begin{center}
            \begin{tikzcd}
H^*(\cB_{2n}) \arrow{d}{i^{*}_{(2a,2b+1,2b+1)}} \arrow{r}{s_{1}^*} & H^*(\cB_{2n-2}) \arrow{d}{i^{*}_{(2a-2,2b+1,2b+1)}} \\
H^*(\cB_{(2a,2b+1,2b+1)}) \arrow{r}{s_{1}^*} & H^*(\cB_{(2a-2,2b+1,2b+1)})
            \end{tikzcd}
            \end{center}
\end{proof}
The next lemma will finish the proof of part (1) in \cref{surjectivity}. In particular, we prove that $i^*_{(2a+4, 2a+1, 2a+1)}$ is not surjective for $a\geqslant 1$. Combining with \cref{inductive non-surjectivity}, we get that $i^*_{(2a, 2b+1, 2b+1)}$ is not surjective when $a -2\geqslant b\geqslant 1$. 

\begin{lemma} \label{3-row}
    The map $i_{(2a+4,2a+1,2a+1)}^*: H^*(\cB^\vee)\rightarrow H^*(\spr)$ is not surjective for $a\geqslant 1$.
\end{lemma}
\begin{proof}
    The idea of this proof is to find a torus $T$ that acts on $\cB^\vee$ and $\spr$ so that the map $H^*({\cB^\vee}^T)\rightarrow H^*(\spr^T)$ is not surjective. In particular, we first choose $T$ and then point out a connected component $\cF_\alpha$ of ${\cB^\vee}^T$ so that $\cF_\alpha\cap \spr^T$ consists of two isolated points. The proof is technical and the general idea is similar to the proof of \cref{inductive non-surjectivity}. All the key points can be seen in \cref{example proof} below.

    Recall that $e^\vee$ is regular in $\fl^\vee= \fsp({2a+4})\times \fgl({2a+1})$. Let $T_{e^\vee}\subset G^\vee$ be the connected centralizer of $\fl^\vee$, then $T_{e^\vee}\simeq \CC^\times$. Next, by the Jacobson-Morozov theorem, we have an $\fs\fl_2$-triple $(e^\vee,h^\vee,f^\vee)$. Let $T_{h^\vee}\subset G^\vee$ be the torus with Lie algebra $\CC h^\vee$, this is the $\CC^\times$-action we consider in the main text. Consider the diagonal embedding $\CC^\times\hookrightarrow T_{e^\vee}\times T_{h^\vee}$, denote its image be $T^\vee$.

    We consider a matrix realization of $e^\vee$ as follows. Let $V^0$ be a symplectic vector space of dimension $2a+4$ with a basis $v^{0}_{2i-1}$ for $-1-a\leqslant i\leqslant a+2$. The symplectic form $\langle, \rangle^0$ is as follows.
    \begin{itemize}
        \item For $-3-2a\leqslant i<j \leqslant 2a+3$, $i,j$ odd, $\langle v^{0}_i, v^{0}_j \rangle^0= \delta_{i,-j}(-1)^{i}$.
        \item For $-3-2a\leqslant i>j \leqslant 2a+3$, $i,j$ odd, $\langle v^{0}_i, v^{0}_j \rangle^0= -\delta_{i,-j}(-1)^{i}$.
    \end{itemize}
    For $j= \pm 1$, let $V^j$ be a vector space of dimension $2a+1$ with a basis $v^{j}_{2i}$, $-a\leqslant  i\leqslant a$. We give $V^1\oplus V^{-1}$ a symplectic form $\langle , \rangle^{|1|}$ as follows.
    \begin{itemize}
        \item For $-2a\leqslant i<j \leqslant 2a$, $i,j$ even, $\langle v^{a}_i, v^{b}_j \rangle^{1}= \delta_{i,-j}\delta_{a,-b}(-1)^{i}$.
        \item For $-2a\leqslant i>j \leqslant 2a$, $i,j$ even, $\langle v^{a}_i, v^{b}_j \rangle^{1}= -\delta_{i,-j}\delta_{a,-b}(-1)^{i}$.
    \end{itemize}
    Let $V_\pv= V^{1}\oplus{-1} \oplus{V^0}$. Then we can realize $\fg^\vee= \fsp(V_\pv)$ so that the restrictions $e^\vee|_{V^i}$ are of Jordan forms. Moreover, $T_{e^\vee}$ and $T_{h^\vee}$ naturally acts on $V_\pv$ so that each vector $v_{i}^{a}$ is an eigenvalue of $T_{e^\vee}\times T_{h^\vee}$. In particular, the weight of the vector $v_{i}^{a}$ is $(a, i)$. As a consequence, the $T$-weight of $v_{i}^{a}$ is $i+a$. Next, we give an example of how the proof works below.
    \begin{example} \label{example proof}
     Consider $a=1$, so $\pv= (6,3,3)$. We express the chosen basis in terms of diagrams so that each vector corresponds to one box. The action of $e$ moves a box to the left. The content of each box is the weight of a certain torus.  
     $$\ytableausetup{boxsize= 3em}
\begin{ytableau}
0 & 0 &  0 & 0 & 0 & 0 \\
1 & 1 &1 \\
-1 &-1 &-1 \\
\end{ytableau} \quad \quad \quad 
\begin{ytableau}
5 & 3 &  1 & -1 & -3 & -5 \\
2 & 0 &-2 \\
2 &0 &-2\\
\end{ytableau}
$$
Here, the diagram on the left (resp. right) records the weight of $T_{e^\vee}$ (resp. $T_{h^\vee}$) on the basis vectors. Next, we give the diagram for the weight of $T$. Because we care about $\spr^T$, we will arrange the box so that each column is a $T$-eigenspace.
 $$\ytableausetup{boxsize= 3em}
\begin{ytableau}
5 & 3 &  1 & -1 & -3 & -5 \\
\none & 3 &  1 & -1 & \none & \none \\
\none & \none &  1 & -1 & -3 & \none \\
\end{ytableau}$$
Then $V_\pv$ has a $T$-weight decomposition $V_{-5}\oplus...\oplus V_{-5}$. For each $T$-stable isotropic flag $U^\bullet$, we represent it as $(u_1,...,u_6)$ where $u_i\in V_j$ for some $j$. For $\alpha= (p_1,...,p_6)$, write $\cF_\alpha\subset \cB^T$ for the subvariety that parametrizes the flags $U^\bullet$ so that the $T$-weights of $u_i$ are $p_i$. Assume $\cF_\alpha$ is nonempty, then it is a connected component of $\cB^T$. Moreover, $\cF_\alpha$ is isomorphic to the flag variety of $\fgl(V_5)\times\fgl(V_3)\times \fgl(V_1)= \fgl(1)\times \fgl(2)\times \fgl(3)$. These descriptions of $\cF_\alpha$ follow from the discussion before \cref{fixed point flag variety}.

Now, consider $\alpha= (3,1,-1,5,3,1)$. Let $X_\alpha$ be the intersection $\cF_\alpha\cap \spr^T$. A flag $(u_1,...,u_6)\in X_\alpha$ must satisfy the following: 
\begin{enumerate}
    \item $u_1\in V_3\cap \Ker(e)$, so $\CC u_1= \CC v_{2}^{1}$.
    \item $e(\CC u_3)= \CC u_2$, $\langle u_3, u_2\rangle = 0$ and $u_2\in\Span (v^{-1}_{2}, v_{0}^{1})$. Consequently $\CC u_2$ is $\CC v^{-1}_{2}$ or $\CC v_{0}^{1}$.
    \item $\CC u_6= \CC v_{1}^{0}$, $\CC u_5= e \CC u_6$ and $\CC u_4= e \CC u_5$.
\end{enumerate}
This means that $X_\alpha$ consists of $2$ points, and the pullback $H^*(\cF_\alpha)\rightarrow H^*(X_\alpha)$ is not surjective.
    \end{example}
We return to the proof of the general case. Now $T$ acts on $V_\pv$ with weights $2a+3,...,-2a-3$. The dimensions of the weight spaces are $\dim V_{2a+3}= 1, \dim V_{2a+1}=2$ and $\dim V_{2a-1}=...=\dim V_1= 3$. We use the same notation $X_\alpha$ and $\cF_\alpha$ as in $\cref{example proof}$. Let $\alpha= (2a+1, 2a-1,...,1-2a, 2a+3,2a+1,...,1)$. Then $X_\alpha$ consists of two points that correspond to two isotropic lines in $V_{1-2a}$ that are annihilated by $e^{2a+1}$. The variety $\cF_\alpha$ is connected, so the map $H^*(\cF_\alpha)\rightarrow H^*(X_\alpha)$ is not surjective.
\end{proof}
Next, we state a conjecture about when $i_{\pv}^*$ is surjective. 
\begin{conj} \label{surj conjecture}
The necessary conditions in \cref{surjectivity} and \cref{surjectivity B C D} are sufficient. In other words, $i_{\pv}^*$ is surjective for the following cases.
\begin{enumerate}
    \item $\fg^\vee= \fsp(2n)$, the partition $\pv$ takes the form $((2a+1)^{2d_1},(2a)^{2d_2+1},(2a-1)^{2d_3}, 1^{2d_4})$ for some $a, d_1,...,d_4\geqslant 0$.
    \item $\fg^\vee= \fs\fo(2n+1)$, the partition $\pv$ takes the form $((2a+2)^{2d_1}, (2a+1)^{2d_2+1}, (2a)^{2d_3})$ for some $a, d_1, d_2, d_3\geqslant 0$.
    \item $\fg^\vee= \fs\fo(2n)$, the partition $\pv$ takes one of the following forms.
    \begin{itemize}
        \item $((2a+2)^{2d_1}, (2a+1)^{2d_2}),$ 
        \item $(2a+1)^{2d_1}, (2a)^{2d_2})$,
        \item $((2a+3)^{2d_1+1}, (2a+2)^{2d_2}, (2a+1)^{2d_3+1}),$
        \item $((2a+1)^{2d_1+1},(2b+1)^{2d_2+1}),$
        \item $((2a+1)^{2d_1+1},(2a)^{2d_2}, 2^{2d_3},1^{2d_4+1}),$
    \end{itemize}
    for some $a, b, d_1, d_2, d_3, d_4\geqslant 0$.
\end{enumerate}
\end{conj}
The next proposition gives some evidence for \cref{surj conjecture}. A more general approach to prove this conjecture is explained in \cite{hoang2}.
\begin{prop} \label{b=1} \leavevmode
\begin{enumerate}
    \item Consider $\fg^\vee= \fs\fp(2n)$. The map $i_{(2a,1^{2n-2a})}^*: H^*(\cB^\vee)\rightarrow H^*(\cB_{(2a,1^{2n-2a})})$ is surjective.
    \item Consider $\fg^\vee= \fs\fo(2n)$. The map $i_{(2a+1,1^{2n-2a-1})}^*: H^*(\cB^\vee)\rightarrow H^*(\cB_{(2a+1,1^{2n-2a-1})})$ is surjective.
\end{enumerate}
    
\end{prop}
\begin{proof} We give the proof of Part (1), the case $\fg^\vee$ of type D is similar. We use induction  on $n$. The statement is trivial for $n= 1$. Consider $n>1$. Similarly to the previous proofs, we have a vector space $V_\pv$ such that $e^\vee \in \fsp(V_\pv)$. Let $W$ be the kernel of $e^\vee\colon V_\pv\rightarrow V_\pv$. Let $W_1$ be the image of $(e^\vee)^{2a-1}$, it is a one-dimensional subspace of $W$. Write $[W_1]$ for the class of $W_1$ in $\PP(W)$. The fibers of the projection $\pi_{e^\vee}\colon \spr \rightarrow \PP(W)= \PP^{2n-2a}$ is as follows.
    \begin{itemize}
        \item If $U^1=[W_1]$ , then $\pi_{e^\vee}^{-1}(U^1)$ is isomorphic to $\cB_{(2a-2,1^{2n-2a})}$.
        \item If $U^1\neq[W_1]$ , then $\pi_{e^\vee}^{-1}(U^1)$ is isomorphic to $\cB_{(2a,1^{2n-2a-2})}$.
    \end{itemize}
    Write $\cB_{W,2n}$ for the subvariety $\pi^{-1}(W)$. We have an embedding $s_{1}\colon \cB_{2n-2}\hookrightarrow \cB_{2n}$ as the subvariety of the flags $U^\bullet$ that have $U^1=$ $[W_1]$. This embedding has the property $s_{1}(\cB_{(2a,1^{2n-2a})})\simeq \cB_{(2a-2,1^{2n-2a})}$. And we have a commutative diagram of compactly supported cohomology.
    \begin{equation*}
         \begin{tikzcd} 
            0\arrow{r}& H_{c}^*(\cB_{W,2n}\setminus s_1(\cB_{2n-2})) \arrow{d}{j^*} \arrow{r} & H_{c}^*(\cB_{W,2n}) \arrow{d}{i^{*}_{(2a,1^{2n-2a})}} \arrow{r}{s_1^{*}}&  H_{c}^*(\cB_{2n-2})\arrow{d}{i^*_{(2a-2,1^{2n-2a})}} \arrow{r}& 0\\
            0\arrow{r}& H^*_{c}(\cB_{(2a,1^{2n-2a})}\setminus s_1(\cB_{(2a-2,1^{2n-2a})})) \arrow{r} &  H_{c}^*(\cB_{(2a,1^{2n-2a})})\arrow{r}{s_1^{*}} &  H_{c}^*(\cB_{(2a-2,1^{2n-2a})}) \arrow{r} & 0
        \end{tikzcd}
    \end{equation*}
   The map $j^*$ in the above diagram can be described as follows. We have a stratification $\PP(W)\setminus [W_1]=\bA^{2n-2a}\sqcup \bA^{2n-2a-1}\sqcup\ldots$ $\sqcup \bA^1$ such that $\pi_{e^\vee}^{-1}(\bA^l)= \bA^l\times \cB_{(2a, 1^{2n-2a-2})}$ for $1\leqslant l\leqslant 2n-2a$. Hence, we have $$H^*_{c}(\cB_{(2a,1^{2n-2a})}\setminus s_1(\cB_{(2a-2,1^{2n-2a})}))= \oplus_{l=1}^{2n-2a} H_c^*(\bA^l\times \cB_{(2a, 1^{2n-2a-2})})$$
   We have a similar decomposition $H_{c}^*(\cB_{W,2n}\setminus s_1(\cB_{2n-2}))= \oplus_{l=1}^{2n-2a} H_c^*(\bA^l\times \cB_{2n-2})$. And the map $j^*$ on each direct summand is induced from $i^*_{(2a,1^{2n-2a-2})}\colon H^*(\cB_{2n-2})\rightarrow \cB_{(2a, 1^{2n-2a-2})}$. By the induction hypothesis, both $i^*_{(2a,1^{2n-2a-2})}$ and $i^*_{(2a-2,1^{2n-2a})}$ are surjective. Therefore, $j^*$ and $i^*_{(2a-2,1^{2n-2a})}$ are surjective, and we have $i^*_{(2a,1^{2n-2a})}$ surjective by the five lemma.   
\end{proof}
    

\section{Surjectivity of the Cartan comoment map}\label{app_surj_rhs}
In this section, we study the surjectivity of the morphism 
\begin{equation}\label{morph_that_study}
\mathbb{C}[\mathfrak{h}^*,\hbar] \otimes_{Z_\hbar} \mathbb{C}[\mathfrak{h}^*,\hbar] \rightarrow \cC_\nu(\mathcal{U}_\hbar(\mathfrak{g},\mathfrak{p}))
\end{equation} 
being specialized to particular points $(\la,\hbar_0) \in \mathfrak{h}^* \oplus \mathbb{C}$. We will restrict ourselves to the case $\hbar_0 \neq 0$, so, without losing the generality, we can assume that $\hbar_0=1$ (use that the morphism \ref{morph_that_study} is graded).  

\begin{rmk}
Note that the algebra $\mathcal{U}_\hbar(\mathfrak{g},\mathfrak{p})$ as well as the homomorphism (\ref{morph_that_study}) do not depend on the choice of the parabolic $\mathfrak{p} \supset \mathfrak{b}$. 
\end{rmk}


All of the results of this section follow directly from \cite[Section 3]{HolandPolo} and references therein. 

Pick $\lambda \in \mathfrak{X}(\mathfrak{l})$ that is  $\mathfrak{p}$-antidominant (note that for every $\la \in \mathfrak{X}(\mathfrak{l})$ there exists $\mathfrak{p}$ such that $\la$ is $\mathfrak{p}$-antidominant). Let $\mathfrak{h} \subset \mathfrak{b} \subset \mathfrak{p}$ be a Borel subalgebra (note that $\la$ is automatically $\mathfrak{b}$-antidominant). Recall the element $\widetilde{\la}=\la-\rho_{\mathfrak{b}}(\mathfrak{l})$ that first appears in Section \ref{quant_G_mod_P}.

The following proposition holds by \cite[Section 3.6]{HolandPolo} and references therein.
\begin{prop}
The natural morphism $\mathcal{U}(\mathfrak{g}) \rightarrow \mathcal{U}_\lambda(\mathfrak{g},\mathfrak{p})$ is surjective if $\widetilde{\la}$ is 
$\mathfrak{p}$-antidominant.  
\end{prop}

\begin{cor}
If both $\la \in \mathfrak{X}(\mathfrak{l})$ and $\widetilde{\la} \in \mathfrak{h}^*$ are $\mathfrak{p}$-antidominant, then the morphism (\ref{morph_that_study}) is surjective.  
\end{cor}

\begin{rmk}
Let us remark again that for every $\la$ we can find $\mathfrak{p}$ such that $\la$ is $\mathfrak{p}$-antidominant (moreover, for $\mathfrak{l}$-regular integral $\la$ such $\mathfrak{p}$ is unique and is given by $\mathfrak{p}=\bigoplus_{\langle \la,\beta^\vee\rangle \leqslant 0}\mathfrak{g}_\alpha$). Although the condition that $\widetilde{\la}$ is $\mathfrak{p}$-antidominant is not automatic and may not be satisfied in general. 
\end{rmk}

\begin{warning}
Note that despite the fact that we know that the surjectivity of the morphism $\mathcal{U}(\mathfrak{g}) \rightarrow \mathcal{U}_\lambda(\mathfrak{g},\mathfrak{p})$ sometime fails (see Example \ref{surj_Phi_hbar_nonzero}), we are still not sure if it fails after we pass to the Cartan subquotients. 
\end{warning}

Let us finish this section by making the following comment. Note that the condition that $\widetilde{\la}$ is $\mathfrak{p}$-antidominant is precisely one of two conditions that we impose in Proposition to have abelian localization for $(\la,T^*\mathcal{P})$. Moreover, note that the condition that $\widetilde{\la}$ is $\mathfrak{p}$-antidominant tells us that the global sections functors 
\begin{equation*}
\Gamma \colon \mathcal{D}_\la(G/P)\text{-mod} \rightarrow \mathcal{U}_\la(\mathfrak{g},\mathfrak{p})\text{-mod},~
\Gamma \colon \mathcal{D}_{\widetilde{\la}}(G/B)\text{-mod} \rightarrow \mathcal{U}_{\widetilde{\la}}(\mathfrak{g})\text{-mod}
\end{equation*}
are exact. It would be great if there is a direct argument deducing the surjectivity of $D_{\widetilde{\la}}(G/B) \rightarrow D_{\la}(G/P)$ from the exactness of $\Gamma$ (compare with \cite[Proposition 2.14]{HolandPolo}).

Note that it is always true that there exists {\emph{some}} parabolic $\mathfrak{p}$ with Levi $\mathfrak{l}$ such that $\la$ is 
$\mathfrak{p}$-antidominant. On the other hand, this $\mathfrak{p}$ might not be standard.



\section{Hikita-Nakajima conjecture for type $A$ parabolic Slodowy varieties}\label{app_HN_type_A}
Assume now that $\mathfrak{g}=\mathfrak{sl}_n$. We claim that in this case, Theorem \ref{main_th_weak_hikita} implies the (equivariant) Hikita-Nakajima conjecture for the pair $\widetilde{S}(\mathfrak{q},\mathfrak{p})$, $\widetilde{S}(\mathfrak{p},\mathfrak{q})$. In fact, the classical Hikita conjecture for parabolic Slodowy slices is already proved in \cite[Appendix A]{Hikita}. The proof by Hikita requires hard techniques from algebraic geometry (see the discussion below for more details). The main purpose of this appendix is to give a transparent ``Lie theoretic'' proof of the Hikita-Nakajima conjecture for type $A$ parabolic Slodowy varieties. To this end, we need to check that 
\begin{itemize}
    \item[(a)]  The homomorphism $\CC[\mathfrak{h}^*,\hbar]^{W_M} \otimes \CC[\mathfrak{h}^*,\hbar] \rightarrow \cC_\nu(\mathcal{A}_{\hbar,\mathfrak{X}(\mathfrak{l})}(\widetilde{S}(\mathfrak{q},\mathfrak{p})))$ is surjective. 
    \item[(b)] The restriction homomorphism $H^*_{T^\vee \times \CC^\times}(T^*\mathcal{Q}^\vee_-) \rightarrow H^*_{Z_L \times \CC^\times}(\widetilde{S}(\mathfrak{p}^\vee_-,\mathfrak{q}^\vee_-))$ is surjective. 
    \item[(c)] Algebra $\cC_\nu(\mathcal{A}_{\hbar,\mathfrak{X}(\mathfrak{l})}(\widetilde{S}(\mathfrak{q},\mathfrak{p})))$ is flat over $\CC[\mathfrak{X}(\mathfrak{l}),\hbar]$.
\end{itemize}

By graded Nakayama lemma, it is enough to check that:
\begin{itemize}
    \item[(a)]  The natural morphism $\on{Spec}\CC[\widetilde{S}(\mathfrak{q},\mathfrak{p})] \rightarrow \mathcal{N}$ is a closed embedding. 
    \item[(b)] The restriction homomorphism $H^*(T^*\mathcal{Q}^\vee_-) \rightarrow H^*(\widetilde{S}(\mathfrak{p}^\vee_-,\mathfrak{q}^\vee_-))$ is surjective. 
    \item[(c)] The dimension of the algebra of functions of $\on{Spec}\CC[\widetilde{S}(\mathfrak{q},\mathfrak{p})]^{Z_M}$ does not exceed $|\widetilde{S}(\mathfrak{q},\mathfrak{p})^{Z_M}|=|(W_M\backslash W/W_L)^{\mathrm{free}}|$ (see \cref{prop fixed point parabolic Slodowy}). This condition is equivalent to the weak flatness condition (\ref{weakflat}).
    
    Note that $\dim H^*(\widetilde{S}(\mathfrak{p}^\vee_-,\mathfrak{q}^\vee_-))$ is precisely $|(W_M\backslash W/W_L)^{\mathrm{free}}|$ because its fixed point locus under a torus action consists of $|(W_M\backslash W/W_L)^{\mathrm{free}}|$ points. So, we can substitute (c) with (c') below.
    \item[(c')]The dimension of the algebra of functions of $\on{Spec}\CC[\widetilde{S}(\mathfrak{q},\mathfrak{p})]^{Z_M}$ does not exceed $\dim H^*(\widetilde{S}(\mathfrak{p}^\vee_-,\mathfrak{q}^\vee_-))$.
\end{itemize}

All of these statements are well-known; we give arguments and references for completeness. 

Part (a) follows from \cite[Lemma 5.2.1 (5)]{Losev_isofquant}, namely, it follows from the proof of Lemma 5.2.1 (5) in loc. cit. that $\on{Spec}\CC[\widetilde{S}(\mathfrak{q},\mathfrak{p})]=S(e_{\mathfrak{q}}) \cap \overline{\mathbb{O}}_{\mathfrak{p}}$ that is clearly closed in $\mathcal{N}$ (recall that $\overline{\mathbb{O}}_{\mathfrak{p}}$ is the image of $T^*\mathcal{P} \rightarrow \mathcal{N}$). 

Part (b) follows from the classical result of Spaltenstein; see, for example, \cite[Corollary 2.5]{bo}. 

Finally, part (c') follows from a combination of \cite[Theorem 1.1]{bo} and \cite[Appendix A]{Hikita}. The highlight of our deformation approach is that we use only the elementary parts of the arguments in \cite{bo} and \cite{Hikita}. Details are explained below.

First, let us briefly recall the statements involved in the proof of the classical Hikita conjecture in \cite[Appendix A]{Hikita}. 

Let $\lambda$ (resp. $\mu$) be the partition (composition) that comes from $\fp$ (resp. $\fq$). In \cite{Brundan2008}, Brundan introduces a partial coinvariant algebra $C^{\mu}_{\lambda}$ attached to each pair $(\lambda, \mu)$. The algebra $C^{\mu}_{\lambda}$ is a quotient of $\CC[\fh]^{W_M}$ by an ideal $I_{\lambda}^{\mu}$ generated by certain sets of partial symmetric polynomials. 

One of the main results of \cite{bo} is an isomorphism of algebras $H^*(\widetilde{S}(\mathfrak{p}^\vee_-,\mathfrak{q}^\vee_-))\simeq C^{\mu}_{\lambda}$. In particular, they prove $\dim C^{\mu}_{\lambda}\leqslant \dim H^*(\widetilde{S}(\mathfrak{p}^\vee_-,\mathfrak{q}^\vee_-))$ in \cite[Sections 2,3]{bo}. The argument for this part is combinatorial. The authors construct a spanning set of $C^{\mu}_{\lambda}$ labeled by column strict tableaux that satisfy certain properties. Then they prove $H^*(\widetilde{S}(\mathfrak{p}^\vee_-,\mathfrak{q}^\vee_-))\simeq C^{\mu}_{\lambda}$ using the construction of the Springer representations by perverse sheaves.

The main result of \cite[Appendix A]{Hikita} is an isomorphism of algebras $C_{\lambda}^{\mu}\simeq \CC[\widetilde{S}(\mathfrak{q},\mathfrak{p})]^{Z_M}$. First, Hikita generalizes the argument in \cite[Lemma 1]{Tanisaki} and obtains a surjection $C_{\lambda}^{\mu}\twoheadrightarrow  \CC[\widetilde{S}(\mathfrak{q},\mathfrak{p})]^{Z_M}$. This step uses only techniques related to minors and elementary divisors of matrices. The fact that this surjection is an isomorphism is a consequence of \cite[Theorem 4.6]{Weyman} on the defining ideals of nilpotent orbits in $\fgl_n$. This result is a hard theorem that requires sophisticated geometric techniques in calculating syzygies. 

Recall that in our approach the last step is to show $\dim \CC[\widetilde{S}(\mathfrak{q},\mathfrak{p})]^{Z_M}\leqslant \dim H^*(\widetilde{S}(\mathfrak{p}^\vee_-,\mathfrak{q}^\vee_-))$. This inequality follows from the  ``easy" parts in \cite{bo} and \cite{Hikita} mentioned above. In other words, we only use
$$\dim \CC[\widetilde{S}(\mathfrak{q},\mathfrak{p})]^{Z_M} \leqslant \dim C_\lambda^{\mu} \leqslant \dim H^*(\widetilde{S}(\mathfrak{p}^\vee_-,\mathfrak{q}^\vee_-)).$$
Although our proof of the Hikita-Nakajima conjecture in this case is much less explicit, it proposes a more viable approach outside of type A. One main reason is that we do not have a complete generalization of \cite[Theorem 4.6]{Weyman} for nilpotent orbits of other types. The readers who are interested in some partial (conjectural) sets of generators for the defining ideals may look at \cite[Section 8.3]{Weymanbook}.

\begin{rmk}
Another approach to the proof of (c) is as follows. Recall that by \cite{maffei}, there exists the isomorphism between $\widetilde{S}(\mathfrak{q},\mathfrak{p})$ and a certain type $A$ Nakajima quiver variety $\widetilde{\mathcal{M}}_H$. This quiver variety, in turn, is isomorphic to the {\em{slice in the affine Grassmannian}} $\overline{\mathcal{W}}^\la_\mu$ corresponding to a pair of dominant coweights $\la$, $\mu$ (see \cite{MV}) and now the claim follows from \cite[Equation (20)]{KamnitzerTingleyWebsterWeeksYacobi} (note that the argument in loc. cit. is very elegant and does not use any combinatorics although it is based on some nontrivial representation theory).  
\end{rmk}




\section{Torus fixed points of parabolic Slodowy varieties}\label{app_torus_fixed}

In this section, we will describe explicitly the sets $Y^{T_X}$ and $Y_\la^{T_X}$ as well as the identification $Y_{\mathfrak{a}}^{T_X} \simeq Y^{T_X} \times \mathfrak{a}$ for $Y=\widetilde{S}(\mathfrak{q},\mathfrak{p})$ being the parabolic Slodowy variety, and $\mathfrak{a}=\mathfrak{X}(\mathfrak{l})$.

\subsection{When Slodowy variety has isolated torus fixed points} Let us start with arbitrary $\widetilde{S}(\chi,\mathfrak{p})$ (i.e., we are not assuming that $e$ is regular in some Levi). Recall that $S(\chi)=\chi+\mathfrak{z}_{\mathfrak{g}}(f)$.  
Recall  also the morphism $\widetilde{\pi}_{\mathfrak{p}} \colon T^*_{\mathfrak{X}(\mathfrak{l})}\mathcal{P} \rightarrow \mathfrak{g}^* \times_{\mathfrak{h}^*/W} \mathfrak{X}(\mathfrak{l})$. 
Then, $ \widetilde{S}_{\mathfrak{X}(\mathfrak{l})}(\chi,\mathfrak{p})$ is equal to $\widetilde{\pi}_{\mathfrak{p}}^{-1}(S(\chi) \times_{\mathfrak{h}^*/W} \mathfrak{X}(\mathfrak{l})) \subset T^*_{\mathfrak{X}(\mathfrak{l})}\mathcal{P}$. 
Using the closed embedding $T^*_{\mathfrak{X}(\mathfrak{l})}\mathcal{P} \hookrightarrow \mathcal{P} \times \mathfrak{g}^*$, we can consider an element of $T^*_{\mathfrak{X}(\mathfrak{l})}\mathcal{P}$ as a pair $(\mathfrak{p}',x)$.

\begin{lemma}
 The set $\widetilde{S}(\chi)^{T_e}$ is finite if and only if $e$ is regular in $\mathfrak{z}_{\mathfrak{g}}(\mathfrak{t}_e)$.
\end{lemma}
\begin{proof}
Since $\widetilde{S}(\chi)$ is a $\CC^\times \times T_e$-equivariant symplectic resolution of $S(\chi)\cap \cN$, \cref{finite fixed points} implies that we only need to show that $(S(\chi)\cap \cN)^{T_e}=\{\chi\}$ if and only if $e$ is regular in $\mathfrak{m}:=\mathfrak{z}_{\mathfrak{g}}(\mathfrak{t}_e)$. Indeed, by the definitions, $(S(\chi)\cap \cN)^{T_e}=S(\chi) \cap \cN\cap \mathfrak{m}^*=S(\chi,\mathfrak{m})$ is the nilpotent Slodowy slice to $\chi$ in $\mathfrak{m}^*$, and $S(\chi,\mathfrak{m})=\{\chi\}$ if and only if $e$ is regular in $\mathfrak{m}$.
\end{proof}

From now on, we will restrict ourselves to the pairs $\widetilde{S}(\chi,\mathfrak{p})$ such that $\widetilde{S}(\chi)^{T_e}$ is finite. 


\begin{rmk}
By \cref{symplectic resolution iff one fixed point} this assumption is natural from the perspective of symplectic duality, namely it corresponds to the fact that the dual variety has a symplectic resolution.
\end{rmk}

Fix $T \subset B \subset P \subset G$. Let $Q \subset G$ be another parabolic subgroup containing $B$. Let $M \subset Q$ be the standard Levi subgroup of $Q$, and let $e \in \mathfrak{m}$ be a regular nilpotent element contained in $\mathfrak{b}_0:=\mathfrak{m} \cap \mathfrak{b}$.  Let $\mathfrak{n}_0 \subset \mathfrak{b}_0$ be the nilpotent radical. Following the convention of \cref{parabolic Slodowy subsection}, we denote the corresponding parabolic Slodowy variety $\widetilde{S}(\chi,\mathfrak{p})$ by $\widetilde{S}(\mathfrak{q},\mathfrak{p})$.

\subsection{Explicit description of the torus fixed points on $\widetilde{S}_{\la}(\mathfrak{q},\mathfrak{p})$}\label{app_descr_fixed_points_parabolic}
\subsubsection{Combinatorial preliminaries: alternative descriptions of the set of free $W_M \times W_L$-orbits on $W$}
Let $(W_M\backslash W/W_{L})^{\mathrm{free}} \subset W_M\backslash W/W_{L}$ be the set of {\em{free}} $W_M \times W_L$-orbits on $W_G$. For $w \in W = N_G(T)/T$, we will denote by $\dot{w} \in N_G(T)$ any representative of $w$. Let ${}^M(W/W_{L}) \subset W/W_L$ be the subset of $W/W_L$ defined as follows: 
\begin{equation*}
{}^M(W/W_{L}) = \{wW_L \in W/W_L\,|\, \dot{w}P\dot{w}^{-1} \cap M = B_0\}.
\end{equation*}

\begin{lemma}
The natural map ${}^M(W/W_{L}) \to W_M\backslash W/W_{L}$ induces a bijection 
 $$
 {}^M(W/W_{L})\iso (W_M\backslash W/W_{L})^{\mathrm{free}}.
 $$
\end{lemma} 
\begin{proof}
 Indeed, we have (see, for example, \cite[Appendix B]{losev_krylov_wc_for_gieseker})
 \begin{equation*}
 (W_M\backslash W/W_{L})^{\mathrm{free}}=\{[w] \in W_M\backslash W/W_{L}\,|\, w^{-1}(\Delta_M) \cap \Delta_{L} = \varnothing\},
 \end{equation*}
 and the condition $w^{-1}(\Delta_M) \cap \Delta_{L} = \varnothing$ is equivalent to the fact that the intersection $\dot{w}P \dot{w}^{-1} \cap M$ is a Borel subgroup of $M$. It remains to note that if $\dot{w}P \dot{w}^{-1} \cap M$ is Borel in $M$, then there exists the unique $w' \in W_Mw$ such that ${\dot{w}'}P({\dot{w}'})^{-1} \cap M = B_0$.
\end{proof}

Let us give another description of the set $(W_M\backslash W/W_{L})^{\mathrm{free}}$.  Borel subgroup $B$ defines the length function on $W$. Let ${}^MW^{\mathrm{sh}} \subset W$ be the subset of $w \in W$ such that $w \in W_Mw$ is the shortest (this is equivalent to $\dot{w}B\dot{w}^{-1} \cap M=B_0$). Let ${}^{\mathrm{lo}}W^{L} \subset W$ be the subset of $W$ such that $w \in wW_L$ is the longest  (equivalently, $\dot{w}^{-1}B\dot{w} \cap L = B_{L}^{-}$). Then the natural morphism ${}^MW^{\mathrm{sh}} \cap {}^{\mathrm{lo}}W^{L} \iso (W_M\backslash W/W_{L})^{\mathrm{free}}$ is an isomorphism.

\begin{rmk}
We can also naturally identify the sets above with $(W_M\backslash W)^L=\{W_Mw \in W_M \backslash W\,|\, \dot{w}^{-1}Q\dot{w} \cap L=B_L^{-}\}$.
\end{rmk}

\subsubsection{Description of the fixed points} Recall the identification $S(\chi,\mathfrak{m}) \iso \mathfrak{h}^*/W_M$.
For $\la \in \mathfrak{h}^*/W_M$, let $x_\la \in S(\chi,\mathfrak{m})$ be the corresponding element. Choose a preimage of $\lambda$ in $\fh^*$. Abusing the notations, we also denote it by $\lambda$. Using the identification $\mathfrak{m} \simeq \mathfrak{m}^*$, we can talk about the semisimple part $(x_\la)_{\mathrm{ss}} \in \mathfrak{m}^*$ of $x_\la$. It lies in the same $M$-orbit as $\la$. Let $m_\la \in M$ be such that $m_\la \cdot \la =(x_\la)_{\mathrm{ss}}$, and $m_\la^{-1} \cdot x_\la \in \mathfrak{n}_0^{\perp}$ (i.e., the corresponding element of $\mathfrak{m}$ lies in $\mathfrak{b}_0$).

\begin{rmk}
Note that the element $m_\la$ is {\emph{not}} unique. It is unique up to the right multiplication by $Z_{B_0}(\la)$.
\end{rmk}


Using the closed embedding $T^*_{\mathfrak{X}(\mathfrak{l})}\mathcal{P} \hookrightarrow \mathcal{P} \times \mathfrak{g}^*$, and the natural embedding $\widetilde{S}_{\la}(\mathfrak{q},\mathfrak{p})^{Z_M} \hookrightarrow T^*_{\la}\mathcal{P}$ we consider any point of $\widetilde{S}_{\la}(\mathfrak{q},\mathfrak{p})^{Z_M}$ as an element of $\cP\times \fg^*$.

We start with the following standard lemma to be used in the proof of Proposition \ref{descr_fixed_parabolic} below. 

\begin{lemma}\label{lemma_fixed_G_L}
We have  $(G/L)^{Z_M}=\bigsqcup_{w \in W_M \backslash W/W_L} MwL/L$.
\end{lemma}
\begin{proof}
The natural morphism $f\colon G/L \rightarrow G/P$ is a $Z_M$-equivariant $P/L$-bundle, so it induces the map $(G/L)^{Z_M} \rightarrow (G/P)^{Z_M}$. It follows from Proposition \ref{fixed pt of partial flag} that $(G/P)^{Z_L}=\bigsqcup_{w \in W_M \backslash W/W_L}MwP/P$. 
Let us describe the action of $Z_M$ on the fiber of $f$ over $mwP/P \in G/P$. This fiber identifies with  $P/L$ via the map 
\begin{equation}\label{ident_fiber_P_mod_L}
P/L \ni [p] \mapsto m\dot{w}pL/L \in f^{-1}(m\dot{w}L/L),
\end{equation}
note that the identification (\ref{ident_fiber_P_mod_L}) depends on the choice of $m \in M$.

Now, for $t \in Z_M$, its action on $P/L$ is via the conjugation by ${\dot{w}}^{-1}t\dot{w}$. Our goal is to show that $(P/L)^{Z_M}=(P \cap {\dot{w}}^{-1}M\dot{w})L/L$. Let $U_P \subset P$ be the unipotent radical. Note that the natural isomorphism $U_P \iso P/L$ is $T$-equivariant. So, it remains to describe the $Z_M$-fixed points of $U_P$ (where the $Z_M$-action is ``twisted'' by $\dot{w}^{-1}$ as above).

Clearly, $U_P$ is a $T$-invariant subvariety of $G$ and the $Z_M$-fixed points of $G$ w.r.t. the $\dot{w}^{-1}$-twisted action are $\dot{w}^{-1}M\dot{w}$. We conclude that $U_P^{Z_M} = U_P \cap \dot{w}^{-1}M\dot{w}$ and the claim follows.
\end{proof}

\begin{prop}\label{descr_fixed_parabolic}
For $\la \in \mathfrak{X}(\mathfrak{l})$, the points of $\widetilde{S}_{\la}(\mathfrak{q},\mathfrak{p})^{Z_M}$ are in bijection with ${}^M(W/W_{L})$, and the bijection is given by: 
\begin{equation}\label{exp_param_fixed_parab_slodowy}
{}^M(W/W_{L}) \ni [w] \mapsto (m_{w\la}\dot{w} \cdot \mathfrak{p},x_{w\la})\in \cP\times \fg^*.
\end{equation}
\end{prop}
\begin{proof}
It follows from Lemma \ref{fixed_pts_deform} that the morphism $\widetilde{S}(\mathfrak{q},\mathfrak{p})^{Z_M} \rightarrow \mathfrak{X}(\mathfrak{l})$ is the trivial fibration. We claim that for every  $[w] \in {}^M(W/W_L)$, the map (\ref{exp_param_fixed_parab_slodowy}) is well-defined and gives a section $\iota_{[w]}\colon \mathfrak{X}(\mathfrak{l}) \rightarrow \widetilde{S}(\mathfrak{q},\mathfrak{p})^{Z_M}$ of this fibration.  
Indeed, let us first of all note that the point $(m_{w\la}\dot{w} \cdot \mathfrak{p},x_{w\la})$ lies in $T^*_{\mathfrak{X}(\mathfrak{l})}\mathcal{P} \subset \mathcal{P} \times \mathfrak{g}^*$, to see that recall that $\dot{w}[\mathfrak{p},\mathfrak{p}]\dot{w}^{-1} \cap \mathfrak{m} = \mathfrak{n}_0$ and $m_{w\la}^{-1}x_{w\la}$  lies in $\mathfrak{n}_0^{\perp}$.  It follows that $m_{w\la}^{-1}x_{w\la} \in (\dot{w}[\mathfrak{p},\mathfrak{p}]\dot{w}^{-1})^{\perp}$. 
It also follows that $m_{w\la}\dot{w} \cdot \mathfrak{p}$ does not depend on the choice of $m_{w\la}$ (use that $\mathfrak{b}_0 \subset \dot{w} \cdot \mathfrak{p}$) so the map (\ref{exp_param_fixed_parab_slodowy}) is well-defined.

Recall the map $\widetilde{\pi}_{\mathfrak{p}}\colon T^*_{\mathfrak{X}(\mathfrak{l})}\mathcal{P} \rightarrow \widetilde{S}(\mathfrak{q},\mathfrak{p}) \times_{\mathfrak{h}^*/W} \mathfrak{X}(\mathfrak{l})$.
We claim that $\widetilde{\pi}_{\mathfrak{p}}(m_{w\la}\dot{w} \cdot \mathfrak{p},x_{w\la}) \in S(\chi,\mathfrak{m}) \times_{\mathfrak{h}^*/W} \mathfrak{X}(\mathfrak{l})$.  Indeed, directly from the definitions, this image is equal to $[(x_{w\la},\la)]$ (use that $\dot{w}^{-1}m_{w\la}^{-1} \cdot (x_{w\la})_{\mathrm{ss}}=\la$). So, we see that $(m_{w\la}\dot{w} \cdot \mathfrak{p},x_{w\la}) \in \widetilde{S}_{\la}(\mathfrak{q},\mathfrak{p})$. Clearly, this point is fixed by $Z_M$. So, we obtain the well-defined map $\iota_{[w]}\colon \mathfrak{X}(\mathfrak{l}) \rightarrow \widetilde{S}(\mathfrak{q},\mathfrak{p})^{Z_M}$. 
It follows from the definitions that the map $\iota_{[w]}$ is continuous, hence, algebraic (use Lemma \ref{fixed_pts_deform}).

So, we have constructed a collection of sections $\iota_{[w]}$.
Clearly, it is then  enough to prove for {\emph{some}} $\la \in \mathfrak{X}(\mathfrak{l})$ that $\iota_{[w]}(\la)$ are all distinct and $\widetilde{S}(\mathfrak{q},\mathfrak{p})^{Z_M}=\{\iota_{[w]}(\la)\,|\, [w] \in {}^M(W/W_L)\}$.

Let $\lambda \in \mathfrak{X}(\mathfrak{l})$ be such that $Z_{\mathfrak{g}}(\la)=\mathfrak{l}$ (i.e., $\la$ is $\mathfrak{l}$-regular, see Section \ref{sec_parabolic_slodowy}).
We use identifications $\mathfrak{g}^* \simeq \mathfrak{g}, \mathfrak{h}^* \simeq \mathfrak{h}, \mathfrak{X}(\mathfrak{l}) \simeq \mathfrak{z}(\mathfrak{l})$ given by the Killing form on $\fg$.

The variety $\widetilde{S}_\la(\mathfrak{q},\mathfrak{p})$ is smooth, affine and isomorphic to the intersection $S(e) \cap \mathbb{O}_{\la}$, where $\mathbb{O}_\lambda=G.\lambda$ (see Lemma \ref{lem_fiber_parab_slod_gen}). We have $G/L \iso \mathbb{O}_{\la}$ given by $[g] \mapsto g\cdot \la$, so, by Lemma \ref{lemma_fixed_G_L}:
\begin{equation*}
\mathbb{O}_{\la}^{Z_M}=\bigsqcup_{w \in W_M \backslash W/W_L}Mw \cdot \la = \bigsqcup_{w \in W_M \backslash W/W_L} \mathbb{O}_{w\la}(\mathfrak{m}).
\end{equation*}
Recall also that $S(e_{\mathfrak{m}})^{Z_M}=S(e_{\mathfrak{m}},\mathfrak{m})$. We identify ${}^M(W/W_L) \iso (W_M \backslash W / W_L)^{\mathrm{free}}$. Then, $\la \in \mathfrak{l} \subset \mathfrak{h}$ has the following property that we will be using: for $\alpha \in \Delta_G$, $\langle \la,\alpha \rangle = 0$ iff $\alpha \in \Delta_{L}$.

Element $e_{\mathfrak{m}} \in \mathfrak{m}$ is regular, so $S(e_{\mathfrak{m}},\mathfrak{m}) \cap \mathbb{O}_{w\la}(\mathfrak{m})$ consists of one point if $w\la$ is $\mathfrak{m}$-regular and is empty otherwise (this is the classical result of Kostant).
Note now that $w\la$ is $\mathfrak{m}$-regular if and only if $\langle w\la, \al\rangle = \langle \la, w^{-1}\al\rangle \neq 0$ for every $\al \in \Delta_{M}$. Using that for $\al \in \Delta_G$, $\langle \la, \alpha \rangle = 0$ if and only if $\al \in \Delta_{L}$, we see that $w\la$ is $\mathfrak{m}$-regular if and only if $w^{-1}(\Delta_M) \cap \Delta_L = \varnothing$. This is equivalent to $\dot{w}P\dot{w}^{-1} \cap M = B_0$, i.e., $w \in  {}^M(W/W_{L})$. The intersection $S(e_{\mathfrak{m}},\mathfrak{m}) \cap \mathbb{O}_{w\la}(\mathfrak{m})$ is then equal to $x_{w\la}$ (abusing notations, we denote by the same symbol $x_{w\la}$ the corresponding element of $\mathfrak{g} \simeq \mathfrak{g}^*$). 
Our goal now is to describe the corresponding point $(\mathfrak{p}',x_{w\la})$ of $\widetilde{S}_\la(e_{\mathfrak{m}},\mathfrak{p})^{Z_M} \iso S(e_{\mathfrak{m}},\mathfrak{m}) \cap {\mathbb{O}}_{w\la}(\mathfrak{m})$.
We claim that $\mathfrak{p}':=m_{w\la} \dot{w} \cdot \mathfrak{p}$ works. 
Indeed, we already know that $(m_{w\la} \dot{w} \cdot \mathfrak{p},x_{w\la}) \in \widetilde{S}_\la(e_{\mathfrak{m}},\mathfrak{p})^{Z_M}$ and its image in $S(e_{\mathfrak{m}},\mathfrak{m}) \cap {\mathbb{O}}_{w\la}(\mathfrak{m})$ is $x_{w\la}$. The claim follows.
\end{proof}

\begin{cor}\label{dist_comp_fixed}
The embedding $\widetilde{S}(\mathfrak{q},\mathfrak{p})^{Z_M} \hookrightarrow (T^*\mathcal{P})^{Z_M}$ sends distinct points $p_1,p_2 \in \widetilde{S}(e,\mathfrak{p})^{Z_M}$ to distinct connected components of $(T^*\mathcal{P})^{Z_M}$. The connected components of the latter are affine fibrations over partial flag varieties of $L$ (see \cref{fixed pt of partial flag}).
\end{cor}

\begin{cor}\label{separating_cor_parabolic}
If $[w_1], [w_2] \in (W_M\backslash W/W_L)^{\mathrm{free}}$ are such that for every $s \in \CC[\mathfrak{h}]^{W_L}$, and $\la \in \mathfrak{z}(\mathfrak{m})$, we have $s(w_1^{-1}\la)=s(w_2^{-1}\la)$, then $[w_1] = [w_2]$.
\end{cor}
\begin{proof}
Consider the composition of pull back homomorphisms 
\begin{equation*}
H^*_{Z_M}(T^*\mathcal{P}) \rightarrow H^*_{Z_M}(\widetilde{S}(e,\mathfrak{p})) \rightarrow H^*_{Z_M}(\widetilde{S}(e,\mathfrak{p})^{Z_M})
\end{equation*}
of modules over $\CC[\mathfrak{t}_e]=H^*_{Z_M}(\on{pt})$.

It is enough to show that this homomorphism is generically surjective. The homomorphism $H^*_{Z_M}(\widetilde{S}(\mathfrak{q},\mathfrak{p})) \rightarrow H^*_{Z_M}(\widetilde{S}(\mathfrak{q},\mathfrak{p})^{Z_M})$ is an isomorphism generically  by the localization theorem. It remains to note that $H^*_{Z_M}(T^*\mathcal{P}) \rightarrow H^*_{Z_M}(\widetilde{S}(\mathfrak{q},\mathfrak{p}))$ is generically surjective by Corollary \ref{dist_comp_fixed} together with the localization theorem (\ref{BV_localization}).
\end{proof}

\begin{rmk}
We expect that Corollary \ref{separating_cor_parabolic} remains correct even if one only restricts to $s \in \mathfrak{z}(\mathfrak{l})^* \subset \CC[\mathfrak{h}]^{W_L}$. We know how to prove this in type $A$ using the realization of $\widetilde{S}(\mathfrak{q},\mathfrak{p})$ as a Nakajima quiver variety (but the argument is quite complicated and uses representation theory of Yangians, we are grateful to Leonid Rybnikov and Alexei Ilin for explaining it to us). We remark that this expected property is equivalent to the fact that $Z_M$-equivariant line bundles on $\widetilde{S}(\mathfrak{q},\mathfrak{p})$ coming as pullbacks of the $G$ equivariant line bundles on $G/Q$ {\emph{separate}} the fixed points of $\widetilde{S}(\mathfrak{q},\mathfrak{p})$. This property is very natural and appears in various places (see, for example, \cite[``Important Assumptions'' after Theorem 2.6]{quantum_K_theory}). It would be interesting and important for further applications to figure out if this holds for $\widetilde{S}(\mathfrak{q},\mathfrak{p})$ outside of type $A$ (let us emphasize that in our formulation this is basically the question about $W$ acting on $\mathfrak{h}$).
\end{rmk}









\section{Notations}

Here is a list of the notation used in the paper.

\subsection{Spaces of parameters and Lie theory}
   \begin{itemize}
       \item $G$, simple Lie group with Lie algebra $\mathfrak{g}$
       \item $T \subset G$, maximal torus;  $\mathfrak{h}=\on{Lie}T \subset \mathfrak{g}$, Cartan subalgebra
       \item $\mathfrak{l} \subset \mathfrak{p}$, standard Levi subalgebra of a parabolic $\mathfrak{p} \supset \mathfrak{b}$
       \item $\mathfrak{X}(\mathfrak{l})=(\mathfrak{l}/[\mathfrak{l},\mathfrak{l}])^*=(\mathfrak{p}/[\mathfrak{p},\mathfrak{p}])^*$ space of characters of $\mathfrak{p}$
       \item $Z_L \subset L$, the connected component of $1$ of the center of $L$
       \item $Z_e=Z_{G}(e,h,f)$, reductive part of the centralizer $Z_G(e)$
       \item $T_e \subset Z_e$, maximal torus
       \item $\mathfrak{h}_e=\on{Lie}T_e$
       \item $\mathfrak{u}_{\mathfrak{p}}$, unipotent radical of $\mathfrak{p}$
       \item $\mathfrak{p}^{\perp}=(\mathfrak{g}/\mathfrak{p})^* \subset \mathfrak{g}^*$, $[\mathfrak{p},\mathfrak{p}]^{\perp}=(\mathfrak{g}/[\mathfrak{p},\mathfrak{p}])^* \subset \mathfrak{g}^*$ 
       \item $\mathfrak{m} \subset \mathfrak{q}$, standard Levi subalgebra of a parabolic $\mathfrak{q} \supset \mathfrak{b}$ 
       \item $\mathfrak{z}(\mathfrak{m})$, center of $\mathfrak{m}$
       \item $\mathfrak{g}=\bigoplus_{i \in \mathbb{Z}}\mathfrak{g}(i)$, decomposition of $\mathfrak{g}$ into the weight spaces of $[h,\bullet]$
       \item $\ell \subset \mathfrak{g}(-1)$, Lagrangian subspace; $\mathfrak{u}_{\ell}=\ell \oplus \bigoplus_{i \leqslant -2} \mathfrak{g}(i)$; $U_\ell \subset G$, unipotent subgroup with the Lie algebra $\mathfrak{u}$
       \item $\mathfrak{n}=\bigoplus_{i \leqslant -1}\mathfrak{g}(i)$; $\mathfrak{u}=\bigoplus_{i \leqslant -2}\mathfrak{g}(i)$; $N, U \subset G$, unipotent subgroups with Lie algebras $\mathfrak{n}$, $\mathfrak{u}$; $\mathfrak{k}=\mathfrak{g}(-1)$ 
       \item $\mathfrak{h}_X=H^2(\widetilde{X}^{\mathrm{reg}},\CC)$, Namikawa space
       \item $Z_X$, reductive part of the group of Hamiltonian graded automorphisms of $X$; $T_X \subset Z_X$, maximal torus 
       \item $\Delta_{\mathfrak{p}}$ (resp. $\Delta_{\mathfrak{p}}^\vee$), the set of roots $\alpha \in \Delta$ (resp. coroots $\alpha^\vee \in \Delta^\vee$) such that $\mathfrak{g}_\alpha \subset \mathfrak{p}$ (resp. $\mathfrak{g}_{\alpha^\vee} \subset \mathfrak{p}^\vee$); $\Delta_{\mathfrak{l}}$ (resp. $\Delta_{\mathfrak{l}}^\vee$), the set of roots $\alpha \in \Delta$ (resp. coroots $\alpha^\vee \in \Delta^\vee$) such that $\mathfrak{g}_{\alpha} \subset \mathfrak{l}$ (resp. $\mathfrak{g}^\vee_{\alpha^\vee} \subset \mathfrak{l}^\vee$)
       \item $\mathfrak{X}(\mathfrak{l})^{\mathrm{reg}}=(\mathfrak{z}(\mathfrak{l})^*)^{\mathrm{reg}}$, the set of $\la \in \mathfrak{z}(\mathfrak{l})^*$ such that $Z_{\mathfrak{g}}(\la)=\mathfrak{l}$
   \end{itemize}

\subsection{Varieties}

\begin{itemize}
    \item $\mathcal{N} \subset \mathfrak{g}^*$, nilpotent cone in $\mathfrak{g}^*$; $\mathcal{N}^\vee \subset (\mathfrak{g}^\vee)^*$, nilpotent cone in $(\mathfrak{g}^\vee)^*$  
    \item $\pi_{\mathfrak{b}}\colon T^*(G/B) \rightarrow \mathcal{N}$, Springer resolution 
    \item $\mathcal{B}_e \subset T^*(G/B)$, fiber of $\pi_{\mathfrak{b}}$ over an  element $\chi \in \mathcal{N}$ corresponding to the nilpotent element $e \in \mathfrak{g}$ via the Killing form
    \item $\widetilde{\pi}_{\mathfrak{b}}\colon \widetilde{\mathfrak{g}}^* \rightarrow \mathfrak{g}^* \times_{\mathfrak{h}^*/W} \mathfrak{h}^*$, Grothendieck resolution (universal deformation of the Springer resolution)
    \item ${\mathbb{O}} \subset \mathfrak{g}^*$, nilpotent orbit; $\widetilde{\mathbb{O}}$, cover of ${\mathbb{O}}$ 
    \item $S(e)=e+\operatorname{Ker}(\operatorname{ad} f)$, Slodowy slice
    \item $S(\chi)\subset \fg^*$, the image of $S(e)$ under the identification $\mathfrak{g}^* \simeq \mathfrak{g}$
    \item $S_{\mathfrak{h}^*}(\chi)=S(\chi) \times_{\mathfrak{h}^*/W} \mathfrak{h}^*$,~$S_{\mathfrak{X}(\mathfrak{l})}(\chi)=S(\chi) \times_{\mathfrak{h}^*/W} \mathfrak{X}(\mathfrak{l})$
    \item $\widetilde{S}(\chi)=\pi_{\mathfrak{b}}^{-1}(S(\chi) \cap \mathcal{N})$, Slodowy variety
    \item $\mathcal{P}=G/P$, $\mathcal{Q}=G/Q$, parabolic flag varieties
    \item $\pi_{\mathfrak{p}}\colon T^*\mathcal{P} \rightarrow \mathcal{N}$, parabolic Springer map
    \item $\widetilde{S}(\chi,\mathfrak{p})=\pi_{\mathfrak{p}}^{-1}(S_\chi \cap \mathcal{N})$, parabolic Slodowy variety
    \item $\widetilde{S}(\fq, \fp)$, a special case of $\widetilde{S}(\chi,\mathfrak{p})$, when $\chi$ is a regular nilpotent element in $\fm$
    \item $T^*_{\mathfrak{X}(\mathfrak{l})}\mathcal{P}=G \times_{P} (\mathfrak{g}/[\mathfrak{p},\mathfrak{p}])^*$, universal deformation of $T^*\mathcal{P}$ 
    \item $S(\mathfrak{q},\mathfrak{p})=\Spec(\CC[\widetilde{S}(\fq, \fp)])$, affinization of parabolic Slodowy variety
    \item $\widetilde{S}_{\mathfrak{h}^*}(\chi)=\widetilde{\pi}^{-1}_{\mathfrak{b}}(S_{\mathfrak{h}^*}(\chi))=\widetilde{\mathfrak{g}}/\!/\!/_{\chi} U$  deformation of the Slodowy variety over $\mathfrak{h}^*$
    \item $\widetilde{S}_{\mathfrak{X}(\mathfrak{l})}(\chi,\mathfrak{p})=\widetilde{\pi}_{\mathfrak{p}}^{-1}(S_{\mathfrak{X}(\mathfrak{l})}(\chi))=T^*_{\mathfrak{X}(\mathfrak{l})}\mathcal{P}/\!/\!/_\chi U$, deformation of the parabolic Slodowy variety
    \item $\pi\colon Y \rightarrow X$, arbitrary symplectic resolution; $\pi\colon \widetilde{X} \rightarrow X$, arbitrary $\mathbb{Q}$-factorial terminalization of a symplectic singularity $X$
    \item $Y_{\mathfrak{h}_X} \rightarrow \mathfrak{h}_X$, universal deformation of $Y$, $X_{\mathfrak{h}_X}$, the corresponding deformation of $X$
    \item For a linear map $\mathfrak{a} \rightarrow \mathfrak{h}_X$, $Y_{\mathfrak{a}} \rightarrow X_{\mathfrak{a}}$, base change of $\pi_{\mathfrak{h}_X}\colon Y_{\mathfrak{h}_X} \rightarrow X_{\mathfrak{h}_X}$
        \item $X^T$, schematic fixed points of $T \curvearrowright X$
     \item $\operatorname{Sat}_{\mathbb{O}_L}^{\mathbb{O}_G}$, $\operatorname{Ind}_{\mathbb{O}_L}^{\mathbb{O}_G}$, saturation and induction of a nilpotent coadjoint orbit $\mathbb{O}_L \subset \mathfrak{l}^*$ to $\mathfrak{g}^*$    
\end{itemize}

\subsection{Explicit elements}
\begin{itemize}
\item $\rho_{\mathfrak{g}}$, half of the sum of all positive roots $\al$ of $\mathfrak{g}$;  $\rho_{\mathfrak{l}}$, half of the sum of all positive roots $\al$ of $\mathfrak{l}$
\item $(-,-)$, a nondegenerate $\mathfrak{g}$-invariant form on $\mathfrak{g}$
\item $\chi \in \mathfrak{g}^*$, functional such that $(e,-)=\chi$
\item $e_{\mathfrak{m}} \in \mathfrak{m}$, a regular nilpotent element of $\fm$ contained in $\mathfrak{b}$
\item $\mathcal{O}_{\mathcal{B}}(\la)$, line bundle on $\mathcal{B}$ corresponding to $\la\colon T \rightarrow \CC^\times$
\end{itemize}


 \subsection{Combinatorics}
 \begin{itemize}
     \item $(W_M\backslash W / W_L)^{\mathrm{free}}$, the set of free $W_M \times W_L$-orbits on $W$
    \item ${}^M(W/W_L) \subset W/W_L$, the set of $wW_L \in W/W_L$ such that $wPw^{-1} \cap M = B_0$
    \item $(W_M\backslash W)^L \subset W_M \backslash W$, the set of $W_Mw \in W_M \backslash W$ such that $w^{-1}Qw \cap L = B_L^-$
    \item $\mathcal{P}(n)$, the set of partitions of $n$; $\mathcal{P}_{B}(2n+1)$ (resp. $\mathcal{P}_{D}(2n)$), partitions of $2n+1$ (resp. $2n$) s.t. every even member comes with an even multiplicity; $\mathcal{P}_{C}(n)$, partitions of $n$ s.t. every odd member comes with an even multiplicity 
    \item ${\pv}$, partition
 \end{itemize}

\subsection{Algebras and sheaves of algebras; morphisms}
   \begin{itemize}
           \item $\cC(\mathcal{A})$, Cartan subquotient of a $\mathbb{Z}$-graded algebra $\mathcal{A}$; $\cC_\nu(\mathcal{A})$, Cartan subquotient of a $T$-graded algebra with a $\mathbb{Z}$-grading induced by a cocharacter $\nu\colon \CC^\times \rightarrow T$     
           \item $\mathcal{U}(\mathfrak{g})$, universal enveloping algebra of $\mathfrak{g}$; $\mathcal{U}_\hbar(\mathfrak{g})=\mathcal{U}_\hbar$, homogenized version over $\CC[\hbar]$
       \item $\mathcal{W}(\chi)=\mathcal{U}(\mathfrak{g})/\!/\!/_\chi U_{\ell}$, finite $W$-algebra quantizing $S(\chi)$, $\mathcal{W}_\hbar(\chi)=\mathcal{U}_\hbar(\mathfrak{g})/\!/\!/_\chi U_\ell$
       \item $Z_\hbar=Z_\hbar(\mathcal{U}_\hbar(\mathfrak{g})) \subset \mathcal{U}_\hbar(\mathfrak{g})$, the center of $\mathcal{U}_\hbar(\mathfrak{g})$; $HC_\hbar\colon Z_\hbar \iso \CC[\mathfrak{h}^*,\hbar]^{W}$, the Harish-Chandra isomorphism; $\widetilde{HC}_\hbar\colon Z_\hbar \iso \CC[\mathfrak{h}^*]^{W} \otimes \CC[\hbar]$, twisted Harish-Chandra isomorphism 
       \item  $\mathcal{U}_\la=\mathcal{U}_\la(\mathfrak{g})$ central reduction of $\mathcal{U}(\mathfrak{g})$ at $\la \in \mathfrak{h}^*$; $\mathcal{U}_{\hbar,\la}$, the homogenized version
       \item $\cW_{\la}(\chi,\mathfrak{p}) = D_{\la}(\mathcal{P})/\!/\!/_\chi U_{\ell}$; $
\cW_{\hbar,\la}(\chi,\mathfrak{p}) = D_{\hbar,\la}(\mathcal{P})/\!/\!/_\chi U_{\ell}$
        \item $I_{\la,\mathfrak{p}}$, kernel of the homomorphism $\mathcal{U}_{{\la}}(\mathfrak{g}) \rightarrow \mathcal{U}_{\la}(\mathfrak{g},\mathfrak{p})$
        \item $\mathcal{D}_{\hbar,\mathfrak{h}_X}(Y)$, universal graded polynomial quantization of $Y_{\mathfrak{h}_X}/\mathfrak{h}_X$
        \item $\mathcal{D}_{\hbar,\la}(Y)$, specialization of $\mathcal{D}_{\hbar,\mathfrak{h}_X}(Y)$ at $\la \in \mathfrak{h}_X$
         \item $\mathcal{D}_{\mathfrak{h}_X}(Y)$, specialization of $\mathcal{D}_{\hbar,\mathfrak{h}_X}(Y)$ at $\hbar=1$
        \item $\mathcal{D}_{\la}(Y)$, specialization of $\mathcal{D}_{\hbar,\la}(Y)$ at $\hbar=1$
        \item $\mathcal{A}_{\hbar,\mathfrak{h}_X}(X)=\Gamma(Y,\mathcal{D}_{\hbar,\mathfrak{h}_X}(Y))$,  graded polynomial quantization of $X_{\mathfrak{h}_X}/\mathfrak{h}_X$
        \item $\mathcal{A}_{\hbar,\la}(Y)=\Gamma(Y,\mathcal{D}_{\hbar,\la}(Y))$, graded polynomial quantization of $X$ with period $\la$
        \item $\mathcal{A}_{\mathfrak{h}_X}(Y)=\Gamma(Y,\mathcal{D}_{\mathfrak{h}_X}(Y))$, filtered quantization of $X_{\mathfrak{h}_X}/\mathfrak{h}_X$
        \item $\mathcal{A}_{\la}(Y)=\Gamma(Y,\mathcal{D}_{\la}(Y))$, filtered quantization of $X$ with period $\la$
         \item $\Phi_{\mathfrak{p}}\colon \mathcal{A}_{\hbar,\mathfrak{X}(\mathfrak{l})}(T^*\mathcal{B}) \rightarrow \mathcal{A}_{\hbar,\mathfrak{X}(\mathfrak{l})}(T^*\mathcal{P})$, natural morphism 
         \item $\Phi_{\chi,\mathfrak{p}}\colon \mathcal{A}_{\hbar,\mathfrak{X}(\mathfrak{l})}(\widetilde{S}(\chi)) \rightarrow \mathcal{A}_{\hbar,\mathfrak{X}(\mathfrak{l})}(\widetilde{S}(\chi,\mathfrak{p}))$, quantum Hamiltonian reduction of $\Phi_{\mathfrak{p}}$
   \end{itemize}

\begin{sloppypar} \printbibliography[title={References}] \end{sloppypar}

\end{document}